\documentclass[letterpaper, oneside]{amsart}
\usepackage{layout}	

\usepackage[
]{geometry}

\usepackage{amssymb}
\usepackage{amsthm}
\usepackage{amsmath}
\usepackage{mathtools}
\usepackage{hyperref}
\usepackage{colonequals}
\usepackage{mathrsfs}
\usepackage{color}
\usepackage{subcaption}
\usepackage{adjustbox}
\usepackage{ytableau}
\usepackage{bbm}
\usepackage{tensor}
\usepackage{float}
\usepackage{cancel}
\usepackage[normalem]{ulem}

\usepackage[inline]{enumitem}
\setlist{nosep}
\setlist[enumerate,1]{label = \textnormal{(\arabic*)}}

\usepackage[ruled,vlined]{algorithm2e}
\SetKwProg{Fn}{Function}{}{}
\AlgoDisplayBlockMarkers\AlgoDisplayGroupMarkers\SetAlgoBlockMarkers{ \{}{ \}\ }%
  \SetAlgoNoEnd\SetAlgoNoLine

\usepackage{tikz}
\usetikzlibrary{math, calc}
\usetikzlibrary{cd}
\usetikzlibrary{decorations.markings}
\usetikzlibrary{arrows.meta}
\usetikzlibrary{patterns}
\usetikzlibrary{positioning}

\ifdim\overfullrule>0pt 
  \usepackage{environ}

  \NewEnviron{tikzpicture}{%
    \begin{pgfpicture}
    \pgfpathrectanglecorners{\pgfpointorigin}{\pgfpoint{3cm}{3cm}}%
     \pgfusepath{stroke}\end{pgfpicture}%
  }
\fi

\tikzset{
	ddarr/.style={double distance=1.5pt, arrows={-Straight Barb[scale=0.6]}},
	circ/.style={dashed,circle,draw=black,minimum size=15pt,inner sep=0},
	equ/.style={double distance=1.5pt, inner sep=0pt, outer sep=10pt},
}

\newcommand{\StairGrid}[1]{
\draw[color=black] (0,0) grid (#1,#1);
\foreach \XX in {1,...,#1}
{
    \draw[very thick] (#1-\XX,\XX) |- ++ (+1,-1);
};
}
\newcommand{\Intv}[3][{}]{
    	\draw[Bracket-Bracket, line width=1.5pt] (#2, #3) -- (#2+1,#3);
    	\node[fill=white] at (#2+0.5, #3) {#1};
}

\newcommand{\Boxfill}[2]{%
	\fill [gray!40!white] ( #1,#2 ) rectangle (#1+1,#2+1) ;
}
\newcommand{\Recfill}[4]{%
	\fill [gray!40!white] ( #1,#2 ) rectangle (#3,#4) ;
}

\theoremstyle{plain}
\newtheorem{thm}{Theorem}[section]
\newtheorem{lem}[thm]{Lemma}
\newtheorem{prop}[thm]{Proposition}
\newtheorem{cor}[thm]{Corollary}

\theoremstyle{definition}
\newtheorem{defn}[thm]{Definition}

\newtheorem{example}[thm]{Example}
\newtheorem{conj}[thm]{Conjecture}

\newcommand{\continuation}{??}

\theoremstyle{remark}
\newtheorem*{rmk}{Remark}

\numberwithin{equation}{section}

\newcommand{\bN}{\mathbb{N}}

\newcommand{\bR}{\mathbb{R}}

\newcommand{\bZ}{\mathbb{Z}}

\newcommand{\cI}{\mathcal{I}}
\newcommand{\cJ}{\mathcal{J}}

\newcommand{\cL}{\mathcal{L}}

\newcommand{\cP}{\mathcal{P}}

\newcommand{\fA}{\mathfrak{A}}

\newcommand{\fC}{\mathfrak{C}}

\newcommand{\fb}{\mathfrak{b}}

\newcommand{\fd}{\mathfrak{d}}

\makeatletter
\newcommand{\pushright}[1]{\ifmeasuring@#1\else\omit\hfill$\displaystyle#1$\fi\ignorespaces}
\makeatother

\newcommand{\alg}{\mathbf{Alg}}

\newcommand{\ccP}{{\widehat{\mathcal{P}}}}

\newcommand{\vup}{\raisebox{1.3ex}{$\vdots$}}

\newcommand{\und}[1]{\underline{#1}}

\newcommand{\hh}[1]{\widehat{#1}}

\newcommand{\pkm}{\stackrel{\cP}{\leftrightsquigarrow}}
\newcommand{\psim}{\sim_\cP}
\newcommand{\dr}{\dashrightarrow}
\newcommand{\dl}{\dashleftarrow}

\newcommand{\stair}{\textnormal{Stair}}
\newcommand{\vn}{\infty}
\newcommand{\cond}{Condition $(\pitchfork)$}

\newcommand{\csi}[1][]{\textnormal{Case I#1}}
\newcommand{\csii}[1][]{\textnormal{Case II#1}}
\newcommand{\csiii}[1][]{\textnormal{Case $\infty$#1}}

\newcommand{\pdr}{\dr_\cP\nobreak}
\newcommand{\pdl}{\dl_\cP\nobreak}
\newcommand{\pr}{\rightarrow_\cP\nobreak}
\newcommand{\pl}{\leftarrow_\cP\nobreak}
\newcommand{\pleq}{\prec_\cP\nobreak}
\newcommand{\pgeq}{\succ_\cP\nobreak}
\newcommand{\pp}{^\bullet}
\newcommand{\ppp}{^\circ}
\newcommand{\dg}{^\dagger}

\newcommand{\cpr}{\rightarrow_{\ccP}\nobreak}
\newcommand{\cpl}{\leftarrow_{\ccP}\nobreak}

\newcommand{\loe}{=\joinrel=\joinrel=}
\newcommand{\eee}[2]{\node at ($(#1)!0.5!(#2)$) {$\loe$};}

\DeclareMathOperator{\des}{\textnormal{des}}

\DeclareMathOperator{\ginv}{\textnormal{g-inv}}
\DeclareMathOperator{\ght}{\textnormal{g-ht}}
\DeclareMathOperator{\finv}{\textnormal{f-inv}}
\DeclareMathOperator{\dash}{\mbox{-\,-\,-}}
\DeclareMathOperator{\sym}{\mathfrak{S}}
\DeclareMathOperator{\SYT}{\textnormal{SYT}}
\DeclareMathOperator{\ptab}{\cP\textnormal{-Tab}}
\DeclareMathOperator{\rw}{\textnormal{read}}
\DeclareMathOperator{\prs}{\cP\textnormal{-RS}}

\DeclareMathOperator{\im}{\textup{im}}

\title{Robinson-Schensted correspondence for unit interval orders}
\author{Dongkwan Kim}
\address{School of Mathematics\\
  University of Minnesota Twin Cities\\
  Minneapolis, MN 55455\\
  U.S.A.}
\email{kim00657@umn.edu}
\author{Pavlo Pylyavskyy}
\address{School of Mathematics\\
  University of Minnesota Twin Cities\\
  Minneapolis, MN 55455\\
  U.S.A.}
\email{ppylyavs@umn.edu}
\date{\today}							

\begin{document}
\begin{abstract}  
The Stanley-Stembridge conjecture associates a symmetric function to each natural unit interval order $\mathcal{P}$. In this paper, we define relations \`a la Knuth on the symmetric group for each $\mathcal{P}$ and conjecture that the associated $\mathcal{P}$-Knuth equivalence classes are Schur-positive, refining theorems of Gasharov, Brosnan-Chow, Guay-Paquet, and Shareshian-Wachs. The resulting equivalence graphs fit into the framework of D graphs studied by Assaf. Furthermore, we conjecture that the Schur expansion is given by column-readings of $\mathcal{P}$-tableaux that occur in the equivalence class. We prove these conjectures for $\mathcal{P}$ avoiding two specific suborders by introducing $\mathcal{P}$-analog of Robinson-Schensted insertion, giving an answer to a long standing question of Chow. 
\end{abstract}

\setcounter{tocdepth}{1}
\maketitle

\renewcommand\contentsname{}
\tableofcontents

\setlength{\parindent}{15pt} 	
\setlength{\parskip}{5pt} 	

\section{Introduction}
Since its formulation in 1993, Stanley-Stembridge conjecture \cite[Conjecture 5.5]{stst93} has been one of the most intriguing problems in algebraic combinatorics. Interest in it was greatly strengthened when Shareshian and Wachs \cite{sw16} related the conjecture to Hessenberg varieties. The original conjecture was shown by Guay-Paquet \cite{g13} to be equivalent to saying that chromatic symmetric functions of incomparability graphs of unit interval orders are positive combinations of elementary symmetric functions. Shareshian and Wachs realized that essentially the same symmetric functions arise as Frobenius characters of actions of symmetric groups on cohomology rings of Hessenberg varieties, as studied by Tymoczko \cite{ty08}. Shareshian-Wachs conjecture was proved by Brosnan and Chow \cite{bc18}, and independently by Guay-Paquet \cite{g16}. On the combinatorial level the results of Brosnan-Chow and Guay-Paquet imply a graded refinement of the Schur positivity result of Gasharov \cite{gas96}. It also provides useful tools to understand combinatorics in terms of geometry, i.e. theory of perverse sheaves and geometric properties of (regular) Hessenberg varieties.

The original Stanley-Stembridge conjecture, nowadays usually stated in terms of positivity in complete homogenous symmetric functions, remains open except for special cases, see Gebhard-Sagan \cite{gs01}, Dahlberg-van Willigenburg \cite{dw18}, Harada-Precup \cite{hp19}, Cho-Huh \cite{ch19_2}, Cho-Hong \cite{ch19}, etc.

In an independent development, Assaf \cite{ass15, ass17} has introduced a beautiful theory of D graphs to address Schur positivity questions in symmetric functions, such as Macdonald polynomials, LLT polynomials, and $k$-Schur functions. While as shown by Blasiak \cite{bl16} getting exactly the right axiomatization to address those questions can be very challenging, Assaf's work provides a very useful framework. In particular her characterization of dual equivalence graphs has been used in a variety of contexts, see for example Chmutov \cite{chm15} and Roberts \cite{rob14}. Assaf's ideas were further developed by Blasiak-Fomin \cite{bf17} and others. 

In this paper we combine the two lines of research. Specifically, for each unit interval order $\cP$ we define an analog of Knuth moves. The resulting $\cP$-Knuth equivalence classes of permutations satisfy correct axioms to fit into the framework of D graphs. We conjecture that via the standard map from permutations to quasisymmetric functions the images of $\cP$-Knuth equivalence classes are symmetric and Schur positive. This is a refinement of results of Gasharov, Brosnan-Chow, and Guay-Paquet. Furthermore, we conjecture that the decomposition into Schur functions can be read off from column reading words of $\cP$-tableaux that occur in the equivalence class. 

We prove this Schur positivity conjecture for a special class of unit interval orders $\cP$. For that purpose we introduce an analog of Robinson-Schensted insertion that preserves descents, solving an open problem dating back to the works of Sundquist-Wagner-West and Chow. The 1997 work of Sundquist-Wagner-West \cite{sww97} constructs a version of Robinson-Schensted insertion for unit interval orders, however in general their algorithm does not preserve descents, and thus cannot be used to derive Schur positivity results. Chow \cite{ch99} proved that Sundquist-Wagner-West does preserve descents under a very restrictive condition --- in our terminology his condition is to avoid a suborder isomorphic to  $\cP_{(2,1),4}$. Chow implicitly states in his paper the question of constructing Robinson-Schensted correspondence that preserves descents when $\cP$ avoids a  less restrictive pattern  $\cP_{(3,1,1),5}$. In this paper we solve this problem for unit interval orders that avoid both  $\cP_{(3,1,1),5}$ and  $\cP_{(4,2,1,1),6}$. As a result, in those cases we are able to prove Schur positivity of the  $\cP$-Knuth equivalence classes.

This project started as an attempt to prove Stanley-Stembridge conjecture. This goal remains elusive, as it would require introducing an affine analog of  $\cP$-Knuth equivalence classes and proving their $h$-positivity. We expect this to be strictly harder than proving Schur positivity of the $\cP$-Knuth equivalence graphs introduced in this paper, and even that remains open in full generality. Nevertheless, $\cP$-Knuth equivalence classes seem to be interesting objects of their own, perhaps having geometric meaning in terms of (equivariant) cohomology and moment graphs of Hessenberg varieties. We hope that understanding $\cP$-Knuth equivalence classes, and in particular proving Conjecture \ref{conj:main}, will shed new light on Stanley-Stembridge conjecture. 

The paper proceeds as follows. In Section \ref{sec:def} we recall some standard combinatorial notions such as partitions, partial orders and tableaux. In Section \ref{sec:unit} we recall properties and characterizations of natural unit interval orders. We also introduce an important class of natural unit interval orders called ladders, as well as  ladder-climbing property. In Section \ref{sec:Knuth} we introduce $\cP$-Knuth equivalence classes and state the main theorem \ref{thm:mainstrong}. In Section \ref{sec:ins} we introduce column insertion procedure, which is then used in Section \ref{sec:RS} to define the full $\cP$-Robinson-Schensted insertion algorithm. Section \ref{sec:examples} is filled with examples illustrating everything introduced in the previous sections. In Sections \ref{sec:proofprop} and \ref{sec:proofthm} we give proofs of the results from previous sections.

\section{Definitions and notations} \label{sec:def}

For $a, b \in \bZ$, we set $[a,b]\colonequals \{x \in \bZ\mid a\leq x \leq b\}$. For a set $X$, we let $|X|$ be its cardinal.

\subsection{Partitions}
A partition is a finite sequence of integers $\lambda=(\lambda_1, \lambda_2, \ldots, \lambda_a)$ such that $\lambda_1\geq \lambda_2 \geq \cdots \geq \lambda_a >0$. In such a case, we set its length to be $a$ (denoted $l(\lambda)$) and its size to be $\lambda_1+\lambda_2+\cdots+\lambda_a$ (denoted $|\lambda|$). When $|\lambda|=n$, we also write $\lambda \vdash n$. If $i > l(\lambda)$, we set $\lambda_i=0$. We write $\lambda'$ to denote the conjugate partition of $\lambda$. We define the staircase partition $\stair(n)$ to be $(n-1, n-2, \ldots, 2, 1)$. For two partitions $\lambda$ and $\mu$, we write $\lambda \subset \mu$ if $\lambda_i\leq \mu_i$ for all $i \in \bZ_{>0}$. Pictorially, it means that the Young diagram of $\mu$ contains that of $\lambda$.

\subsection{Partial orders} We use the symbols $\geq, >, \leq,$ and $<$ for the usual order on $\bR$. However, throughout this paper we discuss various partial orders, for which new symbols are necessary in order to avoid conflict. Namely, suppose that a partial order $\cP$ on $[1,n]$ is given. For $a,b\in [1,n]$, we write

\begin{enumerate}
\item $a \pleq b$ (or $b \pgeq a$) if $a$ is smaller than $b$ with respect to $\cP$,
\item $a \pl b$ (or $b \pr a$) if $a<b$ and $a \pleq b$,
\item $a \pdl b$ (or $b \pdr a$) if $a<b$ but $a \not\pleq b$, and
\item $a \dash_\cP b$ if $a$ and $b$ are not comparable with respect to $\cP$ (and $a\neq b$).
\end{enumerate}
If there is no confusion we drop the subscript ${}_\cP$ from each symbol.

For a partial order $\cP$ on a set $X$ and its subset $Y \subset X$, the restriction of $\cP$ to $Y$, denoted $\cP|_Y$, is well-defined. For two partial orders $\cP$ on $[1,n]$ and $\cP'$ on $[1,m]$, we say that $\cP$ avoids $\cP'$ or $\cP$ is $\cP'$-avoiding if restriction of $\cP$ to any subset of $[1,n]$ (of cardinal $m$) is not isomorphic to $\cP'$.

\subsection{Symmetric groups and words} In this paper, a word means a finite sequence. For a word $\alpha=(\alpha_1, \alpha_2, \ldots, \alpha_k)$, we also write $\alpha=\alpha_1\alpha_2\cdots\alpha_k$ to simplify notations. For a word $\alpha$, we denote by $\und{\alpha}$ the corresponding underlying set. We let $|\alpha|\in \bN$ be the length of $\alpha$. If $|\alpha|=0$, then we also write $\alpha=\emptyset$. By a subword of $\alpha$, we mean a word $(\alpha_{i_1}, \alpha_{i_2}, \ldots, \alpha_{i_s})$ such that $1\leq i_1<i_2<\cdots<i_s\leq k$. For two words $\alpha$ and $\beta$, we define $\alpha+\beta$ to be their concatenation.

Let $\sym_n$ be the symmetric group permuting $[1,n]$. We identify elements in $\sym_n$ with the words in which each of $1, 2, \ldots, n$ appears once. For $w=w_1w_2\cdots w_n \in \sym_n$ and a partial order $\cP$ defined on $[1,n]$, we set 
\begin{enumerate}[label=\textbullet]
\item its $\cP$-descent to be $\des_\cP(w) = \{i \in [1,n-1] \mid w_i\pgeq w_{i+1}\},$
\item its genuine $\cP$-inversion to be 
\begin{align*}
\ginv_\cP(w) =& \{(i,j) \in [1,n]^2 \mid \ i \pr j, \textnormal{ $i$ appears before $j$ in $w$},
\\&\textnormal{ and there do not exist any subword } ia_1a_2\cdots a_kj \textnormal{ of $w$ such that }
\\&\ i \dash_\cP a_1 \dash_\cP a_2 \dash_\cP \cdots \dash_\cP a_k \dash_\cP j\},
\end{align*}
\item its (genuine) $\cP$-height to be (if $\ginv_\cP(w) =\emptyset$ then $\ght_\cP(w)=1$ and otherwise)
\begin{align*}
\ght_\cP(w) = \max \{k \in \bN \mid& \textnormal{ there exist } a_1, a_2, \ldots, a_{k-1}, a_k \textnormal{ such that }
\\&\ (a_1, a_2), \ldots, (a_{k-1},a_k) \in \ginv_\cP(w)\},
\end{align*}
\item its fake $\cP$-inversion to be $\finv_\cP(w)=\{(i,j) \in [1,n]^2\mid i \pdr j, w^{-1}(i)<w^{-1}(j)\}$, and
\item its (fake) $\cP$-inversion number to be $|\finv_\cP(w)|$.
\end{enumerate}
For example, if $\cP$ is the usual order on $[1,n]$ then $\ginv_\cP(w)$ is the set of usual inversions in $w$, $\ght(w)$ is the length of the longest decreasing subword of $w$, and $\finv_\cP(w) = \emptyset$. On the other hand, if $\cP$ is the trivial order on $[1,n]$ then $\ginv_\cP(w)=\emptyset$, $\ght_\cP(w) = 1$, and $\finv_\cP(w)$ is the set of usual inversions in $w$. In addition, if $\cP$ is defined with respect to the Hasse diagram in Figure \ref{fig:defpar} and $w=(9,5,1,8,4,7,3,6,2)$ then 
\begin{enumerate}[label=\textbullet]
\item $\des_\cP(w)=\{1, 2, 4, 6, 8\}$, 
\item $\ginv_\cP(w)=\{ (9,5), (9,4), (9,3), (9,2), (9,1),  (8,4), (8,3), (8,2),  (7, 3), (7,2), (6, 2), (5,1)\}$,
\item $\ght_\cP(w)=3 $ ($=$ the length of (9,5,1)),
\item $\finv_\cP(w)=\{(3,2), (4,3), (5,3), (5,4), (7,6), (8,6), (8,7), (9,8)\}$, and
\item $|\finv_\cP(w)|=8$.
\end{enumerate}
\begin{rmk} We adopt the convention that if $i\to j$ in the Hasse diagram of an order $\cP$ then $i$ is greater than $j$ with respect to $\cP$, which is in accordance with the notation $i \pr j$.
\end{rmk}

\begin{figure}[!htbp]
\begin{tikzpicture}[yscale=0.7]
\node (node_8) at (36.0bp,153.5bp) [draw,draw=none] {$9$};
  \node (node_7) at (66.0bp,104.5bp) [draw,draw=none] {$8$};
  \node (node_6) at (36.0bp,104.5bp) [draw,draw=none] {$7$};
  \node (node_5) at (6.0bp,104.5bp) [draw,draw=none] {$6$};
  \node (node_4) at (96.0bp,55.5bp) [draw,draw=none] {$5$};
  \node (node_3) at (66.0bp,55.5bp) [draw,draw=none] {$4$};
  \node (node_2) at (36.0bp,55.5bp) [draw,draw=none] {$3$};
  \node (node_1) at (51.0bp,6.5bp) [draw,draw=none] {$2$};
  \node (node_0) at (81.0bp,6.5bp) [draw,draw=none] {$1$};
  \draw [black,<-] (node_0) ..controls (77.092bp,19.746bp) and (73.649bp,30.534bp)  .. (node_3);
  \draw [black,<-] (node_1) ..controls (63.27bp,20.316bp) and (74.902bp,32.464bp)  .. (node_4);
  \draw [black,<-] (node_0) ..controls (68.73bp,20.316bp) and (57.098bp,32.464bp)  .. (node_2);
  \draw [black,<-] (node_1) ..controls (54.908bp,19.746bp) and (58.351bp,30.534bp)  .. (node_3);
  \draw [black,<-] (node_4) ..controls (88.047bp,68.96bp) and (80.84bp,80.25bp)  .. (node_7);
  \draw [black,<-] (node_3) ..controls (58.047bp,68.96bp) and (50.84bp,80.25bp)  .. (node_6);
  \draw [black,<-] (node_5) ..controls (13.953bp,117.96bp) and (21.16bp,129.25bp)  .. (node_8);
  \draw [black,<-] (node_4) ..controls (94.265bp,72.769bp) and (90.702bp,95.077bp)  .. (81.0bp,111.0bp) .. controls (73.21bp,123.79bp) and (60.367bp,135.07bp)  .. (node_8);
  \draw [black,<-] (node_0) ..controls (84.908bp,19.746bp) and (88.351bp,30.534bp)  .. (node_4);
  \draw [black,<-] (node_1) ..controls (39.145bp,20.893bp) and (27.69bp,35.052bp)  .. (21.0bp,49.0bp) .. controls (15.015bp,61.478bp) and (11.072bp,76.746bp)  .. (node_5);
  \draw [black,<-] (node_2) ..controls (43.953bp,68.96bp) and (51.16bp,80.25bp)  .. (node_7);
  \draw [black,<-] (node_2) ..controls (36.0bp,68.603bp) and (36.0bp,79.062bp)  .. (node_6);
  \draw [black,<-] (node_6) ..controls (36.0bp,117.6bp) and (36.0bp,128.06bp)  .. (node_8);
  \draw [black,<-] (node_2) ..controls (28.047bp,68.96bp) and (20.84bp,80.25bp)  .. (node_5);
  \draw [black,<-] (node_3) ..controls (66.0bp,68.603bp) and (66.0bp,79.062bp)  .. (node_7);
\end{tikzpicture}
\caption{}\label{fig:defpar}
\end{figure}

\subsection{Standard and $\cP$-tableaux} A tableau $T$ is said to satisfy the $\cP$-tableau condition if for two entries $i,j \in [1,n]$ adjacent in $T$, (1) if $i$ is above $j$ then $i\pleq j$, and (2) if $i$ is left to $j$ then $i \not\pgeq j$.
That is, entries in $T$ are increasing along columns and nondecreasing along rows with respect to $\cP$. 
Such a tableau $T$ is called a $\cP$-tableau if in addition it contains each entry of $[1,n]$ exactly once.
(Note that our definition is a conjugated version of the one in \cite[Theorem 3]{gas96}.) We denote by $\ptab_n$ (resp. $\ptab_\lambda$) the set of $\cP$-tableaux of size $n$ (resp. of shape $\lambda$). Similarly, we denote by $\SYT_n$ (resp. $\SYT_\lambda$) the set of standard Young tableaux of size $n$ (resp. of shape $\lambda$). Note that if $\cP$ is the usual order on $[1,n]$ then $\cP$-tableaux are exactly standard Young tableaux.

\ytableausetup{smalltableaux}
For a tableau $T$, we often identify each of its columns with its reading word from bottom to top and also identify $T$ with the sequence of its columns. In addition, we define the reading word of $T$, denoted $\rw(T)$, to be the concatenation of column reading words from bottom to top. We define the descent of $T\in \SYT_n$ to be $\des(T) =\{i \in [1,n-1]\mid i \textup{ is in a higher row than that of } i+1\}$. For example, if $T=\ytableaushort{1358,267,4}$ then $T$ is identified with $((4,2,1), (6,3), (7,5), (8))$, the reading word of $T$ is $\rw(T)=(4,2,1,6,3,7,5,8)$, and $\des(T)=\{1, 3, 5\}$.

For a partial order $\cP$ on $[1,n]$ and $T\in \ptab_n$, we define its fake $\cP$-inversion to be $\finv_\cP(T)=\{(i,j) \in [1,n]^2\mid i \pdr j, \textnormal{ the column of } i \textnormal{ is left to that of } j\}$ and its fake $\cP$-inversion number to be $|\finv_\cP(T)|$. Since each column of a $\cP$-tableau is a chain in $\cP$, it is easy to show that $\finv_\cP(T)=\finv_\cP(\rw(T))$ for any $\cP$-tableau $T$. For example, if $\cP$ is given again by Figure \ref{fig:defpar} and $T=\ytableaushort{1432,5876,9}$  then $\finv_\cP(T)=\finv_{\cP}(\rw(T))=\{(3,2), (4,3), (5,3), (5,4), (7,6), (8,6), (8,7), (9,8)\}$, and thus $|\finv_\cP(T)|=8$.

\subsection{Schur and fundamental quasi-symmetric functions} For a partition $\lambda$, we set $s_\lambda$ to be the Schur function corresponding to $\lambda$. For a composition $\mu$, we set $F_\mu$ to be the fundamental quasi-symmetric function corresponding to $\mu$ defined by Gessel.

\section{Natural unit interval order} \label{sec:unit}
In this section we recall the notion of natural unit interval orders and some of its properties. Also, we introduce a ladder order which plays a prominent role in this paper.
\subsection{Three equivalent definitions of natural unit interval orders} Here we define natural unit interval orders in three different ways. We refer readers to \cite[Section 4]{sw16} for the proof that these definitions are indeed equivalent.
\subsubsection{Definition in terms of unit intervals} 
\begin{defn} We say that $\cP$ on $[1,n]$ is a natural unit interval order if it there exist $n$ real numbers $y_1<y_2< \cdots<y_n$ such that $i\pleq j \Leftrightarrow y_i+1<y_j$.
\end{defn}
Pictorially, one may regard $y_1, \ldots, y_n$ as the starting points of unit intervals $I_1=[y_1, y_1+1], I_2=[y_2, y_2+1], \ldots, I_n=[y_n, y_n+1]$. Then the above definition translates to the following.
\begin{enumerate}
\item If $i<j$, then $I_i$ should start before $I_j$ in the real line.
\item We have $i\pleq j$ if and only if $I_j$ starts after $I_i$ ends in the real line.
\end{enumerate}
For example, Figure \ref{fig:nuio1} shows an arrangement of unit intervals and the corresponding natural unit interval order.

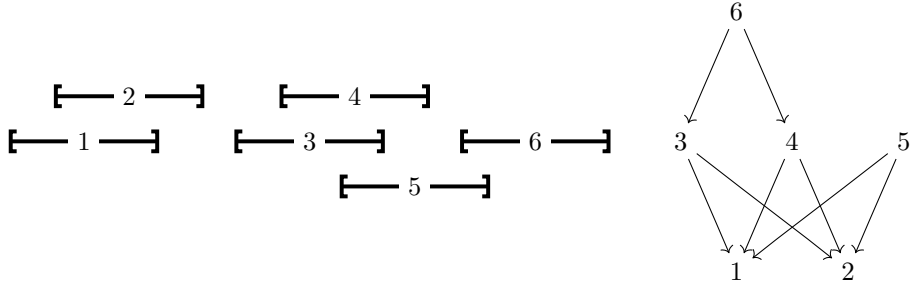
\begin{figure}[!htbp]
\begin{tikzpicture}[scale=2,baseline=(current bounding box.center)]
\Intv[1]{1}{2};
\Intv[2]{1.3}{2.3};
\Intv[3]{2.5}{2};
\Intv[4]{2.8}{2.3};
\Intv[5]{3.2}{1.7};
\Intv[6]{4}{2};
\end{tikzpicture}\qquad
\begin{tikzpicture}[baseline=(current bounding box.center), xscale=1.4]
\node (node_5) at (21.0bp,104.5bp) [draw,draw=none] {$6$};
  \node (node_4) at (66.0bp,55.5bp) [draw,draw=none] {$5$};
  \node (node_3) at (6.0bp,55.5bp) [draw,draw=none] {$3$};
  \node (node_2) at (36.0bp,55.5bp) [draw,draw=none] {$4$};
  \node (node_1) at (51.0bp,6.5bp) [draw,draw=none] {$2$};
  \node (node_0) at (21.0bp,6.5bp) [draw,draw=none] {$1$};
  \draw [black,<-] (node_0) ..controls (17.092bp,19.746bp) and (13.649bp,30.534bp)  .. (node_3);
  \draw [black,<-] (node_1) ..controls (54.908bp,19.746bp) and (58.351bp,30.534bp)  .. (node_4);
  \draw [black,<-] (node_0) ..controls (24.908bp,19.746bp) and (28.351bp,30.534bp)  .. (node_2);
  \draw [black,<-] (node_1) ..controls (38.73bp,20.316bp) and (27.098bp,32.464bp)  .. (node_3);
  \draw [black,<-] (node_1) ..controls (47.092bp,19.746bp) and (43.649bp,30.534bp)  .. (node_2);
  \draw [black,<-] (node_0) ..controls (33.27bp,20.316bp) and (44.902bp,32.464bp)  .. (node_4);
  \draw [black,<-] (node_3) ..controls (9.9083bp,68.746bp) and (13.351bp,79.534bp)  .. (node_5);
  \draw [black,<-] (node_2) ..controls (32.092bp,68.746bp) and (28.649bp,79.534bp)  .. (node_5);
\end{tikzpicture}
\caption{Natural unit interval order in terms of unit intervals}\label{fig:nuio1}
\end{figure}

\subsubsection{Definition using partitions} For $n \in \bZ_{>0}$, let $\lambda$ be a partition contained in $\stair(n)$. We define the partial order $\cP_{\lambda,n}$ on $[1,n]$ such that $a\pleq b$ if and only if $a\leq \lambda_{n+1-b}$. In other words, we have $[1, \lambda_i]=\{x \in [1,n] \mid x \pleq n+1-i\}$. For example, when $n=5$ and $\lambda=(3,1)$ then $\cP_{\lambda,n}$ is given as in Figure \ref{fig:nuio2}.
\begin{figure}[!htbp]
\begin{tikzpicture}[scale=0.5,baseline=(current bounding box.center)]
 \StairGrid{5};
 \Boxfill{0}{0};
 \Boxfill{1}{0};
 \Boxfill{2}{0};
 \Boxfill{0}{1};
 \node at (-0.5,4.5) {$1$};
 \node at (-0.5,3.5) {2};
 \node at (-0.5,2.5) {3};
 \node at (-0.5,1.5) {4};
 \node at (-0.5,.5) {5}; 
 \node at (0.5,-.5) {1};
 \node at (1.5,-.5) {2};
 \node at (2.5,-.5) {3};
 \node at (3.5,-.5) {4};
 \node at (4.5,-.5) {5};
\end{tikzpicture} \qquad
\begin{tikzpicture}[baseline=(current bounding box.center)]
\node (node_4) at (51.0bp,55.5bp) [draw,draw=none] {$5$};
  \node (node_3) at (21.0bp,55.5bp) [draw,draw=none] {$4$};
  \node (node_2) at (66.0bp,6.5bp) [draw,draw=none] {$3$};
  \node (node_1) at (36.0bp,6.5bp) [draw,draw=none] {$2$};
  \node (node_0) at (6.0bp,6.5bp) [draw,draw=none] {$1$};
  \draw [black,<-] (node_0) -- (node_3);
  \draw [black,<-] (node_1) -- (node_4);
  \draw [black,<-] (node_0) -- (node_4);
  \draw [black,<-] (node_2) -- (node_4);
\end{tikzpicture}
\caption{Natural unit interval order in terms of partitions}\label{fig:nuio2}
\end{figure}
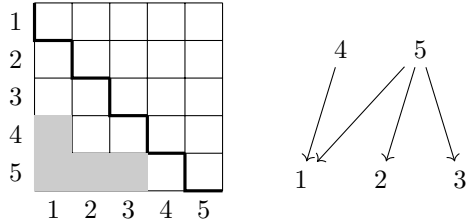

\begin{defn}
We say that a partial order $\cP$ on $[1,n]$ is a natural unit interval order if $\cP=\cP_{\lambda, n}$ for some $\lambda \subset \stair(n)$.
\end{defn}

\subsubsection{Defining properties} One may also define natural unit interval orders by imposing certain conditions on a partial order, namely:
\begin{defn} We say that a partial order $\cP$ on $[1,n]$ is a natural unit interval order if
\begin{enumerate}[label=\textbullet]
\item the usual order is a linearization of $\cP$, i.e. if $a \pleq b$ then $a<b$, and
\item if $b \pl c$, $a \dash_\cP b$, and $a\dash_\cP c$, then $b<a<c$. 
\end{enumerate}
\end{defn}
Because of the first condition, $\pl, \pr$ are equivalent to $\pleq, \pgeq$ when we consider natural unit interval orders. From now on we usually use the former rather than the latter for such orders. Also, hereafter we refer to the second condition as \cond. For example, if $\cP$ is a natural unit interval order and $a, b, c \in [1,n]$ satisfy $b\pl c$ and $a<b$, then $a \pl c$ by \cond{}.

\subsection{$(3+1)$ and $(2+2)$ avoidance}
There is another characterization of natural unit interval orders in terms of suborder avoidance. Indeed, it is essentially proved in \cite{ss58} that a partial order $\cP$ on $[1,n]$ is a natural unit interval order if and only if the usual order is a linearization of $\cP$ and $\cP$ avoids suborders ``$(3+1)$'' (disjoint union of a chain of length 3 and an element) and ``$(2+2)$'' (disjoint union of two chains of length 2). Here we prove only one direction which will be useful later on.
\begin{lem}[See Figure \ref{fig:3122}] \label{lem:3122} Suppose that $\cP$ is a natural unit interval order on $[1,n]$. Then,
\begin{enumerate}
\item there does not exist $a,b,c,d \in [1,n]$ such that $a\pr b\pr c$, $a\dash_\cP d$, and $d\dash_\cP c$, and
\item there does not exist $a,b,c,d \in [1,n]$ such that $a\pr b$, $c\pr d$, $c \dash_\cP b$, and $a \dash_\cP d$.
\end{enumerate}
\end{lem}
\begin{proof} (1) If $b>d$ (resp. $b=d$) then $a\pr d$ by \cond{} applied to $\{a, b, d\}$ (resp. since $a \pr b$) which is absurd. Similarly, if $b<d$ then $d \pr c$ by \cond{} which is again impossible. (2) If $c>a$ (resp. $c=a$) then $c \pr b$ by \cond{} applied to $\{a, b, c\}$ (resp. since $a \pr b$) which is absurd, and thus $c<a$. However, by symmetry we should have $a<c$ as well, which is again impossible.
\end{proof}

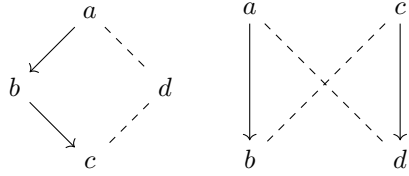
\begin{figure}[!htbp]
\begin{tikzpicture}[yscale=-1]
	\node (a) at (0,0) {$a$};
	\node (b) at (-1,1) {$b$};
	\node (d) at (1,1) {$d$};
	\node (c) at (0,2) {$c$};

	\draw[->] (a) to (b);
	\draw[->] (b) to (c);
	\draw[dashed] (a) to (d);
	\draw[dashed] (d) to (c);
\end{tikzpicture}\qquad
\begin{tikzpicture}[yscale=-1]
	\node (a) at (0,0) {$a$};
	\node (b) at (0,2) {$b$};
	\node (d) at (2,2) {$d$};
	\node (c) at (2,0) {$c$};

	\draw[->] (a) to (b);
	\draw[->] (c) to (d);
	\draw[dashed] (c) to (b);
	\draw[dashed] (a) to (d);
\end{tikzpicture}
\caption{(3+1)- and (2+2)-posets}
\label{fig:3122}
\end{figure}

\subsection{Properties of $\cP$-tableaux} 
We discuss some properties of $\cP$-tableaux for natural unit interval orders. These will be frequently used in the later part of this paper without reminder.

\begin{lem} Suppose that $(a_p, \ldots, a_1)$ and $(b_q, \ldots b_1)$ are two adjacent columns in a $\cP$-tableau such that the former is on the left of the latter. 
\begin{enumerate}
\item If $a_i>b_j$ then $i\geq j$
\item If $a_i \pr b_j$ then $i>j$.
\item If $i<j$ then $a_i <b_j$.
\end{enumerate}
\end{lem}
\begin{proof} For (1), suppose $a_i>b_j$ and $i<j$. Then $b_i \pl b_j$ thus $a_i \pr b_i$  by \cond{}, which contradicts the $\cP$-tableau condition. This proves (1). For (2), suppose $a_i \pr b_j$ and $i\leq j$. Then $b_i \leq b_j$ thus $a_i \pr b_i$ by \cond{}, which contradicts the $\cP$-tableau condition. This proves (2). For (3), suppose $i<j$ and $a_i\geq b_j$. Then $b_i \pl b_j$ thus $a_i \pr b_i$ by \cond{}, which contradicts the $\cP$-tableau condition. This proves (3).
\end{proof}

\subsection{Ladders} We define a special kind of a natural unit interval order called a ladder order.
\begin{defn} A partial order $\cP$ is called a \emph{ladder order} if it is isomorphic to $\cP_{\stair(m-1),m}$ for some $m \in \bN$.
\end{defn}

\begin{figure}[!htbp]
\begin{tikzpicture}[scale=0.4,baseline=(current bounding box.center)]
 \StairGrid{10};
 \Recfill{0}{0}{8}{1};
 \Recfill{0}{1}{7}{2};
 \Recfill{0}{2}{6}{3};
 \Recfill{0}{3}{5}{4};
 \Recfill{0}{4}{4}{5};
 \Recfill{0}{5}{3}{6};
 \Recfill{0}{6}{2}{7};
 \Recfill{0}{7}{1}{8};
\end{tikzpicture}
\adjustbox{stack=cc}{
\begin{tikzpicture}[]
	\node (1) at (0,0) {1};
	\node (3) at (2,0) {3};
	\node (5) at (4,0) {5};
	\node (7) at (6,0) {7};
	\node (9) at (8,0) {9};
	
	\node (2) at (1,1) {2};
	\node (4) at (3,1) {4};
	\node (6) at (5,1) {6};
	\node (8) at (7,1) {8};
	\node (10) at (9,1) {10};

	\draw[<-] (1) to (3);
	\draw[<-] (3) to (5);
	\draw[<-] (5) to (7);
	\draw[<-] (7) to (9);

	\draw[<-] (2) to (4);
	\draw[<-] (4) to (6);
	\draw[<-] (6) to (8);
	\draw[<-] (8) to (10);
	
	\draw[<-] (1) to (4);
	\draw[<-] (2) to (5);
	\draw[<-] (3) to (6);
	\draw[<-] (4) to (7);
	\draw[<-] (5) to (8);
	\draw[<-] (6) to (9);
	\draw[<-] (7) to (10);
\end{tikzpicture}\\
\vspace{20pt}
\begin{tikzpicture}[scale=1.3]
	\tikzmath{\x=0.6;\y=0.3;};
	\Intv[1]{0}{0};
	\Intv[2]{\x}{\y};
	\Intv[3]{2*\x}{0};
	\Intv[4]{3*\x}{\y};
	\Intv[5]{4*\x}{0};
	\Intv[6]{5*\x}{\y};
	\Intv[7]{6*\x}{0};
	\Intv[8]{7*\x}{\y};
	\Intv[9]{8*\x}{0};
	\Intv[10]{9*\x}{\y};
\end{tikzpicture}
}
\caption{$\cP_{\stair(9), 10}$: a ladder of size 10}\label{fig:ladder}
\end{figure}
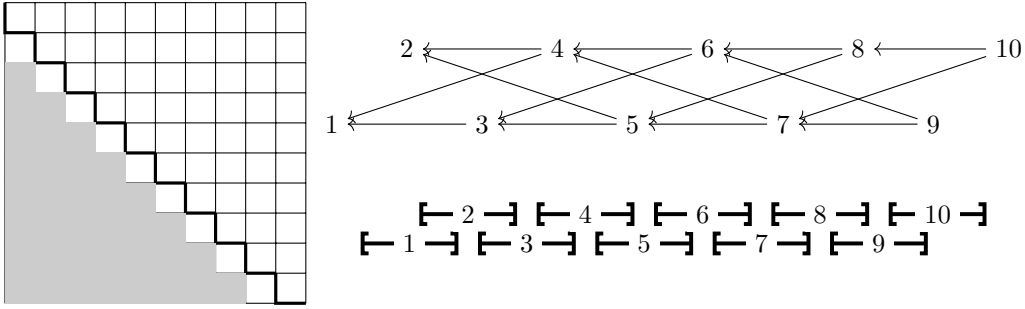
Figure \ref{fig:ladder} shows the partial order $\cP_{\stair(9),10}$ which is by definition a ladder order, and the partition and the unit interval arrangement defining it. The term ``ladder'' is inspired from the shapes of its Hasse diagram and the corresponding unit interval arrangement.

\begin{defn} For a partial order $\cP$ defined on $X$, we say that $A$ is a \emph{ladder in $\cP$} if $\cP|_A$ is a ladder order.
\end{defn}
For example, if $\cP=\cP_{\stair(m-1),m}$ then ladders in $\cP$ are exactly $[a,b]$ for some $a,b \in [1,m]$.
Let us describe some basic properties of ladders. Later we will frequently use these properties without reminder.
\begin{lem} Suppose that $\cP$ is a natural unit interval order on $[1,n]$ and assume that $y_1, \ldots, y_k \in [1,n]$ such that $y_1<y_2<\cdots<y_k$ and $\{y_1, \ldots, y_k\}$ is a ladder in $\cP$. Then,
\begin{enumerate}
\item $\{y_i, \ldots, y_j\}$ is a ladder in $\cP$ for any $1\leq i\leq j \leq k$.
\item $y_i \pdl y_{i+1}$ for $i \in [1,k-1]$.
\item $y_i \pl y_j$ if $j-i\geq 2$.
\item If $x \not\pl y_1$ and $x \pl y_2$, then $\{x, y_1, \ldots, y_k\}$ is a ladder in $\cP$.
\item If $z \not\pr y_k$ and $z \pr y_{k-1}$, then $\{y_1, \ldots, y_k, z\}$ is a ladder in $\cP$.
\item If $y_i<x<y_{i+1}$ for some $i \in [1,k-1]$, then $y_i \pdl x$ and $x \pdl y_{i+1}$.
\end{enumerate}
\end{lem}
\begin{proof} It follows almost directly from the definition of a ladder and \cond{}.
\end{proof}

\subsection{Climbing a ladder} We define a special kind of partial orders called ``ladder-climbing'' orders.
\begin{defn} For a partial order $\cP$ on $[1,n]$, we say that \emph{someone is climbing a ladder in $\cP$} or \emph{$\cP$ is ladder-climbing} if there exist $x, y_1, \ldots, y_k \in [1,n]$ such that
\begin{enumerate}
\item $x \not \in \{y_1, y_2, \ldots, y_k\}$,
\item $\{y_1, y_2, \ldots, y_k\}$ is a ladder in $\cP$, and
\item $y_1 \pl x \pl y_k$.
\end{enumerate}
In this case, we also say that \emph{$x$ is climbing a ladder in $\cP$} or \emph{$x$ is climbing (the ladder) $\{y_1, \ldots, y_k\}$  in $\cP$}. If there is no such $x$, then we say that \emph{no one is climbing a ladder in $\cP$} or simply \emph{$\cP$ is not ladder-climbing}.
\end{defn}

Indeed, there is a characterization of ladder-climbing partial orders in terms of the avoidance of certain suborders as the following proposition shows.
\begin{prop} Let $\cP$ be a partial order on $[1,n]$. Then no one is climbing a ladder in $\cP$ if and only if it avoids both $\cP_{(3,1,1),5}$ and $\cP_{(4,2,1,1),6}$. (See Figure \ref{fig:climbingladders} for the Hasse diagrams of these two orders.)
\end{prop}
\begin{proof} Direct calculation shows that 3 is climbing the ladder $\{1, 2, 4, 5\}$ in $\cP_{(3,1,1), 5}$ and 4 is climbing the ladder $\{1, 2, 4, 5, 6\}$ in $\cP_{(4,2,1,1),6}$. Thus if no one is climbing a ladder in $\cP$ then it should avoid $\cP_{(3,1,1),5}$ and $\cP_{(4,2,1,1),6}$.

It remains to show that if $\cP$ is ladder-climbing then there is a set $X \in [1,n]$ such that $\cP|_X$ is isomorphic to either $\cP_{(3,1,1),5}$ or $\cP_{(4,2,1,1),6}$. Suppose that $x$ is climbing the ladder $\{y_1, y_2, \ldots, y_k\}$ in $\cP$ where $y_1<y_2<\cdots<y_k$, in which case $y_1 \pl x \pl y_k$. Then without loss of generality we may assume that $y_2\not \pl x$ and $y_{k-1}\not\pr x$. Since $y_1 \pdl y_2$ and $y_{k-1} \pdl y_k$, this implies that $y_2 \pdl x$ and $y_{k-1} \pdr x$ by \cond. Moreover, we have $y_i \dash_\cP x$ for $i\in [2,k-1]$. Indeed, if $y_i\pl x$ (resp.  $y_i\pr x$) then $y_2 \pl x$ (resp. $y_{k-1}\pr x$) by \cond{} applied to $(y_2, y_i, x)$ (resp. $(x, y_i, y_{k-1})$), which is a contradiction.

Since $y_2 \pdl x$ and $y_{k-1} \pdr x$ we have $k\geq 4$. On the other hand, if $k\geq 7$ then $y_2, y_4, y_6\dash_\cP x$ and $y_2 \pl y_4 \pl y_6$ which is impossible by Lemma \ref{lem:3122}. Thus it follows that $k \in[4,6]$. If $k=4$, then one can easily show that $\cP$ restricted to $\{x, y_1, y_2, y_3, y_4\}$ is isomorphic to $\cP_{(3,1,1),5}$ where the isomorphism of posets is given by $(x, y_1, y_2, y_3, y_4)\mapsto (3,1,2,4,5) $. If $k=5$, then $\cP$ restricted to $\{x, y_1, y_2, y_3, y_4, y_5\}$ is isomorphic to $\cP_{(4,2,1,1),6}$ where the isomorphism is given by $(x, y_1, y_2, y_3, y_4, y_5) \mapsto (3,1,2,4,5,6)$. Finally, if $k=6$ then $\cP$ restricted to $\{x, y_1, y_2, y_3, y_5\}$ is isomorphic to $\cP_{(3,1,1),5}$, where the isomorphism is given by $(x, y_1, y_2, y_3, y_5) \mapsto (4, 1, 2, 3, 6)$. It suffices for the proof.
\end{proof}

\begin{figure}[!htbp]
\begin{tikzpicture}[baseline=(current bounding box.center)]
  \node (5) at (0,4) [draw,draw=none] {$5$};
  \node (4) at (2,3) [draw,draw=none] {$4$};
  \node (3) at (0,2) [draw,draw=none] {$3$};
  \node (2) at (-2,1) [draw,draw=none] {$2$};
  \node (1) at (0,0) [draw,draw=none] {$1$};
  \draw [black,<-] (1) to (3);
  \draw [black,<-] (3) to (5);
  \draw [black,<-] (2) to (5);
  \draw [black,<-] (1) to (4);
\end{tikzpicture}\qquad
\begin{tikzpicture}[baseline=(current bounding box.center)]
  \node (6) at (0,4) [draw,draw=none] {$6$};
  \node (5) at (2,3) [draw,draw=none] {$5$};
  \node (4) at (0,2) [draw,draw=none] {$4$};
  \node (3) at (-2,2) [draw,draw=none] {$3$};
  \node (2) at (2,1) [draw,draw=none] {$2$};
  \node (1) at (0,0) [draw,draw=none] {$1$};
  \draw [black,<-] (1) to (3);
  \draw [black,<-] (1) to (4);
  \draw [black,<-] (1) to (5);
  \draw [black,<-] (2) to (5);
  \draw [black,<-] (2) to (6);
  \draw [black,<-] (4) to (6);
  \draw [black,<-] (3) to (6);
\end{tikzpicture}
\caption{$\cP_{(3,1,1),5}$ and $\cP_{(4,2,1,1),6}$}\label{fig:climbingladders}
\end{figure}

There is another characterization of $\cP_{(3,1,1),5}$-avoiding partial orders as follows.
\begin{lem} \label{lem:ladjoin} The following two conditions are equivalent.
\begin{enumerate}
\item $\cP$ avoids $\cP_{(3,1,1),5}$.
\item ``A join of two ladders is again a ladder.'' Suppose that $\cL$ and $\cL'$ are two ladders in $\cP$ such that $\cL\cap \cL' = \{x\}$. If $x$ is the maximum in $\cL$ and the minimum in $\cL'$ with respect to the usual order and $|\cL|, |\cL'|\geq 3$ then  $\cL \cup \cL'$ is also a ladder in $\cP$.
\end{enumerate}
\end{lem}
\begin{proof} Suppose that $\cP=\cP_{(3,1,1),5}$. Then $\{1,2,3\}$ and $\{3, 4, 5\}$ are ladders but $\{1, 2, 3, 4, 5\}$ is not a ladder. Thus any order that does not avoid $\cP=\cP_{(3,1,1),5}$ cannot satisfy the second condition. Now suppose that $\cP$ avoids $\cP_{(3,1,1),5}$ and assume that $a_1<\cdots<a_k<x<b_1< \cdots<b_l$ such that $\{a_1, \ldots, a_k, x\}$ and $\{x, b_1, \ldots, b_l\}$ are ladders in $\cP$. Then it suffices to show that $\{a_{k-1}, a_k, x, b_1, b_2\}$ is again a ladder in $\cP$. By assumption we have $a_{k-1} \pl x \pl b_2$ and $a_{k-1} \pdl a_{k} \pdl x \pdl b_1 \pdl b_2$. Also $a_{k-1} \pl b_1$ and $a_k \pl b_2$ by \cond{}. Thus if $a_k \pdl b_1$ then direct calculation shows that $\cP$ restricted to $\{a_{k-1}, a_k, x, b_1, b_2\}$ is isomorphic to $\cP_{(3,1,1),5}$, which is a contradiction. Thus $a_k \pl b_1$ and the result follows.
\end{proof}

\section{$\cP$-Knuth equivalence and the main theorem} \label{sec:Knuth}
In this section we assume that a fixed natural unit interval order $\cP$ on $[1,n]$ is given and define $\cP$-Knuth moves and equivalences. Also we state our main theorem in this paper.
\subsection{Definition of $\cP$-Knuth equivalences} First we define the notions of $\cP$-Knuth moves and equivalences which generalize the ones originally introduced by Knuth.
\begin{defn}\label{def:pknuth} Let $1\leq a<b<c\leq n$ and $a\pl c$. We say that two words $w, w'$ are connected by a $\cP$-Knuth move if they fall into one of the following situations, in which case we write $w \pkm w'$ (or $w {\leftrightsquigarrow} w'$ if there is no confusion).
\begin{enumerate}
\item If $a\pdl b$ and $b\pdl c$, then $[\cdots bca \cdots]\pkm[\cdots cab \cdots]$.
\item If $a\pl b$ and $b\pdl c$, then $[\cdots bca \cdots]\pkm[\cdots bac \cdots]$ and $[\cdots cba \cdots]\pkm[\cdots cab \cdots]$.
\item If $a\pdl b$ and $b\pl c$, then $[\cdots bca \cdots]\pkm[\cdots cba \cdots]$ and $[\cdots acb \cdots]\pkm[\cdots cab \cdots]$.
\item If $a\pl b$ and $b\pl c$, then $[\cdots bca \cdots]\pkm[\cdots bac \cdots]$ and $[\cdots acb \cdots]\pkm[\cdots cab \cdots]$.
\end{enumerate}
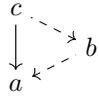
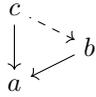
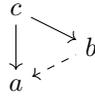
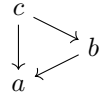
\begin{center}
\begin{tikzpicture}[scale=0.5]
  \node (a) at (0,0) {$a$};
  \node (b) at (2,1) {$b$};
  \node (c) at (0,2) {$c$};
  \node (d) at (1,-1) {(1) $\cP|_{\{a,b,c\}}\simeq \cP_{(1),3}$};
  \draw [black,<-] (a) -- (c);
  \draw [black,dashed,<-] (a) -- (b);
  \draw [black,dashed,<-] (b) -- (c);
\end{tikzpicture}\quad\begin{tikzpicture}[scale=0.5]
  \node (a) at (0,0) {$a$};
  \node (b) at (2,1) {$b$};
  \node (c) at (0,2) {$c$};
  \node (d) at (1,-1) {(2) $\cP|_{\{a,b,c\}}\simeq \cP_{(1,1),3}$};
  \draw [black,<-] (a) -- (c);
  \draw [black,<-] (a) -- (b);
  \draw [black,dashed,<-] (b) -- (c);
\end{tikzpicture}\quad\begin{tikzpicture}[scale=0.5]
  \node (a) at (0,0) {$a$};
  \node (b) at (2,1) {$b$};
  \node (c) at (0,2) {$c$};
  \node (d) at (1,-1) {(3) $\cP|_{\{a,b,c\}}\simeq \cP_{(2),3}$};
  \draw [black,<-] (a) -- (c);
  \draw [black,dashed,<-] (a) -- (b);
  \draw [black,<-] (b) -- (c);
\end{tikzpicture}\quad\begin{tikzpicture}[scale=0.5]
  \node (a) at (0,0) {$a$};
  \node (b) at (2,1) {$b$};
  \node (c) at (0,2) {$c$};
  \node (d) at (1,-1) {(4) $\cP|_{\{a,b,c\}}\simeq \cP_{(2,1),3}$};
  \draw [black,<-] (a) -- (c);
  \draw [black,<-] (a) -- (b);
  \draw [black,<-] (b) -- (c);
\end{tikzpicture}
\end{center}
In each situation, there exists $i \in [2,n-1]$ such that the set of positions of $a, b,$ and $c$ is $\{i-1, i, i+1\}$. In such a case, we say that $i$ is the position of the $\cP$-Knuth move and also that the $\cP$-Knuth move occurs at position $i$.
\end{defn}
\begin{defn}
The $\cP$-Knuth equivalence relation on the set of words is the equivalence relation generated by $\cP$-Knuth moves. If two words $w, w'$ are equivalent under this relation, we say that $w$ and $w'$ are $\cP$-Knuth equivalent and write $w\stackrel{\cP}\sim w'$. (If there is no confusion, we also say that $w$ and $w'$ are equivalent and write $w\sim w'$.)
\end{defn}

Note that $\cP$-Knuth move/equivalence revert to the usual Knuth move/equivalence when $\cP$ is the usual order on $[1,n]$.

\begin{example}[Figure \ref{fig:s3ex}] All the possible $\cP$-Knuth moves for natural unit interval orders $\cP$ on $[1,3]$ are described in Figure \ref{fig:s3ex}. Here, the underlined numbers in each word indicate its $\cP$-descents.
\end{example}

\begin{figure}[!htbp]
\fbox{
\subcaptionbox{$\cP_{\emptyset, 3}$}{
\begin{tikzpicture}[scale=1.2]
  \node[draw] (123) at (0:1) {$123$};
  \node[draw] (132) at (60:1) {$132$};
  \node[draw] (213) at (120:1) {$213$};
  \node[draw] (231) at (180:1) {$231$};
  \node[draw] (312) at (240:1) {$312$};
  \node[draw] (321) at (300:1) {$321$};
\end{tikzpicture}
}
}\ \ 
\fbox{
\subcaptionbox{$\cP_{(1), 3}$}{
\begin{tikzpicture}[scale=1.2]
  \node[draw] (123) at (0:1) {$123$};
  \node[draw] (132) at (60:1) {$132$};
  \node[draw] (213) at (120:1) {$213$};
  \node[draw] (231) at (180:1) {$2\und{3}1$};
  \node[draw] (312) at (240:1) {$\und{3}12$};
  \node[draw] (321) at (300:1) {$321$};
  \draw (312) to (231);
\end{tikzpicture}
}
}\ \ 
\fbox{
\subcaptionbox{$\cP_{(2), 3}$}{
\begin{tikzpicture}[scale=1.2]
  \node[draw] (123) at (0:1) {$123$};
  \node[draw] (132) at (60:1) {$1\und{3}2$};
  \node[draw] (213) at (120:1) {$213$};
  \node[draw] (231) at (180:1) {$2\und{3}1$};
  \node[draw] (312) at (240:1) {$\und{3}12$};
  \node[draw] (321) at (300:1) {$\und{3}21$};
  \draw (231) to (321);
  \draw (132) to (312);
\end{tikzpicture}
}
}

\vspace*{5pt}
\fbox{
\subcaptionbox{$\cP_{(1,1), 3}$}{
\begin{tikzpicture}[scale=1.2]
  \node[draw] (123) at (0:1) {$123$};
  \node[draw] (132) at (60:1) {$132$};
  \node[draw] (213) at (120:1) {$\und{2}13$};
  \node[draw] (231) at (180:1) {$2\und{3}1$};
  \node[draw] (312) at (240:1) {$\und{3}12$};
  \node[draw] (321) at (300:1) {$3\und{2}1$};
  \draw (213) to (231);
  \draw (312) to (321);
\end{tikzpicture}
}
}\ \ 
\fbox{
\subcaptionbox{$\cP_{(2,1), 3}$}{
\begin{tikzpicture}[scale=1.2]
  \node[draw] (123) at (0:1) {$123$};
  \node[draw] (132) at (60:1) {$1\und{3}2$};
  \node[draw] (213) at (120:1) {$\und{2}13$};
  \node[draw] (231) at (180:1) {$2\und{3}1$};
  \node[draw] (312) at (240:1) {$\und{3}12$};
  \node[draw] (321) at (300:1) {$\und{3}\und{2}1$};
  \draw (213) to (231);
  \draw (132) to (312);
\end{tikzpicture}
}
}
\caption{The $\cP$-Knuth moves on $\sym_3$}\label{fig:s3ex}
\end{figure}

\subsection{Relation to $\cP$-descents and D graphs}
Here we relate $\cP$-Knuth equivalences with study of dual equivalence graphs by Assaf. More precisely, we show that the graphs obtained from the $\cP$-Knuth moves are D graphs in the sense of \cite[Definition 4.5]{ab12}. First let us discuss how $\cP$-Knuth moves affect the $\cP$-descents of words.
\begin{lem} \label{lem:pmovedes} Assume that $w,w'\in \sym_n$ are connected by a $\cP$-Knuth move at position $i$. 
\begin{enumerate}
\item We have $\{\des_\cP(w)\cap\{i-1, i\}, \des_\cP(w')\cap\{i-1, i\}\} = \{\{i-1\}, \{i\}\}$.
\item If $i>2$, then $\{\des_\cP(w)\cap\{i-2, i-1\}, \des_\cP(w')\cap\{i-2, i-1\}\}$ is equal to one of $\{\emptyset, \{i-1\}\}$,  $\{\{i-2\}, \{i-1\}\}$, or $\{\{i-2\}, \{i-2,i-1\}\}$.
\item If $i<n-1$, then $\{\des_\cP(w)\cap\{i, i+1\}, \des_\cP(w')\cap\{i, i+1\}\}$ is equal to one of $\{\emptyset, \{i\}\}$,  $\{\{i\}, \{i+1\}\}$, or $\{\{i+1\}, \{i,i+1\}\}$.
\item If $j\in [1,n-1]-[i-2, i+1]$, then $\des_\cP(w)\cap\{j\}=\des_\cP(w')\cap\{j\}$.
\end{enumerate}
\end{lem}
\begin{proof} (1) is checked case-by-case. For (2), we only need to check that $\{\des_\cP(w)\cap\{i-2, i-1\}, \des_\cP(w')\cap\{i-2, i-1\}\} \neq \{\emptyset, \{i-2, i-1\}\}$ thanks to (1). This is also checked case-by-case. (3) is proved similarly to (2). (4) is trivial from the definition of $\cP$-Knuth moves.
\end{proof}

We recall the notion of signed colored graphs following \cite[4.2]{ab12} and \cite[Definition 3.1]{ass15}.
\begin{defn} A signed colored graph of degree $n$ is a collection $(V, \sigma, \{E_i\}_{1<i<n})$ where $V$ is a set, $\sigma$ is a function $\sigma: V\rightarrow 2^{[1,n-1]}$, and each $E_i$ is a set of unordered pairs of different elements in $V$. (Here $2^{[1,n-1]}$ denotes the power set of $[1,n-1]$.) Each element in $V$ is called a vertex, and each element in $E_i$ is called an edge colored $i$.
\end{defn}
\begin{rmk}
In \cite[4.2]{ab12} and \cite[Definition 3.1]{ass15} the function $\sigma$ assigns to each vertex $v \in V$ a sequence of length $n-1$ consisting of $+$ and $-$. Their definition is equivalent to ours if we define a new sigma function, say $\sigma': V \rightarrow 2^{[1,n-1]}$, such that $\sigma'(v) \owns i$ (resp. $\sigma'(v) \not\owns i$) if and only if the $i$-th component of $\sigma(v)$ equals $+$ (resp. $-$).
\end{rmk}

\begin{defn} Suppose that $V \subset \sym_n$ is closed under $\cP$-Knuth moves. Then we define the $\cP$-Knuth equivalence graph $\Gamma_V$ attached to $V$ to be $\Gamma_V=(V, \des_\cP, \{E_i\}_{1<i<n})$ where each $E_i$ is the set of pairs in $V$ connected by a $\cP$-Knuth move at position $i$.
\end{defn}

It is clear that $\cP$-Knuth equivalence graph is a signed colored graph of degree $n$. Now we recall the notion of D graphs following \cite[Definition 4.2, 4.5]{ab12}.
\begin{defn} A signed colored graph $(V, \sigma, \{E_i\}_{1<i<n})$ of degree $n$ is called a D graph if the following axioms hold.
\begin{enumerate}[label=Ax\arabic*.]
\item For $w \in V$ and $1<i<m$, $|\sigma(w) \cap \{i-1, i\}|=1$ if and only if there exists $x\in V$ such that $\{w, x\} \in E_i$. Moreover, $x$ is unique when it exists.
\item Whenever $\{w, x\} \in E_i$, $\sigma(w)\cap \{i\} \neq \sigma(x) \cap \{i\}$ and $\sigma(w)\cap \{h\} =\sigma(x) \cap \{h\}$ for $h \not\in [i-2, i+1]$.
\item For $\{w, x\}\in E_i$, if $\sigma(w) \cap\{i-2\}\neq \sigma(x) \cap \{i-2\}$ then $|\sigma(w) \cap\{i-2,i-1\}|=1$. Also, if $\sigma(w) \cap\{i+1\}\neq \sigma(x) \cap \{i+1\}$ then $|\sigma(w) \cap\{i,i+1\}|=1$.
\item[Ax5.] Whenever $|i-j|\geq 3$, $\{w,x\} \in E_i$, and $\{x,y\} \in E_j$, there exists $v\in V$ such that $\{w, v\} \in E_j$ and $\{v, y\} \in E_i$.
\end{enumerate}
\end{defn}
We claim that the $\cP$-Knuth equivalence graphs are indeed D graphs.
\begin{prop} A $\cP$-Knuth equivalence graph is a D graph.
\end{prop}
\begin{proof} We need to check that Ax1, Ax2, Ax3, and Ax5 hold for $\cP$-Knuth equivalence graphs. For Ax1, it follows from the fact that the $\cP$-Knuth move at position $i$ in Definition \ref{def:pknuth} occurs in all the possible cases of words satisfying $|\sigma(w) \cap \{i-1, i\}|=1$. Ax2 follows from (1) and (4) of Lemma \ref{lem:pmovedes}. Ax3 follows from (2) and (3) of Lemma \ref{lem:pmovedes}. Lastly, Ax5 clearly follows from the definition of $\cP$-Knuth moves.
\end{proof}

In \cite[Definition 3.2]{ass15} and \cite[Definition 4.2]{ab12}, they defined dual equivalence graphs which are a special kind of D graphs by imposing two additional axioms. This framework is used to study Schur positivity of certain quasi-symmetric functions. In particular, the ``generating functions'' attached to a dual equivalence graph is a single Schur function by \cite[Corollary 4.4]{ass15}. However, our graphs are not dual equivalence graphs in general.

\begin{example} Figure \ref{fig:214} shows all the connected $\cP$-Knuth equivalence graphs on $\sym_4$ with $\cP=\cP_{(2,1),4}$, where underlined numbers denote $\cP$-descents and numbers above edges indicate their colors. (Written above each connected component is the corresponding generating function which we will define in a moment.) There is one connected component with 5 vertices which satisfies neither Axiom 4 nor Axiom 6 of \cite[Definition 3.2]{ass15} for dual equivalence graphs. 
\end{example}

\begin{figure}[!htbp]
\begin{tikzpicture}[yscale=1.3, xscale=1.7]

    \draw[rounded corners] (0.5,0.5) rectangle (1.5,1.5);
    \node[fill=white] at (1,1.5) {$s_4$};
    \node[draw] (1) at  (1,1) {$1234$};
    
    \draw[rounded corners] (2,0.5) rectangle (3,1.5);
    \node[fill=white] at (2.5,1.5) {$ts_4$};
    \node[draw] (2) at  (2.5,1) {$1243$};
    
    \draw[rounded corners] (3.5,0.5) rectangle (4.5,1.5);
    \node[fill=white] at (4,1.5) {$ts_4$};
    \node[draw] (3) at  (4,1) {$1324$};

    \draw[rounded corners] (5,0.5) rectangle (8,1.5);
    \node[fill=white] at (6.5,1.5) {$ts_{31}$};
    \node[draw] (4) at  (5.5,1) {$13\und{4}2$};
    \node[draw] (5) at  (6.5,1) {$1\und{4}23$};
    \node[draw] (19) at (7.5,1) {$\und{4}123$};
    \draw (4) -- (5) node [midway, above] {$3$};
    \draw (5) -- (19) node [midway, above] {$2$};

    \draw[rounded corners] (0.5,0) rectangle (1.5,-1);
    \node[fill=white] at (1,0) {$t^2s_4$};
    \node[draw] (6) at  (1,-.5) {$1432$};

    \draw[rounded corners] (2,0) rectangle (3,-1);
    \node[fill=white] at (2.5,0) {$ts_4$};
    \node[draw] (7) at  (2.5,-.5) {$2134$};

    \draw[rounded corners] (3.5,0) rectangle (4.5,-1);
    \node[fill=white] at (4,0) {$t^2s_4$};
    \node[draw] (8) at  (4,-.5) {$2143$};
    
    \draw[rounded corners] (5,0) rectangle (8,-1);
    \node[fill=white] at (6.5,0) {$ts_{31}$};
    \node[draw] (9) at  (6.5,-.5) {$2\und{3}14$};
    \node[draw] (10) at (5.5,-.5) {$23\und{4}1$};
    \node[draw] (13) at (7.5,-.5) {$\und{3}124$};
    \draw (9) -- (10) node [midway, above] {$3$};
    \draw (9) -- (13) node [midway, above] {$2$};

    \draw[rounded corners] (0.5,-1.5) rectangle (3.5,-2.5);
    \node[fill=white] at (2,-1.5) {$t^2s_{31}$};
    \node[draw] (11) at (2,-2) {$2\und{4}13$};
    \node[draw] (12) at (1,-2) {$24\und{3}1$};
    \node[draw] (21) at (3,-2) {$\und{4}213$};
    \draw (11) -- (12) node [midway, above] {$3$};
    \draw (11) -- (21) node [midway, above] {$2$};

    \draw[rounded corners] (4.25,-1.5) rectangle (6.25,-2.5);
    \node[fill=white] at (5.25,-1.5) {$ts_{22}$};
    \node[draw] (14) at (4.75,-2) {$\und{3}1\und{4}2$};
    \node[draw] (17) at (5.75,-2) {$3\und{4}12$};
    \draw (14) -- (17) node [midway, above] {$2,3$};

    \draw[rounded corners] (7,-1.5) rectangle (8,-2.5);
    \node[fill=white] at (7.5,-1.5) {$t^2s_4$};
    \node[draw] (15) at (7.5,-2) {$3214$};

    \draw[rounded corners] (1,-3) rectangle (6,-4);
    \node[fill=white] at (3.5,-3) {$t^2(s_{31}+s_{22})$};
    \node[draw] (16) at (1.5,-3.5) {$32\und{4}1$};
    \node[draw] (18) at (2.5,-3.5) {$3\und{4}21$};
    \node[draw] (20) at (5.5,-3.5) {$\und{4}132$};
    \node[draw] (22) at (3.5,-3.5) {$\und{4}2\und{3}1$};
    \node[draw] (23) at (4.5,-3.5) {$4\und{3}12$};
    \draw (16) -- (18) node [midway, above] {$3$};
    \draw (18) -- (22) node [midway, above] {$2$};
    \draw (20) -- (23) node [midway, above] {$2$};
    \draw (22) -- (23) node [midway, above] {$3$};
    
    \draw[rounded corners] (6.5,-3) rectangle (7.5,-4);
    \node[fill=white] at (7,-3) {$t^3s_4$};
    \node[draw] (24) at (7,-3.5) {$4321$};
\end{tikzpicture}
\caption{$\cP_{(2,1),4}$-Knuth equivalence graph} \label{fig:214}
\end{figure}
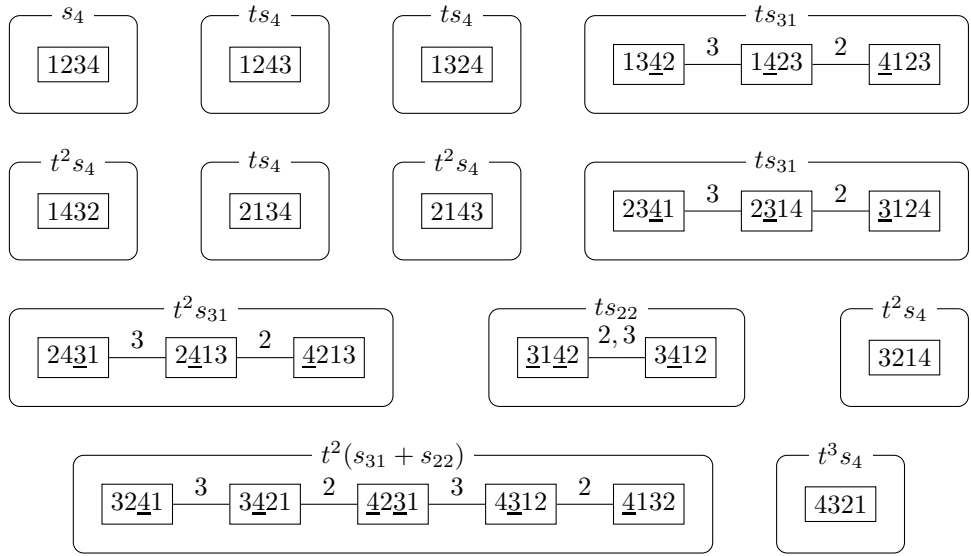

\subsection{Genuine $\cP$-height and fake $\cP$-inversion number}
Here we prove that the $\cP$-Knuth move preserves genuine $\cP$-heights and fake $\cP$-inversion numbers of permutations. For the former claim, we need to impose assumption that $\cP$ is not ladder-climbing.

\begin{prop}\label{prop:ght} Suppose that $\cP$ avoids $\cP_{(3,1,1),5}$ and $\cP_{(4,2,1,1), 6}$. If $w \pkm w'$, then $\ght(w) = \ght(w')$. As a result, the genuine $\cP$-height is constant on any connected $\cP$-Knuth equivalence graph.
\end{prop}
The proof of this proposition will be given in Section \ref{sec:proofprop}.
\begin{rmk} If we allow that $\cP$ is ladder-climbing, then the above proposition is no longer true. For example, when $\cP=\cP_{(3,1,1),5}$ we have $53241 \pkm 53412$ but $\ght_\cP(53241)=|(5,3)|=2\neq 3=|(5,3,1)|=\ght_\cP(53412)$. Likewise, when $\cP=\cP_{(4,2,1,1),6}$ we have $563241 \pkm 635241$ but $\ght_\cP(563241)=|(6,4,1)|=3\neq 2=|(6,3)|=\ght_\cP(635241)$. 
\end{rmk}

\begin{lem}\label{lem:presfinv} If $w \pkm w'$, then $|\finv_\cP(w)| = |\finv_\cP(w')|$. As a result, the fake $\cP$-inversion number is constant on any connected $\cP$-Knuth equivalence graph.
\end{lem}
\begin{proof} Suppose that the $\cP$-Knuth move $w \pkm w'$ occurs at position $i$. If either $x \not\in \{w_{i-1}, w_i, w_{i+1}\}$ or $y \not\in \{w_{i-1}, w_i, w_{i+1}\}$ then it is clear that $\finv_\cP(w) \cap\{(x,y)\} =\finv_\cP(w') \cap\{(x,y)\}$ since the relative position of $x$ and $y$ does not change. Thus for the verification of this lemma we may restrict our attention to words of length 3, e.g. Figure \ref{fig:s3ex}. Now the lemma follows from case-by-case observation.
\end{proof}

\subsection{Generating functions and the main theorem}
Let us define a generating function of a $\cP$-Knuth equivalence graph. (cf. \cite[Theorem 3.1]{sw16})
\begin{defn} For a $\cP$-Knuth equivalence graph $\Gamma_V=(V, \des_\cP, \{E_i\})$, its generating function is defined to be $\gamma_V\colonequals \sum_{w \in V} t^{|\finv_\cP(w)|}F_{\des_{\cP}(w)}$. 
\end{defn}
If we consider a connected $\cP$-Knuth equivalence graph, then we may factor out $t^{|\finv_\cP(w)|}$ from the formula due to Lemma \ref{lem:presfinv}.
More precisely, if $\Gamma_V$ is a connected graph then $\gamma_V\colonequals t^{|\finv_\cP(V)|}\sum_{w \in V} F_{\des_{\cP}(w)}$ where $|\finv_\cP(V)|$ is the fake $\cP$-inversion number of any element in $V$. Now we state the main theorem of this paper. Its proof is given in Section \ref{sec:proofthm}. Note that this strengthens \cite[Theorem 4]{gas96} for a natural unit interval order which avoids $\cP_{(3,1,1),5}$ and $\cP_{(4,2,1,1),6}$.
\begin{thm}[Main theorem] \label{thm:mainstrong} Suppose that $\cP$ is a natural unit interval order on $[1,n]$ which avoids $\cP_{(3,1,1),5}$ and $\cP_{(4,2,1,1),6}$. Let $\Gamma=(V, \des_\cP, \{E_i\})$ be a connected $\cP$-Knuth equivalence graph and $\gamma_V$ be its generating function. Let $w_1, \ldots, w_k$ be all the elements in $V$ each of which is the reading word of the $\cP$-tableau $PT_1, \ldots, PT_k$, of shape $\lambda_1, \ldots, \lambda_k$, respectively. Then we have $\gamma_V = t^{|\finv_\cP(V)|}(s_{\lambda_1}+\cdots +s_{\lambda_k})$, where $|\finv_\cP(V)|$ is the fake $\cP$-inversion number of any element in $V$. Furthermore, we have $l(\lambda_1)=\cdots =l(\lambda_k)$ which is also equal to the genuine $\cP$-height of any $w \in V$.
\end{thm}
\begin{conj}[Main conjecture] \label{conj:main}
The claim of the Theorem \ref{thm:mainstrong} is true for all unit interval orders $\cP$ (except the last sentense).
\end{conj}
See Section \ref{sec:equivgraph} for some examples of $\cP$-Knuth equivalence graphs and their generating functions.
The following corollary is a direct consequence, which generalizes both \cite[Theorem 3]{gas96} and \cite[Theorem 6.3]{sw16} for a natural unit interval order which avoids $\cP_{(3,1,1),5}$ and $\cP_{(4,2,1,1),6}$.
\begin{cor} Suppose that $\cP$ is a natural unit interval order on $[1,n]$ which avoids $\cP_{(3,1,1),5}$ and $\cP_{(4,2,1,1),6}$.
Then the generating function of any $\cP$-Knuth equivalence graph is Schur positive, i.e. it is a symmetric function and its coefficients with respect to the expansion of Schur functions are polynomials in $t$ with nonnegative integer coefficients.
\end{cor}

\section{Column insertion algorithm} \label{sec:ins}
In this section, we assume that a fixed natural unit interval order $\cP$ is given and define a column insertion algorithm.

\subsection{Column insertion algorithm $\alg_\Phi$} 
For convenience, we add $\vn, -\vn$ to the poset $([1,n], \cP)$ so that $\vn \pr i$ (resp. $-\vn\pl i$) for any $i\in[1,n]$. We define 
\begin{align*}
\fA&\colonequals\{(a_m, \ldots, a_1) \mid m \in \bN, a_i \in [1,n] \cup \{\vn\}, a_i\neq a_j \textnormal{ if } i\neq j \textnormal{ and } a_i, a_j \neq \vn\},
\\\fC&\colonequals\{(c_l, \ldots,  c_1) \mid l \in \bN, c_i \in [1,n], c_i\pl c_j \textnormal{ if }  i<j\},
\\\fA\fC&\colonequals\{(\alpha, c) \in \fA\times \fC \mid a_i \neq c_j \textnormal{ for any } i, j\}, and
\\\fC\fA&\colonequals\{(c, \alpha) \in \fC\times \fA \mid a_i \neq c_j \textnormal{ for any } i, j\}.
\end{align*}
One may regard $\fA$ as a set of (input/output) words and $\fC$ as a set of chains, i.e. one-column $\cP$-tableaux. (Recall that we read columns from bottom to top.) 

We introduce the column insertion algorithm $\alg_\Phi$. This defines a function $\Phi: \fA\fC \rightarrow \fC\fA$ and is described in terms of the pseudocode Algorithm \ref{alg1}.

\begin{algorithm}
\DontPrintSemicolon
\caption{Column insertion algorithm $\alg_\Phi$}\label{alg1}
\Fn(\tcp*[f]{$((a_m, \ldots, a_1),(c_l, \ldots, c_1)) \in \fA\fC$}){$\Phi((a_m, \ldots, a_1),(c_l, \ldots, c_1))$}{
$m \gets |(a_m, \ldots, a_1)|$\;
\lFor(\tcp*[f]{Initialize $(b_m, \ldots, b_1)$ to $(\vn,\ldots, \vn)$}){$i \gets 1$ \KwTo $m$}{$b_i \gets \vn$}
$l \gets |(c_l, \ldots, c_1)|$\;
$d_0 \gets -\vn$\;
\lFor(\tcp*[f]{Initialize $(d_l, \ldots, d_1, d_0)$ to $(c_l, \ldots, c_1,-\vn)$}){$i \gets 1$ \KwTo $l$}{$d_i \gets c_i$}
$p \gets 1$\;
\While {$p\leq m$}{
	\lIf(\tcp*[f]{\csiii[(a)]: pass if $a_p=\vn$}) {$a_p=\vn$}{$p \gets p+1$}
	\Else{
            	$r\gets \max \{i \in [0, l] \mid d_i<a_p\}$ \tcp*{Choose $r$ so that $d_r<a_p<d_{r+1}$}
            	\If(\tcp*[f]{\csi}) {$a_p \pr d_r$}
            	{
            		\lIf(\tcp*[f]{\csi[(a)]}){$r=l$}{
            			$l \gets l+1$  
            		}
            		\lElse(\tcp*[f]{\csi[(b)]}){
                    		$b_p\gets d_{r+1}$
            		}
			$d_{r+1} \gets a_p$\; 
            		$p \gets p+1$ 
            	}
            	\Else(\tcp*[f]{\csii})
		{
			$(h,q) \gets \max \{(i,j) \in \bN^2 \mid \{d_r, \ldots, d_{r+i}, a_p, \ldots, a_{p+j}\} $ is a ladder in $\cP$\; $\qquad\qquad\qquad\qquad\qquad\qquad \textnormal{ and } a_p<a_{p+1}<\cdots<a_{p+j}\}$\;
			\tcp*{The maximum is w.r.t. lexicographic order}
            		\If(\tcp*[f]{\csii[(a)]}) {$a_{p+q}<d_{r+h}$}{
            			\lFor{$j \gets 0$ \KwTo $q$}{{$b_{p+j} \gets a_{p+j}$}} 
            		}
			\Else(\tcp*[f]{\csii[(b)]})
			{
            			\For{$i \gets 0$ \KwTo $h$}
				{
            				$j \gets \min\{t \in [0, q] \mid a_{p+t}>d_{r+i}\}$\;
            				\lIf {$i=h$}{$k \gets q$}
					\lElse{$k \gets \max\{t \in [0, q-1] \mid a_{p+t}<d_{r+i+1}\}$}
            				$b_{p+j} \gets d_{r+i}$\;
            				\lFor {$t \gets j$ \KwTo $k-1$}{$b_{p+t+1} \gets a_{p+t}$}
            				$d_{r+i} \gets a_{p+k}$\;
            			}
            		}
			$p\gets p+q+1$\;
            	}
	}
}
\Return $((d_l, \ldots, d_1), (b_m, \ldots, b_1))$\;
}
\end{algorithm}

Let us investigate this algorithm in more detail. It takes the input $(\alpha, c) \in \fA\fC$ where $\alpha=(a_m, \ldots, a_1)$ and $c=(c_l, \ldots, c_1)$. Initiate $\fb=(b_m, \ldots, b_1)$ with $(a_m, \ldots, a_1)$ and $\fd=(d_l, \ldots, d_1)$ with $(c_l, \ldots, c_1)$. (Here $l=|\fd|$, which may change as the algorithm is performed.) Also we set $d_0 \colonequals -\vn$ to simplify our argument. Initialize $p$ with $1$. 

\subsubsection{}If $p>m=|\alpha|$, then we terminate the algorithm and return $(d, \beta)=(\fd, \fb)$.

\subsubsection{\und{\csiii[(a)]}} Suppose that $a_p=\vn$. Then we increase $p$ by 1 and repeat the algorithm from the beginning.
\begin{rmk} The reason why we call it \csiii[(a)] (instead of \csiii{}) shall become apparent when we describe another algorithm $\Psi_X$ in the proof of Proposition \ref{prop:phiinj}.
\end{rmk}

From now on we suppose $a_p \neq \vn$ and choose $r \in [0,l]$ such that $d_r<a_p<d_{r+1}$. (If $d_l<a_p$, then we set $r=l$.) First suppose that $a_p$ and $d_r$ are comparable, i.e. $d_r \pl a_p$.

\subsubsection{\und{\csi[(a)]}} We first consider the case when $r= l$, i.e. $a_p$ is bigger than any element in $\fd$ with respect to $\cP$. (This includes the case when $r=l=0$, i.e. $\fd$ is an empty chain.) In this case we ``add $a_p$ to the end of the chain $\fd$'', i.e. set $d_{l+1}\colonequals a_p$ and replace $\fd$ with $(d_{l+1}=a_p, d_l, \cdots, d_1)$. (As a result, the length of $\fd$ increases by 1.) After this, we increase $p$ by 1 and repeat the algorithm from the beginning.

\subsubsection{\und{\csi[(b)]}} Now suppose that  $r<l$. (This include the case when $-\vn=d_0< a_p< d_1$.) Then ``$a_p$ bumps $d_{r+1}$''; we set $b_p\colonequals d_{r+1}$ and then replace $d_{r+1}$ in $\fd$ with $a_p$. For example, if $1\pl 3$ and $1\pl 2$, then $\Phi((2), (3,1)) = ((2,1),(3))$. (Whether $2\pl 3$ or not does not affect the result here.) After this, we increase $p$ by 1 and repeat the algorithm from the beginning.

Now we suppose that $a_p$ and $d_r$ are not comparable (which forces that $r>0$). We set
\begin{align*}
A=\{&(i,j) \in \bN^2 \mid \ a_p<a_{p+1}<\cdots<a_{p+j} \textup{ and } \{d_{r}, \ldots, d_{r+i}, a_p, \ldots, a_{p+j}\} \textnormal{ is a ladder in } \cP\}.
\end{align*}
Note that $(0,0) \in A$ as $\{d_r, a_p\}$ is a ladder in $\cP$ by assumption (since $\cP_{\stair(1),2}=\cP_{\emptyset,2}$). We set $(h,q)$ to be the maximum of $A$ with respect to lexicographic order. In other words, we choose $(h,q)$ such that 
\begin{enumerate}[label=\textbullet]
\item $a_p<a_{p+1}<\cdots<a_{p+q}$ with respect to the usual order,
\item $\{d_r, \ldots, d_{r+h}, a_p, \ldots, a_{p+q}\}$ is a ladder in $\cP$,
\item $h$ is the biggest among such possible $(h,q)$'s, and
\item $q$ is the biggest among such possible $(h,q)$'s with $h$ chosen above.
\end{enumerate}
For later use, we define:
\begin{defn} We assume the situation above. Then the phrase ``maximality in \csii{}'' indicates the maximality of $(h,q)$ in $A$.
\end{defn}
\subsubsection{\und{\csii[(a)]}} First we suppose that $a_{p+q}<d_{r+h}$, i.e. $\max\{d_r, \ldots, d_{r+h}, a_p, \ldots, a_{p+q}\}=d_{r+h}$. (e.g. Example \ref{stairex0} and \ref{stairex2}) In this case we do not alter the chain $\fd$ and simply let $a_p, \ldots, a_{p+q}$ ``pass through the chain'', i.e. set $b_i\colonequals a_i$ for $i \in [p,p+q]$.
After this, we increase $p$ by $q+1$ and repeat the algorithm from the beginning.
\begin{rmk} Here, the maximality in \csii{} means that either $p+q=m$ or $a_{p+q+1}$ does not satisfy both $a_{p+q+1}\pr a_{p+q}$ and $a_{p+q+1} \pdr d_{r+h}$.
\end{rmk}

\subsubsection{\und{\csii[(b)]}} The remaining case is when $a_{p+q}>d_{r+h}$, i.e.  $\max\{d_r, \ldots, d_{r+h}, a_p, \ldots, a_{p+q}\}=a_{p+q}$. (e.g. Example \ref{stairex1} and \ref{stairex2}) We set $p-1=u(r-1)< u(r)< u(r+1)< \cdots< u(r+h)=p+q$ such that $d_{i}<a_{u(i-1)+1}<\cdots<a_{u(i)}$ for $i\in [r, r+h]$. 
Then we replace $d_r, \ldots, d_{r+h}$ on the chain in $\fd$ with $a_{u(r)}, \ldots, a_{u(r+h)}$, respectively. Furthermore, for $j \in [p, p+q]$ we set
$$b_j \colonequals \left\{\begin{aligned} &d_{i} &\textnormal{ if } j=u(i-1)+1 \textnormal{ for some } i \in [r,r+h],
\\ &a_{j-1} &\textnormal{ otherwise}.
\end{aligned}\right.$$ 
After this, we increase $p$ by $q+1$ and repeat the algorithm from the beginning.

\begin{rmk} Here, the maximality in \csii{} means that
\begin{enumerate}[label=\textbullet] 
\item either $p+q=m$ or $a_{p+q+1}$ does not satisfy both $a_{p+q+1}\pdr a_{p+q}$ and $a_{p+q+1} \pr a_{p+q}$ (or $a_{p+q+1} \pr d_{r+h}$ if $a_{p+q} \pdr d_{r+h}$), and
\item either $r+h=l$ or $d_{r+h+1}\not\pdr a_i$ for any $i \in [p,p+q]$.
\end{enumerate}
One may check that the second condition is equivalent to
\begin{enumerate}[label=\textbullet] 
\item either $r+h=l$ or $d_{r+h+1}\pr a_i$ for any $i \in [p,p+q]$.
\end{enumerate}
\end{rmk}
For later use, we define:
\begin{defn}
We say that $a_p$ is in \csi[(a)], 1(b), etc. if the step in the column insertion algorithm $\alg_\Phi$ processing $a_p$ corresponds to \csi[(a)], 1(b), etc.
\end{defn}
This finishes the description of the algorithm $\alg_\Phi$. See Section \ref{sec:exphi} for some examples about this algorithm. Before we proceed, we need to check that:
\begin{thm}\label{thm:wellphi} The algorithm $\alg_\Phi$ is well-defined, i.e. $\Phi(\alpha, c) \in \fC\fA$.
\end{thm}
\begin{proof} First suppose that $a_1, a_m \neq \vn$ and only one step of $\alg_\Phi$ is performed when calculating $\Phi((a_m, \ldots, a_1), (c_l, \ldots, c_1))=((d_{l'}, \ldots, d_1), (b_m, \ldots, b_1))$. Then we need to show that $b_i, d_j$ for $i \in[1,m]$, $j\in[1,l']$ are pairwise different (possibly except $\vn$) and $d_1\pl d_2 \pl \cdots \pl d_{l'}$. But the first part is clear from the assumption that $a_1, \ldots, a_m, c_1, \ldots, c_l$ are pairwise different. The second part is also easily checked case-by-case using \cond. Now the statement in the general case follows from induction on the number of steps.
\end{proof}

\subsection{Properties of $\alg_\Phi$}
Here we discuss some properties of $\alg_\Phi$. Firstly, if $\Phi(\alpha,c) = (d, \beta)$ then it is easy to observe the following. (We will use these facts without reminder later on.)
\begin{enumerate}[label=\textbullet]
\item if $a_i$ is in \csiii[(a)] or \csi[(a)], we have $b_i=\vn$,
\item if $a_i$ is in \csi[(b)], we have $a_i<b_i\neq \vn$,
\item if $a_i$ is in \csii[(a)], we have $a_i=b_i$, and
\item if $a_i$ is in \csii[(b)], we have $a_i\pdr b_i$.
\end{enumerate}
The following lemma is less trivial.
\begin{lem}\label{lem:smallelt} For $\alpha=(a_m, \ldots, a_1)$ and $c=(c_l, \ldots, c_1)$, suppose that there exists $i \in [1,l]$ such that $c_i \pl a_j$ for all $j\in [1,m]$. If $\Phi(\alpha, c) = ((d_{l'}, \ldots, d_1), -)$, then we have  $c_j=d_j$ for $j \in [1,i]$.
\end{lem}
\begin{proof} It is shown by case-by-case observation.
\end{proof}

The proofs of the following two propositions are provided in Section \ref{sec:proofprop}.

\begin{prop} \label{prop:column}
Suppose that $\Phi(\alpha, c) = (d, \beta)$ where $\alpha=(a_m, \ldots, a_1)$, $\beta=(b_m, \ldots, b_1)$, and $c=(c_l, \ldots, c_1)$. Write $\alpha^f$ (resp. $\beta^f$) to be the word obtained by removing $\vn$ from $\alpha$ (resp. $\beta$).
\begin{enumerate}[label=\textnormal{(\Alph*)}]
\item\label{mainprop1} $\alpha^f+c$ and $d+\beta^f$ are $\cP$-Knuth equivalent. In particular, $\und{\alpha^f+c}=\und{d+\beta^f}$ as sets.
\item\label{mainprop2} Suppose that $\alpha \in \fC$, $m\geq l$, and $(\alpha, c)$ satisfies the $\cP$-tableau condition, i.e. $a_i \not\pr c_i$ for $i \in [1,l]$. Then $d=\alpha$ and $\beta=(\vn,\ldots, \vn) +c$.
\item\label{mainprop3} If $a_i, a_{i+1} \neq \vn$ and $a_i \pl a_{i+1}$, then either $[b_{i+1} = \vn]$ or $[b_i, b_{i+1} \neq \vn$ and $b_i \pl b_{i+1}]$.
\item\label{mainprop4} If $a_i, a_{i+1} \neq \vn$ and $a_i \not\pl a_{i+1}$, then either $[b_{i} = \vn, b_{i+1}\neq \vn]$ or $[b_i, b_{i+1} \neq \vn$ and $b_i \not\pl b_{i+1}]$.
\end{enumerate}
\end{prop}

\begin{prop} \label{prop:phiinj} Suppose that $(\alpha=(a_m, \ldots, a_1), c), (\alpha'=(a'_m, \ldots, a'_1), c')\in \fA\fC$ satisfy $\Phi(\alpha,c) = \Phi(\alpha',c')$ and $a_i =\vn \Leftrightarrow a_i'=\vn$. Then we have $(\alpha,c) = (\alpha',c')$.
\end{prop}

\subsection{Another algorithm $\alg_\Psi$}
Here we introduce another column insertion algorithm $\alg_\Psi$ which resembles $\alg_\Phi$. This will not be used for the definition of the $\cP$-Robinson-Schensted algorithm in the next section, but it will play an important role when we prove Proposition \ref{prop:phiinj}. Also see Section \ref{sec:exphi} for some examples about this algorithm.

%

For a subset $X \subset \bZ_{>0}$, the algorithm $\alg_\Psi$ defines a function $\Psi_X: \fA\fC\rightarrow \fC\fA$ and is described by the pseudocode Algorithm \ref{algrev}. Note that the only difference between $\alg_{\Phi}$ and $\alg_{\Psi}$ is when $a_p=\vn$, $p \in X$, and $\fd\neq \emptyset$, which is as follows.

\noindent{\und{\csiii[(b)]}} Suppose that $a_p=\vn$, $p \in X$, and $\fd\neq \emptyset$. Then we ``drag the first entry of $\fd$ to $\fb$'', i.e. set $b_p\colonequals d_1$ and replace $\fd$ with $(d_l, \ldots, d_2)$. (As a result, the length of $\fd$ decreases by 1.) After this, we increase $p$ by 1 and repeat the algorithm from the beginning.

Indeed, if $X=\emptyset$ then $\alg_{\Psi}$ and $\Psi_X$ revert to $\alg_{\Phi}$ and $\Phi$, respectively.

\begin{algorithm}
\DontPrintSemicolon
\caption{Another algorithm $\alg_{\Psi}$}\label{algrev}
\Fn(\tcp*[f]{$((a_m, \ldots, a_1),(c_l, \ldots, c_1)) \in \fA\fC$}){$\Psi_X((a_m, \ldots, a_1),(c_l, \ldots, c_1))$}{
$m \gets |(a_m, \ldots, a_1)|$\;
\lFor(\tcp*[f]{Initialize $(b_m, \ldots, b_1)$ to $(\vn,\ldots, \vn)$}){$i \gets 1$ \KwTo $m$}{$b_i \gets \vn$}
$l \gets |(c_l, \ldots, c_1)|$\;
$d_0 \gets -\vn$\;
\lFor(\tcp*[f]{Initialize $(d_l, \ldots, d_1, d_0)$ to $(c_l, \ldots, c_1,-\vn)$}){$i \gets 1$ \KwTo $l$}{$d_i \gets c_i$}
$p \gets 1$\;
\While {$p\leq m$}{
	\If{$a_p=\vn$}{
		\If(\tcp*[f]{\csiii[(b)]}){$p\in X$ {\bf and} $l> 0$}{
			$b_p\gets d_1$\;
			\lFor{$i \gets 1$ \KwTo $l-1$}{$d_i \gets d_{i+1}$}
			$l\gets l-1$\;
		}
		$p \gets p+1$  \tcp*{\csiii[(a)]: pass if $a_p=\vn$ and $p\not\in X$}
	}
	\Else{
            	$r\gets \max \{i \in [0, l] \mid d_i<a_p\}$ \tcp*{Choose $r$ so that $d_r<a_p<d_{r+1}$}
            	\If(\tcp*[f]{\csi}) {$a_p \pr d_r$}
            	{
            		\lIf(\tcp*[f]{\csi[(a)]}){$r=l$}{
            			$l \gets l+1$  
            		}
            		\lElse(\tcp*[f]{\csi[(b)]}){
                    		$b_p\gets d_{r+1}$
            		}
			$d_{r+1} \gets a_p$\; 
            		$p \gets p+1$ 
            	}
            	\Else(\tcp*[f]{\csii})
		{
			$(h,q) \gets \max \{(i,j) \in \bN^2 \mid \{d_r, \ldots, d_{r+i}, a_p, \ldots, a_{p+j}\}$ is a ladder in $\cP$\; $\qquad\qquad\qquad\qquad\qquad\qquad \textnormal{ and } a_p<a_{p+1}<\cdots<a_{p+j}\}$\;
			\tcp*{The maximum is w.r.t. lexicographic order}
            		\If(\tcp*[f]{\csii[(a)]}) {$a_{p+q}<d_{r+h}$}{
            			\lFor{$j \gets 0$ \KwTo $q$}{{$b_{p+j} \gets a_{p+j}$}} 
            		}
			\Else(\tcp*[f]{\csii[(b)]})
			{
            			\For{$i \gets 0$ \KwTo $h$}
				{
            				$j \gets \min\{t \in [0, q] \mid a_{p+t}>d_{r+i}\}$\;
            				\lIf {$i=h$}{$k \gets q$}
					\lElse{$k \gets \max\{t \in [0, q-1] \mid a_{p+t}<d_{r+i+1}\}$}
            				$b_{p+j} \gets d_{r+i}$\;
            				\lFor {$t \gets j$ \KwTo $k-1$}{$b_{p+t+1} \gets a_{p+t}$}
            				$d_{r+i} \gets a_{p+k}$\;
            			}
            		}
			$p\gets p+q+1$\;
            	}
	}
}
\Return $((d_l, \ldots, d_1),(b_m, \ldots, b_1)) $\;
}
\end{algorithm}
\begin{lem} \label{lem:wellpsi}The algorithm $\Psi_X$ is well-defined.
\end{lem}
\begin{proof} It is proved in the same manner as Theorem \ref{thm:wellphi}.
\end{proof}

\section{$\cP$-Robinson-Schensted algorithm} \label{sec:RS}
In this section, we assume that a fixed natural unit interval order $\cP$ is given and define a $\cP$-Robinson-Schensted algorithm.
\subsection{$\cP$-Robinson-Schensted algorithm}
We identify $\ptab$ with the subset of $\fC^n$ such that $(c^1, c^2, \ldots, c^n)\in \fC^n$ corresponds to the $\cP$-tableau whose reading word is $c^1+c^2+\cdots+ c^n$ if such a $\cP$-tableau exists.
We define the $\cP$-Robinson-Schensted algorithm $\alg_{\prs}$ as in Algorithm \ref{alg2}.

\begin{algorithm}
\DontPrintSemicolon
\caption{$\cP$-Robinson-Schensted algorithm $\alg_{\prs}$}\label{alg2}
\Fn(\tcp*[f]{$(a_m, \ldots, a_1) \in \fA$}){$\prs(a_m, \ldots, a_1)$}{
$m \gets  |(a_m, \ldots, a_1)|$\;
\lFor(\tcp*[f]{Initialize $(b_m, \ldots, b_1)$ to $(a_m,\ldots, a_1)$}){$i \gets 1$ \KwTo $m$}{$b_i \gets a_i$}
$p \gets 0$\;
\While{$(b_m, \ldots, b_1)\neq (\vn,\ldots, \vn)$}{
	$p\gets p+1$\;
	$(PT_p, (t_m, \ldots, t_1)) \gets \Phi((b_m, \ldots, b_1),\emptyset)$\;
	\tcp*{$PT_p$ is a new column of the $\cP$-tableau}
	$QT_p \gets (i(k), \ldots, i(1)) \textnormal{ where } \left\{
	\begin{aligned} & \{i(1), \ldots, i(k)\} = \{j \in [1,m] \mid t_{j}=\vn, b_{j} \neq \vn\},
	\\ &1\leq i(1)<\cdots< i(k)\leq m
	\end{aligned}\right.$\;
	\tcp*{$QT_p$ is a new column of the standard Young tableau}
	\lFor{$i \gets 1$ \KwTo $m$}{$b_i \gets t_i$}
	\tcp*{The output of $\Phi$ is the new input on the next step}
}
\Return $((PT_1, \ldots, PT_p), (QT_1, \ldots, QT_p))$
}
\end{algorithm}

Let us describe the algorithm in detail. This algorithm takes an input $\alpha=(a_m, \ldots, a_1) \in \fA$ and produces an output $(PT, QT)$. Initialize $p$ with 0 and $(b_m, \ldots, b_1)$ with $(a_m, \ldots, a_1)$.

\subsubsection{} If $(b_m, \ldots, b_1)=(\vn, \ldots, \vn)$ then terminate the algorithm and return $(PT, QT)$ where $PT=(PT_1, \ldots, PT_p)$ and $QT=(QT_1, \ldots, QT_p)$.

\subsubsection{} Otherwise, we increase p by 1 and set $(PT_p, (t_m, \ldots, t_1))$ to be $\Phi((b_m, \ldots, b_1),\emptyset)$. Also we set $QT_p=(i(k), \ldots, i(1))$ where $i(1)< \cdots< i(k)$ are chosen such that $\{i(1), \ldots, i(k)\} = \{j \in [1.m] \mid t_j=\vn, b_j\neq \vn\}$, i.e. they are indices where \csi[(a)] of $\alg_\Phi$ occured in the calculation of $\Phi((b_m, \ldots, b_1),\emptyset)$. After this, we set $(b_m, \ldots, b_1)$ to be $(t_m, \ldots, t_1)$ and repeat the algorithm from the beginning.

This finishes the description of the algorithm $\alg_{\prs}$. It is clear that each column of $PT$ (resp. $QT$) is a chain with respect to $\cP$ (resp. the usual order). However, it is not clear at this moment that $PT$ (resp. $QT$) is a $\cP$-tableau (standard Young tableau). Indeed, it is not always so; see Section \ref{sec:patho} for such examples. However, we will observe that this algorithm behaves well when the given partial order on $[1,n]$ avoids $\cP_{(3,1,1),5}$ and $\cP_{(4,2,1,1),6}$.

\subsection{Properties of $\prs$}
Let $T_\lambda \in \SYT_\lambda$ be the standard Young tableau of shape $\lambda$ where $\lambda'=(l_1, l_2, \ldots, l_p)$ such that the $i$-th column of $T_\lambda$ consists of $(\sum_{k=1}^{i-1}l_k)+1, (\sum_{k=1}^{i-1}l_k)+2, \ldots, (\sum_{k=1}^{i}l_k)$. For example, we have $T_{(4,3,1)} = \ytableaushort{1468,257,3}$. The following theorem summarizes important properties of the algorithm $\alg_{\prs}$ which is proved in Section \ref{sec:proofthm}, together with the main theorem (Theorem \ref{thm:mainstrong}) of this paper.
\begin{thm} \label{thm:main}
Suppose that $\cP$ avoids $\cP_{(3,1,1),5}$ and $\cP_{(4,2,1,1),6}$, i.e. $\cP$ is not ladder-climbing. Then the following are satisfied.
\begin{enumerate}[label=\textnormal{(\Alph*)}]
\item \label{mainthm1} For $w\in \sym_n$, if $\prs(w)=(PT, QT)$ then $PT$ is a $\cP$-tableau and $QT$ is a standard Young tableau.
\item \label{mainthm2} For $w\in \sym_n$, if $\prs(w) = (PT, QT)$ then $\{n-x \mid x \in \des_\cP(w)\} = \des(QT)$.
\item \label{mainthm3} For $w\in \sym_n$, If $\prs(w)=(PT, QT)$ then $w\psim \rw(PT)$.
\item \label{mainthm3.5} For $w \in \sym_n$, if $\prs(w)=(PT, QT)$ then the length of the first column of $PT$ is equal to $\ght_\cP(w)$. Furthermore, if $w' \in \sym_n$ satisfies $w\sim_\cP w'$ and $\prs(w') = (PT', QT')$, then the lengths of the first column of $PT$ and $PT'$ are the same.
\item \label{mainthm4} If $w=\rw(PT)$ for some $PT \in \ptab_\lambda$ then $\prs(w) = (PT, \omega(T_\lambda))$ where $\omega : \SYT_\lambda \rightarrow \SYT_\lambda$ is Sch\"utzenberger's evacuation. 
\item \label{mainthm5} If $\alpha=(a_m, \ldots, a_1)$ and $\alpha'=(a'_m, \ldots, a_1')$ are two words of the same length then $\prs(\alpha)=\prs(\alpha') \Leftrightarrow \alpha=\alpha'$.
\item \label{mainthm6} $\prs$ restricts to a bijection $\prs: \sym_n \xrightarrow{\sim} \bigsqcup_{\lambda \vdash n} \ptab_\lambda \times \SYT_\lambda$.
\end{enumerate}
\end{thm}

\section{Examples}\label{sec:examples}
In this section we give various examples of the objects that we introduced so far.
\subsection{$\cP$-Knuth equivalence graphs} \label{sec:equivgraph}
Here we provide some examples of $\cP$-Knuth equivalence graphs whose generating functions are not a single Schur function. In Figure \ref{fig:graphexbegin}--\ref{fig:graphexend}, underlined numbers in each word denote its descents and numbers above edges indicate their colors. Vertices with bold borders are reading words of some $\cP$-tableaux and vertices of the same colors are the ones that give the same $\cP$-tableau under $\prs$. Also, two gray vertices in Figure \ref{fig:311graph} are the ones that insert to \ytableaushort{312,54}, which is not a $\cP$-tableau for $\cP=\cP_{(3,1,1),5}$.

\begin{figure}[H]
\begin{tikzpicture}[xscale=3,yscale=1.2]
	\node[draw,very thick] (node_5) at (0,0)  {$\und{4}2\und{3}15$};
	\node[draw,red] (node_0) at (0,-1)  {$4\und{3}125$};
	\node[draw,very thick,red] (node_2) at (0,-2)  {$\und{4}1325$};
	\node[draw] (node_1) at (-1,0)  {$3\und{4}215$};
	\node[draw,red] (node_8) at (-1,-1)  {$32\und{4}15$};
	\node[draw,red] (node_3) at (-1,-2)  {$324\und{5}1$};
	\node[draw] (node_6) at (1,0)  {$\und{4}23\und{5}1$};
	\node[draw] (node_7) at (1,-1)  {$3\und{4}2\und{5}1$};
	\node[draw] (node_4) at (1,-2)  {$34\und{5}21$};
  
  \draw [black,] (node_7) -- (node_6) node [midway,left]  {$1$};
  \draw [black,] (node_8) -- (node_1) node [midway,left] {$2$};
  \draw [black,] (node_1) -- (node_5) node [midway,above] {$1$};
  \draw [black,] (node_4) -- (node_7) node [midway,left] {$2,3$};
  \draw [black,] (node_3) -- (node_8) node [midway,left] {$3$};
  \draw [black,] (node_0) -- (node_5) node [midway,left] {$2$};
  \draw [black,] (node_6) -- (node_5) node [midway,above] {$3$};
  \draw [black,] (node_0) -- (node_2) node [midway,left] {$1$};
\end{tikzpicture}
\caption{$\cP=\cP_{(2,2,1),5}, \gamma_V=t^2(s_{32}+s_{41})$}\label{fig:graphexbegin}
\end{figure}

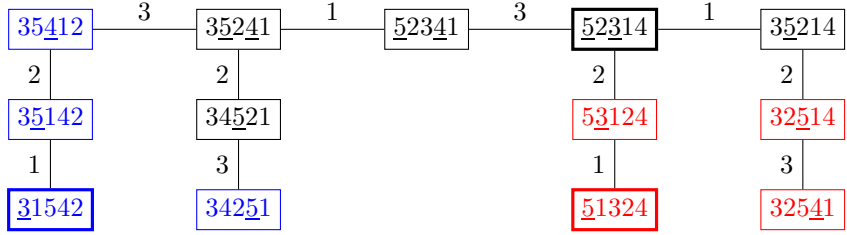
\begin{figure}[H]
\begin{tikzpicture}[xscale=2.5,yscale=1.2]
	\node[draw] (node_10) at (0,0)  {$\und{5}23\und{4}1$};
	\node[draw] (node_7) at (-1,0)  {$3\und{5}2\und{4}1$};
	\node[draw, very thick] (node_2) at (1,0)  {$\und{5}2\und{3}14$};
	
	\node[draw,blue] (node_5) at (-2,0)  {$35\und{4}12$};
	\node[draw,blue] (node_0) at (-2,-1)  {$3\und{5}142$};
	\node[draw, very thick,blue] (node_3) at (-2,-2)  {$\und{3}1542$};
	
	\node[draw] (node_6) at (-1,-1)  {$34\und{5}21$};
	\node[draw,blue] (node_4) at (-1,-2)  {$342\und{5}1$};
	
	\node[draw] (node_1) at (2,0)  {$3\und{5}214$};
	\node[draw,red] (node_12) at (2,-1)  {$32\und{5}14$};
	\node[draw,red] (node_9) at (2,-2)  {$325\und{4}1$};
	
	\node[draw,red] (node_11) at (1,-1)  {$5\und{3}124$};
	\node[draw, very thick,red] (node_8) at (1,-2)  {$\und{5}1324$};

  \draw [black,] (node_0) -- (node_3) node [midway,left]   {$1$};
  \draw [black,] (node_5) -- (node_7) node [midway,above]   {$3$};
  \draw [black,] (node_10) -- (node_2) node [midway,above]  {$3$};
  \draw [black,] (node_12) --  (node_1) node [midway,left]  {$2$};
  \draw [black,] (node_6) --  (node_7) node [midway,left] {$2$};
  \draw [black,] (node_4) --  (node_6) node [midway,left] {$3$};
  \draw [black,] (node_11) -- (node_8) node [midway,left] {$1$};
  \draw [black,] (node_1) -- (node_2) node [midway,above] {$1$};
  \draw [black,] (node_5) -- (node_0) node [midway,left]  {$2$};
  \draw [black,] (node_11) --  (node_2) node [midway,left]  {$2$};
  \draw [black,] (node_9) --  (node_12) node [midway,left] {$3$};
  \draw [black,] (node_7)--  (node_10) node [midway,above]  {$1$};
\end{tikzpicture}
\caption{$\cP=\cP_{(2,1,1),5}, \gamma_V=t^3(s_{32} + 2s_{41})$}
\end{figure}

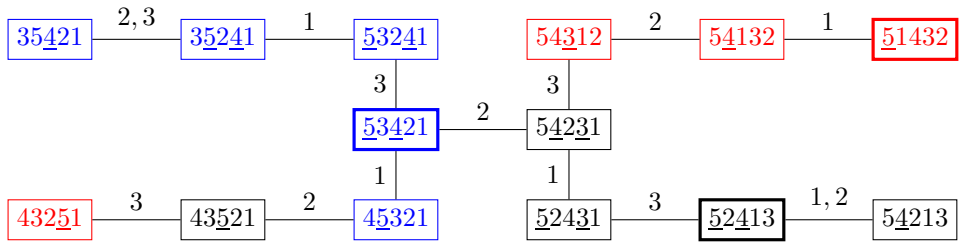
\begin{figure}[H]
\begin{tikzpicture}[xscale=2.3,yscale=1.2]
	\node[draw,red] (node_0) at (0,0) {$432\und{5}1$};
	\node[draw] (node_3) at (1,0) {$43\und{5}21$};
	\node[draw,blue] (node_2) at (2,0)  {$4\und{5}321$};
	\node[draw,very thick,blue] (node_1) at (2,1) {$\und{5}3\und{4}21$};
	\node[draw,blue] (node_10) at (2,2)  {$\und{5}32\und{4}1$};
	\node[draw,blue] (node_7) at (1,2) {$3\und{5}2\und{4}1$};
	\node[draw,blue] (node_11) at (0,2) {$35\und{4}21$};
	
	\node[draw] (node_13) at (3,1) {$5\und{4}2\und{3}1$};
	\node[draw] (node_9) at (3,0) {$\und{5}24\und{3}1$};
	\node[draw,very thick] (node_5) at (4,0) {$\und{5}2\und{4}13$};
	\node[draw] (node_12) at (5,0)  {$5\und{4}213$};
	\node[draw,red] (node_8) at (3,2) {$54\und{3}12$};
	\node[draw,red] (node_6) at (4,2) {$5\und{4}132$};
	\node[draw,very thick,red] (node_4) at (5,2) {$\und{5}1432$};

  \draw [black,] (node_0) -- (node_3) node [midway,above] {$3$};
  \draw [black,] (node_13)--(node_1) node [midway,above] {$2$};
  \draw [black,] (node_2) -- (node_1) node [midway,left] {$1$};
  \draw [black,] (node_9) -- (node_5) node [midway,above] {$3$};
  \draw [black,] (node_8) -- (node_13) node [midway,left] {$3$};
  \draw [black,] (node_13)--(node_9) node [midway,left] {$1$};
  \draw [black,] (node_3) -- (node_2) node [midway,above] {$2$};
  \draw [black,] (node_6) -- (node_4) node [midway,above] {$1$};
  \draw [black,] (node_10) -- (node_1) node [midway,left] {$3$};
  \draw [black,] (node_12) --(node_5) node [midway,above] {$1, 2$};
  \draw [black,] (node_11) -- (node_7) node [midway,above] {$2, 3$};
  \draw [black,] (node_7) --(node_10) node [midway,above]  {$1$};
  \draw [black,] (node_8) -- (node_6) node [midway,above] {$2$};
\end{tikzpicture}
\caption{$\cP=\cP_{(3,2,1),5}, \gamma_V=t^3(2s_{32}+s_{41})$}
\end{figure}

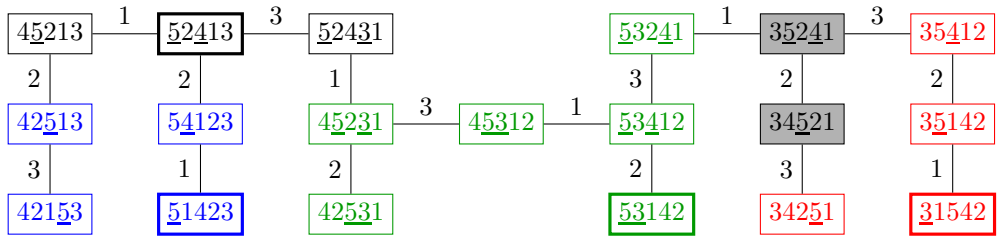
\begin{figure}[H]
\begin{tikzpicture}[xscale=2,yscale=1.2]
	\node[draw,black!40!green] (node_7) at (0,0)  {$4\und{5}\und{3}12$};
	\node[draw,black!40!green] (node_14) at (-1,0)  {$4\und{5}2\und{3}1$};
	\node[draw,black!40!green] (node_18) at (1,0)  {$\und{5}3\und{4}12$};
	\node[draw,black!40!green] (node_15) at (-1,-1)  {$42\und{5}\und{3}1$};
	\node[draw, very thick,black!40!green] (node_3) at (1,-1)  {$\und{5}\und{3}142$};
	
	\node[draw,black!40!green] (node_1) at (1,1)  {$\und{5}32\und{4}1$};
	\node[draw,fill=black!30!white] (node_8) at (2,1)  {$3\und{5}2\und{4}1$};
	\node[draw,red] (node_13) at (3,1)  {$35\und{4}12$};
	\node[draw,fill=black!30!white] (node_10) at (2,0)  {$34\und{5}21$};
	\node[draw,red] (node_11) at (2,-1)  {$342\und{5}1$};
	\node[draw,red] (node_0) at (3,0)  {$3\und{5}142$};
	\node[draw, very thick,red] (node_5) at (3,-1)  {$\und{3}1542$};
	
	\node[draw] (node_9) at (-1,1)  {$\und{5}24\und{3}1$};
	\node[draw, very thick] (node_12) at (-2,1)  {$\und{5}2\und{4}13$};
	\node[draw] (node_2) at (-3,1)  {$4\und{5}213$};
	\node[draw,blue] (node_4) at (-3,0)  {$42\und{5}13$};
	\node[draw,blue] (node_17) at (-3,-1)  {$421\und{5}3$};
	\node[draw,blue] (node_6) at (-2,0)  {$5\und{4}123$};
	\node[draw, very thick,blue] (node_16) at (-2,-1)  {$\und{5}1423$};

  \draw [black,] (node_6) -- (node_12) node [midway,left] {$2$};
  \draw [black,] (node_8) --  (node_1) node [midway,above]  {$1$};
  \draw [black,] (node_15) -- (node_14) node [midway,left]  {$2$};
  \draw [black,] (node_2) -- (node_12) node [midway,above]  {$1$};
  \draw [black,] (node_14) -- (node_7) node [midway,above]  {$3$};
  \draw [black,] (node_13) -- (node_0) node [midway,left]  {$2$};
  \draw [black,] (node_1) --  (node_18) node [midway,left]  {$3$};
  \draw [black,] (node_6) -- (node_16) node [midway,left]  {$1$};
  \draw [black,] (node_0) -- (node_5) node [midway,left]  {$1$};
  \draw [black,] (node_18) --  (node_3) node [midway,left]  {$2$};
  \draw [black,] (node_11) --  (node_10) node [midway,left]  {$3$};
  \draw [black,] (node_10) -- (node_8) node [midway,left] {$2$};
  \draw [black,] (node_7) --  (node_18) node [midway,above]  {$1$};
  \draw [black,] (node_9) -- (node_12) node [midway,above]  {$3$};
  \draw [black,] (node_17) -- (node_4) node [midway,left]  {$3$};
  \draw [black,] (node_13) -- (node_8) node [midway,above]  {$3$};
  \draw [black,] (node_14) --  (node_9) node [midway,left]  {$1$};
  \draw [black,] (node_4) -- (node_2) node [midway,left]  {$2$};
\end{tikzpicture}
\caption{$\cP=\cP_{(3,1,1),5}, \gamma_V=t^3(s_{311} + s_{32} + 2s_{41})$}\label{fig:311graph}
\end{figure}

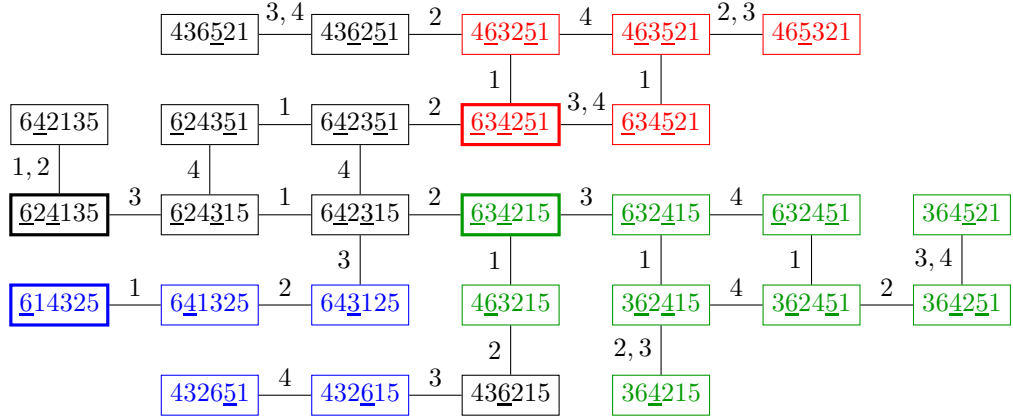
\begin{figure}[H]
\begin{tikzpicture}[xscale=1,yscale=0.6]
	\node (node_2) at (0,0) [draw] {$6\und{4}2\und{3}15$};
	\node (node_20) at (0,2) [draw] {$6\und{4}23\und{5}1$};
	\node (node_9) at (-2,0) [draw] {$\und{6}24\und{3}15$};
	\node (node_0) at (0,-2) [draw,blue] {$64\und{3}125$};
	\node (node_17) at (2,0) [draw,very thick,black!40!green] {$\und{6}3\und{4}215$};
	
	\node (node_11) at (-2,2) [draw] {$\und{6}243\und{5}1$};
	\node (node_3) at (-4,0) [draw,very thick] {$\und{6}2\und{4}135$};
	\node (node_23) at (-4,2) [draw] {$6\und{4}2135$};
	\node (node_24) at (-2,-2) [draw,blue] {$6\und{4}1325$};
	\node (node_25) at (-4,-2) [draw,very thick,blue] {$\und{6}14325$};
	
	\node (node_16) at (2,2) [draw,very thick,red] {$\und{6}3\und{4}2\und{5}1$};
	\node (node_5) at (2,4) [draw,red] {$4\und{6}32\und{5}1$};
	\node (node_7) at (0,4) [draw] {$43\und{6}2\und{5}1$};
	\node (node_26) at (-2,4) [draw] {$436\und{5}21$};
	\node (node_14) at (4,2) [draw,red] {$\und{6}34\und{5}21$};
	\node (node_21) at (4,4) [draw,red] {$4\und{6}3\und{5}21$};
	\node (node_6) at (6,4) [draw,red] {$46\und{5}321$};
	
	\node (node_8) at (2,-2) [draw,black!40!green] {$4\und{6}3215$};
	\node (node_4) at (2,-4) [draw] {$43\und{6}215$};
	\node (node_1) at (0,-4) [draw,blue] {$432\und{6}15$};
	\node (node_19) at (-2,-4) [draw,blue] {$4326\und{5}1$};
	
	\node (node_10) at (4,0) [draw,black!40!green] {$\und{6}32\und{4}15$};
	\node (node_27) at (4,-2) [draw,black!40!green] {$3\und{6}2\und{4}15$};
	\node (node_18) at (4,-4) [draw,black!40!green] {$36\und{4}215$};

	\node (node_15) at (6,-2) [draw,black!40!green] {$3\und{6}24\und{5}1$};
	\node (node_13) at (6,0) [draw,black!40!green] {$\und{6}324\und{5}1$};
	\node (node_22) at (8,-2) [draw,black!40!green] {$36\und{4}2\und{5}1$};
	\node (node_12) at (8,0) [draw,black!40!green] {$364\und{5}21$};

  \draw [black,] (node_0) -- (node_2) node [midway,left]  {$3$};
  \draw [black,] (node_21) -- (node_14) node [midway,left]  {$1$};
  \draw [black,] (node_15) -- (node_13) node [midway,left] {$1$};
  \draw [black,] (node_24) --(node_25) node [midway,above]  {$1$};
  \draw [black,] (node_5) -- (node_21) node [midway,above]  {$4$};
  \draw [black,] (node_12) -- (node_22) node [midway,left] {$3,4$};
  \draw [black,] (node_26) -- (node_7) node [midway,above]  {$3,4$};
  \draw [black,] (node_7) -- (node_5) node [midway,above]  {$2$};
  \draw [black,] (node_6) --(node_21) node [midway,above]  {$2,3$};
  \draw [black,] (node_22)-- (node_15) node [midway,above]  {$2$};
  \draw [black,] (node_2) -- (node_17) node [midway,above]  {$2$};
  \draw [black,] (node_19) -- (node_1) node [midway,above]  {$4$};
  \draw [black,] (node_20) --(node_2) node [midway,left]  {$4$};
  \draw [black,] (node_14) -- (node_16) node [midway,above]  {$3,4$};
  \draw [black,] (node_2) --(node_9) node [midway,above]  {$1$};
  \draw [black,] (node_10) -- (node_17) node [midway,above]  {$3$};
  \draw [black,] (node_18) -- (node_27) node [midway,left]  {$2,3$};
  \draw [black,] (node_9) --(node_3) node [midway,above]  {$3$};
  \draw [black,] (node_20)-- (node_11) node [midway,above]  {$1$};
  \draw [black,] (node_8) -- (node_17) node [midway,left]  {$1$};
  \draw [black,] (node_5) -- (node_16) node [midway,left]  {$1$};
  \draw [black,] (node_1) --(node_4) node [midway,above]  {$3$};
  \draw [black,] (node_4)--(node_8) node [midway,left]  {$2$};
  \draw [black,] (node_13) --(node_10) node [midway,above]  {$4$};
  \draw [black,] (node_15)-- (node_27) node [midway,above]  {$4$};
  \draw [black,] (node_11)-- (node_9) node [midway,left] {$4$};
  \draw [black,] (node_20)-- (node_16) node [midway,above]  {$2$};
  \draw [black,] (node_0) -- (node_24) node [midway,above]  {$2$};
  \draw [black,] (node_23) -- (node_3) node [midway,left]  {$1,2$};
  \draw [black,] (node_27)-- (node_10) node [midway,left]  {$1$};
\end{tikzpicture}
\caption{$\cP=\cP_{(3,3,2,1),6}, \gamma_V=t^4(s_{33} + 2s_{42} + s_{51})$}
\end{figure}

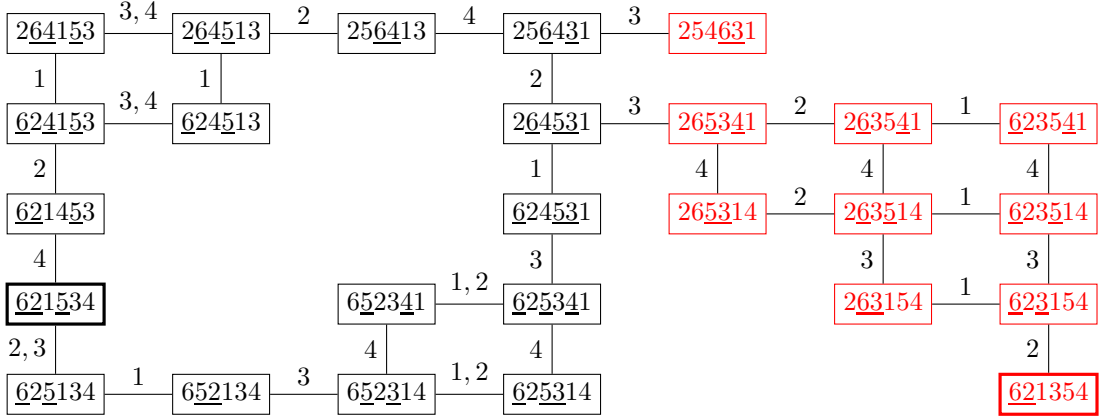
\begin{figure}[H]
\begin{tikzpicture}[xscale=1.1, yscale=0.6]
	\node (node_14) at (-4,2) [draw] {$\und{6}2\und{4}1\und{5}3$};
	\node (node_12) at (-4,4) [draw] {$2\und{6}\und{4}1\und{5}3$};
	\node (node_10) at (-2,2) [draw] {$\und{6}24\und{5}13$};
	\node (node_5) at (-2,4) [draw] {$2\und{6}4\und{5}13$};
	
	\node (node_22) at (-4,0) [draw] {$\und{6}\und{2}14\und{5}3$};
	\node (node_4) at (-4,-2) [draw, very thick] {$\und{6}\und{2}1\und{5}34$};
	\node (node_9) at (-4,-4) [draw] {$\und{6}2\und{5}134$};
	\node (node_1) at (-2,-4) [draw] {$6\und{5}\und{2}134$};
	
	\node (node_18) at (0,-4) [draw] {$6\und{5}2\und{3}14$};
	\node (node_0) at (2,-4) [draw] {$\und{6}2\und{5}\und{3}14$};
	\node (node_8) at (0,-2) [draw] {$6\und{5}23\und{4}1$};
	\node (node_20) at (2,-2) [draw] {$\und{6}2\und{5}3\und{4}1$};
	
	\node (node_6) at (2,0) [draw] {$\und{6}24\und{5}\und{3}1$};
	\node (node_13) at (2,2) [draw] {$2\und{6}4\und{5}\und{3}1$};
	\node (node_24) at (4,4) [draw,red] {$254\und{6}\und{3}1$};
	\node (node_2) at (0,4) [draw] {$25\und{6}\und{4}13$};
	\node (node_15) at (2,4) [draw] {$25\und{6}4\und{3}1$};
	
	\node (node_23) at (4,2) [draw,red] {$26\und{5}3\und{4}1$};
	\node (node_21) at (4,0) [draw,red] {$26\und{5}\und{3}14$};
	\node (node_25) at (6,0) [draw,red] {$2\und{6}3\und{5}14$};
	\node (node_19) at (6,2) [draw,red] {$2\und{6}35\und{4}1$};
	
	\node (node_17) at (8,2) [draw,red] {$\und{6}235\und{4}1$};
	\node (node_16) at (8,0) [draw,red] {$\und{6}23\und{5}14$};
	\node (node_11) at (6,-2) [draw,red] {$2\und{6}\und{3}154$};
	\node (node_7) at (8,-2) [draw,red] {$\und{6}2\und{3}154$};
	\node (node_3) at (8,-4) [draw, very thick,red] {$\und{6}\und{2}1354$};

  \draw [black,] (node_20) -- (node_0) node [midway,left]  {$4$};
  \draw [black,] (node_25)-- (node_11) node [midway,left]  {$3$};
  \draw [black,] (node_17) -- (node_16) node [midway,left]  {$4$};
  \draw [black,] (node_8) -- (node_20) node [midway,above]  {$1,2$};
  \draw [black,] (node_11)-- (node_7) node [midway,above]   {$1$};
  \draw [black,] (node_19)--(node_17) node [midway,above]  {$1$};
  \draw [black,] (node_16) -- (node_7) node [midway,left]   {$3$};
  \draw [black,] (node_15) --(node_13) node [midway,left]  {$2$};
  \draw [black,] (node_22)-- (node_4) node [midway,left]   {$4$};
  \draw [black,] (node_5)-- (node_10) node [midway,left]   {$1$};
  \draw [black,] (node_13) -- (node_6) node [midway,left]  {$1$};
  \draw [black,] (node_12) -- (node_14) node [midway,left]  {$1$};
  \draw [black,] (node_19) -- (node_25) node [midway,left]  {$4$};
  \draw [black,] (node_23) -- (node_13) node [midway,above]  {$3$};
  \draw [black,] (node_25) -- (node_16) node [midway,above]  {$1$};
  \draw [black,] (node_18) -- (node_0) node [midway,above]   {$1,2$};
  \draw [black,] (node_9) -- (node_4) node [midway,left]  {$2,3$};
  \draw [black,] (node_23) --(node_21) node [midway,left]  {$4$};
  \draw [black,] (node_6) -- (node_20) node [midway,left]   {$3$};
  \draw [black,] (node_1) --(node_9) node [midway,above]   {$1$};
  \draw [black,] (node_14)--(node_22) node [midway,left]   {$2$};
  \draw [black,] (node_2) -- (node_5) node [midway,above]  {$2$};
  \draw [black,] (node_21) --(node_25) node [midway,above]  {$2$};
  \draw [black,] (node_18) --(node_1) node [midway,above]   {$3$};
  \draw [black,] (node_7) -- (node_3) node [midway,left]   {$2$};
  \draw [black,] (node_15)-- (node_2) node [midway,above]   {$4$};
  \draw [black,] (node_8) -- (node_18) node [midway,left]  {$4$};
  \draw [black,] (node_10) -- (node_14) node [midway,above]   {$3,4$};
  \draw [black,] (node_23) --(node_19) node [midway,above]   {$2$};
  \draw [black,] (node_5) --(node_12) node [midway,above]   {$3,4$};
  \draw [black,] (node_24)-- (node_15) node [midway,above]  {$3$};
\end{tikzpicture}
\caption{$\cP=\cP_{(4, 3, 2, 1, 1),6}, \gamma_V=t^2(s_{321} + s_{411})$}\label{fig:graphexend}
\end{figure}

\subsection{Some examples of $\alg_\Phi$ and $\alg_\Psi$}\label{sec:exphi}
Here we provide some examples how the algorithms $\alg_\Phi$ and $\alg_\Psi$ work.
\begin{example}[Figure \ref{fig:stairex0}] \label{stairex0} Suppose that $\cP=\cP_{(5,3,2,1),6}$, $\alpha=(4,3,2)$, and $c=(6,5,1)$. In this case only one step of $\alg_\Phi$ is required to calculate $\Phi(\alpha, c)$, i.e. \csii[(a)]. Here, the set $A$ in the description of \csii{} is equal to $\{(0,0),(0,1),(0,2),(1,2)\}$, and thus we have $(h,q)=(1,2)$. It follows that $\Phi(\alpha, c) = ((6,5,1), (4,3,2))$.
\end{example}
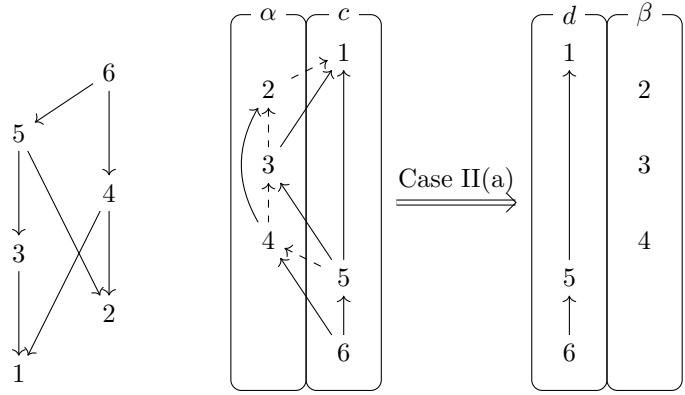
\begin{figure}[!htbp]
\begin{tikzpicture}[rotate=270, xscale=-0.8, yscale=1.2]
	\node (1) at (0,0) {1};
	\node (3) at (2,0) {3};
	\node (5) at (4,0) {5};
	
	\node (2) at (1,1) {2};
	\node (4) at (3,1) {4};
	\node (6) at (5,1) {6};

	\draw[<-] (1) to (3);
	\draw[<-] (3) to (5);

	\draw[<-] (2) to (4);
	\draw[<-] (4) to (6);
	
	\draw[<-] (1) to (4);
	\draw[<-] (2) to (5);
	\draw[<-] (5) to (6);
\end{tikzpicture}\qquad\qquad 
\begin{tikzpicture}[]
	\tikzmath{\x=0.5;\y=1;}
	\node (1) at (\y, 0) {1};
	\node (2) at (0, -0.5) {2};
	\node (3) at (0, -1.5) {3};
	\node (4) at (0, -2.5) {4};
	\node (5) at (\y, -3) {5};
	\node (6) at (\y, -4) {6};
	\draw [<-] (1) to (3);
	\draw [<-, out=240, in=120] (2) to (4);
	\draw [<-] (3) to (5);
	\draw [<-] (1) to (5);
	\draw [<-] (4) to (6);
	\draw [<-] (5) to (6);
	\draw [dashed,<-] (1) to (2);
	\draw [dashed,<-] (2) to (3);
	\draw [dashed,<-] (3) to (4);
	\draw [dashed,<-] (4) to (5);
	\draw[rounded corners] (-.5, 0.5) rectangle (.5, -4.5);
	\node[fill=white] at (0,0.5) {$\alpha$};
	\draw[rounded corners] (\y-.5, 0.5) rectangle (\y+0.5, -4.5);
	\node[fill=white] at (\y,0.5) {$c$};
	
	\draw[ddarr] (\y+0.5+0.2,-2) --  (\y+0.5+1.8,-2) node [midway,above] {\csii[(a)]};

	\node (1) at (4*\y, 0) {1};
	\node (2) at (5*\y, -0.5) {2};
	\node (3) at (5*\y, -1.5) {3};
	\node (4) at (5*\y, -2.5) {4};
	\node (5) at (4*\y, -3) {5};
	\node (6) at (4*\y, -4) {6};
	\draw [<-] (1) to (5);
	\draw [<-] (5) to (6);
	\draw[rounded corners] (-.5+4*\y, 0.5) rectangle (.5+4*\y, -4.5);
	\node[fill=white] at (4*\y,0.5) {$d$};
	\draw[rounded corners] (5*\y-.5, 0.5) rectangle (5*\y+0.5, -4.5);
	\node[fill=white] at (5*\y,0.5) {$\beta$};
\end{tikzpicture} 
\caption{$\alg_\Phi$: $\cP=\cP_{(5,3,2,1),6}$, $\alpha=(4,3,2)$, and $c=(6,5,1)$}\label{fig:stairex0}
\end{figure}

\begin{example}[Figure \ref{fig:stairex1}] \label{stairex1} Suppose that $\cP=\cP_{\stair(6),7}$, $\alpha=(7,5,4,2)$, and $c=(6,3,1)$. Similarly to above, in this case only one step of $\alg_\Phi$ is required to calculate $\Phi(\alpha, c)$, i.e. \csii[(b)]. Here, the set $A$ in the description of \csii{} is equal to $A=\{(0,0),(1,0),(1,1),(1,2), (2,2), (2,3)\}$, and thus we have $(h,q)=(2,3)$. It follows that $\Phi(\alpha, c) = ((7,5,2), (6,4,3,1))$.
\end{example}
\begin{figure}[!htbp]
\begin{tikzpicture}[rotate=270, xscale=-0.8, yscale=1.2]
	\node (1) at (0,0) {1};
	\node (3) at (2,0) {3};
	\node (5) at (4,0) {5};
	\node (7) at (6,0) {7};
	
	\node (2) at (1,1) {2};
	\node (4) at (3,1) {4};
	\node (6) at (5,1) {6};

	\draw[<-] (1) to (3);
	\draw[<-] (3) to (5);
	\draw[<-] (5) to (7);

	\draw[<-] (2) to (4);
	\draw[<-] (4) to (6);
	
	\draw[<-] (1) to (4);
	\draw[<-] (2) to (5);
	\draw[<-] (3) to (6);
	\draw[<-] (4) to (7);
\end{tikzpicture}\qquad\qquad 
\begin{tikzpicture}[]
	\tikzmath{\x=0.5;\y=1;}
	\node (1) at (\y, 0) {1};
	\node (2) at (0, -0.5) {2};
	\node (3) at (\y, -1) {3};
	\node (4) at (0, -1.5) {4};
	\node (5) at (0, -2.5) {5};
	\node (6) at (\y, -3) {6};
	\node (7) at (0, -3.5) {7};
	\draw [<-] (1) to (3);
	\draw [<-] (2) to (4);
	\draw [<-] (3) to (5);
	\draw [<-] (4) to (6);
	\draw [<-] (3) to (6);
	\draw [<-] (5) to (7);
	\draw [dashed,<-] (1) to (2);
	\draw [dashed,<-] (2) to (3);
	\draw [dashed,<-] (3) to (4);
	\draw [dashed,<-] (4) to (5);
	\draw [dashed,<-] (5) to (6);
	\draw [dashed,<-] (6) to (7);
	\draw[rounded corners] (-.5, 0.5) rectangle (.5, -4.5);
	\node[fill=white] at (0,0.5) {$\alpha$};
	\draw[rounded corners] (\y-.5, 0.5) rectangle (\y+0.5, -4.5);
	\node[fill=white] at (\y,0.5) {$c$};
	
	\draw[ddarr] (\y+0.5+0.2,-2) --  (\y+0.5+1.8,-2) node [midway,above] {\csii[(b)]};

	\node (2) at (5*\y, 0) {1};
	\node (1) at (4*\y, -0.5) {2};

	\node (4) at (5*\y, -1) {3};
	\node (5) at (5*\y, -2) {4};
	\node (3) at (4*\y, -2.5) {5};
	\node (7) at (5*\y, -3) {6};
	\node (6) at (4*\y, -3.5) {7};

	\draw [<-] (1) to (3);
	\draw [<-] (3) to (6);
	\draw[rounded corners] (-.5+4*\y, 0.5) rectangle (.5+4*\y, -4.5);
	\node[fill=white] at (4*\y,0.5) {$d$};
	\draw[rounded corners] (5*\y-.5, 0.5) rectangle (5*\y+0.5, -4.5);
	\node[fill=white] at (5*\y,0.5) {$\beta$};
\end{tikzpicture} 
\caption{$\alg_\Phi$: $\cP=\cP_{\stair(6),7}$, $\alpha=(7,5,4,2)$, and $c=(6,3,1)$}\label{fig:stairex1}
\end{figure}

\begin{example}[Figure \ref{fig:stairex2}] \label{stairex2} Suppose that $\cP=\cP_{(2,1,1),5}$, $\alpha=(5,4,2)$ and $c=(3,1)$. We need to process three steps of $\alg_\Phi$ to calculate $\Phi(\alpha,c)$ in this case. 
\begin{enumerate}
\item Since $a_1=2 \pdr 1=d_1$, $a_1$ is in \csii{}. Here $A=\{(0,0), (1,0), (0,1), (0,2)\}$, and thus we have $(h,q)=(1,0)$ that is the maximum of $A$ with respect to the lexicographic order even if the choice of $(0,2)$ produces a bigger ladder. We set $b_1=2$, $p=2$ and continue.
\item Since $a_2=4 \pdr 3=d_2$, $a_2$ is in \csii{}. Here $A=\{(0,0)\}$, and thus $(h,q)=(1,0)$. Note that $\{3, 4, 5\}$ is not a ladder in $\cP$ since $3 \not \pl 5$. We set $b_2=3$, $d_2=4$, $p=3$ and continue.
\item Since $a_3=5\pdr 4=d_2$, $a_3$ is in \csii{}. Here $A=\{(0,0)\}$, and thus $(h,q)=(1,0)$. We set $b_3=4$, $d_3=5$, and terminate the algorithm.
\end{enumerate}
As a result, we have $\Phi((5,4,2), (3,1)) = ((5,1), (4,3,2))$.
\end{example}
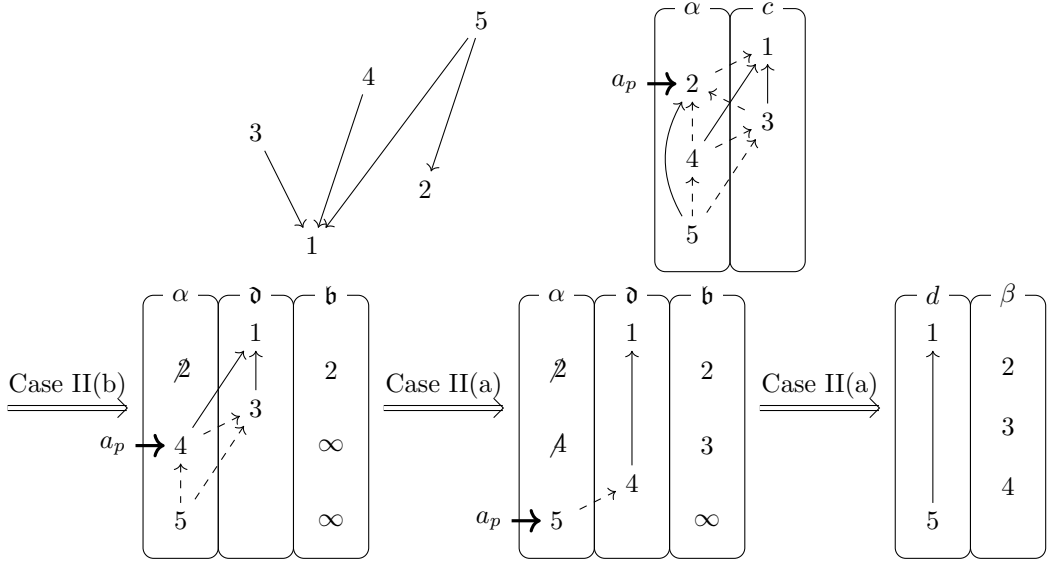
\begin{figure}[!htbp]
\begin{tikzpicture}[baseline=(current bounding box.center), scale=1.5]
\node (node_4) at (1.5,2) [draw,draw=none] {$5$};
  \node (node_3) at (0.5,1.5) [draw,draw=none] {$4$};
  \node (node_2) at (-0.5,1) [draw,draw=none] {$3$};
  \node (node_1) at (1,0.5) [draw,draw=none] {$2$};
  \node (node_0) at (0,0) [draw,draw=none] {$1$};
  \draw [black,<-] (node_0) to (node_3);
  \draw [black,<-] (node_0) to (node_4);
  \draw [black,<-] (node_1) to (node_4);
  \draw [black,<-] (node_0) to (node_2);
\end{tikzpicture}\qquad\qquad
\begin{tikzpicture}[baseline=(current bounding box.center)]
	\tikzmath{\x=0.5;\y=1;\z=4;\zz=9;\zzz=13;}	
	\node (d1) at (\y, 0) {1};
	\node (d2) at (\y, -2*\x) {3};

	\node (a1) at (0, -\x) {2};
	\node (a2) at (0, -3*\x) {4};
	\node (a3) at (0, -5*\x) {5};
	\draw [very thick,->] ($(a1)-(0.6,0)$) to (a1);
	\node at ($(a1)-(0.9,0)$)  {$a_p$};
	
	\draw [<-] (d1) to (d2);
	\draw [<-] (d1) to (a2);

	\draw [<-, out=240, in=120] (a1) to (a3);
	\draw [dashed,<-] (d1) to (a1);
	\draw [dashed,<-] (a1) to (d2);
	\draw [dashed,<-] (a1) to (a2);	
	\draw [dashed,<-] (d2) to (a2);
	\draw [dashed,<-] (a2) to (a3);
	\draw [dashed,<-] (d2) to (a3);
	\draw[rounded corners] (-.5, 0.5) rectangle (.5, -3);
	\node[fill=white] at (0,0.5) {$\alpha$};
	\draw[rounded corners] (0.5, 0.5) rectangle (1.5, -3);
	\node[fill=white] at (1,0.5) {$c$};
\end{tikzpicture}

\begin{tikzpicture}[baseline=(current bounding box.center)]
	\tikzmath{\x=0.5;\y=1;\z=4;\zz=9;\zzz=13;}	
	\draw[ddarr] (\y+0.5+0.2,-1) --  (\y+0.5+1.8,-1) node [midway,above] {\csii[(b)]};

	\node (2) at (\z, -0.5) {$\not2$};	
	\node (1) at (\z+1, 0) {1};
	\node (3) at (\z+1, -1) {3};
	\node (4) at (\z+0, -1.5) {4};
	\node (5) at (\z+0, -2.5) {5};
	\node (b1) at (\z+2, -0.5) {2};
	\node (b2) at (\z+2, -1.5) {$\vn$};
	\node (b3) at (\z+2, -2.5) {$\vn$};
	
	\draw [very thick,->] ($(4)-(0.6,0)$) to (4);
	\node at ($(4)-(0.9,0)$)  {$a_p$};
	
	\draw [<-] (1) to (3);
	\draw [<-] (1) to (4);
	\draw [dashed,<-] (3) to (4);
	\draw [dashed,<-] (4) to (5);
	\draw [dashed,<-] (3) to (5);
	\draw[rounded corners] (\z+-.5, 0.5) rectangle (\z+.5, -3);
	\node[fill=white] at (\z+0,0.5) {$\alpha$};
	\draw[rounded corners] (\z+0.5, 0.5) rectangle (\z+1.5, -3);
	\node[fill=white] at (\z+1,0.5) {$\fd$};
	\draw[rounded corners] (\z+1.5, 0.5) rectangle (\z+2.5, -3);
	\node[fill=white] at (\z+2,0.5) {$\fb$};

	\draw[ddarr] (\z+2*\y+0.5+0.2,-1) --  (\z+2*\y+0.5+1.8,-1) node [midway,above] {\csii[(a)]};

	\node (2) at (\zz, -0.5) {$\not2$};	
	\node (4) at (\zz, -1.5) {$\not4$};	
	\node (1) at (\zz+1, 0) {1};
	\node (4) at (\zz+1, -4*\x) {4};
	\node (5) at (\zz+0, -2.5) {5};
	\node (b1) at (\zz+2, -\x) {2};
	\node (b2) at (\zz+2, -3*\x) {3};
	\node (b3) at (\zz+2, -2.5) {$\vn$};

	\draw [very thick,->] ($(5)-(0.6,0)$) to (5);
	\node at ($(5)-(0.9,0)$)  {$a_p$};

	\draw [<-] (1) to (4);
	\draw [dashed,<-] (4) to (5);
	\draw[rounded corners] (\zz+-.5, 0.5) rectangle (\zz+.5, -3);
	\node[fill=white] at (\zz+0,0.5) {$\alpha$};
	\draw[rounded corners] (\zz+0.5, 0.5) rectangle (\zz+1.5, -3);
	\node[fill=white] at (\zz+1,0.5) {$\fd$};
	\draw[rounded corners] (\zz+1.5, 0.5) rectangle (\zz+2.5, -3);
	\node[fill=white] at (\zz+2,0.5) {$\fb$};

	\draw[ddarr] (\zz+2*\y+0.5+0.2,-1) --  (\zz+2*\y+0.5+1.8,-1) node [midway,above] {\csii[(a)]};
	\node (1) at (\zzz+1, 0) {1};
	\node (5) at (\zzz+1, -2.5) {5};
	\node (b1) at (\zzz+2, -0.45) {2};
	\node (b2) at (\zzz+2, -1.25) {3};
	\node (b3) at (\zzz+2, -2.05) {4};
	\draw [<-] (1) to (5);
	\draw[rounded corners] (\zzz+0.5, 0.5) rectangle (\zzz+1.5, -3);
	\node[fill=white] at (\zzz+1,0.5) {$d$};
	\draw[rounded corners] (\zzz+1.5, 0.5) rectangle (\zzz+2.5, -3);
	\node[fill=white] at (\zzz+2,0.5) {$\beta$};
\end{tikzpicture}
\caption{$\alg_\Phi$: $\cP=\cP_{(2,1,1),5}$, $\alpha=(5,4,2)$, and $c=(3,1)$} \label{fig:stairex2}
\end{figure}

\begin{example}[Figure \ref{fig:stairex3}] \label{stairex3} Suppose that $\cP=\cP_{(7, 6, 5, 4, 3, 2, 1),9}$, $\alpha=(9, 8, 7, 5, 6, 3, 2, 4, 1)$ and $c=\emptyset$. We need to process five steps of $\alg_\Phi$ to calculate $\Phi(\alpha,c)$ in this case. 
\begin{enumerate}
\item Since the chain is empty, $a_1$ is in \csi[(a)]. We set $d_1=1$, $p=2$ and continue.
\item Since $a_2=4>d_1=1$ and $4 \pr 1$, $a_2$ is in \csi[(a)]. We set $d_2=4$, $p=3$ and continue.
\item Since $d_1=1<a_3=2<d_2=4$ and $2 \pdr 1$, $a_3$ is in \csii{}. The set $A$ is equal to $\{(0,0), (0,1), (1,1)\}$ thus $(h,q)=(1,1)$, in which case it is in \csii[(a)]. We set $b_3=2$, $b_4=3$, $p=5$ and continue.
\item Since $a_4=6>d_2=4$ and $6 \pr 4$, $a_4$ is in \csi[(a)]. We set $d_3=6$, $p=6$ and continue.
\item Since $d_2=4<a_6=5<d_3=6$ and $5 \pdr 4$, $a_6$ is in \csii{}. The set $A$ is equal to $\{(0,0), (1,0), (1,1), (1,2), (1,3)\}$ thus $(h,q)=(1,3)$ in which case it is in \csii[(b)]. We set $d_2=5$, $d_3=9$, $b_6=4$, $b_7=6$, $b_8=7$, $b_9=8$ and terminate the algorithm.
\end{enumerate}
As a result, we have $\Phi((9,8,7,5,6,3,2,4,1), \emptyset) = ((9,5,1), (8,7,6,4,\vn,3,2,\vn,\vn))$.
\end{example}
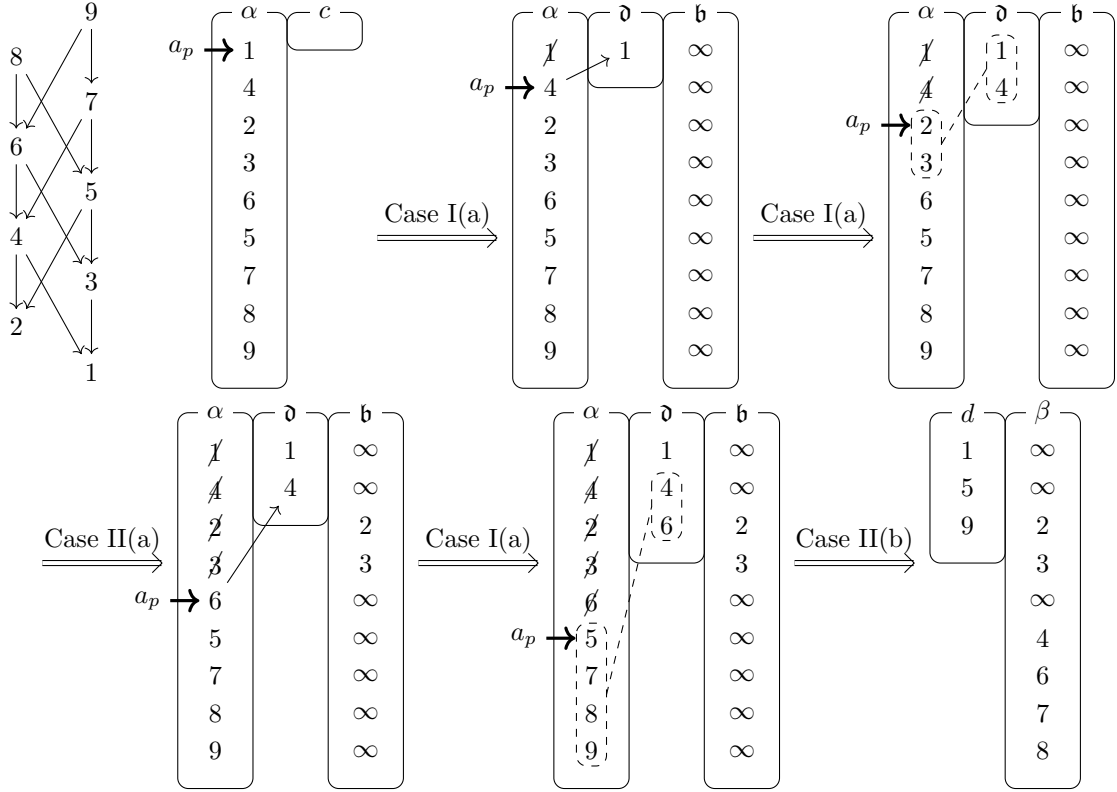
\begin{figure}[!htbp]
\begin{tikzpicture}[baseline=(current bounding box.center), xscale=1,yscale=.6, rotate=90]
	\node (1) at (0,0) {1};
	\node (3) at (2,0) {3};
	\node (5) at (4,0) {5};
	\node (7) at (6,0) {7};
	\node (9) at (8,0) {9};
	
	\node (2) at (1,1) {2};
	\node (4) at (3,1) {4};
	\node (6) at (5,1) {6};
	\node (8) at (7,1) {8};

	\draw[<-] (1) to (3);
	\draw[<-] (3) to (5);
	\draw[<-] (5) to (7);
	\draw[<-] (7) to (9);

	\draw[<-] (2) to (4);
	\draw[<-] (4) to (6);
	\draw[<-] (6) to (8);
	
	\draw[<-] (1) to (4);
	\draw[<-] (2) to (5);
	\draw[<-] (3) to (6);
	\draw[<-] (4) to (7);
	\draw[<-] (5) to (8);
	\draw[<-] (6) to (9);
\end{tikzpicture}\qquad
\begin{tikzpicture}[baseline=(current bounding box.center)]
	\tikzmath{\x=0.5;\z=4;\zz=9;}	
	
	\node (a1) at (0, -1*\x) {1};
	\node (a2) at (0, -2*\x) {4};
	\node (a3) at (0, -3*\x) {2};
	\node (a4) at (0, -4*\x) {3};
	\node (a5) at (0, -5*\x) {6};
	\node (a6) at (0, -6*\x) {5};
	\node (a7) at (0, -7*\x) {7};
	\node (a8) at (0, -8*\x) {8};
	\node (a9) at (0, -9*\x) {9};	
	
	\draw [very thick,->] ($(a1)-(0.6,0)$) to (a1);
	\node at ($(a1)-(0.9,0)$)  {$a_p$};

	\draw[rounded corners] (-.5, 0) rectangle (.5, -9*\x-0.5);
	\node[fill=white] at (0,0) {$\alpha$};
	\draw[rounded corners] (0.5, 0) rectangle (1.5, -.5);
	\node[fill=white] at (1,0) {$c$};
	
	\draw[ddarr] (1.5+0.2,-3) --  (1.5+1.8,-3) node [midway,above] {\csi[(a)]};
	
	\node (a1) at (\z, -1*\x) {$\cancel{1}$};
	\node (a2) at (\z, -2*\x) {4};
	\node (a3) at (\z, -3*\x) {2};
	\node (a4) at (\z, -4*\x) {3};
	\node (a5) at (\z, -5*\x) {6};
	\node (a6) at (\z, -6*\x) {5};
	\node (a7) at (\z, -7*\x) {7};
	\node (a8) at (\z, -8*\x) {8};
	\node (a9) at (\z, -9*\x) {9};	
	
	\node (c1) at (\z+1, -1*\x) {$1$};
	\draw[->] (a2)--(c1);
	
	\node (b1) at (\z+2, -1*\x) {$\vn$};
	\node (b2) at (\z+2, -2*\x) {$\vn$};
	\node (b3) at (\z+2, -3*\x) {$\vn$};
	\node (b4) at (\z+2, -4*\x) {$\vn$};
	\node (b5) at (\z+2, -5*\x) {$\vn$};
	\node (b6) at (\z+2, -6*\x) {$\vn$};
	\node (b7) at (\z+2, -7*\x) {$\vn$};
	\node (b8) at (\z+2, -8*\x) {$\vn$};
	\node (b9) at (\z+2, -9*\x) {$\vn$};	
	
	\draw [very thick,->] ($(a2)-(0.6,0)$) to (a2);
	\node at ($(a2)-(0.9,0)$)  {$a_p$};

	\draw[rounded corners] (\z-.5, 0) rectangle (\z+.5, -9*\x-0.5);
	\node[fill=white] at (\z,0) {$\alpha$};
	\draw[rounded corners] (\z+0.5, 0) rectangle (\z+1.5, -\x-.5);
	\node[fill=white] at (\z+1,0) {$\fd$};
	\draw[rounded corners] (\z+1.5, 0) rectangle (\z+2.5, -9*\x-0.5);
	\node[fill=white] at (\z+2,0) {$\fb$};
	
	\draw[ddarr] (6.5+0.2,-3) --  (6.5+1.8,-3) node [midway,above] {\csi[(a)]};
	
	\node (a1) at (\zz, -1*\x) {$\cancel{1}$};
	\node (a2) at (\zz, -2*\x) {$\cancel{4}$};
	\node (a3) at (\zz, -3*\x) {2};
	\node (a4) at (\zz, -4*\x) {3};
	\node (a5) at (\zz, -5*\x) {6};
	\node (a6) at (\zz, -6*\x) {5};
	\node (a7) at (\zz, -7*\x) {7};
	\node (a8) at (\zz, -8*\x) {8};
	\node (a9) at (\zz, -9*\x) {9};	
	
	\node (c1) at (\zz+1, -1*\x) {$1$};
	\node (c2) at (\zz+1, -2*\x) {$4$};
	
	\node (b1) at (\zz+2, -1*\x) {$\vn$};
	\node (b2) at (\zz+2, -2*\x) {$\vn$};
	\node (b3) at (\zz+2, -3*\x) {$\vn$};
	\node (b4) at (\zz+2, -4*\x) {$\vn$};
	\node (b5) at (\zz+2, -5*\x) {$\vn$};
	\node (b6) at (\zz+2, -6*\x) {$\vn$};
	\node (b7) at (\zz+2, -7*\x) {$\vn$};
	\node (b8) at (\zz+2, -8*\x) {$\vn$};
	\node (b9) at (\zz+2, -9*\x) {$\vn$};	
	
	\draw [very thick,->] ($(a3)-(0.6,0)$) to (a3);
	\node at ($(a3)-(0.9,0)$)  {$a_p$};

	\draw[rounded corners] (\zz-.5, 0) rectangle (\zz+.5, -9*\x-0.5);
	\node[fill=white] at (\zz,0) {$\alpha$};
	\draw[rounded corners] (\zz+0.5, 0) rectangle (\zz+1.5, -2*\x-.5);
	\node[fill=white] at (\zz+1,0) {$\fd$};
	\draw[rounded corners] (\zz+1.5, 0) rectangle (\zz+2.5, -9*\x-0.5);
	\node[fill=white] at (\zz+2,0) {$\fb$};
	\draw[rounded corners,dashed] (\zz-.2, -3*\x+0.2) rectangle (\zz+.2, -4*\x-0.2);
	\draw[rounded corners,dashed] (\zz+1-.2, -1*\x+0.2) rectangle (\zz+1+.2, -2*\x-0.2);
	\draw[dashed] (\zz+.2, -3.5*\x) -- (\zz+1-.2, -1.5*\x);
\end{tikzpicture}

\begin{tikzpicture}[baseline=(current bounding box.center)]
	\draw[ddarr] (-2.5+0.2,-2) --  (-2.5+1.8,-2) node [midway,above] {\csii[(a)]};
	\tikzmath{\x=0.5;\z=5;\zz=9;}	
	
	\node (a1) at (0, -1*\x) {$\cancel{1}$};
	\node (a2) at (0, -2*\x) {$\cancel{4}$};
	\node (a3) at (0, -3*\x) {$\cancel{2}$};
	\node (a4) at (0, -4*\x) {$\cancel{3}$};
	\node (a5) at (0, -5*\x) {6};
	\node (a6) at (0, -6*\x) {5};
	\node (a7) at (0, -7*\x) {7};
	\node (a8) at (0, -8*\x) {8};
	\node (a9) at (0, -9*\x) {9};	
	
	\node (c1) at (1, -1*\x) {$1$};
	\node (c2) at (1, -2*\x) {$4$};
	
	\node (b1) at (2, -1*\x) {$\vn$};
	\node (b2) at (2, -2*\x) {$\vn$};
	\node (b3) at (2, -3*\x) {$2$};
	\node (b4) at (2, -4*\x) {$3$};
	\node (b5) at (2, -5*\x) {$\vn$};
	\node (b6) at (2, -6*\x) {$\vn$};
	\node (b7) at (2, -7*\x) {$\vn$};
	\node (b8) at (2, -8*\x) {$\vn$};
	\node (b9) at (2, -9*\x) {$\vn$};	
	
	\draw [very thick,->] ($(a5)-(0.6,0)$) to (a5);
	\node at ($(a5)-(0.9,0)$)  {$a_p$};
	\draw[->] (a5)--(c2);

	\draw[rounded corners] (-.5, 0) rectangle (.5, -9*\x-0.5);
	\node[fill=white] at (0,0) {$\alpha$};
	\draw[rounded corners] (0.5, 0) rectangle (1.5, -2*\x-.5);
	\node[fill=white] at (1,0) {$\fd$};
	\draw[rounded corners] (1.5, 0) rectangle (2.5, -9*\x-0.5);
	\node[fill=white] at (2,0) {$\fb$};
	
	\draw[ddarr] (2.5+0.2,-2) --  (2.5+1.8,-2) node [midway,above] {\csi[(a)]};
	
	\node (a1) at (\z, -1*\x) {$\cancel{1}$};
	\node (a2) at (\z, -2*\x) {$\cancel{4}$};
	\node (a3) at (\z, -3*\x) {$\cancel{2}$};
	\node (a4) at (\z, -4*\x) {$\cancel{3}$};
	\node (a5) at (\z, -5*\x) {$\cancel{6}$};
	\node (a6) at (\z, -6*\x) {5};
	\node (a7) at (\z, -7*\x) {7};
	\node (a8) at (\z, -8*\x) {8};
	\node (a9) at (\z, -9*\x) {9};	
	
	\node (c1) at (\z+1, -1*\x) {$1$};
	\node (c2) at (\z+1, -2*\x) {$4$};
	\node (c3) at (\z+1, -3*\x) {$6$};
	
	\node (b1) at (\z+2, -1*\x) {$\vn$};
	\node (b2) at (\z+2, -2*\x) {$\vn$};
	\node (b3) at (\z+2, -3*\x) {$2$};
	\node (b4) at (\z+2, -4*\x) {$3$};
	\node (b5) at (\z+2, -5*\x) {$\vn$};
	\node (b6) at (\z+2, -6*\x) {$\vn$};
	\node (b7) at (\z+2, -7*\x) {$\vn$};
	\node (b8) at (\z+2, -8*\x) {$\vn$};
	\node (b9) at (\z+2, -9*\x) {$\vn$};	
	
	\draw [very thick,->] ($(a6)-(0.6,0)$) to (a6);
	\node at ($(a6)-(0.9,0)$)  {$a_p$};

	\draw[rounded corners] (\z-.5, 0) rectangle (\z+.5, -9*\x-0.5);
	\node[fill=white] at (\z+0,0) {$\alpha$};
	\draw[rounded corners] (\z+0.5, 0) rectangle (\z+1.5, -3*\x-.5);
	\node[fill=white] at (\z+1,0) {$\fd$};
	\draw[rounded corners] (\z+1.5, 0) rectangle (\z+2.5, -9*\x-0.5);
	\node[fill=white] at (\z+2,0) {$\fb$};
	
	\draw[rounded corners,dashed] (\z-.2, -6*\x+0.2) rectangle (\z+.2, -9*\x-0.2);
	\draw[rounded corners,dashed] (\z+1-.2, -2*\x+0.2) rectangle (\z+1+.2, -3*\x-0.2);
	\draw[dashed] (\z+.2, -7.5*\x) -- (\z+1-.2, -2.5*\x);
	
	\draw[ddarr] (7.5+0.2,-2) --  (7.5+1.8,-2) node [midway,above] {\csii[(b)]};

	\node (c1) at (\zz+1, -1*\x) {$1$};
	\node (c2) at (\zz+1, -2*\x) {$5$};
	\node (c3) at (\zz+1, -3*\x) {$9$};
	
	\node (b1) at (\zz+2, -1*\x) {$\vn$};
	\node (b2) at (\zz+2, -2*\x) {$\vn$};
	\node (b3) at (\zz+2, -3*\x) {$2$};
	\node (b4) at (\zz+2, -4*\x) {$3$};
	\node (b5) at (\zz+2, -5*\x) {$\vn$};
	\node (b6) at (\zz+2, -6*\x) {$4$};
	\node (b7) at (\zz+2, -7*\x) {$6$};
	\node (b8) at (\zz+2, -8*\x) {$7$};
	\node (b9) at (\zz+2, -9*\x) {$8$};

	\draw[rounded corners] (\zz+0.5, 0) rectangle (\zz+1.5, -3*\x-.5);
	\node[fill=white] at (\zz+1,0) {$d$};
	\draw[rounded corners] (\zz+1.5, 0) rectangle (\zz+2.5, -9*\x-0.5);
	\node[fill=white] at (\zz+2,0) {$\beta$};

\end{tikzpicture}
\caption{$\alg_\Phi$: $\cP=\cP_{(7,6,5,4,3,2,1),9}$, $\alpha=(9,8,7,5,6,3,2,4,1)$, and $c=\emptyset$} \label{fig:stairex3}
\end{figure}

\begin{example}[Figure \ref{fig:stairex3inv}] \label{stairex3inv} Suppose that $\cP=\cP_{(7, 6, 5, 4, 3, 2, 1),9}$, $\alpha=(\vn,\vn,8,7,\vn,6,4,3,2)$, $c=(9, 5, 1)$, and $X=\{5, 8, 9\}$. We need to process five steps of $\alg_\Psi$ to calculate $\Psi_X(\alpha,c)$ in this case. 
\begin{enumerate}
\item Since $d_1=1<a_1=2<d_2=5$ and $2 \pdr 1$, $a_1$ is in \csii{}. The set $A$ is equal to $\{(0,0), (0,1), (0,2), (1,2), (1,3)\}$ thus $(h,q)=(1,3)$ in which case it is in \csii[(b)]. We set $d_1=4$, $d_2=6$, $b_1=1$, $b_2=2$, $b_3=3$, $b_4=5$, $p=5$ and continue.
\item Since $a_5=\vn$, $5\in X$,  and $\fd \neq \emptyset$, $a_4$ is in \csiii[(b)]. We set $\fd=(9,6)$, $b_5=4$, $p=6$ and continue.
\item Since $d_1=6<a_6=7<d_2=9$ and $7 \pdr 6$, $a_6$ is in \csii{}. The set $A$ is equal to $\{(0,0), (0,1), (1,1)\}$ thus $(h,q)=(1,1)$, in which case it is in \csii[(a)]. We set $b_6=7$, $b_7=8$, $p=8$ and continue.
\item Since $a_8=\vn$, $8\in X$,  and $\fd \neq \emptyset$, $a_8$ is in \csiii[(b)]. We set $\fd=(9)$, $b_8=6$ and continue.
\item Since $a_9=\vn$, $9\in X$,  and $\fd \neq \emptyset$, $a_9$ is in \csiii[(b)]. We set $\fd=\emptyset$, $b_9=9$ and continue.
\end{enumerate}
As a result, we have $\Psi_{\{5, 8, 9\}}((\vn,\vn,8,7,\vn,6,4,3,2), (9, 5, 1)) = (\emptyset, (9,6,8,7,4,5,3,2,1))$.
\end{example}
\begin{figure}[!htbp]
\begin{tikzpicture}[baseline=(current bounding box.center), xscale=1,yscale=.6, rotate=90]
	\node (1) at (0,0) {1};
	\node (3) at (2,0) {3};
	\node (5) at (4,0) {5};
	\node (7) at (6,0) {7};
	\node (9) at (8,0) {9};
	
	\node (2) at (1,1) {2};
	\node (4) at (3,1) {4};
	\node (6) at (5,1) {6};
	\node (8) at (7,1) {8};

	\draw[<-] (1) to (3);
	\draw[<-] (3) to (5);
	\draw[<-] (5) to (7);
	\draw[<-] (7) to (9);

	\draw[<-] (2) to (4);
	\draw[<-] (4) to (6);
	\draw[<-] (6) to (8);
	
	\draw[<-] (1) to (4);
	\draw[<-] (2) to (5);
	\draw[<-] (3) to (6);
	\draw[<-] (4) to (7);
	\draw[<-] (5) to (8);
	\draw[<-] (6) to (9);
\end{tikzpicture}\qquad
\begin{tikzpicture}[baseline=(current bounding box.center)]
	\tikzmath{\x=0.5;\z=4;\zz=9;}	
	
	\node (a1) at (0, -1*\x) {2};
	\node (a2) at (0, -2*\x) {3};
	\node (a3) at (0, -3*\x) {4};
	\node (a4) at (0, -4*\x) {6};
	\node (a5) at (0, -5*\x) {$\vn$};
	\node (a6) at (0, -6*\x) {7};
	\node (a7) at (0, -7*\x) {8};
	\node (a8) at (0, -8*\x) {$\vn$};
	\node (a9) at (0, -9*\x) {$\vn$};	
	
	\node (c1) at (1, -1*\x) {$1$};
	\node (c1) at (1, -2*\x) {$5$};
	\node (c1) at (1, -3*\x) {$9$};
	
	\draw (0.3,-5*\x) -- (.7,-5*\x) -- (.7,-9*\x) -- (0.3,-9*\x);
	\draw (0.3,-8*\x) -- (.7,-8*\x);
	\node[fill=white] at (.7,-7*\x) {$X$};
	
	\draw [very thick,->] ($(a1)-(0.6,0)$) to (a1);
	\node at ($(a1)-(0.9,0)$)  {$a_p$};

	\draw[rounded corners] (-.5, 0) rectangle (.5, -9*\x-0.5);
	\node[fill=white] at (0,0) {$\alpha$};
	\draw[rounded corners] (0.5, 0) rectangle (1.5, -3*\x-.5);
	\node[fill=white] at (1,0) {$c$};
	\draw[rounded corners,dashed] (-.2, -1*\x+0.2) rectangle (+.2, -4*\x-0.2);
	\draw[rounded corners,dashed] (+1-.2, -1*\x+0.2) rectangle (+1+.2, -2*\x-0.2);
	\draw[dashed] (+.2, -2.5*\x) -- (+1-.2, -1.5*\x);

	\draw[ddarr] (1.5+0.2,-2) --  (1.5+1.8,-2) node [midway,above] {\csii[(b)]};
	
	\node (a1) at (\z, -1*\x) {$\cancel{2}$};
	\node (a2) at (\z, -2*\x) {$\cancel{3}$};
	\node (a3) at (\z, -3*\x) {$\cancel{4}$};
	\node (a4) at (\z, -4*\x) {$\cancel{6}$};
	\node (a5) at (\z, -5*\x) {$\vn$};
	\node (a6) at (\z, -6*\x) {7};
	\node (a7) at (\z, -7*\x) {8};
	\node (a8) at (\z, -8*\x) {$\vn$};
	\node (a9) at (\z, -9*\x) {$\vn$};	
	
	\node (c1) at (\z+1, -1*\x) {$4$};
	\node (c2) at (\z+1, -2*\x) {$6$};
	\node (c3) at (\z+1, -3*\x) {$9$};
	
	\node (b1) at (\z+2, -1*\x) {$1$};
	\node (b2) at (\z+2, -2*\x) {$2$};
	\node (b3) at (\z+2, -3*\x) {$3$};
	\node (b4) at (\z+2, -4*\x) {$5$};
	\node (b5) at (\z+2, -5*\x) {$\vn$};
	\node (b6) at (\z+2, -6*\x) {$\vn$};
	\node (b7) at (\z+2, -7*\x) {$\vn$};
	\node (b8) at (\z+2, -8*\x) {$\vn$};
	\node (b9) at (\z+2, -9*\x) {$\vn$};	
	
	\draw (\z+0.3,-5*\x) -- (\z+.7,-5*\x) -- (\z+.7,-9*\x) -- (\z+0.3,-9*\x);
	\draw (\z+0.3,-8*\x) -- (\z+.7,-8*\x);
	\node[fill=white] at (\z+.7,-7*\x) {$X$};
	
	\draw [very thick,->] ($(a5)-(0.6,0)$) to (a5);
	\node at ($(a5)-(0.9,0)$)  {$a_p$};

	\draw[rounded corners] (\z-.5, 0) rectangle (\z+.5, -9*\x-0.5);
	\node[fill=white] at (\z,0) {$\alpha$};
	\draw[rounded corners] (\z+0.5, 0) rectangle (\z+1.5, -3*\x-.5);
	\node[fill=white] at (\z+1,0) {$\fd$};
	\draw[rounded corners] (\z+1.5, 0) rectangle (\z+2.5, -9*\x-0.5);
	\node[fill=white] at (\z+2,0) {$\fb$};
	
	\draw[ddarr] (6.5+0.2,-2) --  (6.5+1.8,-2) node [midway,above] {\csiii[(b)]};
	
	\node (a1) at (\zz, -1*\x) {$\cancel{2}$};
	\node (a2) at (\zz, -2*\x) {$\cancel{3}$};
	\node (a3) at (\zz, -3*\x) {$\cancel{4}$};
	\node (a4) at (\zz, -4*\x) {$\cancel{6}$};
	\node (a5) at (\zz, -5*\x) {$\cancel{\vn}$};
	\node (a6) at (\zz, -6*\x) {7};
	\node (a7) at (\zz, -7*\x) {8};
	\node (a8) at (\zz, -8*\x) {$\vn$};
	\node (a9) at (\zz, -9*\x) {$\vn$};	
	
	\node (c1) at (\zz+1, -1*\x) {$6$};
	\node (c2) at (\zz+1, -2*\x) {$9$};
	
	\node (b1) at (\zz+2, -1*\x) {$1$};
	\node (b2) at (\zz+2, -2*\x) {$2$};
	\node (b3) at (\zz+2, -3*\x) {$3$};
	\node (b4) at (\zz+2, -4*\x) {$5$};
	\node (b5) at (\zz+2, -5*\x) {$4$};
	\node (b6) at (\zz+2, -6*\x) {$\vn$};
	\node (b7) at (\zz+2, -7*\x) {$\vn$};
	\node (b8) at (\zz+2, -8*\x) {$\vn$};
	\node (b9) at (\zz+2, -9*\x) {$\vn$};	

	\draw (\zz+0.3,-5*\x) -- (\zz+.7,-5*\x) -- (\zz+.7,-9*\x) -- (\zz+0.3,-9*\x);
	\draw (\zz+0.3,-8*\x) -- (\zz+.7,-8*\x);
	\node[fill=white] at (\zz+.7,-7*\x) {$X$};
	
	\draw [very thick,->] ($(a6)-(0.6,0)$) to (a6);
	\node at ($(a6)-(0.9,0)$)  {$a_p$};

	\draw[rounded corners] (\zz-.5, 0) rectangle (\zz+.5, -9*\x-0.5);
	\node[fill=white] at (\zz,0) {$\alpha$};
	\draw[rounded corners] (\zz+0.5, 0) rectangle (\zz+1.5, -2*\x-.5);
	\node[fill=white] at (\zz+1,0) {$\fd$};
	\draw[rounded corners] (\zz+1.5, 0) rectangle (\zz+2.5, -9*\x-0.5);
	\node[fill=white] at (\zz+2,0) {$\fb$};
	\draw[rounded corners,dashed] (\zz-.2, -6*\x+0.2) rectangle (\zz+.2, -7*\x-0.2);
	\draw[rounded corners,dashed] (\zz+1-.2, -1*\x+0.2) rectangle (\zz+1+.2, -2*\x-0.2);
	\draw[dashed] (\zz+.2, -6.5*\x) -- (\zz+1-.2, -1.5*\x);
\end{tikzpicture}

\begin{tikzpicture}[baseline=(current bounding box.center)]
	\draw[ddarr] (-2.5+0.2,-2) --  (-2.5+1.8,-2) node [midway,above] {\csii[(a)]};
	\tikzmath{\x=0.5;\z=5;\zz=9;}	
	
	\node (a1) at (0, -1*\x) {$\cancel{2}$};
	\node (a2) at (0, -2*\x) {$\cancel{3}$};
	\node (a3) at (0, -3*\x) {$\cancel{4}$};
	\node (a4) at (0, -4*\x) {$\cancel{6}$};
	\node (a5) at (0, -5*\x) {$\cancel{\vn}$};
	\node (a6) at (0, -6*\x) {$\cancel{7}$};
	\node (a7) at (0, -7*\x) {$\cancel{8}$};
	\node (a8) at (0, -8*\x) {$\vn$};
	\node (a9) at (0, -9*\x) {$\vn$};	
	
	\node (c1) at (1, -1*\x) {$6$};
	\node (c2) at (1, -2*\x) {$9$};
	
	\node (b1) at (2, -1*\x) {$1$};
	\node (b2) at (2, -2*\x) {$2$};
	\node (b3) at (2, -3*\x) {$3$};
	\node (b4) at (2, -4*\x) {$5$};
	\node (b5) at (2, -5*\x) {$4$};
	\node (b6) at (2, -6*\x) {$7$};
	\node (b7) at (2, -7*\x) {$8$};
	\node (b8) at (2, -8*\x) {$\vn$};
	\node (b9) at (2, -9*\x) {$\vn$};	
	
	\draw [very thick,->] ($(a8)-(0.6,0)$) to (a8);
	\node at ($(a8)-(0.9,0)$)  {$a_p$};
	
	\draw (0.3,-5*\x) -- (.7,-5*\x) -- (.7,-9*\x) -- (0.3,-9*\x);
	\draw (0.3,-8*\x) -- (.7,-8*\x);
	\node[fill=white] at (.7,-7*\x) {$X$};

	\draw[rounded corners] (-.5, 0) rectangle (.5, -9*\x-0.5);
	\node[fill=white] at (0,0) {$\alpha$};
	\draw[rounded corners] (0.5, 0) rectangle (1.5, -2*\x-.5);
	\node[fill=white] at (1,0) {$\fd$};
	\draw[rounded corners] (1.5, 0) rectangle (2.5, -9*\x-0.5);
	\node[fill=white] at (2,0) {$\fb$};
	
	\draw[ddarr] (2.5+0.2,-2) --  (2.5+1.8,-2) node [midway,above] {\csiii[(b)]};
	
	\node (a1) at (\z, -1*\x) {$\cancel{2}$};
	\node (a2) at (\z, -2*\x) {$\cancel{3}$};
	\node (a3) at (\z, -3*\x) {$\cancel{4}$};
	\node (a4) at (\z, -4*\x) {$\cancel{6}$};
	\node (a5) at (\z, -5*\x) {$\cancel{\vn}$};
	\node (a6) at (\z, -6*\x) {$\cancel{7}$};
	\node (a7) at (\z, -7*\x) {$\cancel{8}$};
	\node (a8) at (\z, -8*\x) {$\cancel{\vn}$};
	\node (a9) at (\z, -9*\x) {$\vn$};	
	
	\node (c1) at (\z+1, -1*\x) {$9$};
	
	\node (b1) at (\z+2, -1*\x) {$1$};
	\node (b2) at (\z+2, -2*\x) {$2$};
	\node (b3) at (\z+2, -3*\x) {$3$};
	\node (b4) at (\z+2, -4*\x) {$5$};
	\node (b5) at (\z+2, -5*\x) {$4$};
	\node (b6) at (\z+2, -6*\x) {$7$};
	\node (b7) at (\z+2, -7*\x) {$8$};
	\node (b8) at (\z+2, -8*\x) {$6$};
	\node (b9) at (\z+2, -9*\x) {$\vn$};
	
	\draw [very thick,->] ($(a9)-(0.6,0)$) to (a9);
	\node at ($(a9)-(0.9,0)$)  {$a_p$};
	\draw (\z+0.3,-5*\x) -- (\z+.7,-5*\x) -- (\z+.7,-9*\x) -- (\z+0.3,-9*\x);
	\draw (\z+0.3,-8*\x) -- (\z+.7,-8*\x);
	\node[fill=white] at (\z+.7,-7*\x) {$X$};

	\draw[rounded corners] (\z-.5, 0) rectangle (\z+.5, -9*\x-0.5);
	\node[fill=white] at (\z+0,0) {$\alpha$};
	\draw[rounded corners] (\z+0.5, 0) rectangle (\z+1.5, -1*\x-.5);
	\node[fill=white] at (\z+1,0) {$\fd$};
	\draw[rounded corners] (\z+1.5, 0) rectangle (\z+2.5, -9*\x-0.5);
	\node[fill=white] at (\z+2,0) {$\fb$};

	\draw[ddarr] (7.5+0.2,-2) --  (7.5+1.8,-2) node [midway,above] {\csiii[(b)]};

	\node (b1) at (\zz+2, -1*\x) {$1$};
	\node (b2) at (\zz+2, -2*\x) {$2$};
	\node (b3) at (\zz+2, -3*\x) {$3$};
	\node (b4) at (\zz+2, -4*\x) {$5$};
	\node (b5) at (\zz+2, -5*\x) {$4$};
	\node (b6) at (\zz+2, -6*\x) {$7$};
	\node (b7) at (\zz+2, -7*\x) {$8$};
	\node (b8) at (\zz+2, -8*\x) {$6$};
	\node (b9) at (\zz+2, -9*\x) {$9$};

	\draw[rounded corners] (\zz+0.5, 0) rectangle (\zz+1.5, -0*\x-.5);
	\node[fill=white] at (\zz+1,0) {$d$};
	\draw[rounded corners] (\zz+1.5, 0) rectangle (\zz+2.5, -9*\x-0.5);
	\node[fill=white] at (\zz+2,0) {$\beta$};

\end{tikzpicture}
\captionsetup{justification=centering}
\caption{$\alg_\Psi$: $\cP=\cP_{(7,6,5,4,3,2,1),9}$, $\alpha=(\vn,\vn,8,7,\vn,6,4,3,2)$, \\$c=(9,5,1)$, and $X=\{5, 8, 9\}$} \label{fig:stairex3inv}
\end{figure}

\begin{rmk} Indeed, Example \ref{stairex3} and Example \ref{stairex3inv} are ``mirror images'' to each other. This is not a coincidence but explained in detail in Section \ref{sec:proofinj}.
\end{rmk}

\subsection{Some examples of $\alg_{\prs}$}
Here we provide some examples of the algorithm $\alg_{\prs}$.

\begin{example} Figure \ref{fig:21-5} shows the steps of $\alg_{\prs}$ when $\cP=\cP_{(2,1),4}$ and $w$ is an element of $\{3241, 3421, 4231, 4312, 4132\}$, the set of vertices of a connected $\cP$-Knuth equivalence graph that is not a dual equivalence graph (cf. Figure \ref{fig:214}). Here we have $\Phi(3241) = (\ytableaushort{132,4},\ytableaushort{134,2})$, $\Phi(3421) = (\ytableaushort{21,43},\ytableaushort{12,34})$, $\Phi(4231) = (\ytableaushort{21,43},\ytableaushort{13,24})$, $\Phi(4312) = (\ytableaushort{132,4},\ytableaushort{124,3})$, and $\Phi(4132) = (\ytableaushort{132,4},\ytableaushort{123,4})$. It is easy to observe in this case that (cf. Theorem \ref{thm:main})
\begin{enumerate}
\item[(A)] the outputs are pairs $(PT, QT)$ where $PT \in \ptab_\lambda$, $QT\in \SYT_\lambda$ for some $\lambda \vdash 4$,
\item[(B)] if $\prs(w) = (PT, QT)$ then $\des_\cP(w) = \{4-x\mid x \in \des(QT)\}$,
\item[(C)] $\rw(\ytableaushort{21,43}) = 4231$ and $\rw(\ytableaushort{132,4}) = 4132$ are the vertices of the given connected $\cP$-Knuth equivalence graph,
\item[(D)] $\Phi(4231)=(\ytableaushort{21,43}, \omega(\ytableaushort{13,24}))=(\ytableaushort{21,43}, \ytableaushort{13,24})$ and $\Phi(4132) = (\ytableaushort{132,4},\omega(\ytableaushort{134,2}))=(\ytableaushort{132,4},\ytableaushort{123,4})$ where $\omega$ is Sch\"utzenberger's evacuation, and
\item[(E), (F)] $\Phi$ gives a bijection between the given set of vertices and $\{\ytableaushort{21,43}\} \times \SYT_{(2,2)} \sqcup \{\ytableaushort{132,4}\}\times \SYT_{(3,1)}$.
\end{enumerate}
Furthermore, its generating function is $t^2(s_{31}+s_{22})$ as expected by Theorem \ref{thm:mainstrong}.

\end{example}

\begin{figure}[!htbp]
\begin{tikzpicture}[baseline=(current bounding box.center), rotate=90, xscale=0.8]
	\node (1) at (0,0) {1};
	\node (3) at (2,0) {3};
	\node (2) at (1,1) {2};
	\node (4) at (3,1) {4};
	\draw[<-] (1) to (3);
	\draw[<-] (2) to (4);
	\draw[<-] (1) to (4);
\end{tikzpicture}\qquad\qquad
\begin{tikzpicture}[baseline=(current bounding box.center)]
	\tikzmath{\nnn=3;\x=0.5;\y=0.4;\z=1.8;\zz=4.4;\zzz=7;\zzzz=11;}
	\node at (0, -0*\x) {$1$};
	\node at (0, -1*\x) {$4$};
	\node at (0, -2*\x) {$2$};
	\node at (0, -3*\x) {$3$};

	\draw[rounded corners] (-\y, 0.5) rectangle (\y, -\nnn*\x-.5);
	\node[fill=white] at (0,0.5) {$w$};
	\draw[ddarr] (\y+.2, -1*\x) -- (\y+.8, -1*\x);
	
	\node at (\z, 0) {$1$};
	\node at (\z, -\x) {$4$};
	\node[circ] at (\z+2*\y, -0*\x) {$\vn$};
	\node[circ] at (\z+2*\y, -1*\x) {$\vn$};
	\node at (\z+2*\y, -2*\x) {$2$};
	\node at (\z+2*\y, -3*\x) {$3$};
	\draw[rounded corners] (-\y+\z, 0.5) rectangle (\y+\z,-1*\x-.5);
	\node[fill=white] at (\z,0.5) {$PT_1$};
	\draw[rounded corners] (\y+\z, 0.5) rectangle (3*\y+\z, -\nnn*\x-.5);
	\draw[ddarr] (\z+3*\y+.2, -1*\x) -- (\z+3*\y+.8, -1*\x);
	
	\node at (\zz, 0) {$3$};
	\node at (\zz+2*\y, -0*\x) {$\vn$};
	\node at (\zz+2*\y, -1*\x) {$\vn$};
	\node[circ] at (\zz+2*\y, -2*\x) {$\vn$};
	\node at (\zz+2*\y, -3*\x) {$2$};

	\draw[rounded corners] (-\y+\zz, 0.5) rectangle (\y+\zz, -0*\x-.5);
	\node[fill=white] at (\zz,0.5) {$PT_2$};
	\draw[rounded corners] (\y+\zz, 0.5) rectangle (3*\y+\zz, -\nnn*\x-.5);
	\draw[ddarr] (\zz+3*\y+.2, -1*\x) -- (\zz+3*\y+.8, -1*\x);
	
	\node at (\zzz, 0) {$2$};
	\node at (\zzz+2*\y, -0*\x) {$\vn$};
	\node at (\zzz+2*\y, -1*\x) {$\vn$};
	\node at (\zzz+2*\y, -2*\x) {$\vn$};
	\node[circ] at (\zzz+2*\y, -3*\x) {$\vn$};

	\draw[rounded corners] (-\y+\zzz, 0.5) rectangle (\y+\zzz, -0*\x-.5);
	\node[fill=white] at (\zzz,0.5) {$PT_3$};
	\draw[rounded corners] (\y+\zzz, 0.5) rectangle (3*\y+\zzz, -\nnn*\x-.5);
\end{tikzpicture}

\begin{tikzpicture}[baseline=(current bounding box.center)]
	\tikzmath{\nnn=3;\x=0.5;\y=0.4;\z=1.8;\zz=4.4;\zzz=7;\zzzz=11;}
	\node at (0, -0*\x) {$1$};
	\node at (0, -1*\x) {$2$};
	\node at (0, -2*\x) {$4$};
	\node at (0, -3*\x) {$3$};

	\draw[rounded corners] (-\y, 0.5) rectangle (\y, -\nnn*\x-.5);
	\node[fill=white] at (0,0.5) {$w$};
	\draw[ddarr] (\y+.2, -1*\x) -- (\y+.8, -1*\x);
	
	\node at (\z, 0) {$2$};
	\node at (\z, -\x) {$4$};
	\node[circ] at (\z+2*\y, -0*\x) {$\vn$};
	\node at (\z+2*\y, -1*\x) {$1$};
	\node[circ] at (\z+2*\y, -2*\x) {$\vn$};
	\node at (\z+2*\y, -3*\x) {$3$};
	\draw[rounded corners] (-\y+\z, 0.5) rectangle (\y+\z,-1*\x-.5);
	\node[fill=white] at (\z,0.5) {$PT_1$};
	\draw[rounded corners] (\y+\z, 0.5) rectangle (3*\y+\z, -\nnn*\x-.5);
	\draw[ddarr] (\z+3*\y+.2, -1*\x) -- (\z+3*\y+.8, -1*\x);
	
	\node at (\zz, 0) {$1$};
	\node at (\zz, -\x) {$3$};
	\node at (\zz+2*\y, -0*\x) {$\vn$};
	\node[circ] at (\zz+2*\y, -1*\x) {$\vn$};
	\node at (\zz+2*\y, -2*\x) {$\vn$};
	\node[circ] at (\zz+2*\y, -3*\x) {$\vn$};

	\draw[rounded corners] (-\y+\zz, 0.5) rectangle (\y+\zz, -1*\x-.5);
	\node[fill=white] at (\zz,0.5) {$PT_2$};
	\draw[rounded corners] (\y+\zz, 0.5) rectangle (3*\y+\zz, -\nnn*\x-.5);
\end{tikzpicture}\qquad
\begin{tikzpicture}[baseline=(current bounding box.center)]
	\tikzmath{\nnn=3;\x=0.5;\y=0.4;\z=1.8;\zz=4.4;\zzz=7;\zzzz=11;}
	\node at (0, -0*\x) {$1$};
	\node at (0, -1*\x) {$3$};
	\node at (0, -2*\x) {$2$};
	\node at (0, -3*\x) {$4$};

	\draw[rounded corners] (-\y, 0.5) rectangle (\y, -\nnn*\x-.5);
	\node[fill=white] at (0,0.5) {$w$};
	\draw[ddarr] (\y+.2, -1*\x) -- (\y+.8, -1*\x);
	
	\node at (\z, 0) {$2$};
	\node at (\z, -\x) {$4$};
	\node[circ] at (\z+2*\y, -0*\x) {$\vn$};
	\node[circ] at (\z+2*\y, -1*\x) {$\vn$};
	\node at (\z+2*\y, -2*\x) {$1$};
	\node at (\z+2*\y, -3*\x) {$3$};
	\draw[rounded corners] (-\y+\z, 0.5) rectangle (\y+\z,-1*\x-.5);
	\node[fill=white] at (\z,0.5) {$PT_1$};
	\draw[rounded corners] (\y+\z, 0.5) rectangle (3*\y+\z, -\nnn*\x-.5);
	\draw[ddarr] (\z+3*\y+.2, -1*\x) -- (\z+3*\y+.8, -1*\x);
	
	\node at (\zz, 0) {$1$};
	\node at (\zz, -\x) {$3$};
	\node at (\zz+2*\y, -0*\x) {$\vn$};
	\node at (\zz+2*\y, -1*\x) {$\vn$};
	\node[circ] at (\zz+2*\y, -2*\x) {$\vn$};
	\node[circ] at (\zz+2*\y, -3*\x) {$\vn$};

	\draw[rounded corners] (-\y+\zz, 0.5) rectangle (\y+\zz, -1*\x-.5);
	\node[fill=white] at (\zz,0.5) {$PT_2$};
	\draw[rounded corners] (\y+\zz, 0.5) rectangle (3*\y+\zz, -\nnn*\x-.5);
\end{tikzpicture}

\begin{tikzpicture}[baseline=(current bounding box.center)]
	\tikzmath{\nnn=3;\x=0.5;\y=0.4;\z=1.8;\zz=4.4;\zzz=7;\zzzz=11;}
	\node at (0, -0*\x) {$2$};
	\node at (0, -1*\x) {$1$};
	\node at (0, -2*\x) {$3$};
	\node at (0, -3*\x) {$4$};

	\draw[rounded corners] (-\y, 0.5) rectangle (\y, -\nnn*\x-.5);
	\node[fill=white] at (0,0.5) {$w$};
	\draw[ddarr] (\y+.2, -1*\x) -- (\y+.8, -1*\x);
	
	\node at (\z, 0) {$1$};
	\node at (\z, -\x) {$4$};
	\node[circ] at (\z+2*\y, -0*\x) {$\vn$};
	\node at (\z+2*\y, -1*\x) {$2$};
	\node[circ] at (\z+2*\y, -2*\x) {$\vn$};
	\node at (\z+2*\y, -3*\x) {$3$};
	\draw[rounded corners] (-\y+\z, 0.5) rectangle (\y+\z,-1*\x-.5);
	\node[fill=white] at (\z,0.5) {$PT_1$};
	\draw[rounded corners] (\y+\z, 0.5) rectangle (3*\y+\z, -\nnn*\x-.5);
	\draw[ddarr] (\z+3*\y+.2, -1*\x) -- (\z+3*\y+.8, -1*\x);
	
	\node at (\zz, 0) {$3$};
	\node at (\zz+2*\y, -0*\x) {$\vn$};
	\node[circ] at (\zz+2*\y, -1*\x) {$\vn$};
	\node at (\zz+2*\y, -2*\x) {$\vn$};
	\node at (\zz+2*\y, -3*\x) {$2$};

	\draw[rounded corners] (-\y+\zz, 0.5) rectangle (\y+\zz, -0*\x-.5);
	\node[fill=white] at (\zz,0.5) {$PT_2$};
	\draw[rounded corners] (\y+\zz, 0.5) rectangle (3*\y+\zz, -\nnn*\x-.5);
	\draw[ddarr] (\zz+3*\y+.2, -1*\x) -- (\zz+3*\y+.8, -1*\x);
	
	\node at (\zzz, 0) {$2$};
	\node at (\zzz+2*\y, -0*\x) {$\vn$};
	\node at (\zzz+2*\y, -1*\x) {$\vn$};
	\node at (\zzz+2*\y, -2*\x) {$\vn$};
	\node[circ] at (\zzz+2*\y, -3*\x) {$\vn$};

	\draw[rounded corners] (-\y+\zzz, 0.5) rectangle (\y+\zzz, -0*\x-.5);
	\node[fill=white] at (\zzz,0.5) {$PT_3$};
	\draw[rounded corners] (\y+\zzz, 0.5) rectangle (3*\y+\zzz, -\nnn*\x-.5);
\end{tikzpicture}

\begin{tikzpicture}[baseline=(current bounding box.center)]
	\tikzmath{\nnn=3;\x=0.5;\y=0.4;\z=1.8;\zz=4.4;\zzz=7;\zzzz=11;}
	\node at (0, -0*\x) {$2$};
	\node at (0, -1*\x) {$3$};
	\node at (0, -2*\x) {$1$};
	\node at (0, -3*\x) {$4$};

	\draw[rounded corners] (-\y, 0.5) rectangle (\y, -\nnn*\x-.5);
	\node[fill=white] at (0,0.5) {$w$};
	\draw[ddarr] (\y+.2, -1*\x) -- (\y+.8, -1*\x);
	
	\node at (\z, 0) {$1$};
	\node at (\z, -\x) {$4$};
	\node[circ] at (\z+2*\y, -0*\x) {$\vn$};
	\node at (\z+2*\y, -1*\x) {$2$};
	\node at (\z+2*\y, -2*\x) {$3$};
	\node[circ] at (\z+2*\y, -3*\x) {$\vn$};
	\draw[rounded corners] (-\y+\z, 0.5) rectangle (\y+\z,-1*\x-.5);
	\node[fill=white] at (\z,0.5) {$PT_1$};
	\draw[rounded corners] (\y+\z, 0.5) rectangle (3*\y+\z, -\nnn*\x-.5);
	\draw[ddarr] (\z+3*\y+.2, -1*\x) -- (\z+3*\y+.8, -1*\x);
	
	\node at (\zz, 0) {$3$};
	\node at (\zz+2*\y, -0*\x) {$\vn$};
	\node[circ] at (\zz+2*\y, -1*\x) {$\vn$};
	\node at (\zz+2*\y, -2*\x) {$2$};
	\node at (\zz+2*\y, -3*\x) {$\vn$};

	\draw[rounded corners] (-\y+\zz, 0.5) rectangle (\y+\zz, -0*\x-.5);
	\node[fill=white] at (\zz,0.5) {$PT_2$};
	\draw[rounded corners] (\y+\zz, 0.5) rectangle (3*\y+\zz, -\nnn*\x-.5);
	\draw[ddarr] (\zz+3*\y+.2, -1*\x) -- (\zz+3*\y+.8, -1*\x);
	
	\node at (\zzz, 0) {$2$};
	\node at (\zzz+2*\y, -0*\x) {$\vn$};
	\node at (\zzz+2*\y, -1*\x) {$\vn$};
	\node[circ] at (\zzz+2*\y, -2*\x) {$\vn$};
	\node at (\zzz+2*\y, -3*\x) {$\vn$};

	\draw[rounded corners] (-\y+\zzz, 0.5) rectangle (\y+\zzz, -0*\x-.5);
	\node[fill=white] at (\zzz,0.5) {$PT_3$};
	\draw[rounded corners] (\y+\zzz, 0.5) rectangle (3*\y+\zzz, -\nnn*\x-.5);
\end{tikzpicture}
\caption{$\alg_{\prs}$: $\cP=\cP_{(2,1),4}, w\in \{3241, 3421, 4231, 4312, 4132\}$}\label{fig:21-5}
\end{figure}

\begin{example} Figure \ref{fig:7654321,987563241} shows the steps of $\alg_{\prs}$ when $\cP=\cP_{(7,6,5,4,3,2,1),9}$ and $w=(9,8,7,5,6,3,2,4,1)$. Here we have $\Phi(w) = (\ytableaushort{1432,5876,9}, \ytableaushort{1346,2789,5})$.  Note that $\{9-x \mid x\in \des_\cP(w)\} = \{1,4,6\}=\des(\ytableaushort{1346,2789,5})$.
\end{example}

\begin{figure}[!htbp]
\begin{tikzpicture}[baseline=(current bounding box.center), xscale=.8,yscale=.6, rotate=90]
	\node (1) at (0,0) {1};
	\node (3) at (2,0) {3};
	\node (5) at (4,0) {5};
	\node (7) at (6,0) {7};
	\node (9) at (8,0) {9};
	
	\node (2) at (1,1) {2};
	\node (4) at (3,1) {4};
	\node (6) at (5,1) {6};
	\node (8) at (7,1) {8};

	\draw[<-] (1) to (3);
	\draw[<-] (3) to (5);
	\draw[<-] (5) to (7);
	\draw[<-] (7) to (9);

	\draw[<-] (2) to (4);
	\draw[<-] (4) to (6);
	\draw[<-] (6) to (8);
	
	\draw[<-] (1) to (4);
	\draw[<-] (2) to (5);
	\draw[<-] (3) to (6);
	\draw[<-] (4) to (7);
	\draw[<-] (5) to (8);
	\draw[<-] (6) to (9);
\end{tikzpicture}\qquad
\begin{tikzpicture}[baseline=(current bounding box.center)]
	\tikzmath{\nnn=3;\x=0.5;\y=0.4;\z=1.8;\zz=4.4;\zzz=7;\zzzz=9.5;}
	\node at (0, -0*\x) {$1$};
	\node at (0, -1*\x) {$2$};
	\node at (0, -2*\x) {$4$};
	\node at (0, -3*\x) {$3$};
	\node at (0, -4*\x) {$6$};
	\node at (0, -5*\x) {$5$};
	\node at (0, -6*\x) {$7$};
	\node at (0, -7*\x) {$8$};
	\node at (0, -8*\x) {$9$};
	
	\draw[rounded corners] (-\y, 0.5) rectangle (\y, -8*\x-.5);
	\node[fill=white] at (0,0.5) {$w$};
	\draw[double distance=1.5pt, arrows={-Straight Barb[scale=0.6]}] (\y+.2, -2*\x) -- (\y+.8, -2*\x);
	
	\node at (\z, 0) {$1$};
	\node at (\z, -\x) {$5$};
	\node at (\z, -2*\x) {$9$};
	\node[circ] at (\z+2*\y, -0*\x) {$\vn$};
	\node[circ] at (\z+2*\y, -1*\x) {$\vn$};
	\node at (\z+2*\y, -2*\x) {$2$};
	\node at (\z+2*\y, -3*\x) {$3$};
	\node[circ] at (\z+2*\y, -4*\x) {$\vn$};
	\node at (\z+2*\y, -5*\x) {$4$};
	\node at (\z+2*\y, -6*\x) {$6$};
	\node at (\z+2*\y, -7*\x) {$7$};
	\node at (\z+2*\y, -8*\x) {$8$};
	\draw[rounded corners] (-\y+\z, 0.5) rectangle (\y+\z,-2*\x-.5);
	\node[fill=white] at (\z,0.5) {$PT_1$};
	\draw[rounded corners] (\y+\z, 0.5) rectangle (3*\y+\z, -8*\x-.5);
	\draw[ddarr] (\z+3*\y+.2, -2*\x) -- (\z+3*\y+.8, -2*\x);
	
	\node at (\zz, 0) {$4$};
	\node at (\zz, -\x) {$8$};
	\node at (\zz+2*\y, -0*\x) {$\vn$};
	\node at (\zz+2*\y, -1*\x) {$\vn$};
	\node[circ] at (\zz+2*\y, -2*\x) {$\vn$};
	\node at (\zz+2*\y, -3*\x) {$2$};
	\node at (\zz+2*\y, -4*\x) {$\vn$};
	\node at (\zz+2*\y, -5*\x) {$3$};
	\node[circ] at (\zz+2*\y, -6*\x) {$\vn$};
	\node at (\zz+2*\y, -7*\x) {$6$};
	\node at (\zz+2*\y, -8*\x) {$7$};
	\draw[rounded corners] (-\y+\zz, 0.5) rectangle (\y+\zz, -1*\x-.5);
	\node[fill=white] at (\zz,0.5) {$PT_2$};
	\draw[rounded corners] (\y+\zz, 0.5) rectangle (3*\y+\zz, -8*\x-.5);
	\draw[ddarr] (\zz+3*\y+.2, -2*\x) -- (\zz+3*\y+.8, -2*\x);

	\node at (\zzz, 0) {$3$};
	\node at (\zzz, -\x) {$7$};
	\node at (\zzz+2*\y, -0*\x) {$\vn$};
	\node at (\zzz+2*\y, -1*\x) {$\vn$};
	\node at (\zzz+2*\y, -2*\x) {$\vn$};
	\node[circ] at (\zzz+2*\y, -3*\x) {$\vn$};
	\node at (\zzz+2*\y, -4*\x) {$\vn$};
	\node at (\zzz+2*\y, -5*\x) {$2$};
	\node at (\zzz+2*\y, -6*\x) {$\vn$};
	\node[circ] at (\zzz+2*\y, -7*\x) {$\vn$};
	\node at (\zzz+2*\y, -8*\x) {$6$};
	\draw[rounded corners] (-\y+\zzz, 0.5) rectangle (\y+\zzz, -1*\x-.5);
	\node[fill=white] at (\zzz,0.5) {$PT_3$};
	\draw[rounded corners] (\y+\zzz, 0.5) rectangle (3*\y+\zzz, -8*\x-.5);
	\draw[ddarr] (\zzz+3*\y+.2, -2*\x) -- (\zzz+3*\y+.8, -2*\x);
	
	\node at (\zzzz, 0) {$2$};
	\node at (\zzzz, -\x) {$6$};
	\node at (\zzzz+2*\y, -0*\x) {$\vn$};
	\node at (\zzzz+2*\y, -1*\x) {$\vn$};
	\node at (\zzzz+2*\y, -2*\x) {$\vn$};
	\node at (\zzzz+2*\y, -3*\x) {$\vn$};
	\node at (\zzzz+2*\y, -4*\x) {$\vn$};
	\node[circ] at (\zzzz+2*\y, -5*\x) {$\vn$};
	\node at (\zzzz+2*\y, -6*\x) {$\vn$};
	\node at (\zzzz+2*\y, -7*\x) {$\vn$};
	\node[circ] at (\zzzz+2*\y, -8*\x) {$\vn$};
	\draw[rounded corners] (-\y+\zzzz, 0.5) rectangle (\y+\zzzz, -1*\x-.5);
	\node[fill=white] at (\zzzz,0.5) {$PT_4$};
	\draw[rounded corners] (\y+\zzzz, 0.5) rectangle (3*\y+\zzzz, -8*\x-.5);
\end{tikzpicture}
\caption{$\alg_{\prs}$: $\cP=\cP_{(7,6,5,4,3,2,1),9}, w=(9,8,7,5,6,3,2,4,1)$}\label{fig:7654321,987563241}
\end{figure}

\begin{example} Figure \ref{fig:986643221} shows the steps of $\alg_{\prs}$ when $\cP=\cP_{(9,8,6,6,4,3,2,2,1),10}$ and $w=(8, 4, 6, 7, 10, 1, 2, 5, 3, 9)$. Here we have $\Phi(w) = (\ytableaushort{1239,465{10},87}, \ytableaushort{1245,3789,6{10}})$. Note that $\{10-x \mid x\in \des_\cP(w)\} = \{2,5,9\}=\des(\ytableaushort{1245,3789,6{10}})$.
\end{example}

\begin{figure}[!htbp]
\begin{tikzpicture}[baseline=(current bounding box.center),xscale=0.5,yscale=0.9]
	\node (1) at (0,0) {1};
	\node (2) at (0,1) {2};
	\node (3) at (-1,2) {3};
	\node (4) at (1,2) {4};
	\node (5) at (-1,3) {5};
	\node (6) at (1,3) {6};
	\node (7) at (-1,4) {7};
	\node (8) at (1,4) {8};
	\node (9) at (0,5) {9};
	\node (10) at (0,6) {10};

	\draw[<-] (1) to (2);
	\draw[<-] (2) to (3);
	\draw[<-] (2) to (4);
	\draw[<-] (3) to (5);
	\draw[<-] (3) to (6);
	\draw[<-] (4) to (6);
	
	\draw[<-] (5) to (7);
	\draw[<-] (5) to (8);
	\draw[<-] (6) to (7);
	\draw[<-] (6) to (8);
	\draw[<-] (7) to (9);	
	\draw[<-] (8) to (9);
	\draw[<-] (9) to (10);	
\end{tikzpicture}\qquad
\begin{tikzpicture}[baseline=(current bounding box.center)]
	\tikzmath{\nnn=9;\x=0.5;\y=0.4;\z=1.8;\zz=4.4;\zzz=7;\zzzz=9.5;}
	\node at (0, -0*\x) {$2$};
	\node at (0, -1*\x) {$5$};
	\node at (0, -2*\x) {$3$};
	\node at (0, -3*\x) {$9$};
	\node at (0, -4*\x) {$1$};
	\node at (0, -5*\x) {$10$};
	\node at (0, -6*\x) {$7$};
	\node at (0, -7*\x) {$6$};
	\node at (0, -8*\x) {$4$};
	\node at (0, -9*\x) {$8$};
	
	\draw[rounded corners] (-\y, 0.5) rectangle (\y, -\nnn*\x-.5);
	\node[fill=white] at (0,0.5) {$w$};
	\draw[double distance=1.5pt, arrows={-Straight Barb[scale=0.6]}] (\y+.2, -2*\x) -- (\y+.8, -2*\x);
	
	\node at (\z, 0) {$1$};
	\node at (\z, -\x) {$4$};
	\node at (\z, -2*\x) {$8$};
	\node[circ] at (\z+2*\y, -0*\x) {$\vn$};
	\node at (\z+2*\y, -1*\x) {$9$};
	\node[circ] at (\z+2*\y, -2*\x) {$\vn$};
	\node at (\z+2*\y, -3*\x) {$3$};
	\node at (\z+2*\y, -4*\x) {$2$};
	\node[circ] at (\z+2*\y, -5*\x) {$\vn$};
	\node at (\z+2*\y, -6*\x) {$10$};
	\node at (\z+2*\y, -7*\x) {$5$};
	\node at (\z+2*\y, -8*\x) {$6$};
	\node at (\z+2*\y, -9*\x) {$7$};
	\draw[rounded corners] (-\y+\z, 0.5) rectangle (\y+\z,-2*\x-.5);
	\node[fill=white] at (\z,0.5) {$PT_1$};
	\draw[rounded corners] (\y+\z, 0.5) rectangle (3*\y+\z, -\nnn*\x-.5);
	\draw[ddarr] (\z+3*\y+.2, -2*\x) -- (\z+3*\y+.8, -2*\x);
	
	\node at (\zz, 0) {$2$};
	\node at (\zz, -\x) {$6$};
	\node at (\zz, -2*\x) {$7$};
	\node at (\zz+2*\y, -0*\x) {$\vn$};
	\node [circ] at (\zz+2*\y, -1*\x) {$\vn$};
	\node at (\zz+2*\y, -2*\x) {$\vn$};
	\node at (\zz+2*\y, -3*\x) {$9$};
	\node at (\zz+2*\y, -4*\x) {$3$};
	\node at (\zz+2*\y, -5*\x) {$\vn$};
	\node [circ] at (\zz+2*\y, -6*\x) {$\vn$};
	\node at (\zz+2*\y, -7*\x) {$10$};
	\node at (\zz+2*\y, -8*\x) {$5$};
	\node [circ] at (\zz+2*\y, -9*\x) {$\vn$};
	\draw[rounded corners] (-\y+\zz, 0.5) rectangle (\y+\zz, -2*\x-.5);
	\node[fill=white] at (\zz,0.5) {$PT_2$};
	\draw[rounded corners] (\y+\zz, 0.5) rectangle (3*\y+\zz, -\nnn*\x-.5);
	\draw[ddarr] (\zz+3*\y+.2, -2*\x) -- (\zz+3*\y+.8, -2*\x);

	\node at (\zzz, 0) {$3$};
	\node at (\zzz, -\x) {$5$};
	\node at (\zzz+2*\y, -0*\x) {$\vn$};
	\node at (\zzz+2*\y, -1*\x) {$\vn$};
	\node at (\zzz+2*\y, -2*\x) {$\vn$};
	\node [circ] at (\zzz+2*\y, -3*\x) {$\vn$};
	\node at (\zzz+2*\y, -4*\x) {$9$};
	\node at (\zzz+2*\y, -5*\x) {$\vn$};
	\node at (\zzz+2*\y, -6*\x) {$\vn$};
	\node [circ] at (\zzz+2*\y, -7*\x) {$\vn$};
	\node at (\zzz+2*\y, -8*\x) {$10$};
	\node at (\zzz+2*\y, -9*\x) {$\vn$};
	\draw[rounded corners] (-\y+\zzz, 0.5) rectangle (\y+\zzz, -1*\x-.5);
	\node[fill=white] at (\zzz,0.5) {$PT_3$};
	\draw[rounded corners] (\y+\zzz, 0.5) rectangle (3*\y+\zzz, -\nnn*\x-.5);
	\draw[ddarr] (\zzz+3*\y+.2, -2*\x) -- (\zzz+3*\y+.8, -2*\x);
	
	\node at (\zzzz, 0) {$9$};
	\node at (\zzzz, -\x) {$10$};
	\node at (\zzzz+2*\y, -0*\x) {$\vn$};
	\node at (\zzzz+2*\y, -1*\x) {$\vn$};
	\node at (\zzzz+2*\y, -2*\x) {$\vn$};
	\node at (\zzzz+2*\y, -3*\x) {$\vn$};
	\node [circ] at (\zzzz+2*\y, -4*\x) {$\vn$};
	\node at (\zzzz+2*\y, -5*\x) {$\vn$};
	\node at (\zzzz+2*\y, -6*\x) {$\vn$};
	\node at (\zzzz+2*\y, -7*\x) {$\vn$};
	\node [circ] at (\zzzz+2*\y, -8*\x) {$\vn$};
	\node at (\zzzz+2*\y, -9*\x) {$\vn$};
	\draw[rounded corners] (-\y+\zzzz, 0.5) rectangle (\y+\zzzz, -1*\x-.5);
	\node[fill=white] at (\zzzz,0.5) {$PT_4$};
	\draw[rounded corners] (\y+\zzzz, 0.5) rectangle (3*\y+\zzzz, -\nnn*\x-.5);
\end{tikzpicture}
\caption{$\alg_{\prs}$: $\cP=\cP_{(9,8,6,6,4,3,2,2,1),10}, w=(8, 4, 6, 7, 10, 1, 2, 5, 3, 9)$}\label{fig:986643221}
\end{figure}

\subsection{Some pathologies for ladder-climbing partial orders} \label{sec:patho} Here we provide some examples when $\prs$ does not produce a desired output when $\cP$ is ladder-climbing.

\ytableausetup{smalltableaux}
\begin{example} Figure \ref{fig:311,34521} shows the steps of $\alg_{\prs}$ when $\cP=\cP_{(3,1,1),5}$ and $w=(3,4,5,2,1)$. Here we see that $\Phi(w) = (\ytableaushort{312,54}, \ytableaushort{125,34})$. However, $PT$ in this case is not a $\cP$-tableau since $3\pr 1$.
\end{example}

\begin{figure}[!htbp]
\begin{tikzpicture}[baseline=(current bounding box.center), xscale=.8,yscale=.6, rotate=90]
	\node (1) at (0,0) {1};
	\node (3) at (2,0) {3};
	\node (5) at (4,0) {5};
	\node (2) at (1,1) {2};
	\node (4) at (3,1) {4};

	\draw[<-] (1) to (3);
	\draw[<-] (3) to (5);

	\draw[<-] (1) to (4);
	\draw[<-] (2) to (5);

\end{tikzpicture}\qquad
\begin{tikzpicture}[baseline=(current bounding box.center)]
	\tikzmath{\nnn=3;\x=0.5;\y=0.4;\z=1.8;\zz=4.4;\zzz=7;\zzzz=9.5;}
	\node at (0, 0) {$1$};
	\node at (0, -\x) {$2$};
	\node at (0, -2*\x) {$5$};
	\node at (0, -3*\x) {$4$};
	\node at (0, -4*\x) {$3$};
	
	\draw[rounded corners] (-\y, 0.5) rectangle (\y, -2.5);
	\node[fill=white] at (0,0.5) {$w$};
	\draw[double distance=1.5pt, arrows={-Straight Barb[scale=0.6]}] (\y+.2, -2*\x) -- (\y+.8, -2*\x);
	
	\node at (\z, 0) {$3$};
	\node at (\z, -\x) {$5$};
	\node[circ] at (\z+2*\y, 0) {$\vn$};
	\node at (\z+2*\y, -\x) {$1$};
	\node[circ] at (\z+2*\y, -2*\x) {$\vn$};
	\node at (\z+2*\y, -3*\x) {$4$};
	\node at (\z+2*\y, -4*\x) {$2$};
	\draw[rounded corners] (-\y+\z, 0.5) rectangle (\y+\z, -1);
	\node[fill=white] at (\z,0.5) {$PT_1$};
	\draw[rounded corners] (\y+\z, 0.5) rectangle (3*\y+\z, -2.5);
	\draw[double distance=1.5pt, arrows={-Straight Barb[scale=0.6]}] (\z+3*\y+.2, -2*\x) -- (\z+3*\y+.8, -2*\x);
	
	\node at (\zz, 0) {$1$};
	\node at (\zz, -\x) {$4$};
	\node at (\zz+2*\y, 0) {$\vn$};
	\node[circ] at (\zz+2*\y, -\x) {$\vn$};
	\node at (\zz+2*\y, -2*\x) {$\vn$};
	\node[circ] at (\zz+2*\y, -3*\x) {$\vn$};
	\node at (\zz+2*\y, -4*\x) {$2$};
	\draw[rounded corners] (-\y+\zz, 0.5) rectangle (\y+\zz, -1);
	\node[fill=white] at (\zz,0.5) {$PT_2$};
	\draw[rounded corners] (\y+\zz, 0.5) rectangle (3*\y+\zz, -2.5);
	\draw[double distance=1.5pt, arrows={-Straight Barb[scale=0.6]}] (\zz+3*\y+.2, -2*\x) -- (\zz+3*\y+.8, -2*\x);

	\node at (\zzz, 0) {$2$};
	\node at (\zzz+2*\y, 0) {$\vn$};
	\node at (\zzz+2*\y, -\x) {$\vn$};
	\node at (\zzz+2*\y, -2*\x) {$\vn$};
	\node at (\zzz+2*\y, -3*\x) {$\vn$};
	\node[circ] at (\zzz+2*\y, -4*\x) {$\vn$};
	\draw[rounded corners] (-\y+\zzz, 0.5) rectangle (\y+\zzz, -.5);
	\node[fill=white] at (\zzz,0.5) {$PT_3$};
	\draw[rounded corners] (\y+\zzz, 0.5) rectangle (3*\y+\zzz, -2.5);	
\end{tikzpicture}
\caption{$\alg_{\prs}$: $\cP=\cP_{(3,1,1),5}, w=(3,4,5,2,1)$}\label{fig:311,34521}
\end{figure}

\begin{example} Figure \ref{fig:4211,436521} shows the steps of $\alg_{\prs}$ when $\cP=\cP_{(4,2,1,1),6}$ and $w=(4,3,6,5,2,1)$. Here we see that $\Phi(w) = (\ytableaushort{4132,65}, \ytableaushort{1256,34})$. However, $PT$ in this case is not a $\cP$-tableau since $4\pr 1$.
\end{example}

\begin{figure}[!htbp]
\begin{tikzpicture}[baseline=(current bounding box.center), xscale=.8,yscale=.6, rotate=90]
	\node (1) at (0,0) {1};
	\node (2) at (1,1) {2};
	\node (3) at (2,0) {3};
	\node (4) at (3,-1) {4};
	\node (5) at (4,1) {5};
	\node (6) at (5,0) {6};

	\draw[<-] (1) to (3);
	\draw[<-] (1) to (4);
	\draw[<-] (1) to (5);
	\draw[<-] (2) to (5);
	\draw[<-] (2) to (6);
	
	\draw[<-] (3) to (6);
	\draw[<-] (4) to (6);

\end{tikzpicture}\qquad
\begin{tikzpicture}[baseline=(current bounding box.center)]
	\tikzmath{\nnn=3;\x=0.5;\y=0.4;\z=1.8;\zz=4.4;\zzz=7;\zzzz=9.5;}
	\node at (0, 0) {$1$};
	\node at (0, -\x) {$2$};
	\node at (0, -2*\x) {$5$};
	\node at (0, -3*\x) {$6$};
	\node at (0, -4*\x) {$3$};
	\node at (0, -5*\x) {$4$};
	
	\draw[rounded corners] (-\y, 0.5) rectangle (\y, -3);
	\node[fill=white] at (0,0.5) {$w$};
	\draw[double distance=1.5pt, arrows={-Straight Barb[scale=0.6]}] (\y+.2, -2*\x) -- (\y+.8, -2*\x);
	
	\node at (\z, 0) {$4$};
	\node at (\z, -\x) {$6$};
	\node[circ] at (\z+2*\y, 0) {$\vn$};
	\node at (\z+2*\y, -\x) {$1$};
	\node[circ] at (\z+2*\y, -2*\x) {$\vn$};
	\node at (\z+2*\y, -3*\x) {$5$};
	\node at (\z+2*\y, -4*\x) {$2$};
	\node at (\z+2*\y, -5*\x) {$3$};
	\draw[rounded corners] (-\y+\z, 0.5) rectangle (\y+\z, -1);
	\node[fill=white] at (\z,0.5) {$PT_1$};
	\draw[rounded corners] (\y+\z, 0.5) rectangle (3*\y+\z, -3);
	\draw[double distance=1.5pt, arrows={-Straight Barb[scale=0.6]}] (\z+3*\y+.2, -2*\x) -- (\z+3*\y+.8, -2*\x);
	
	\node at (\zz, 0) {$1$};
	\node at (\zz, -\x) {$5$};
	\node at (\zz+2*\y, 0) {$\vn$};
	\node[circ] at (\zz+2*\y, -\x) {$\vn$};
	\node at (\zz+2*\y, -2*\x) {$\vn$};
	\node[circ] at (\zz+2*\y, -3*\x) {$\vn$};
	\node at (\zz+2*\y, -4*\x) {$2$};
	\node at (\zz+2*\y, -5*\x) {$3$};
	\draw[rounded corners] (-\y+\zz, 0.5) rectangle (\y+\zz, -1);
	\node[fill=white] at (\zz,0.5) {$PT_2$};
	\draw[rounded corners] (\y+\zz, 0.5) rectangle (3*\y+\zz, -3);
	\draw[double distance=1.5pt, arrows={-Straight Barb[scale=0.6]}] (\zz+3*\y+.2, -2*\x) -- (\zz+3*\y+.8, -2*\x);

	\node at (\zzz, 0) {$3$};
	\node at (\zzz+2*\y, 0) {$\vn$};
	\node at (\zzz+2*\y, -\x) {$\vn$};
	\node at (\zzz+2*\y, -2*\x) {$\vn$};
	\node at (\zzz+2*\y, -3*\x) {$\vn$};
	\node[circ] at (\zzz+2*\y, -4*\x) {$\vn$};
	\node at (\zzz+2*\y, -5*\x) {$2$};
	\draw[rounded corners] (-\y+\zzz, 0.5) rectangle (\y+\zzz, -.5);
	\node[fill=white] at (\zzz,0.5) {$PT_3$};
	\draw[rounded corners] (\y+\zzz, 0.5) rectangle (3*\y+\zzz, -3);	
	\draw[double distance=1.5pt, arrows={-Straight Barb[scale=0.6]}] (\zzz+3*\y+.2, -2*\x) -- (\zzz+3*\y+.8, -2*\x);
	
	\node at (\zzzz, 0) {$2$};
	\node at (\zzzz+2*\y, 0) {$\vn$};
	\node at (\zzzz+2*\y, -\x) {$\vn$};
	\node at (\zzzz+2*\y, -2*\x) {$\vn$};
	\node at (\zzzz+2*\y, -3*\x) {$\vn$};
	\node at (\zzzz+2*\y, -4*\x) {$\vn$};
	\node[circ] at (\zzzz+2*\y, -5*\x) {$\vn$};
	\draw[rounded corners] (-\y+\zzzz, 0.5) rectangle (\y+\zzzz, -.5);
	\node[fill=white] at (\zzzz,0.5) {$PT_4$};
	\draw[rounded corners] (\y+\zzzz, 0.5) rectangle (3*\y+\zzzz, -3);	
\end{tikzpicture}
\caption{$\alg_{\prs}$: $\cP=\cP_{(4,2,1,1),6}, w=(4,3,6,5,2,1)$}\label{fig:4211,436521}
\end{figure}

\begin{example} Figure \ref{fig:53211,3156742} shows the steps of $\alg_{\prs}$ when $\cP=\cP_{(5,3,2,1,1),7}$ and $w=(3,1,5,6,7,4,2)$. Here we see that $\Phi(w) = (\ytableaushort{1264,35,{\none}7}, \ytableaushort{1256,34,{\none}7})$. However, $PT$ and $QT$ in this case are not even tableaux.
\end{example}

\begin{figure}[H]
\begin{tikzpicture}[baseline=(current bounding box.center), xscale=.8,yscale=.6, rotate=90]
	\node (1) at (0,0) {1};
	\node (2) at (1,1) {2};
	\node (3) at (2,0) {3};
	\node (4) at (3,-1) {4};
	\node (5) at (4,1) {5};
	\node (6) at (5,0) {6};
	\node (7) at (6,1) {7};

	\draw[<-] (1) to (3);
	\draw[<-] (1) to (4);
	\draw[<-] (1) to (5);
	\draw[<-] (2) to (5);
	\draw[<-] (2) to (6);
	\draw[<-] (3) to (6);
	\draw[<-] (3) to (7);
	\draw[<-] (4) to (7);
	\draw[<-] (5) to (7);
\end{tikzpicture}\qquad
\begin{tikzpicture}[baseline=(current bounding box.center)]
	\tikzmath{\nnn=3;\x=0.5;\y=0.4;\z=1.8;\zz=4.4;\zzz=7;\zzzz=9.5;}
	\node at (0, 0) {$2$};
	\node at (0, -\x) {$4$};
	\node at (0, -2*\x) {$7$};
	\node at (0, -3*\x) {$6$};
	\node at (0, -4*\x) {$5$};
	\node at (0, -5*\x) {$1$};
	\node at (0, -6*\x) {$3$};
	
	\draw[rounded corners] (-\y, 0.5) rectangle (\y, -3.5);
	\node[fill=white] at (0,0.5) {$w$};
	\draw[ddarr] (\y+.2, -2*\x) -- (\y+.8, -2*\x);
	
	\node at (\z, 0) {$1$};
	\node at (\z, -\x) {$3$};
	\node[circ] at (\z+2*\y, 0) {$\vn$};
	\node at (\z+2*\y, -\x) {$2$};
	\node[circ] at (\z+2*\y, -2*\x) {$\vn$};
	\node at (\z+2*\y, -3*\x) {$6$};
	\node at (\z+2*\y, -4*\x) {$4$};
	\node at (\z+2*\y, -5*\x) {$5$};
	\node at (\z+2*\y, -6*\x) {$7$};
	\draw[rounded corners] (-\y+\z, 0.5) rectangle (\y+\z, -1);
	\node[fill=white] at (\z,0.5) {$PT_1$};
	\draw[rounded corners] (\y+\z, 0.5) rectangle (3*\y+\z, -3.5);
	\draw[ddarr] (\z+3*\y+.2, -2*\x) -- (\z+3*\y+.8, -2*\x);
	
	\node at (\zz, 0) {$2$};
	\node at (\zz, -\x) {$5$};
	\node at (\zz, -2*\x) {$7$};
	\node at (\zz+2*\y, 0) {$\vn$};
	\node[circ] at (\zz+2*\y, -\x) {$\vn$};
	\node at (\zz+2*\y, -2*\x) {$\vn$};
	\node[circ] at (\zz+2*\y, -3*\x) {$\vn$};
	\node at (\zz+2*\y, -4*\x) {$4$};
	\node at (\zz+2*\y, -5*\x) {$6$};
	\node[circ] at (\zz+2*\y, -6*\x) {$\vn$};
	\draw[rounded corners] (-\y+\zz, 0.5) rectangle (\y+\zz, -1.5);
	\node[fill=white] at (\zz,0.5) {$PT_2$};
	\draw[rounded corners] (\y+\zz, 0.5) rectangle (3*\y+\zz, -3.5);
	\draw[ddarr] (\zz+3*\y+.2, -2*\x) -- (\zz+3*\y+.8, -2*\x);

	\node at (\zzz, 0) {$6$};
	\node at (\zzz+2*\y, 0) {$\vn$};
	\node at (\zzz+2*\y, -\x) {$\vn$};
	\node at (\zzz+2*\y, -2*\x) {$\vn$};
	\node at (\zzz+2*\y, -3*\x) {$\vn$};
	\node[circ] at (\zzz+2*\y, -4*\x) {$\vn$};
	\node at (\zzz+2*\y, -5*\x) {$4$};
	\node at (\zzz+2*\y, -6*\x) {$\vn$};
	\draw[rounded corners] (-\y+\zzz, 0.5) rectangle (\y+\zzz, -.5);
	\node[fill=white] at (\zzz,0.5) {$PT_3$};
	\draw[rounded corners] (\y+\zzz, 0.5) rectangle (3*\y+\zzz, -3.5);	
	\draw[ddarr] (\zzz+3*\y+.2, -2*\x) -- (\zzz+3*\y+.8, -2*\x);
	
	\node at (\zzzz, 0) {$4$};
	\node at (\zzzz+2*\y, 0) {$\vn$};
	\node at (\zzzz+2*\y, -\x) {$\vn$};
	\node at (\zzzz+2*\y, -2*\x) {$\vn$};
	\node at (\zzzz+2*\y, -3*\x) {$\vn$};
	\node at (\zzzz+2*\y, -4*\x) {$\vn$};
	\node[circ] at (\zzzz+2*\y, -5*\x) {$\vn$};
	\node at (\zzzz+2*\y, -6*\x) {$\vn$};
	\draw[rounded corners] (-\y+\zzzz, 0.5) rectangle (\y+\zzzz, -.5);
	\node[fill=white] at (\zzzz,0.5) {$PT_4$};
	\draw[rounded corners] (\y+\zzzz, 0.5) rectangle (3*\y+\zzzz, -3.5);	
\end{tikzpicture}
\caption{$\alg_{\prs}$: $\cP=\cP_{(5,3,2,1,1),7}, w=(3, 1, 5, 6, 7, 4, 2)$}\label{fig:53211,3156742}
\end{figure}

\section{Proof of Proposition  \ref{prop:ght}, \ref{prop:column}, and \ref{prop:phiinj}}  \label{sec:proofprop}

\subsection{Proof of Proposition \ref{prop:ght}}
We start with the following lemma.
\begin{lem} \label{lem:ght} Let $\cP$ be a natural unit interval order on $[1,n]$. Suppose that we are given $a, b \in [1,n]$ and $w \in \sym_n$ that satisfy $w^{-1}(a)<w^{-1}(b)$ and $a\pr b$. Then the following are equivalent:
\begin{enumerate}
\item $(a,b) \not\in \ght_\cP(w)$.
\item there exists a subword $ad_1\cdots d_kb$ of $w$ such that $a \dash_\cP d_1 \dash_\cP \cdots \dash_\cP d_k \dash_\cP b$.
\item there exists a subword $ad_1\cdots d_kb$ of $w$ such that $a \pdr d_1 \pdr \cdots \pdr d_k \pdr b$.
\item there exists a subword $ad_1\cdots d_kb$ of $w$ such that $a \pdr d_1 \pdr \cdots \pdr d_k \pdr b$ and $\{a, d_1, \ldots, d_k, b\}$ is a ladder in $\cP$.
\end{enumerate}
\end{lem}
\begin{proof}
(1) $\Leftrightarrow$ (2) by definition, and (4) $\Rightarrow$ (3) $\Rightarrow$ (2) is clear. It remains to show (2) $\Rightarrow$ (4). To this end, suppose that $ad_1\cdots d_kb$ is a subword of $w$ such that $a \dash_\cP d_1 \dash_\cP \cdots \dash_\cP d_k \dash_\cP b$ and we construct another subword satisfying the condition of (4).

For simplicity we set $d_0=a$ and $d_{k+1}=b$. First, let $d_i$ be the last element in the subword $d_0d_1\cdots d_kd_{k+1}$ such that $d_0\dash_\cP d_i$. By removing $d_1, d_2, \ldots, d_{i-1}$ if necessary, without loss of generality we may assume that it is $d_1$. Similarly we let $d_j$ be the last element in the subword $d_0d_1\cdots d_kd_{k+1}$ such that $d_1 \dash_\cP d_j$. By removing $d_2, \ldots, d_{j-1}$ if necessary, without loss of generality we may assume that it is $d_2$. We iterate this process, and thus we may assume that $d_{i+1}$ is the last element in $d_0d_1\cdots d_kd_{k+1}$ such that $d_i \dash_\cP d_{i+1}$ for $i \in [0,k]$.

We claim that $d_i \pdr d_{i+1}$ for any $i \in [0,k]$. For the sake of contradiction suppose otherwise, i.e. $d_i \pdl d_{i+1}$ and let $i \in [0,k]$ be the smallest element such that $d_i \pdl d_{i+1}$. If $i >0$, then $d_{i-1} \pdr d_i$ by assumption and thus $d_{i-1} \dash_\cP d_{i+1}$ by \cond{}, which contradicts the fact that $d_{i}$ is the last element satisfying $d_{i-1} \dash_\cP d_i$. If $i=0$, i.e. $a \pdl d_1$, then we claim that $a \pl d_j$ for any $j \geq 2$, which in particular contradicts the fact that $a \pr b=d_{k+1}$. Thus suppose otherwise. Then there exists $j$ such that $a \pr d_j$, and we choose such $j$ to be the minimum one. Then we have $a<d_{j-1}$, $d_{j-1} \dash_\cP d_j$, and $a \pr d_j$, which contradicts \cond{}. (Note that this argument is valid even when $j=2$.)

Now for any $i \in [1,k]$, we have $d_{i-1} \pdr d_i \pdr d_{i+1}$, and thus $d_{i-1}\pr d_{i+1}$ since again $d_i$ is the last element satisfying $d_{i-1} \dash_\cP d_i$. But this means that $\{d_0,d_1, \ldots, d_k, d_{k+1}\}$ is a ladder in $\cP$, which is what we want to prove.
\end{proof}

%
%
%
%
%

Let us start proving Proposition \ref{prop:ght}. Let $w = \cdots xyz \cdots$ and $w'=\cdots x'y'z' \cdots$ such that the $\cP$-Knuth move connecting $w$ and $w'$ shuffles $\{x, y, z\}=\{x',y',z'\}$. 
Suppose that $\cI\colonequals(p_1, \ldots, p_s, r_1, \ldots, r_u, q_1, \ldots, q_t)$ is the longest subword of genuine $\cP$-inversions in $w$ such that $\und{\cI} \cap \{x, y, z\} = \{r_1, \ldots, r_u\}$. In particular we have $\ght_\cP(w) = s+t+u$. Note that $s, t \geq 0$ and $u\in [0,2]$. ($u$ cannot be 3 since otherwise $x\pr y\pr z$ in which case there is no $\cP$-Knuth move shuffling $x,y,z$.) 

We may assume $s, t \geq 2$. Indeed, we add $n+1, n+2$ and $0, -1$ to the poset $([1,n], \cP)$ so that $n+2 \pr n+1 \pr i \pr 0 \pr -1$ for any $i \in [1,n]$. Then we replace $w$ and $w'$ with $(n+2, n+1)+w+(0, -1)$ and $(n+2, n+1)+w'+(0, -1)$, respectively. After this procedure, $\cP$ still remains to avoid $\cP_{(3,1,1),5}$ and $\cP_{(4,2,1,1), 6}$. Also, any longest subword of genuine $\cP$-inversions in $w$ and $w'$ contains $n+2, n+1, 0,$ and $-1$, which in particular increases $s$ and $t$ by 2, respectively. (Here, we have $p_1=n+2, p_2=n+1, q_{t-1}=0,$ and $q_t=-1$.)

We will construct the subword $\cI'$ of genuine $\cP$-inversions of length $s+t+u$ on $w'$, which contains $p_1, \ldots, p_{s-1}, q_2, \ldots, q_t$. It means in particular that $\ght_\cP(w') \geq s+t+u$, which proves the claim by symmetry. Note that we possibly change only $p_s, r_1, \ldots, r_u, q_1$ part from $\cI$ to obtain $\cI'$. Therefore, the entries before $p_{s-1}$ and after $q_2$ in $w$ and $w'$ do not affect this process. By removing such entries if necessary, it suffices to assume that $s=t=2$, i.e. we have $\cI=(p_1, p_2, r_1, \ldots, r_u, q_1, q_2)$. 

Let $a, b, c \in [1,n]$ be such that $a<b<c$ and $\{a, b, c\}=\{x,y,z\}=\{x',y',z'\}$. From now on we argue case-by-case based on $\cP|_{\{a, b, c\}}$ and $u\in [0,2]$.

\subsubsection{$\cP|_{\{a, b, c\}} \simeq \cP_{\emptyset,3}$} There is no $\cP$-Knuth move in this case, so there is nothing to prove.

\subsubsection{$\cP|_{\{a,b,c\}}\simeq \cP_{(1),3}$}
We have $a\pdl b \pdl c$ and $a \pl c$. There is only one $\cP$-Knuth move in this case: $[\cdots bca \cdots] \pkm [\cdots cab \cdots]$.

\noindent{\bfseries \textbullet\ $u=0$ case.} In this case we have $\cI=p_1p_2q_1q_2$ and in particular $(p_2,q_1) \in \ginv_\cP(w)$. If $(p_2, q_1) \in \ginv_\cP(w')$ then we are done since we may set $\cI'=\cI$. From now on we assume $(p_2, q_1) \not\in \ginv_\cP(w')$.

First consider the case when $w=[\cdots bca \cdots]$ and $w'=[\cdots cab \cdots]$. By Lemma \ref{lem:ght}, there exists a subword $\cJ\colonequals p_2d_1\cdots d_kq_1$ ($k\geq 1$) in $w'$ where $p_2 \pdr d_1\pdr \cdots \pdr d_k\pdr q_1$ and $\und{\cJ}$ is a ladder in $\cP$. As $(p_2,q_1) \in \ginv_\cP(w)$, it means that the order of some elements in $\cJ$ should be switched under the $\cP$-Knuth move so that it no longer prohibits $(p_2, q_1)$ from being a genuine $\cP$-inversion of $w$. Then the only possibility is that $\cJ=p_2d_1\cdots d_{v-1} cb d_{v+2} \cdots d_kq_1$, i.e. $d_v=c$ and $d_{v+1}=b$ for some $v \in [1,k-1]$. 

We claim that we may choose $\cI'\colonequals p_1p_2aq_2$. To this end we need to check that $(p_2,a), (a, q_2) \in \ght_\cP(w')$. First if $(p_2,a) \not\in \ght_\cP(w')$ then as $p_2>c\pr a$ there exists a subword $p_2e_1\cdots e_la$ of $w'$ such that $p_2 \dash_\cP e_1 \dash_\cP \cdots \dash_\cP e_l \dash_\cP a$. Since $c \pr a$, $e_l$ appears before $c$ in $w'$ which means that $p_2e_1\cdots e_la$ is also a subword of $w$. However, this contradicts $(p_2, q_1) \in \ght_\cP(w)$ as $a \dash_\cP b \dash_\cP d_{v+2} \dash_\cP \cdots \dash_\cP d_{k} \dash_\cP q_1$. 

This time suppose that $(a,q_2) \not\in \ght_\cP(w')$. Since $b>q_1 \pr q_2$ and $b \pdr a$, by \cond{} we should have $a>q_2$. First suppose that $a \pdr q_2$. If $v<k-1$ then $b \pr q_1 \pr q_2$ and $b\pdr a\pdr q_2$, which contradicts Lemma \ref{lem:3122}. Thus $v=k-1$, i.e. $\cJ=\cdots cbq_1$. However, direct calculation shows that $\cP$ restricted to $\{c, b, q_1, a, q_2\}$ is isomorphic to $\cP_{(3,1,1),5}$, which is again a contradiction.

Thus we have $a \pr q_2$, and by Lemma \ref{lem:ght} there exists a subword $ae_1\cdots e_lq_2$ of $w'$ ($l\geq 1$) such that $a \pdr e_1 \pdr \cdots \pdr e_l \pdr q_2$ and $\{a, e_1, \ldots, e_l, q_2\}$ is a ladder in $\cP$. By Lemma \ref{lem:ladjoin} it follows that $\{c,b,a, e_1, \ldots, e_l, q_2\}$ is also a ladder in $\cP$. By assumption we have $c \pr q_1 \pr q_2$, and also $q_1 \not \in \{c,b,a, e_1, \ldots, e_l, q_2\}$; it is clear that $q_1\neq c,b,a, q_2$, and if $q_1=e_i$ for some $i$ then it means that $(q_1, q_2) \not\in \ght(w)$. In other words, $q_1$ is climbing the ladder $\{c,b,a, e_1, \ldots, e_l, q_2\}$, which contradicts the assumption on $\cP$.

It remains to consider the case when $w=[\cdots cab \cdots]$ and $w'=[\cdots bca \cdots]$. By arguing similarly, we should be able to find a subword $\cJ=p_2d_1\cdots d_{v-1}bad_{v+2}\cdots d_kq_1$ of $w'$ with the same properties as above. Then we may choose $\cI'=p_1cq_1q_2$, which can also be shown analogously. We omit the details.

\noindent{\bfseries \textbullet\ $u=1$ case.} We have $\cI=p_1p_2r_1q_1q_2$ where $r_1 \in \{a,b,c\}$. First consider the case when $w=[\cdots bca \cdots]$ and $w'=[\cdots cab \cdots]$. Using Lemma \ref{lem:ght}, one can easily show that if $r_1=a$ then we may set $\cI'=\cI$ and we are done. 

Now suppose that $r_1=b$. 
We claim that we may choose $\cI'=p_1p_2aq_1q_2$. To this end we need to verify that $(p_2,a), (a, q_1) \in \ght_\cP(w')$. First if $(p_2, a) \not\in\ght_\cP(w')$ then as $p_2\pr b > a$ there exists a subword $p_2e_1\cdots e_la$ of $w'$ such that $p_2 \pdr e_1\pdr \cdots \pdr e_l \pdr a$. Since $b\pdr a$, either $e_l=b$ or $e_l \dash_\cP b$ by \cond{}, which means that  $(p_2, b) \not\in \ght_\cP(w)$. This is a contradiction and thus we have $(p_2, a) \in\ght_\cP(w')$. On the other hand, since $b\pr q_1$ and $b\pdr a$ we have $a>q_1$ by \cond{}. Now if $(a, q_1) \not\in \ght_\cP(w')$, i.e. there exists a subword $ae_1\cdots e_lq_1$ ($l\geq 0$) of $w'$ such that $a\pdr e_1 \pdr \cdots \pdr e_l \pdr q_1$ (note that $e_1$ cannot be $b$), then since $b \pdr a$ it follows that $(b, q_1) \not\in \ght_\cP(w)$. This is a contradiction, and thus we have $(a, q_1) \in \ght_\cP(w')$ as desired.

If $r_1=c$, then by the same argument one can easily show that we may choose $\cI'=p_1p_2bq_1q_2$. We omit the details.

%
%

It remains to consider the case when $w=[\cdots cab \cdots]$ and $w'=[\cdots bca \cdots]$. Similarly to above, we may set $\cI'=p_1p_2cq_1 q_2$ if $r_1\in \{b,c\}$ and $\cI'=p_1q_2bq_1q_2$ if $r_1=a$. Again we omit the details.

\noindent{\bfseries \textbullet\ $u=2$ case.} We have $\cI = p_1p_2caq_1q_2$. Then using Lemma \ref{lem:ght}, one can easily show that $\cI'=\cI$ satisfies the desired properties.

\subsubsection{$\cP|_{\{a,b,c\}}\simeq \cP_{(1,1),3}$} 
We have $a \pl b \pdl c$ and $a \pl c$. Here we have two kinds of $\cP$-Knuth moves: $[\cdots bca \cdots]\pkm[\cdots bac \cdots]$ and $[\cdots cba \cdots]\pkm[\cdots cab \cdots]$.

\noindent{\bfseries \textbullet\ $u=0$ case.} In this case we have $\cI=p_1p_2q_1q_2$ and in particular $(p_2,q_1) \in \ginv_\cP(w)$. Since the relative position of $b$ and $c$ does not change by the $\cP$-Knuth moves in this case, one can easily show that $(p_2, q_1) \in \ginv_\cP(w')$ by using Lemma \ref{lem:ght}. Then we are done since we may set $\cI'=\cI$. 

\noindent{\bfseries \textbullet\ $u=1$ case.} We have $\cI=p_1p_2r_1q_1q_2$ where $r_1 \in \{a,b,c\}$. Again, since the relative position of $b$ and $c$ does not change by the $\cP$-Knuth moves in this case, one can easily show that we may set $\cI'=\cI$ by using Lemma \ref{lem:ght}.

\noindent{\bfseries \textbullet\ $u=2$ case.} We have $\cI = p_1p_2r_1r_2q_1q_2$ where either $(r_1, r_2) = (b,a)$ or $(r_1, r_2) = (c,a)$. Since the relative position of $b$ and $c$ does not change by the $\cP$-Knuth moves in this case, we only need to consider the situations when the relative position of $r_1$ and $r_2$ changes under the moves. We have two cases to consider.

The first case is when $\cI=p_1p_2baq_1q_2$, $w=[\cdots cba \cdots]$, and $w'=[\cdots cab \cdots]$. We claim that we may choose $\cI'=p_1p_2 caq_1q_2$. To this end we need to check $(p_2,c) \in \ght_\cP(w')$. (It is trivial that $(c,a) \in \ght_\cP(w')$ in this case.) Since $p_2 \pr b$ and $c \pdr b$, we have $p_2>c$ by \cond{}. Thus if $(p_2,c) \not\in \ght_\cP(w')$ then there exists a subword $p_2e_1\cdots e_lc$ ($l\geq 0$) of $w'$ such that $p_2 \pdr e_1 \pdr \cdots \pdr e_l \pdr c$. However, since $c \pdr b$, it implies $(p_2,b) \not\in \ght_\cP(w)$, which is a contradiction.

The remaining case is when $\cI=p_1p_2caq_1q_2$, $w=[\cdots bca \cdots]$, and $w'=[\cdots bac \cdots]$. Here one can show that we may choose $\cI'=p_1p_2baq_1q_2$ similarly to above. We omit the details.

\subsubsection{$\cP|_{\{a,b,c\}}\simeq \cP_{(2),3}$} This case is completely analogous to the above case when $\cP|_{\{a,b,c\}}\simeq \cP_{(1,1),3}$ if one ``reverses'' the words therein. We again omit the details.

\subsubsection{$\cP|_{\{a,b,c\}}\simeq \cP_{(2,1),3}$} We have $a\pl b\pl c$. Here we have two kinds of $\cP$-Knuth moves: $[\cdots bca \cdots]\pkm[\cdots bac \cdots]$ and $[\cdots acb \cdots]\pkm[\cdots cab \cdots]$.

\noindent{\bfseries \textbullet\ $u=0$ case.} In this case we have $\cI=p_1p_2q_1q_2$. It is easy to show that $(p_2, q_1) \in \ght_\cP(w)$ implies $(p_2, q_1) \in \ght_\cP(w')$, and thus we may set $\cI'=\cI$. 

\noindent{\bfseries \textbullet\ $u=1$ case.} We have $\cI=p_1p_2r_1q_1q_2$ where $r_1 \in \{a,b,c\}$. Again, one can easily show that we may set $\cI'=\cI$.

\noindent{\bfseries \textbullet\ $u=2$ case.} We have $\cI = p_1p_2r_1r_2q_1q_2$ where $(r_1, r_2) \in \{(b,a), (c,a), (c,b)\}$. Here we only need to consider the situations when the relative position of $r_1$ and $r_2$ changes under the $\cP$-Knuth moves. We have two cases to consider.

The first case is when $\cI = p_1p_2caq_1q_2$, $w=[\cdots bca \cdots]$, and $w'=[\cdots bac \cdots]$. Here we may choose $\cI'=p_1p_2baq_1q_2$. Trivially $(b,a) \in \ght_\cP(w')$, so it suffices to check that  $(p_2,b) \in \ght_\cP(w')$. If we suppose otherwise then since $p_2 \pr c \pr b$ there exists a subword $p_2e_1\cdots e_lb$ ($l>0$) in $w'$ such that $p_2\pdr e_1 \pdr \cdots \pdr e_l \pdr b$ and $\{p_2, e_1, \ldots, e_l, b\}$ is a ladder in $\cP$. However, $c \not \in \{p_2, e_1, \ldots, e_l, b\}$ and $p_2 \pr c \pr b$ which means that $c$ is climbing the aforementioned ladder, which is a contradiction. Thus we have $(p_2,b) \in \ght_\cP(w')$.

The remaining case is when  $\cI = p_1p_2caq_1q_2$, $w=[\cdots cab \cdots]$, and $w'=[\cdots acb \cdots]$. Here we may choose $\cI'=p_1p_2cbq_1q_2$ similarly to above. We omit the details.

We exhausted all the cases and thus conclude the statement.

\subsection{Proof of Proposition \ref{prop:column}\ref{mainprop1}} 
Here, we write $(\cdots \und{abc} \cdots) \pkm (\cdots a'b'c' \cdots)$ for $\{a', b' , c'\}=\{a,b,c\}$ to indicate the location of a $\cP$-Knuth move (that is the underlined part). We start with the following lemma. After this, we use wavy underlines, e.g. $\cdots \uwave{x_1\cdots x_k} \cdots \psim \cdots y_1\cdots y_k \cdots$ for $\{x_1, \ldots, x_k\}=\{y_1, \ldots, y_k\}$ to indicate the part where Lemma \ref{lem:simple1} is applied.
\begin{lem} Suppose that a fixed natural unit interval order $\cP$ is given. 
\label{lem:simple1}
\begin{enumerate} 
\item If $b_1\pl b_2\pl\cdots\pl b_k$ and $a<b_1$ then $ab_k\cdots b_1 \psim  b_k \cdots b_2ab_1$.
\item If $b_1\pl b_2\pl \cdots \pl b_k$ and $a>b_k$ then $b_k\cdots b_1a \psim  b_kab_{k-1} \cdots b_1$.
\item If $b_1\pdl b_2 \pdl\cdots  \pdl b_k$, $x\pl y$, and $y<b_1$ then $yxb_k\cdots b_1\psim yb_k\cdots b_1x$.
\item If $b_1\pdl b_2 \pdl\cdots \pdl b_k$, $x\pl y$, and $x>b_k$ then $b_k\cdots b_1yx\psim yb_k\cdots b_1x$.
\end{enumerate}
\end{lem}
\begin{proof} For (1), we have 
$
\und{ab_kb_{k-1}}\cdots b_1 \pkm b_k\und{ab_{k-1}b_{k-2}} \cdots b_1 \pkm\cdots\pkm b_k\cdots b_3\und{ab_2b_1}\pkm b_k\cdots b_3b_2ab_1
$. (2) is proved similarly. For (3), we have $\und{yxb_k}\cdots b_1  \pkm y\und{b_kxb_{k-1}}\cdots b_1 \pkm \cdots \pkm yb_k \cdots \und{b_2xb_1} \pkm yb_k \cdots b_2b_1x$. (4) is proved similarly.
\end{proof}

We prove (A) by showing that each step in the column insertion algorithm respects the $\cP$-Knuth equivalence. It is sufficient to consider when the input is given by $(\alpha=(a_m, \ldots, a_1), c=(c_l, \ldots, c_1))$ such that $c_1\pl c_2 \pl \cdots \pl c_l$, $\alpha=\alpha^f$, and only one step of $\alg_\Phi$ is required to obtain the output. We argue case-by-case.

\noindent\und{\csi[(a)].} There is nothing to prove since $\alpha+c=d$ and $\beta=(\vn)$.

\noindent\und{\csi[(b)].} The length of $\alpha$ equals 1, i.e. $\alpha=(a)$ for some $a$. First suppose that  $c_i<a<c_{i+1}$ and $c_i \pl a$ for some $i \in [1,l-1]$. We have
\begin{align*}
&\uwave{ac_l\cdots c_{i+2} c_{i+1}}c_i \cdots c_1 \psim c_l\cdots c_{i+2}\und{ac_{i+1}c_i}c_{i-1} \cdots c_1
\\&\pkm c_l\cdots c_{i+2}a\uwave{c_ic_{i+1}c_{i-1} \cdots c_1}\psim c_l\cdots c_{i+2}ac_i\cdots c_1c_{i+1}.
\end{align*}
which proves the claim.
It remains to show that $ac_l\cdots c_1 \psim c_l\cdots c_2ac_1$ if $c_1 \pl c_2 \pl \cdots \pl c_{l}$ and $a<c_1$, but it follows directly from Lemma \ref{lem:simple1}.

\noindent\und{\csii[(a)].}
It suffices to prove the following lemma.
\begin{lem} \label{lem:simple2} Suppose that $c_1 \pl \cdots \pl c_l$, $c_i<a_1<a_2<\cdots<a_m<c_{i+1}$ for some $i \in [1,l-1]$, and  $\{c_i, a_1, \ldots, a_m, c_{i+1}\}$ is a ladder in $\cP$. Then $a_m\cdots a_1 c_l\cdots c_1 \psim c_l\cdots c_1 a_m \cdots a_1$.
\end{lem}
\begin{proof} Since we have
\begin{align*}
&\uwave{a_m\cdots a_1c_lc_{l-1}}c_{l-2} \cdots c_1 \psim c_l\uwave{a_m\cdots a_1c_{l-1}c_{l-2}}c_{l-3}\cdots c_1\psim\cdots 
\\&\psim c_l\cdots c_{i+3}\uwave{a_m\cdots a_1c_{i+2}c_{i+1}}c_i\cdots c_1 \psim c_l\cdots c_{i+2}a_m\cdots a_1c_{i+1}c_i \cdots c_1 \qquad\textnormal{ and}
\\&c_l\cdots c_{i+1}\uwave{c_{i}a_m\cdots a_1c_{i-1}}c_{i-2} \cdots c_1 \psim c_l\cdots c_i\uwave{c_{i-1}a_m\cdots a_1c_{i-2}}c_{i-3} \cdots c_1 \psim \cdots
\\&\psim c_l\cdots c_3\uwave{c_{2}a_m\cdots a_1c_{1}} \psim c_l\cdots c_3c_{2}c_1a_m\cdots a_1,
\end{align*}
it is enough to consider the case when $i=1$ and $l=2$. However, in this case we have
\begin{align*}
&a_m\cdots a_2\und{a_1c_2c_1} \pkm  a_m\cdots a_3\und{a_2c_2a_1}c_1 \pkm \cdots \pkm \und{a_mc_2a_{m-1}} a_{m-2} \cdots a_1 c_1
\\&\pkm c_2\und{a_{m-1}a_ma_{m-2}}a_{m-3} \cdots a_1 c_1 
\pkm \cdots \pkm  c_2a_ma_{m-1}\cdots a_4 \und{a_2a_3a_1} x \pkm  c_2a_ma_{m-1}\cdots a_3\und{a_1a_2 c_1}
\\&\pkm  c_2a_ma_{m-1}\cdots \und{a_3a_2 c_1}a_1\pkm \cdots \pkm  \und{c_2a_mc_1}a_{m-1}\cdots a_1 \pkm  c_2c_1a_m\cdots a_1
\end{align*}
as desired.
\end{proof}

\noindent\und{\csii[(b)].} Suppose that there exist $s,t$ and $0=u(0)<u(1)<\cdots< u(s)=m$ such that $c_{t+i}<a_{u(i-1)+1}<\cdots< a_{u(i)}$ for $i \in [1,s]$. Also we assume that $\{c_{t+1}, \ldots, c_{t+s}, a_1, \ldots, a_{m}\}$ is a ladder in $\cP$. Note that if $t+s<l$ then we may assume that $c_{t+s+1} \pr a_i$ for any $i \in [1,m]$ by maximality in \csii.


We argue by induction on $s\geq 1$. Assume that $s=1$, i.e. $c_{t+1}<a_1<\cdots<a_m$ and $c_{t+2} \pr a_i$ for any $i \in [1,m]$ if $t+1<l$. First if $m=1$, then we have
\begin{align*}
\uwave{a_1 c_l \cdots c_{t+2}}c_{t+1}\cdots c_1 &\psim c_l \cdots c_{t+3}\und{a_1c_{t+2}c_{t+1}}c_t \cdots c_1 
\\&\pkm c_l \cdots c_{t+2}\und{a_1c_{t+1}c_{t}}c_{t-1} \cdots c_1 
\\&\pkm c_l \cdots c_{t+2}a_1\und{c_tc_{t+1}c_{t-1}}c_{t-2} \cdots c_1 \pkm \cdots
\\&\pkm c_l \cdots c_{t+2}a_1c_{t} \cdots \und{c_2 c_{t+1} c_1}
\\&\pkm c_l \cdots c_{t+2}a_1c_{t} \cdots c_2 c_1 c_{t+1}
\end{align*}
which proves the claim. In general, we have
\begin{align*}
a_m \cdots a_1 c_l \cdots c_1 &\psim a_m \cdots a_2 c_l \cdots c_{t+2}a_1c_t\cdots  c_1 c_{t+1}
\\&\psim a_m \cdots a_3 c_l \cdots c_{t+2}a_2c_t\cdots  c_1 a_1c_{t+1} \psim \cdots
\\&\psim c_l \cdots c_{t+2}a_mc_t\cdots  c_1 a_{m-1}\cdots a_1c_{t+1}
\end{align*}
by iterating the above process $m$ times.

Now suppose that the induction step is valid up to $s-1$. Then we have
\begin{align*}
&a_{u(s)}\cdots a_{1}c_l\cdots c_1 
\\&\psim a_{u(s)}\cdots a_{u(1)+1}c_l\cdots c_1 a_{u(1)} \cdots a_1
\\&\psim c_l \cdots c_{t+s+1}a_{u(s)}a_{u(s-1)}\cdots a_{u(2)}c_{t+1} \cdots c_1 + a_{u(s)-1}\cdots a_{u(s-1)+1}c_{t+s}
\\&\qquad + a_{u(s-1)-1} \cdots a_{u(s-2)+1} c_{t+s-1}+\cdots+ a_{u(2)-1}\cdots a_{u(1)+1}c_{t+2}a_{u(1)} \cdots a_1
\\&\psim a_{u(s)-1}\cdots a_{u(s-1)+1}c_{t+s} + a_{u(s-1)-1} \cdots a_{u(s-2)+1} c_{t+s-1}+\cdots
\\&\qquad + a_{u(2)-1}\cdots a_{u(1)+1}c_{t+2}a_{u(1)} \cdots a_1 + c_l \cdots c_{t+s+1}a_{u(s)}a_{u(s-1)}\cdots a_{u(2)}c_{t+1} \cdots c_1
\\&\psim a_{u(s)-1}\cdots a_{u(s-1)+1}c_{t+s} + a_{u(s-1)-1} \cdots a_{u(s-2)+1} c_{t+s-1}+\cdots+ a_{u(2)-1}\cdots a_{u(1)+1}c_{t+2} 
\\&\qquad + c_l \cdots c_{t+s+1}a_{u(s)}a_{u(s-1)}\cdots a_{u(2)}a_{u(1)}c_t \cdots c_1 + a_{u(1)} \cdots a_1c_{t+1}
\\&\psim  c_l \cdots c_{t+s+1}a_{u(s)}a_{u(s-1)}\cdots a_{u(2)}a_{u(1)}c_t \cdots c_1+ a_{u(s)-1}\cdots a_{u(s-1)+1}c_{t+s}
\\&\qquad + a_{u(s-1)-1} \cdots a_{u(s-2)+1} c_{t+s-1}+\cdots+ a_{u(2)-1}\cdots a_{u(1)+1}c_{t+2} + a_{u(1)} \cdots a_1c_{t+1}
\end{align*}
which completes the induction step. (Here, the second $\sim_\cP$ is from induction assumption, the fourth one is from $s=1$ case, and the others are from Lemma \ref{lem:simple2}.)

\subsection{Proof of Proposition \ref{prop:column}\ref{mainprop2}}
 We argue by induction on $l=|c|$. First suppose that $l=1$. If $a_1$ is in \csi{}, then it should be in $\csi[(b)]$ by assumption and the result is trivial. If $a_1$ is in \csii{}, then the only possible case is when $a_1$ is in \csii[(b)] and either $a_1$ is not processed in the same step as $a_2$ or $\alpha=(a_1)$. Again the result is trivial in this case.

From now on suppose that $l\geq 2$ and the result holds up to $l-1$. Note that $a_1 <c_2$ since otherwise $a_1 \pr c_1$ by \cond{}. If $a_1<c_1$, then $a_1$ is in \csi[(b)] and $a_1$ bumps $c_1$. Then $a_1 \pl a_j$ for any $j \in [2,m]$ and $a_i \pl c_j$ for $j \in [2,l]$, and thus we may apply induction assumption on $(a_m,\ldots, a_2)$ and $(c_l,\ldots, c_2)$ to prove the claim. Thus it suffices to assume that $c_1<a_1<c_2$. Since $a_1 \not\pr c_1$, it implies that $a_1 \pdr c_1$, and thus $a_1$ is in \csii{}. Suppose that $a_1$ is in \csii[(b)] and $a_1, \ldots, a_k$ are processed in the same step but $a_{k+1}$ is not. As $a_1 \pl \cdots \pl a_k$, it means that $a_1, \ldots, a_k$ bump $c_1, \ldots, c_k$, respectively, and we may apply induction hypothesis to $(a_m, \ldots, a_{k+1})$ and $(c_l, \ldots, c_{k+1})$ similarly to above.

It remains to assume that $a_1$ is in \csii[(a)]. Since $a_1 \pl a_2$, this only happens when $a_1 \pdl c_2$. However, as $a_1 \pr a_2$ and $a_2 \not \pr c_2$, this forces that $a_1<c_2<a_2$ and $a_2 \pdr c_2$ by \cond{}. In other words, $\{c_1, a_1, c_2, a_2\}$ is a ladder in $\cP$. Thus by the assumption that $a_1$ is in \csii[(a)] we have $a_2 \pdl c_3$. We may iterate this argument and observe that $\{c_1, a_1, c_2, a_2, \ldots, c_l, a_l\}$ is a ladder. However, in this case $a_1$ is in \csii[(b)], which is a contradiction. Thus $a_1$ cannot be in \csii[(a)] and the claim is proved.

\subsection{Proof of Proposition \ref{prop:column}\ref{mainprop3}} 
We assume that the input is given by $(\alpha=(a_m, \ldots, a_1), c=(c_l, \ldots, c_1))$ and the output is given by $(-, \beta)$ where $\beta=(b_m, \ldots, b_1)$. Furthermore, we suppose that $a_i$ is processed in the first step of the algorithm. For example, if $a_i$ is in \csi[(a)] or \csi[(b)] then it means that $a_j=\vn$ for $j<i$.

In order to prove \ref{mainprop3}, we need to show that $b_i \neq \vn$ and $b_i \pl b_{i+1}$ under the assumption that $a_i, a_{i+1} \neq \vn$, $a_i \pl a_{i+1}$, and $b_{i+1}\neq \vn$. 
First we consider the case when $a_i$ and $a_{i+1}$ are processed in the same step. If this step is in \csii[(a)], then $a_i=b_i$ and $a_{i+1}=b_{i+1}$ thus the result is obvious. On the other hand, if this step is in \csii[(b)] then one may easily observe that $b_i<a_i<b_{i+1}<a_{i+1}$ and $\{b_i, a_i, b_{i+1},a_{i+1}\}$ is a ladder in $\cP$. In particular we have $b_i \pl b_{i+1}$ as well.

Therefore, it suffices to assume that $a_i$ and $a_{i+1}$ are processed in different steps. We let $\fd=(d_{l'}, \ldots, d_1)$ be the chain obtained after the first step (which processes $a_i$) is performed. (See Figure \ref{fig:descabdd}.) For example, if $a_i$ is in \csi[(a)] then $l'=l+1, d_{l'}=a_i$, and $d_j=c_j$ for $j \in [1,l]$.
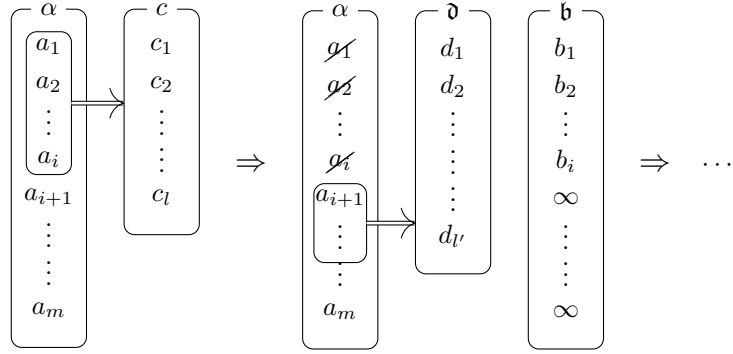
\begin{figure}[!htbp]
\begin{tikzpicture}[baseline=(current bounding box.center)]
	\node at (1.5, 0) {$c_1$};
	\node at (1.5, -.5) {$c_2$};
	\node at (1.5, -1) {$\vup$};
	\node at (1.5, -1.5) {$\vup$};
	\node at (1.5, -2) {$c_l$};
	\node at (0, 0) {$a_1$};
	\node at (0, -.5) {$a_2$};
	\node at (0, -1) {$\vup$};
	\node at (0, -1.5) {$a_i$};
	\node at (0, -2) {$a_{i+1}$};
	\node at (0, -2.5) {$\vup$};
	\node at (0, -3) {$\vup$};
	\node at (0, -3.5) {$a_{m}$};
	\draw[rounded corners] (-.5, 0.5) rectangle (.5, -4);
	\node[fill=white] at (0,0.5) {$\alpha$};
	\draw[rounded corners] (1, 0.5) rectangle (2, -2.5);
	\node[fill=white] at (1.5,0.5) {$c$};
	\draw[rounded corners] (-.3, 0.2) rectangle (.3, -1.7);
	\draw[double distance = 1pt,->] (.3,-.75) -- (1,-.75);
\end{tikzpicture} $\quad \textnormal{\Large $\Rightarrow$} \quad $
\begin{tikzpicture}[baseline=(current bounding box.center)]
	\node at (1.5, 0) {$d_1$};
	\node at (1.5, -.5) {$d_2$};
	\node at (1.5, -1) {$\vup$};
	\node at (1.5, -1.5) {$\vup$};
	\node at (1.5, -2) {$\vup$};
	\node at (1.5, -2.5) {$d_{l'}$};
	\node at (0, 0) {$\cancel{a_1}$};
	\node at (0, -.5) {$\cancel{a_2}$};
	\node at (0, -1) {$\vup$};
	\node at (0, -1.5) {$\cancel{a_i}$};
	\node at (0, -2) {$a_{i+1}$};
	\node at (0, -2.5) {$\vup$};
	\node at (0, -3) {$\vup$};
	\node at (0, -3.5) {$a_{m}$};
	\node at (3, 0) {$b_1$};
	\node at (3, -.5) {$b_2$};
	\node at (3, -1) {$\vup$};
	\node at (3, -1.5) {$b_i$};
	\node at (3, -2) {$\vn$};
	\node at (3, -2.5) {$\vup$};
	\node at (3, -3) {$\vup$};
	\node at (3, -3.5) {$\vn$};
	\draw[rounded corners] (-.5, 0.5) rectangle (.5, -4);
	\node[fill=white] at (0,0.5) {$\alpha$};
	\draw[rounded corners] (1, 0.5) rectangle (2, -3);
	\node[fill=white] at (1.5,0.5) {$\fd$};
	\draw[rounded corners] (2.5, 0.5) rectangle (3.5, -4);
	\node[fill=white] at (3,0.5) {$\fb$};
	\draw[rounded corners] (-.35, -1.8) rectangle (.35, -2.85);
	\draw[double distance = 1pt,->] (.35,-2.325) -- (1,-2.325);
\end{tikzpicture} $\quad \textnormal{\Large $\Rightarrow$} \quad \cdots$
\caption{Description of the calculation}\label{fig:descabdd}
\end{figure}

First, we note that $b_i \neq \vn$; otherwise, $a_i$ is in \csi[(a)] which means that $a_i$ becomes the largest element in the chain $\fd$. Since  $a_i \pl a_{i+1}$, it means that $a_{i+1}$ is also in \csi[(a)], which contradicts the assumption that $b_{i+1}\neq \vn$. In particular, it follows that $l=l'$, i.e. the length of $c$ and $\fd$ should be equal.

It remains to show that $b_i \pl b_{i+1}$. First assume that $a_i$ is in \csi[(b)], i.e. there exists $j \in [0,l-1]$ such that $c_j<a_i<c_{j+1}$ and $c_j\pl a_i$. (Here we put $c_0=-\vn$ as before.) In this case $b_i = c_{j+1}$, $d_{j+1}=a_i$, and $d_k=c_k$ for $k \neq j+1$. Now, if $a_{i+1}$ is in
\begin{enumerate}[label=\textbullet, leftmargin=*]
\item \csi[(a)]: this is impossible as we assumed that $b_{i+1} \neq \vn$.
\item \csi[(b)]: $a_{i+1}$ should bump $d_k=c_k$ for some $k \geq j+2$ since $d_{j+1}=a_i \pl a_{i+1}$. As a result, $b_i=c_{j+1} \pl c_k =b_{i+1}$.
\item \csii[(a)]: since $a_i \pl a_{i+1}$, this is only possible when there exists $d_k=c_k$ for some $k\geq j+2$ such that $c_k \pdl a_{i+1}$. As $c_{j+1} \pl c_k$, we have $c_{j+1} \pl a_{i+1}$ by \cond{} applied to $(c_{j+1}, c_k, a_{i+1})$. Therefore we have $b_i = c_{j+1} \pl  a_{i+1} = b_{i+1}$.
\item \csii[(b)]: it is shown in the same way as \csii[(a)].
\end{enumerate}

Now we assume that $a_i$ is in \csii[(a)], i.e. $a_i=b_i$ and $d=\fd$. Then for any $x\geq a_{i+1}$ we have $a_i=b_i \pl x$ by \cond{} applied to $(a_i, a_{i+1}, x)$, and thus the only nontrivial case occurs when $a_{i+1}$ is in \csii[(b)]. However, this is only possible when there exists $j \in [1,l]$ such that $a_i\pdl c_j \pdl a_{i+1}$, in which case $a_i$ and $a_{i+1}$ should be processed in the same step because of the maximality in \csii. This violates the aforementioned assumption.

Lastly we assume that $a_i$ is in \csii[(b)], i.e. $a_i \pdr b_i$. Again, by \cond{} the only nontrivial case occurs when $a_{i+1}$ is in \csii[(b)] and this is only possible when there exists $j \in [1,l]$ such that $a_i\pl c_j=d_{j} \pdl a_{i+1}$. In this case we have $b_{i+1}=c_j$ and thus $b_i \pl c_j=b_{i+1}$ by \cond{} applied to $(b_i, a_i, c_j)$.

We exhaust all the possibilities and thus completed the proof of \ref{mainprop3}.

\subsection{Proof of Proposition \ref{prop:column}\ref{mainprop4}}
We keep the setup in the proof of\ref{mainprop3} above.
First we consider the case when $a_i$ and $a_{i+1}$ are processed in the same step. If this step is in \csii[(a)], then $a_i=b_i$ and $a_{i+1}=b_{i+1}$ thus the result is obvious. On the other hand, if this step is in \csii[(b)] then one may easily observe that $b_i<a_i=b_{i+1}<a_{i+1}$ and $\{b_i, a_i=b_{i+1},a_{i+1}\}$ is a ladder in $\cP$. In particular we have $b_i \pdl b_{i+1}$ as well. Therefore, it suffices to assume that $a_i$ and $a_{i+1}$ are processed in different steps.

We suppose that $b_{i+1}=\vn$, i.e. $a_{i+1}$ is in \csi[(a)]. Since $a_i \not\pl a_{i+1}$, it follows that $d_{l'} \neq a_i$, i.e. $l=l'$ and $d_l=c_l$ (and also $l\neq 0$). However, it is only possible when $a_i<c_l$ and $c_l \pl a_{i+1}$, which implies $a_i \pl a_{i+1}$ by \cond{} applied to $(a_i, c_l, a_{i+1})$. This violates our assumption, and thus $b_{i+1} \neq \vn$ as expected.

It remains to show that $b_i \not\pl b_{i+1}$ if $b_i\neq \vn$, i.e. the length of $c$ is equal to $\fd$. We first assume that $a_i$ is in \csi[(b)], i.e. there exists $j\in [0,l-1]$ such that $c_j<a_i<c_{j+1}$ and $c_j \pl a_i$. (Here we put $c_0=-\vn$ as before.) In this case we have $b_i =c_{j+1}>a_i$, $d_{j+1}=a_i$, and $d_k=c_k$ for any $k \neq j+1$. As $b_i \not\pl a_{i+1}$ by \cond{} applied to $(a_i, b_i=c_{j+1}, a_{i+1})$, it follows that $b_i \not\pl x$ for any $x \leq a_{i+1}$ again by \cond{} applied to $(b_i, x, a_{i+1})$. Therefore the only nontrivial case occurs when $a_{i+1}$ is in \csi[(b)] as well. This is only possible when either [there exists $k\in[0,j-1]$ such that $c_k<a_{i+1}<c_{k+1}$ and $c_k \pl a_{i+1}$] or [$c_{j}<a_{i+1}<d_{j+1}=a_i$ and $c_j \pl a_{i+1}$]. Thus $b_{i+1}$ is equal to either $c_{k+1}$ for $k<j$ or $a_i$. In either case, we have $b_i=c_{j+1}\not\pl b_{i+1}$.

Let us assume that $a_i$ is in \csii[(a)], i.e. $c=\fd$. As above, the only nontrivial case occurs when $a_{i+1}$ is in \csi[(b)], i.e. there exists $j \in [0,l-1]$ such that $c_j<a_{i+1}<c_{j+1}$ and $c_j \pl a_{i+1}$, in which case $b_{i+1}=c_{j+1}$. Now if $a_i = b_i \pl b_{i+1}=c_{j+1}$, then as $a_i$ is in \csii[(a)] there exists $k <j+1$ such that $a_i\pdl c_k$. But this is contradiction since $c_j \pl a_{i+1}$ implies $c_k \pl a_{i+1}$, which means $a_i \pl a_{i+1}$ by \cond{} applied to $(a_i, c_k, a_{i+1})$.

Lastly, we assume that $a_i$ is in \csii[(b)] so that $a_i \pdr b_i$. As $a_i$ and $a_{i+1}$ are not processed in the same step, $a_i$ should be in $\fd$. If $a_{i+1}$ is in
\begin{enumerate}[label=\textbullet, leftmargin=*]
\item \csi[(a)]: this is impossible as we assumed that $b_{i+1}=\vn$.
\item \csi[(b)]:  since $a_i \not \pl a_{i+1}$, $a_{i+1}$ either bumps $a_i$ or some element above $a_i$ in the chain $\fd$, i.e. $b_{i+1}=a_i$ or $b_{i+1}\pl a_i$. In either case, we should have $b_i \not\pl b_{i+1}$ since otherwise $a_i \pr b_i$.
\item \csii[(a)]: suppose that $b_{i}\pl b_{i+1}=a_{i+1}$. As $b_i \pdl a_i$, we have $a_i<a_{i+1}$ by \cond{} applied to $(b_i, a_i, a_{i+1})$, and thus  $\{b_i, a_i, a_{i+1}\}$ is a ladder in $\cP$. However, this violates the maximality of \csii{} as $a_i$ and $a_{i+1}$ are not processed in the same step. Thus we should have $b_{i}\not\pl b_{i+1}=a_{i+1}$
\item \csii[(b)]: %
$b_{i+1}$ is some element in the chain $d'$ satisfying $a_{i+1} \pdr b_{i+1}$. If $b_i \pl b_{i+1}$, then we should have $b_i<a_i<b_{i+1}<a_{i+1}$ by \cond{} applied to $(b_i, a_i, b_{i+1})$ and also $b_i \pl a_{i+1}$  by \cond{} applied to $(b_i, b_{i+1}, a_{i+1})$.
%
However, in such a case $\{b_i, a_i, a_{i+1}\}$ is a ladder in $\cP$, which violates the maximality in \csii. Thus we have $b_i \not\pl b_{i+1}$
\end{enumerate}

We exhaust all the possibilities and thus completed the proof of \ref{mainprop4}.

\subsection{Proof of Proposition \ref{prop:phiinj}}  \label{sec:proofinj}

Hereafter we write $\hh{\vn}=\vn$ and $\hh{a} \colonequals n+1-a$ for $a \in [1,n]$. ($n$ is assumed to be fixed.) We define a new order $\ccP$ on $[1,n]$ such that $a\cpr b$ if and only if $\hh{a} \pl \hh{b}$. Then one can easily check that $\cP=\cP_{\lambda,m}$ if and only if $\ccP=\cP_{\lambda', m}$. Also, from now on we write $\fC^\cP, \fC^\ccP, \Phi^\cP, \Psi_X^\ccP$, etc. to clarify which partial order is used in their definitions. We define 
\begin{align*}
\hh{\bullet}:&\ \fA \rightarrow \fA: \alpha=(a_m, \ldots, a_1) \mapsto \hh{\alpha}=(\hh{a_1}, \ldots, \hh{a_m})
\\\hh{\bullet}:&\ \fC^\cP \rightarrow \fC^\ccP: c=(c_l, \ldots, c_1) \mapsto \hh{c}=(\hh{c_1}, \ldots, \hh{c_l})
\end{align*}
and $\hh{\bullet}: \fC^\ccP \rightarrow \fC^\cP$ similarly. Also we set 
\begin{align*}
\omega:\fC^\cP\fA \rightarrow \fA\fC^\ccP: (c, \alpha) \mapsto (\hh{\alpha}, \hh{c}),
\end{align*}
and $\omega:\fC^\ccP\fA \rightarrow \fA\fC^\cP$ similarly. (Here we abuse notations and denote both functions by $\omega$.) Note that these functions are well-defined bijections.

Suppose that $\Phi^\cP(\alpha,c) = (d, \beta)$ for some $(\alpha,c) \in \fA\fC^\cP$ and $(d,\beta) \in \fC^\cP\fA$ where $\alpha=(a_m, \ldots, a_1)$ and $\beta=(b_m, \ldots, b_1)$. We set $X=\{m+1-i \mid i\in [1,m], a_i\neq \vn, b_i=\vn\}$. (This set records in which step \csi[(a)] occured when calculating $\Phi^\cP(\alpha, c)$.) 
We claim the following.
\begin{lem}\label{lem:invrel} Keep the assumptions above. Then we have $(\alpha,c) = (\omega\circ\Psi^\ccP_X \circ \omega \circ\Phi^\cP)(\alpha, c)=(\omega\circ\Psi^\ccP_X \circ \omega)(d, \beta)$, i.e. $(\hh{c}, \hh{\alpha}) = \Psi^\ccP_X(\hh{\beta}, \hh{d})$.
\end{lem}
\begin{rmk} Since the set $X$ depends on $(\alpha, c)$, the lemma above does \emph{not} show that $\omega\circ\Psi^\ccP_X \circ \omega \circ\Phi^\cP$ is the identity. However, it shows that one can recover $(\alpha, c)$ from $(d,\beta)$ and $X$, which essentially proves Proposition \ref{prop:phiinj}.
\end{rmk}
\begin{proof} We may assume that $m \geq 1$ and $a_m\neq \vn$. Suppose that it only needs one step. If it is in
\begin{enumerate}[label=\textbullet]
\item \csi[(a)]: in this case it is clear that $\alpha=(a_1)$, $\beta=(\vn)$, $d=(a_1)+c$, and $X=\{1\}$. Thus $\Psi_X^\ccP(\hh{\beta} ,\hh{d}) = \Psi_{\{1\}}^\ccP((\vn) ,\hh{c}+(\hh{a_1}))=(\hh{c}, (\hh{a_1}))=(\hh{c}, \hh{\alpha})$ as desired. (This corresponds to \csiii[(b)] of $\alg_\Psi$.)
\item \csi[(b)]: we have $\alpha=(a_1)$ and $X=\emptyset$. There exists $r \in [0,l-1]$ such that $c_r < a_1<c_{r+1}$ and $c_r \pl a_1$. (If $a_1 <c_1$ then we set $r=0$.) Then we have $\beta=(c_{r+1})$ and $d=(c_l, \ldots, c_{r+2}, a_1, c_{r}, \ldots, c_1)$. Since $\hh{c_{r+2}}<\hh{c_{r+1}}<\hh{a_1}$ and $\hh{c_{r+2}}\cpl\hh{c_{r+1}}$, we have
$\Psi_X^\ccP(\hh{\beta} ,\hh{d})
= \Phi^\ccP((\hh{c_{r+1}}) ,(\hh{c_1}, \ldots, \hh{c_{r}}, \hh{a_1}, \hh{c_{r+2}}, \ldots, \hh{c_l}))
=((\hh{c_1}, \ldots, \hh{c_{r}}, \hh{c_{r+1}}, \hh{c_{r+2}}, \ldots, \hh{c_l}), (\hh{a_1}))
=(\hh{c}, \hh{\alpha})$ as desired. (This corresponds to \csi[(b)] of $\alg_\Psi$, and it is still valid when $r+2=l+1$.)
\item \csii[(a)]: we have $X=\emptyset$, $\alpha=\beta$, $c=d$, and $a_1<a_2<\cdots<a_m$. There exists $r,h$ such that $\{c_r, \ldots, c_{r+h}\}\cup \und{\alpha}$ is a ladder in $\cP$ and $c_r<a_i<c_{r+h}$ for any $i$. Then it is easy to see that  $\hh{a_m}<\cdots<\hh{a_2}<\hh{a_1}$, $\{\hh{c_{r}}, \ldots, \hh{c_{r+h}}\} \cup \und{\hh{\alpha}}$ is a ladder in $\ccP,$ and $\hh{c_{r+h}}<\hh{a_i}<\hh{c_{r}}$ for any $i$. Thus
$\Psi_X^\ccP(\hh{\beta} ,\hh{d})
= \Phi^\ccP(\hh{\alpha} ,\hh{c})
=(\hh{c}, \hh{\alpha})
$ as desired. (This corresponds to \csii[(a)] of $\alg_\Psi$.)
\item \csii[(b)]: we have $X=\emptyset$. As in the description of $\alg_\Phi$, we choose $r, h$ and $0=u(r-1)< u(r)< u(r+1)< \cdots< u(r+h)=m$ such that $c_{i}<a_{u(i-1)+1}<\cdots<a_{u(i)}$ for $i\in [r, r+h]$ and $\{c_{r}, \ldots, c_{r+h}, a_1, \ldots, a_m\}$ is a ladder in $\cP$. 
Then it follows that $d=(c_l, \ldots, c_{r+h+1}, a_{u(r+h)}, \ldots, a_{u(r)}, c_{r-1}, \ldots, c_1)$ and 
$$b_j =\left\{\begin{aligned} &c_{i} &\textnormal{ if } j=u(i-1)+1 \textnormal{ for some } i \in [r,r+h],
\\ &a_{j-1} &\textnormal{ otherwise}.
\end{aligned}\right.$$ 
However, it implies that $\hh{b_1}>\hh{b_2}>\cdots>\hh{b_m}$, $\{\hh{a_{u(r)}}, \ldots, \hh{a_{u(r+h)}}, \hh{b_1}, \ldots, \hh{b_m}\}$ is a ladder in $\ccP$, and $\hh{b_{u(i-1)+1}}=\hh{c_{i}}>\hh{a_{u(i-1)+1}}>\cdots>\hh{a_{u(i)}}$ for $i\in [r, r+h]$. Also note that $\hh{c_{r-1}} \cpr x$ for any $x \in \{\hh{a_{u(r)}}, \ldots, \hh{a_{u(r+h)}}, \hh{b_1}, \ldots, \hh{b_m}\}$ when $r>1$. Thus
\begin{align*}
\Psi_X^\ccP(\hh{\beta} ,\hh{d})&= \Phi^\ccP((\hh{b_1}, \ldots, \hh{b_m}) ,(\hh{c_1}, \ldots, \hh{c_{r-1}}, \hh{a_{u(r)}},\ldots, \hh{a_{u(r+h)}}, \hh{c_{r+h+1}} \ldots, \hh{c_l}))
\\&=((\hh{c_1}, \ldots, \hh{c_{r-1}}, \hh{c_{r}},\ldots, \hh{c_{r+h}}, \hh{c_{r+h+1}}, \ldots, \hh{c_l}), (\hh{a_1}, \ldots, \hh{a_m}))=(\hh{c}, \hh{\alpha})
\end{align*}
as desired. (This corresponds to \csii[(b)] of $\alg_\Psi$.)
\end{enumerate}

To prove the general case we argue by induction on the number of steps of $\alg_{\Phi}$ for the calculation of $\Phi^\cP(\alpha,c)$. Let $m\pp\in[0,m-1]$ be the smallest element such that $a_{m\pp+1}, \ldots, a_m$ are processed in the same step. Set $\alpha\pp=(a_{m\pp}, \ldots, a_1)$, $\alpha\dg=(a_m, \ldots, a_{m\pp+1})$, $\beta\pp=(b_{m\pp}, \ldots, b_1)$ and $\beta\dg=(b_m, \ldots, b_{m\pp+1})$, so that $\alpha=\alpha\dg+\alpha\pp$ and $\beta=\beta\dg+\beta\pp$. Then there exists $d\pp\in \fC^\cP$ such that $\Phi^\cP(\alpha\pp, c)=(d\pp, \beta\pp)$ and $\Phi^\cP(\alpha\dg, d\pp) = (d, \beta\dg)$ because of the choice of $m\pp$.  Thus by induction assumption we have $(\hh{c}, \hh{\alpha\pp}) = \Psi^\ccP_{X\pp}(\hh{\beta\pp}, \hh{d\pp})$ and $(\hh{d\pp}, \hh{\alpha\dg})=\Psi^\ccP_{X\dg}(\hh{\beta\dg}, \hh{d})$ where $X\pp =\{i+m\pp-m\mid i \in X\cap [m-m\pp+1, m]\}$ and $X\dg=X\cap[1,m-m\pp]$.

Therefore, in order to prove the lemma, it suffices to show that the first step of the calculation of $\Psi^\ccP_X(\hh{\beta}, \hh{d})$ processes $\hh{b_{m\pp+1}}, \ldots, \hh{b_m}$ but not $\hh{b_{m\pp}}$. (See Figure \ref{fig:descinv}.) If $a_m$ is in \csi{} of $\alg_\Phi$ then $\hh{b_m}$ is in \csi{} or \csiii{} of $\alg_\Psi$ as shown above, in which case the statement is trivial. Thus we assume that $a_m$ is in \csii{} of $\alg_\Phi$. Then we have $|d|=|d\pp|$ and $X\dg=\emptyset$. Thus in particular $\Psi^\ccP_{X\dg}(\hh{\beta\dg}, \hh{d})=\Phi^\ccP(\hh{\beta\dg}, \hh{d})$. We set $d=(d_q, \ldots, d_1)$ and $d\pp=(d\pp_{q}, \ldots, d\pp_1)$.

\begin{figure}[!htbp]
\begin{tikzpicture}[baseline=(current bounding box.center)]
	\tikzmath{\x=1.2;\y=0.5;\m=3.8;\l=2.5;\sts=4.6;\tst=10.4;}
	\node at (\x, 0) {$c_1$};
	\node at (\x, -\y) {$c_2$};
	\node at (\x, -2*\y) {$\vup$};
	\node at (\x, -3*\y) {$\vup$};
	\node at (\x, -4*\y) {$c_l$};
	\node at (0, 0) {$a_1$};
	\node at (0, -\y) {$a_2$};
	\node at (0, -2*\y) {$\vup$};
	\node at (0, -3*\y) {$a_{m\pp}$};
	\node at (0, -4*\y) {$a_{m\pp+1}$};
	\node at (0, -5*\y) {$\vup$};
	\node at (0, -6*\y) {$\vup$};
	\node at (0, -7*\y) {$a_{m}$};
	\draw[rounded corners] (-.5*\x, \y) rectangle (.5*\x, -\m);
	\node[fill=white] at (0,0.5) {$\alpha$};
	\draw[rounded corners] (.5*\x, \y) rectangle (1.5*\x, -\l);
	\node[fill=white] at (\x,0.5) {$c$};
	\draw[rounded corners] (-.35, 0.2) rectangle (.35, -1.8);
	\draw[-{Latex[width=3mm]}] (1.7*\x,-.4*\m) -- (3.3*\x,-.4*\m) node [midway, above] {$\Phi^\cP(\alpha\pp, c)$};
	\draw[{Latex[width=3mm]}-] (1.7*\x,-.55*\m) -- (3.3*\x,-.55*\m) node [midway, below] {$\Psi^\ccP_{X\pp}(\hh{\beta\pp}, \hh{d\pp})$};
	
	\node at (\sts+\x, 0) {$d_1\pp$};
	\node at (\sts+\x, -\y) {$d_2\pp$};
	\node at (\sts+\x, -2*\y) {$\vup$};
	\node at (\sts+\x, -3*\y) {$\vup$};
	\node at (\sts+\x, -4*\y) {$d_{q\pp}\pp$};
	\node at (\sts+0, 0) {$\cancel{a_1}$};
	\node at (\sts+0, -\y) {$\cancel{a_2}$};
	\node at (\sts+0, -2*\y) {$\vup$};
	\node at (\sts+0, -3*\y) {$\cancel{a_{m\pp}}$};
	\node at (\sts+0, -4*\y) {$a_{m\pp+1}$};
	\node at (\sts+0, -5*\y) {$\vup$};
	\node at (\sts+0, -6*\y) {$\vup$};
	\node at (\sts+0, -7*\y) {$a_{m}$};
	\node at (\sts+2*\x, 0) {$b_1$};
	\node at (\sts+2*\x, -\y) {$b_2$};
	\node at (\sts+2*\x, -2*\y) {$\vup$};
	\node at (\sts+2*\x, -3*\y) {$b_{m\pp}$};
	\draw[rounded corners] (\sts+-.5*\x, \y) rectangle (\sts+.5*\x, -\m);
	\node[fill=white] at (\sts+0,0.5) {$\alpha$};
	\draw[rounded corners] (\sts+.5*\x, \y) rectangle (\sts+1.5*\x, -\l);
	\node[fill=white] at (\sts+\x,0.5) {$d\pp$};
	\draw[rounded corners] (\sts+1.5*\x, \y) rectangle (\sts+2.5*\x, -1.8);

	\draw[rounded corners] (\sts-.5, -1.8) rectangle (\sts+.5, -3.7);
	\node[fill=white] at (\sts+2*\x,\y) {$\beta\pp$};
	\draw[-{Latex[width=3mm]}] (\sts+2.7*\x,-.4*\m) -- (\sts+4.3*\x,-.4*\m) node [midway, above] {$\Phi^\cP(\alpha\dg, d\pp)$};
	\draw[{Latex[width=3mm]}-] (\sts+2.7*\x,-.55*\m) -- (\sts+4.3*\x,-.55*\m) node [midway, below] {$\Psi^\ccP_{X\dg}(\hh{\beta\dg}, \hh{d})$};

	\node at (\tst+\x, 0) {$d_1$};
	\node at (\tst+\x, -\y) {$d_2$};
	\node at (\tst+\x, -2*\y) {$\vup$};
	\node at (\tst+\x, -3*\y) {$\vup$};
	\node at (\tst+\x, -4*\y) {$d_{q}$};
	\node at (\tst+0, 0) {$\cancel{a_1}$};
	\node at (\tst+0, -\y) {$\cancel{a_2}$};
	\node at (\tst+0, -2*\y) {$\vup$};
	\node at (\tst+0, -3*\y) {$\cancel{a_{m\pp}}$};
	\node at (\tst+0, -4*\y) {$\cancel{a_{m\pp+1}}$};
	\node at (\tst+0, -5*\y) {$\vup$};
	\node at (\tst+0, -6*\y) {$\vup$};
	\node at (\tst+0, -7*\y) {$\cancel{a_{m}}$};
	\node at (\tst+2*\x, 0) {$b_1$};
	\node at (\tst+2*\x, -\y) {$b_2$};
	\node at (\tst+2*\x, -2*\y) {$\vup$};
	\node at (\tst+2*\x, -3*\y) {$b_{m\pp}$};
	\node at (\tst+2*\x, -4*\y) {$b_{m\pp+1}$};
	\node at (\tst+2*\x, -5*\y) {$\vup$};
	\node at (\tst+2*\x, -6*\y) {$\vup$};
	\node at (\tst+2*\x, -7*\y) {$b_m$};
	\draw[rounded corners] (\tst+-.5*\x, \y) rectangle (\tst+.5*\x, -\m);
	\node[fill=white] at (\tst+0,0.5) {$\alpha$};
	\draw[rounded corners] (\tst+.5*\x, \y) rectangle (\tst+1.5*\x, -\l);
	\node[fill=white] at (\tst+\x,0.5) {$d$};
	\draw[rounded corners] (\tst+1.5*\x, \y) rectangle (\tst+2.5*\x, -\m);
	\node[fill=white] at (\tst+2*\x,0.5) {$\beta$};
	\draw[rounded corners] (\tst+2*\x-.5, -1.8) rectangle (\tst+2*\x+.5, -3.7);
	
	\draw[-{Latex[width=3mm]}] (.7*\x,-\m) -- (.7*\x,-1.2*\m) -- node [midway, above] {$\Phi^\cP(\alpha, c)$} (\tst+.8*\x,-1.2*\m)  --  (\tst+.8*\x,-\m);	
	\draw[{Latex[width=3mm]}-] (.3*\x,-\m) -- (.3*\x,-1.3*\m) -- node [midway, below] {$\Psi^\ccP_{X}(\hh{\beta}, \hh{d})$?} (\tst+1.2*\x,-1.3*\m)  --  (\tst+1.2*\x,-\m);	

\end{tikzpicture}
\caption{Strategy of the proof}\label{fig:descinv}
\end{figure}
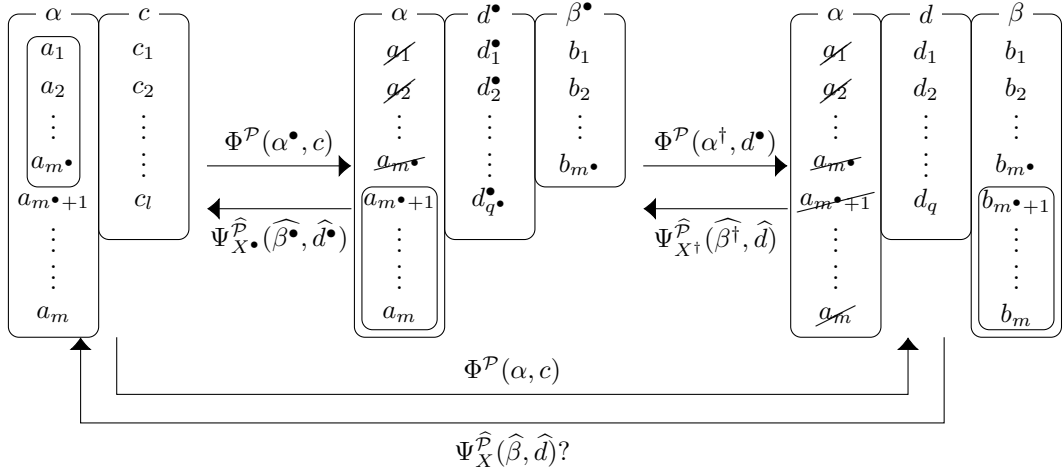

\noindent{\und{\csii[(a)]}}. Suppose that $a_m$ is in \csii[(a)] of $\alg_\Phi$. Then $d=d\pp$, $\beta\dg=\alpha\dg$ and there exist $r, h$ such that $\cL\colonequals \{d_{r}\pp, \ldots, d_{r+h}\pp, a_{m\pp+1}, \ldots, a_{m}\}$ is a ladder in $\cP$ and $d\pp_r\leq x\leq d\pp_{r+h}$ for any $x\in\cL$.
Here we need to check that $\{\hh{b_{m\pp}}\} \cup \{\hh{x} \mid x \in \cL\}$ is not a ladder in $\hh{\cP}$, i.e. either $m\pp=0$ or $b_{m\pp}$ does not satisfy both $b_{m\pp} \pl b_{m\pp+1}$ and $b_{m\pp} \pdl d_{r}\pp$. For the sake of contradiction we suppose otherwise. By Proposition \ref{prop:column}\ref{mainprop3}, it also means that $a_{m\pp} \pl a_{m\pp+1}$. Also $a_{m\pp}$ cannot be in \csi[(a)] of $\alg_\Phi$.

Suppose that $a_{m\pp}$ is in \csi[(b)] of $\alg_\Phi$. Then $b_m\pp$ is bumped out from the chain $\fd$, and thus $d_r\pp \not\pdr b_{m\pp}$ unless $a_{m\pp}$ bumps the $r$-th element of $\fd$ and $a_{m\pp}=d_{r}\pp$. But this is absurd since it means that $a_{m\pp}<b_{m\pp}$ but $b_{m\pp} \pdl d_{r}\pp=a_{m\pp}$.

Now we assume that $a_{m\pp}$ is in \csii[(a)] of $\alg_\Phi$ and $\cL'\colonequals\{d_{r'}\pp, \ldots, d_{r'+h'}\pp, a_{m'}, \ldots, a_{m\pp}\}$ is the corresponding ladder in $\cP$ where $d_{r'}\pp<\cdots<d_{r'+j}\pp$ and $d_{r'}\pp<a_{m'} < \cdots < a_{m\pp}<d_{r'+h'}\pp$. Then $b_{m\pp}=a_{m\pp}$, and thus $d_r\pp \pdr b_{m\pp}=a_{m\pp}$ which is only possible when $r'+h'=r$, i.e. $\cL\cup\cL'$ is again a ladder in $\cP$. (Note that $a_{m\pp}=b_m\pp \pl a_{m\pp+1}$.) However, it contradicts the minimality of $m\pp$.

Assume that $a_{m\pp}$ is in \csii[(b)] of $\alg_\Phi$. Then $a_{m\pp}=d_i\pp$ for some $i$, and as $a_{m\pp} \pl a_{m\pp+1}$ and $a_{m\pp+1} \pdr d_r\pp$ we should have $i<r$ and $a_{m\pp}=d_i\pp\pl d_r\pp$. On the other hand, we always have $a_{m\pp} \pdr b_{m\pp}$, and thus $d_r\pp \pr b_{m\pp}$ by \cond{}. However, it contradicts the assumption that $d_r\pp \pdr b_{m\pp}$.

\noindent{\und{\csii[(b)]}}. Suppose that $a_m$ is in \csii[(b)] of $\alg_\Phi$ and let 
$\cL\colonequals \{d_{r}\pp, \ldots, d_{r+h}\pp, a_{m\pp+1}, \ldots, a_{m}\}$ be the corresponding ladder so that $d_r\pp\leq x \leq a_m$ for any $x\in \cL$. Then direct calculation as above shows that $\cL=\{d_r,\ldots, d_{r+h},   b_{m\pp+1}, \ldots, b_m\}$. Note that $\cL \cup\{d_{r-1}\pp\}$ cannot be a ladder, and thus in order for $\hh{b_{m\pp}}$ to be processed with $\hh{b_{m\pp+1}}, \ldots, \hh{b_m}$ we should have that $\{\hh{x}\mid x\in \cL\}\cup\{\hh{b_{m\pp}}\}$ is a ladder in $\hh{\cP}$, i.e. $b_{m\pp+1} \pdr b_{m\pp}$ and $x \pr b_{m\pp}$ for any $x\in \cL-\{b_{m\pp+1}\}$ by maximality in \csii. By Proposition \ref{prop:column}\ref{mainprop4}, it also means that $a_{m\pp} \pdl a_{m\pp+1}$.

Assume that $a_{m\pp}$ is in \csi[(b)] of $\alg_\Phi$. Then $a_{m\pp}=d_i\pp$ for some $i$. Since $a_{m\pp} \pdl a_{m\pp+1}$ and $d_{r}\pp \pdl a_{m\pp+1}$, we should have $r=i$ and $a_{m\pp}=d_{r}\pp$. But it contradicts that $b_{m\pp}>a_{m\pp}$ as $b_{m\pp} \pdl b_{m\pp+1}=d_r\pp$.

Assume that $a_{m\pp}$ is in \csii[(a)] of $\alg_\Phi$. Then $b_{m\pp}=a_{m\pp} \pdl a_{m\pp+1}\neq b_{m\pp+1}$ which contradicts that $x \pr b_{m\pp}$ for any $x\in \cL-\{b_{m\pp+1}\}$.

Assume that $a_{m\pp}$ is in \csii[(b)] of $\alg_\Phi$ and $\cL'$ is the corresponding ladder where $a_{m\pp}>b_{m\pp}>x$ for any $x \in \cL'-\{a_{m\pp}, b_{m\pp}\}$. Then $a_{m\pp}=d_i\pp$ for some $i$. Since $a_{m\pp} \pdl a_{m\pp+1}$ and $d_{r} \pdl a_{m\pp+1}$, it follows that $r=i$. Together with the fact that $x \pr b_{m\pp}$ for any $x\in \cL-\{b_{m\pp+1}\}$, it follows that $\cL \cup \cL'$ is again a ladder. However, it contradicts the minimality of $m\pp$.

We exhaust all the cases and finish the proof.
\end{proof}

We are ready to prove Proposition \ref{prop:phiinj}.
\begin{proof}[Proof of Proposition \ref{prop:phiinj}] Recall that we have $(\alpha=(a_m, \ldots, a_1), c), (\alpha'=(a'_m, \ldots, a'_1), c')\in \fA\fC$ such that $a_i =\vn \Leftrightarrow a_i'=\vn$ and $\Phi(\alpha,c) = \Phi(\alpha',c')$. Let $(d, \beta) \in \fC\fA$ be such a result. Let $X=\{m+1-i\mid i \in [1,m], a_i\neq \vn, b_i=\vn\}=\{m+1-i \mid i \in [1,m], a_i'\neq \vn, b_i=\vn\}$. Then by Lemma \ref{lem:invrel} we have $\omega(\Psi^\ccP_X(\hh{\beta}, \hh{d}))=(\alpha,c) = (\alpha',c')$ as desired.
\end{proof}

\section{Proof of Theorem \ref{thm:main} and \ref{thm:mainstrong}} \label{sec:proofthm}
\subsection{Preliminary lemmas}
Before the proof we first state the following series of lemmas which will be frequently used later on.
\begin{lem}\label{lem:extlad}
Suppose that $\{a_1, \ldots, a_k\}$ is a ladder in $\cP$ where $k \geq 2$ and $a_1< \cdots < a_k$.
\begin{enumerate}
\item If $b\pdr a_k$, then either $\{a_1, \ldots, a_{k-1},a_k, b\}$ or $\{a_1, \ldots, a_{k-2}, a_{k-1}, b\}$ is a ladder in $\cP$.
\item If $b\pdl a_{1}$, then either $\{b, a_1, a_2, \ldots, a_{k}\}$ or $\{b, a_2, a_3, \ldots, a_{k}\}$ is a ladder in $\cP$.
\end{enumerate}
\end{lem}
\begin{proof} For (1), first note that  we have $b \pr a_{k-2}$ which follows from \cond{} as $a_{k-2} \pl a_k$. Now if $b \pr a_{k-1}$, then  $\{a_1, \ldots, a_{k-1},a_k, b\}$ is a ladder in $\cP$. Otherwise, we have $b \pdr a_{k-1}$ and thus $\{a_1, \ldots, a_{k-2}, a_{k-1}, b\}$ is a ladder in $\cP$. (2) is proved similarly.
\end{proof}

\begin{lem} \label{lem:ladend}
Suppose that $\{a_1, \ldots, a_k\}$ is a ladder in $\cP$ where $a_1< \cdots < a_k$. If $x \not\in \{a_1, \ldots, a_k\}$, $a_1\pl x$, and no one is climbing a ladder in $\cP$, then $a_{i} \pl x$ and $a_{i+1} < x$ for $i \leq k-3$.
\end{lem}
\begin{proof} Since no one is climbing a ladder in $\cP$, we have $a_{k}\not\pr x$. Since $a_{k-2} \pl a_k$, this means that $a_{k-2} < x$ by \cond{}. Now suppose that there exists $i \leq k-3$ such that $a_i \not \pl x$ (and thus $k\geq 5$). By \cond{}, it is equivalent to assuming that $a_{k-3} \pdl x$. Since $a_{k-3} \pl a_{k-1}$, this means that $x<a_{k-1}$ by \cond{}. Thus we have $x \pdl a_{k-1}, a_k$ and $x \pdr a_{k-2}, a_{k-3}$. Since $a_{k-4} \pl a_{k-2}$, we also have $x \pr a_{k-4}$ by \cond{}. Now one may check that $\cP$ restricted to $\{a_{k-4}, a_{k-3}, a_{k-2}, x, a_k\}$ is isomorphic to $\cP_{(3,1,1),5}$, which is a contradiction. It remains to check that $a_i<x$ for $i \leq k-2$, and thus suppose otherwise (and thus $k\geq 4$). It implies that $a_{i-1} \pl a_i$ by \cond{} since $a_{i-1}\pl x$ by above (or by assumption if $i=2$), which is a contradiction.
\end{proof}

\begin{lem} \label{lem:ladtail4} Suppose that $\{b_1,\ldots, b_k, a_1, a_2, a_3\}$ is a ladder in $\cP$ where $b_1<\cdots<b_k<a_1<a_2<a_3$, $x \not\in\{b_1,\ldots, b_k, a_1, a_2, a_3\}$, $b_1 \pl x$, and no one is climbing a ladder in $\cP$. Then the relation between $\{a_1, a_2\}$ and $x$ falls into one of the following. (Also see Figure \ref{fig:ladtail4}.)
\begin{enumerate}[label=\textnormal{(\ref{lem:ladtail4}.\arabic*)}]
\item $a_1\pdl x$ and $a_2\pdr x$
\item $a_1\pdl x$ and $a_2\pdl x$
\item $a_1\pl x$ and $a_2\pdl x$
\item $a_1\pl x$ and $a_2\pl x$
\end{enumerate}
\end{lem}
\begin{proof} By Lemma \ref{lem:ladend} we may assume that $k=1$. Then one may check case-by-case using \cond{} and $a_3 \not \pr x$ since otherwise $x$ is climbing $\{b_1,\ldots, b_k, a_1, a_2, a_3\}$.
\end{proof}

\begin{figure}[!htbp]
\subcaptionbox*{(\ref{lem:ladtail4}.1)}{
\begin{tikzpicture}[scale=2]
	\node (a2) at (1,-1) {$a_1$};
	\node (a3) at (1,-2) {$a_2$};
	\node (x) at (2,-1.5) {$x$};
	
	\draw[<-, dashed] (a2) to (a3);
	
	\draw[<-, dashed] (a2) to (x);
	\draw[<-, dashed] (x) to (a3);
\end{tikzpicture}
}\subcaptionbox*{(\ref{lem:ladtail4}.2)}{
\begin{tikzpicture}[scale=2]
	\node (a2) at (1,-1) {$a_1$};
	\node (a3) at (1,-2) {$a_2$};
	\node (x) at (2,-2.5) {$x$};
	
	\draw[<-, dashed] (a2) to (a3);
	
	\draw[<-, dashed] (a2) to (x);
	\draw[<-, dashed] (a3) to (x);
\end{tikzpicture}
}\subcaptionbox*{(\ref{lem:ladtail4}.3)}{
\begin{tikzpicture}[scale=2]
	\node (a2) at (1,-1) {$a_1$};
	\node (a3) at (1,-2) {$a_2$};
	\node (x) at (2,-2.5) {$x$};
	
	\draw[<-, dashed] (a2) to (a3);
	
	\draw[<-] (a2) to (x);
	\draw[<-, dashed] (a3) to (x);
\end{tikzpicture}}
\subcaptionbox*{(\ref{lem:ladtail4}.4)}{
\begin{tikzpicture}[scale=2]
	\node (a2) at (1,-1) {$a_1$};
	\node (a3) at (1,-2) {$a_2$};
	\node (x) at (2,-2.5) {$x$};
	
	\draw[<-, dashed] (a2) to (a3);
	
	\draw[<-] (a2) to (x);
	\draw[<-] (a3) to (x);
\end{tikzpicture}
}
\caption{Possible cases in Lemma \ref{lem:ladtail4}}\label{fig:ladtail4}
\end{figure}

\begin{lem} \label{lem:ladmid}
Suppose that $\{a_1, \ldots, a_k\}$ is a ladder in $\cP$ where $a_1< \cdots < a_k$ and $k\geq 4$, $x \not\in \{a_1, \ldots, a_k\}$, and no one is climbing a ladder in $\cP$.
If $a_2\pdl x$, then it falls into one of the following cases. (Also see Figure \ref{fig:ladmid})
\begin{enumerate}[label=\textnormal{(\ref{lem:ladmid}.\arabic*)}]
\item $a_1 \pdl x$, $a_3\pdr x$
\item $k=4$, $a_{1} \pl x$, $a_3 \pdr x$, $a_4\pdr x$
\item $k=4$, $a_{1} \pl x$, $a_3 \pdl x$, $a_4 \pdr x$
\end{enumerate}
\end{lem}
\begin{proof} Using Lemma \ref{lem:ladend}, \cond{}, and the assumption that $x$ is not climbing $\{a_1, \ldots, a_k\}$ in $\cP$, one may deduce that the only possibilities are the above three cases and possibly [$k=5$, $a_{1} \pl x$, $a_3 \pdl x$, $a_4 \pdr x$, $a_5 \pdr x$] (see the last diagram in Figure \ref{fig:ladmid}). However, the latter is impossible since in this case $a_3$ is climbing the ladder $\{a_1, a_2, x, a_5\}$.
\end{proof}

\begin{figure}[!htbp]
\subcaptionbox*{(\ref{lem:ladmid}.1)}{
\begin{tikzpicture}[scale=1.3]
	\node (a1) at (0,0) {$a_1$};
	\node (a2) at (1,-1) {$a_2$};
	\node (a3) at (0,-2) {$a_3$};
	\node (a4) at (1,-3) {$\vup$};
	\node (x) at (2,-1.5) {$x$};
	
	\draw[<-] (a1) to (a3);
	\draw[<-] (a2) to (a4);
	\draw[<-] (a1) to (a4);
	\draw[<-,dashed,out=0,in=90] (a1) to (x);
	\draw[<-, dashed] (a1) to (a2);
	\draw[<-, dashed] (a2) to (a3);
	\draw[<-, dashed] (a3) to (a4);
	\draw[<-, dashed] (a2) to (x);
	\draw[<-, dashed] (x) to (a3);
\end{tikzpicture}
}\subcaptionbox*{(\ref{lem:ladmid}.2)}{
\begin{tikzpicture}[scale=1.3]
	\node (a1) at (0,0) {$a_1$};
	\node (a2) at (1,-1) {$a_2$};
	\node (a3) at (0,-2) {$a_3$};
	\node (a4) at (1,-3) {$a_4$};
	\node (x) at (2,-1.5) {$x$};
	
	\draw[<-] (a1) to (a3);
	\draw[<-] (a2) to (a4);
	\draw[<-] (a1) to (a4);
	\draw[<-,out=0,in=90] (a1) to (x);
	\draw[->, dashed, out=0,in=270] (a4) to (x);
	\draw[<-, dashed] (a1) to (a2);
	\draw[<-, dashed] (a2) to (a3);
	\draw[<-, dashed] (a3) to (a4);
	
	\draw[<-, dashed] (a2) to (x);
	\draw[<-, dashed] (x) to (a3);
\end{tikzpicture}
}\subcaptionbox*{(\ref{lem:ladmid}.3)}{
\begin{tikzpicture}[scale=1.3]
	\node (a1) at (0,0) {$a_1$};
	\node (a2) at (1,-1) {$a_2$};
	\node (a3) at (0,-2) {$a_3$};
	\node (a4) at (1,-3) {$a_4$};
	\node (x) at (2,-2.5) {$x$};
	
	\draw[<-] (a1) to (a3);
	\draw[<-] (a2) to (a4);
	\draw[<-] (a1) to (a4);
	\draw[<-,out=0,in=90] (a1) to (x);
	\draw[->, dashed] (a4) to (x);
	\draw[<-, dashed] (a1) to (a2);
	\draw[<-, dashed] (a2) to (a3);
	\draw[<-, dashed] (a3) to (a4);
	
	\draw[<-, dashed] (a2) to (x);
	\draw[<-, dashed] (a3) to (x);
\end{tikzpicture}
}\subcaptionbox*{not possible}{
\begin{tikzpicture}[xscale=1.3]
	\node (a1) at (0,0) {$a_1$};
	\node (a2) at (1,-1) {$a_2$};
	\node (a3) at (0,-2) {$a_3$};
	\node (a4) at (1,-3) {$a_4$};
	\node (a5) at (0,-4) {$a_5$};
	\node (x) at (2,-2.5) {$x$};
	
	\draw[<-] (a1) to (a3);
	\draw[<-] (a2) to (a4);

	\draw[<-] (a1) to (a4);

	\draw[<-,out=0,in=90] (a1) to (x);
	\draw[<-,dashed,out=270,in=0] (x) to (a5);
	\draw[->, dashed] (a4) to (x);
	\draw[<-, dashed] (a1) to (a2);
	\draw[<-, dashed] (a2) to (a3);
	\draw[<-, dashed] (a3) to (a4);
	
	\draw[<-] (a3) to (a5);
	\draw[<-] (a2) to (a5);
	\draw[<-, dashed] (a4) to (a5);
	
	\draw[<-, dashed] (a2) to (x);
	\draw[<-, dashed] (a3) to (x);
\end{tikzpicture}
}
\caption{(Im)Possible cases in Lemma \ref{lem:ladmid}}\label{fig:ladmid}
\end{figure}
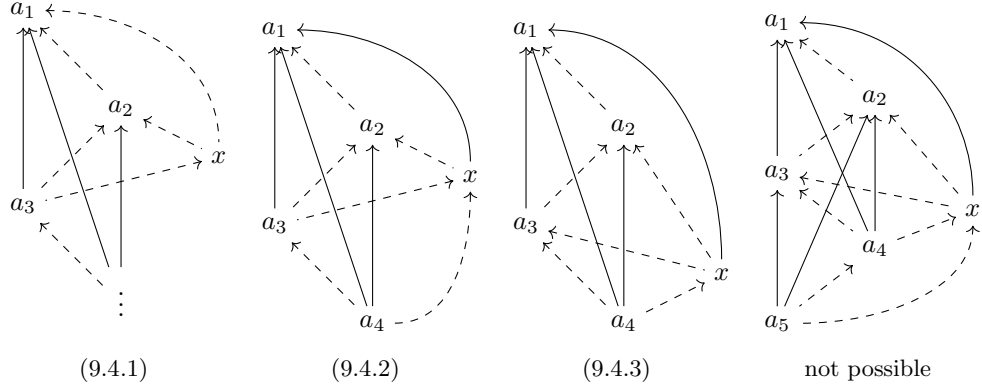

\subsection{Proof of Theorem \ref{thm:main}\ref{mainthm1}}
We are ready to prove Theorem \ref{thm:main}\ref{mainthm1}.
Since each column of $PT$ (resp. $QT$) is a chain with respect to $\cP$ (resp. the usual order), we may restrict our attention to comparing elements in two adjacent columns of $PT$ and $QT$, respectively. To this end, we set $\alpha, \beta \in \fA$ to be $\alpha=(a_m, a_{m-1}, \ldots, a_1)$, $\beta=(b_m, \ldots, b_1)$ and set $PT=(PT_1, PT_2)$ where $PT_1, PT_2 \in \fC$ are defined to be $PT_1=(d_p, \ldots, d_1)$ and $PT_2=(e_q, \ldots, e_1)$. We assume that $\Phi(\alpha, \emptyset) = (PT_1, \beta)$, $\Phi(\beta, \emptyset) = (PT_2, \gamma)$ for some $\gamma \in \fA$. We define $QT=(QT_1, QT_2)$ to be as in the algorithm of $\prs$.

\ytableausetup{nosmalltableaux}
\begin{figure}[!htbp]
\begin{tikzpicture}
	\tikzmath{\x = 0.5; \y=3; \z=2.5;\w=-4;\xx=4;}
	\node at (0,-1-\x) {$a_1$};
	\node at (0,-1-2*\x) {$a_2$};
	\node at (0,-1-3*\x) {$\vup$};
	\node at (0,-1-4*\x) {$a_{m\pp}$};
	\node at (0,-1-5*\x) {$a_{m\pp+1}$};
	\node at (0,-1-6*\x) {$\vup$};
	\node at (0,-1-7*\x) {$a_{m}$};
	\draw[rounded corners] (-.5, -1) rectangle (.5, -1-4.5*\x);
	\draw[rounded corners] (-.7, \x-1) rectangle (.7, -1-7.5*\x);
	\node[fill=white] at (0,-1+\x) {$\alpha$};
	\node[fill=white] at (0,-1) {$\alpha\pp$};
	\draw[ddarr] (.5,-2*\x-1) -- (\z,0);
	\draw[ddarr] (.7,-2.5*\x-1-1) -- (\z,\w-0.5);
	
	\node[below] at (\z+.5,.5) {$\begin{ytableau}d\pp_1\\d\pp_2\\\vdots\\d_{p\pp}\end{ytableau}$};
	\node[above] at (\z+.5,.5) {$PT_1\pp$};
	\draw[dashed] (\z, 1+.5) rectangle (\z+2.25, -2);
	\node at (\z+1.5,1-\x) {$b_1\pp$};
	\node at (\z+1.5,1-2*\x) {$b_2\pp$};
	\node at (\z+1.5,1-3*\x) {$\vup$};
	\node at (\z+1.5,1-4*\x) {$b_{m\pp}\pp$};
	\draw[rounded corners] (\z+1, 1) rectangle (\z+2, 1-4*\x-.5 );
	\node[fill=white] at (\z+1.5,1) {$\beta\pp$};

	\node[below] at (\z+.5,\w+.5) {$\begin{ytableau}d_1\\d_2\\\vdots\\d_{p}\end{ytableau}$};
	\node[above] at (\z+.5,\w+.5) {$PT_1$};
	\draw[dashed] (\z, \w+1+.5) rectangle (\z+2.25, \w-2.5);
	\node at (\z+1.5,\w+1-\x) {$b_1$};
	\node at (\z+1.5,\w+1-2*\x) {$b_2$};
	\node at (\z+1.5,\w+1-3*\x) {$\vup$};
	\node at (\z+1.5,\w+1-4*\x) {$b_{m-1}$};
	\node at (\z+1.5,\w+1-5*\x) {$b_{m}$};
	\draw[rounded corners] (\z+1, \w+1) rectangle (\z+2, \w+1-5*\x-.5 );
	\node[fill=white] at (\z+1.5,\w+1) {$\beta$};
	\draw[ddarr] (\z+2,0) -- (\z+\xx,0);
	\draw[ddarr] (\z+2,\w-0.5) -- (\z+\xx,\w-0.5);
	
	\node[below] at (\xx+\z+.5,.5) {$\begin{ytableau}e\pp_1\\e\pp_2\\\vdots\\e_{q\pp}\end{ytableau}$};
	\node[above] at (\xx+\z+.5,.5) {$PT_2\pp$};
	\draw[dashed] (\xx+\z, 1+.5) rectangle (\xx+\z+1, -2);
	
	\node[below] at (\xx+\z+.5,\w+.5) {$\begin{ytableau}e_1\\e_2\\\vdots\\e_{q}\end{ytableau}$};
	\node[above] at (\xx+\z+.5,\w+.5) {$PT_2$};
	\draw[dashed] (\xx+\z, \w+1+.5) rectangle (\xx+\z+1, \w-2.5);
	
\end{tikzpicture}
\caption{Setup for induction argument}\label{fig:setupind}
\end{figure}
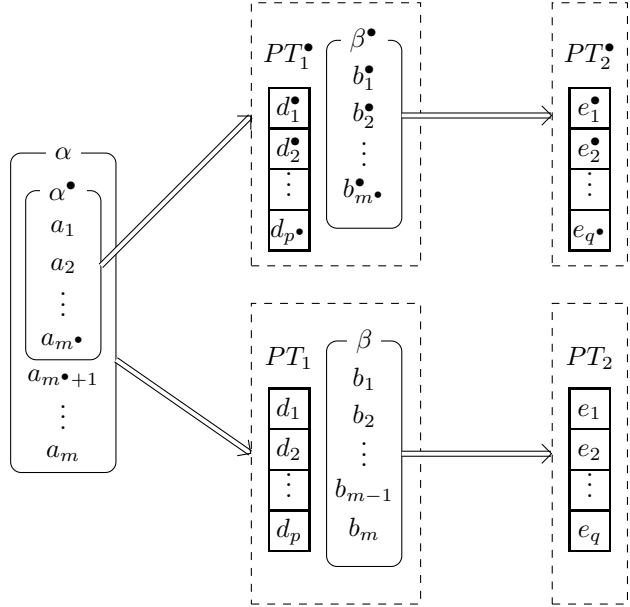

We argue by induction on $m=|\alpha|$. There is nothing to prove when its length is 0, and thus suppose that $m\geq 1$ and the statement is true up to $m-1$. We set $m\pp<m$ (specified later), and define $\alpha\pp=(a_{m\pp}, \ldots, a_1)$, $\beta\pp=(b_{m\pp}\pp, \ldots, b_2\pp, b_1\pp)$, $PT\pp=(PT_1\pp, PT_2\pp)$ where $PT_1\pp=(d_{p\pp}\pp, \ldots, d_{2}\pp, d_1\pp)$, $PT_2\pp=(e_{q\pp}\pp, \ldots, e_2\pp, e_1\pp)$, $QT\pp=(QT_1\pp, QT_2\pp)$, and $\gamma\pp\in \fA$ analogously. (See Figure \ref{fig:setupind}.) 

It suffices to assume that $a_m\neq\vn$. For $a_m$ in \csi[(a)], we set $m\pp=m-1$. Then we have $p=p\pp+1$, $q=q\pp$, $PT_1=(a_m)+PT_1\pp$, $QT_1=(m)+QT_1\pp$, $PT_2=PT_2\pp$, and $QT_2=QT_2\pp$. In this case there is nothing to prove. From now on, we divide all the remaining possibilities into the following cases. (Note that we have $p=p\pp$ and $QT_1=QT_1\pp$ in these cases.)

\begin{enumerate}[label=\Roman*., leftmargin=*]
\item $a_m$ is in \csi[(b)] (see Figure \ref{fig:amc1}): we set $m\pp\colonequals m-1$. There exists $\varrho \in [1,p]$ such that $d\pp_{\varrho-1} < a_m<d\pp_{\varrho}$ and $d\pp_{\varrho-1} \pl a_m$. (Here we set $d\pp_0=-\vn$.) Here it is easy to observe that $\Phi((a_m), PT_1\pp)=(PT_1, (b_m))$. Also we have $d_\varrho=a_m<d_\varrho\pp$, $d_k=d_k\pp$ if $k\neq \varrho$, $\beta\pp=(b_{m-1}, \ldots, b_2, b_1)$, and $b_m=d_\varrho\pp$.

\begin{figure}[!htbp]
\begin{tikzpicture}[baseline=(current bounding box.center)]
	\tikzmath{\x=0.7;\y=1.2;}
	\node (1) at (\y, 0) {$\vup$};
	\node (2) at (\y, -\x) {$d_{\varrho-1}\pp$};
	\node (4) at (\y, -3*\x) {$d_\varrho\pp$};
	\node (5) at (\y, -4*\x) {$\vup$};
	
	\node (3) at (0, -2*\x) {$a_m$};

	\node at (2*\y, -0) {$\vup$};
	\node at (2*\y, -\x) {$b_{m-1}$};
	
	\draw [->] (3) to (2);
	\draw [->] (4) to (2);
	\draw[rounded corners] (\y-0.5, 0.5) rectangle (\y+.5, -3.5);
	\node[fill=white] at (\y,0.5) {$PT_1\pp$};
	\draw[rounded corners] (2*\y-.5, 0.5) rectangle (2*\y+.5, -1.2);
	\node[fill=white] at (2*\y,0.5) {$\beta\pp$};
	\draw[dashed] (\y-0.7, 1) rectangle (2*\y+0.7, -3.7);
	\node at (\y-1.7,0) {$\Phi(\alpha\pp,\emptyset)=$};
\end{tikzpicture} 
$\quad \textnormal{\Large $\Rightarrow$} \quad $
\begin{tikzpicture}[baseline=(current bounding box.center)]
	\tikzmath{\x=0.7;\y=1.2;}
	\node (1) at (\y, 0) {$\vup$};
	\node (2) at (\y, -\x) {$d_{\varrho-1}\pp$};
	\node (4) at (\y, -3*\x) {$a_m$};
	\node (5) at (\y, -4*\x) {$\vup$};
	

	\node at (2*\y, -0) {$\vup$};
	\node at (2*\y, -\x) {$b_{m-1}$};
	\node at (2*\y, -2*\x) {$d_{\varrho}\pp$};
	
	\draw [->] (4) to (2);
	\draw[rounded corners] (\y-0.5, 0.5) rectangle (\y+.5, -3.5);
	\node[fill=white] at (\y,0.5) {$PT_1$};
	\draw[rounded corners] (2*\y-.5, 0.5) rectangle (2*\y+.5, -2);
	\node[fill=white] at (2*\y,0.5) {$\beta$};
\end{tikzpicture} 
\caption{I. $a_m$ is in \csi[(b)]}\label{fig:amc1}
\end{figure}

\item $a_m$ is in \csii[(a)] (see Figure \ref{fig:amc2}): we set $m\pp \in [0,m-1]$ to be the smallest integer such that $a_{m\pp+1}, \ldots, a_m$ are processed in the same step and $a_i \pdl a_{i+1}$ if $i\in [m\pp+1,m-1]$. Then,
\begin{enumerate}[label=--]
\item $\Phi((a_m, \ldots, a_{m\pp+1}), PT_1\pp)=(PT_1, (b_m, \ldots, b_{m\pp+1}))$,
\item there exists $\varrho\in [1, p-1]$ such that $d_{\varrho}< a_{m\pp+1}<\cdots< a_{m}<d_{\varrho+1}$, and
\item $\cL_{ad}\colonequals \{d_{\varrho}, a_{m\pp+1},\cdots, a_{m},d_{\varrho+1}\}$ is a ladder in $\cP$.
\end{enumerate}
Then we have $PT_1=PT_1\pp$, $\beta\pp=(b_{m\pp}, \ldots, b_2, b_1)$,  and $b_i=a_i$ for $i \in [m\pp+1, m]$.

\begin{figure}[!htbp]
\begin{tikzpicture}[baseline=(current bounding box.center)]
	\tikzmath{\x=0.7;\y=1.2;}
	\node (1) at (\y, 0) {$\vup$};
	\node (2) at (\y, -\x) {$d_{\varrho}\pp$};
	\node (8) at (\y, -6*\x) {$d_{\varrho+1}\pp$};
	\node at (\y, -7*\x) {$\vup$};

	\node (3) at (-\y, -1.5*\x) {$a_{m\pp+1}$};
	\node (4) at (0, -2.5*\x) {$a_{m\pp+2}$};
	\node (5) at (-\y, -3.5*\x) {$\vup$};
	\node (6) at (0, -4.5*\x) {$\vup$};
	\node (7) at (-\y, -5.5*\x) {$a_m$};

	\node at (2*\y, -0) {$\vup$};
	\node at (2*\y, -\x) {$b_{m\pp}$};
	
	\draw [->] (8) to (2);
	\draw [->] (4) to (2);
	\draw [->] (5) to (3);
	\draw [->] (6) to (4);
	\draw [->] (7) to (5);
	\draw [->] (8) to (6);
	\draw [->,dashed] (3) to (2);
	\draw [->,dashed] (4) to (3);
	\draw [->,dashed] (5) to (4);
	\draw [->,dashed] (6) to (5);
	\draw [->,dashed] (7) to (6);
	\draw [->,dashed] (8) to (7);
	\draw[rounded corners] (\y-0.5, 0.5) rectangle (\y+.5, -5.5);
	\node[fill=white] at (\y,0.5) {$PT_1\pp$};
	\draw[rounded corners] (2*\y-.5, 0.5) rectangle (2*\y+.5, -1.2);
	\node[fill=white] at (2*\y,0.5) {$\beta\pp$};
	\draw[dashed] (\y-0.7, 1) rectangle (2*\y+0.7, -5.7);
	\node at (\y-1.7,0) {$\Phi(\alpha\pp,\emptyset)=$};
\end{tikzpicture} 
$\quad \textnormal{\Large $\Rightarrow$} \quad $
\begin{tikzpicture}[baseline=(current bounding box.center)]
	\tikzmath{\x=0.7;\y=1.2;}
	\node (1) at (\y, 0) {$\vup$};
	\node (2) at (\y, -\x) {$d_{\varrho}\pp$};
	\node (8) at (\y, -6*\x) {$d_{\varrho+1}\pp$};
	\node at (\y, -7*\x) {$\vup$};


	\node at (2*\y, -0) {$\vup$};
	\node at (2*\y, -\x) {$b_{m\pp}$};
	\node at (2*\y, -2*\x) {$a_{m\pp+1}$};
	\node at (2*\y, -3*\x) {$a_{m\pp+2}$};
	\node at (2*\y, -4*\x) {$\vup$};
	\node at (2*\y, -5*\x) {$a_{m}$};
	
	\draw [->] (8) to (2);
	\draw[rounded corners] (\y-0.5, 0.5) rectangle (\y+.5, -5.5);
	\node[fill=white] at (\y,0.5) {$PT_1$};
	\draw[rounded corners] (2*\y-.5, 0.5) rectangle (2*\y+.5, -4);
	\node[fill=white] at (2*\y,0.5) {$\beta$};
\end{tikzpicture} 
\caption{II. $a_m$ is in \csii[(a)]}\label{fig:amc2}
\end{figure}

\item $a_m$ is in \csii[(b)] (see Figure \ref{fig:amc3}): we set $m\pp \in [0,m-1]$ to be the smallest integer such that $a_{m\pp+1}, \ldots, a_m$ are processed in the same step and $a_i \pl a_{i+1}$ if $i\in [m\pp+1,m-1]$. Then,
\begin{enumerate}[label=--]
\item $\Phi((a_m, \ldots, a_{m\pp+1}), PT_1\pp)=(PT_1, (b_m, \ldots, b_{m\pp+1}))$,
\item there exists $\varrho \in [m-m\pp,p]$ such that $d_{\sigma+1}\pp< a_{m\pp+1}<d_{\sigma+2}\pp<a_{m\pp+2}<\cdots< d_\varrho\pp<a_{m}$ where $\sigma\colonequals \varrho-m+m\pp$,
\item either $\varrho=q\pp$ or $d_{\varrho+1} \pr a_m$ (by maximality in \csii{}), and
\item $\cL_{ad}\colonequals\{d_{\sigma+1}\pp, a_{m\pp+1},d_{\sigma+2}\pp,a_{m\pp+2},\ldots, d_\varrho\pp,a_{m}\}$ is a ladder in $\cP$.
\end{enumerate}
Then we have $d_i=d_i\pp$ for $i \in [1,\sigma]\cup[\varrho+1, p]$ and $d_i=a_{i+m-\varrho}$ for $i \in [\sigma+1, \varrho]$, $\beta\pp=(b_{m\pp}, \ldots, b_2, b_1)$, and $b_i=d_{i+\varrho-m}\pp$ for $i \in [m\pp+1, m]$.

\begin{figure}[!htbp]
\begin{tikzpicture}[baseline=(current bounding box.center)]
	\tikzmath{\x=0.7;\y=1.2;}
	\node (1) at (\y, 0) {$\vup$};
	\node (2) at (\y, -\x) {$d_{\sigma+1}\pp$};
	\node (4) at (\y, -3*\x) {$d_{\sigma+2}\pp$};
	\node (6) at (\y, -5*\x) {$\vup$};
	\node (8) at (\y, -7*\x) {$d_{\varrho}\pp$};
	\node at (\y, -8*\x) {$\vup$};

	\node (3) at (-.5*\y, -2*\x) {$a_{m\pp+1}$};
	\node (5) at (-.5*\y, -4*\x) {$a_{m\pp+2}$};
	\node (7) at (-.5*\y, -6*\x) {$\vup$};
	\node (9) at (-.5*\y, -8*\x) {$a_m$};

	\node at (2*\y, -0) {$\vup$};
	\node at (2*\y, -\x) {$b_{m\pp}$};

	\draw [->] (4) to (2);
	\draw [->] (5) to (3);
	\draw [->] (6) to (4);
	\draw [->] (7) to (5);
	\draw [->] (8) to (6);
	\draw [->] (9) to (7);
	\draw [->,dashed] (3) to (2);
	\draw [->,dashed] (4) to (3);
	\draw [->,dashed] (5) to (4);
	\draw [->,dashed] (6) to (5);
	\draw [->,dashed] (7) to (6);
	\draw [->,dashed] (8) to (7);
	\draw [->,dashed] (9) to (8);
	\draw[rounded corners] (\y-0.5, 0.5) rectangle (\y+.5, -6.5);
	\node[fill=white] at (\y,0.5) {$PT_1\pp$};
	\draw[rounded corners] (2*\y-.5, 0.5) rectangle (2*\y+.5, -1.2);
	\node[fill=white] at (2*\y,0.5) {$\beta\pp$};
	\draw[dashed] (\y-0.7, 1) rectangle (2*\y+0.7, -6.7);
	\node at (\y-1.7,0) {$\Phi(\alpha\pp,\emptyset)=$};
\end{tikzpicture} 
$\quad \textnormal{\Large $\Rightarrow$} \quad $
\begin{tikzpicture}[baseline=(current bounding box.center)]
	\tikzmath{\x=0.7;\y=1.2;}
	\node (1) at (\y, 0) {$\vup$};
	\node (2) at (\y, -\x) {$a_{m\pp+1}$};
	\node (4) at (\y, -3*\x) {$a_{m\pp+2}$};
	\node (6) at (\y, -5*\x) {$\vup$};
	\node (8) at (\y, -7*\x) {$a_m$};
	\node at (\y, -8*\x) {$\vup$};

	\node at (2*\y, -0) {$\vup$};
	\node at (2*\y, -\x) {$b_{m\pp}$};
	\node at (2*\y, -2*\x) {$d_{\sigma+1}\pp$};
	\node at (2*\y, -3*\x) {$d_{\sigma+2}\pp$};
	\node at (2*\y, -4*\x) {$\vup$};
	\node at (2*\y, -5*\x) {$d_{\varrho}\pp$};

	\draw [->] (4) to (2);
	\draw [->] (6) to (4);
	\draw [->] (8) to (6);

	\draw[rounded corners] (\y-0.5, 0.5) rectangle (\y+.5, -6.5);
	\node[fill=white] at (\y,0.5) {$PT_1\pp$};
	\draw[rounded corners] (2*\y-.5, 0.5) rectangle (2*\y+.5, -4);
	\node[fill=white] at (2*\y,0.5) {$\beta\pp$};
\end{tikzpicture} 
\caption{III. $a_m$ is in \csii[(b)]}\label{fig:amc3}
\end{figure}

\end{enumerate}
The condition $\Phi((a_m, \ldots, a_{m\pp+1}), PT_1\pp)=(PT_1, (b_m, \ldots, b_{m\pp+1}))$ implies that in each case $QT_2$ is obtained from $QT_2\pp$ by adding some entries in $[m\pp+1,m]$ in order, which in turn means that $QT$ satisfies the standard Young tableau condition if and only if $p\geq q$. Thus in each case it suffices to prove that $p\geq q$ and $d_i \not \pr e_i$ for $i\in [1,q]$. Now we also divide possibilities into following cases. (Here we use Lemma \ref{lem:smallelt} in each case. Note that we always have $b_{m\pp+1}<\cdots< b_m$.)

\begin{enumerate}[label=\textbullet, leftmargin=*]
\item $b_{m\pp+1}$ is in \csi[(a)]: we have $e_{q\pp} \pl b_i$ for $i \in [m\pp+1, m]$. Thus $e_i = e_i\pp$ for $i \in [1, q\pp]$
\item $b_{m\pp+1}$ is in \csi[(b)]: there exists $r\in[1,q]$ such that $e\pp_{r-1}<b_{m\pp+1}<e\pp_{r}$ and  $e\pp_{r-1} \pl b_{m\pp+1}$. (Here we set $e_0\pp=-\vn$.) Then we have $e_{i}=e_i\pp$ for $i \in [1, r-1]$.

\item $b_{m\pp+1}$ is in \csii: let $m' \in [1,{m\pp+1}]$ be the smallest integer such that $b_{m'}, \ldots, b_{m\pp+1}$ are processed in the same step (thus in particular $b_{m'}<\cdots<b_{m\pp+1}$), and we set $PT_2\ppp=(e_{q\pp}\ppp, \ldots, e_2\ppp, e_1\ppp) \in \fC$ to be ``$PT_2$ right before the step processing $b_{m'}, \ldots, b_{m}$'', i.e. such that $\Phi((b_m, \ldots, b_{m'}), PT_2\ppp) = (PT_2, -)$ and also $\Phi((b_{m\pp}, \ldots, b_{m'}), PT_2\ppp) = (PT_2\pp, -)$. (As \csii{} does not increase the length of the chain, $PT_2\ppp$ is of length $q\pp$. See Figure \ref{fig:bmc2}.)

\begin{figure}[!htbp]
\begin{tikzpicture}
	\tikzmath{\x = 0.5; \y=3; \z=5.5;\w=1;\zz=8;\ww=0;}
	\node at (-4,-1) {$\emptyset$};
	\draw[ddarr] (-3.7,-1) -- (-.7,-1) node[midway,above]{$\Phi((b_{m'-1},\ldots, b_1),\emptyset)$};
	\node at (0,-\x) {$e_1\ppp$};
	\node at (0,-2*\x) {$e_2\ppp$};
	\node at (0,-3*\x) {$\vup$};
	\node at (0,-4*\x) {$e_{q\pp}\ppp$};
	\draw[rounded corners] (-.5, 0) rectangle (.5, -4*\x-.5);
	\node[fill=white] at (0,0) {$PT_2\ppp$};
	\draw[ddarr] (.7,-.5) -- (4.8,-.5)  node[midway,above]{$\Phi((b_{m\pp},\ldots, b_{m'}),PT_2\ppp)$};
	\draw[ddarr] (.7,-2.2) -- (7.3,-2.2)  node[midway,above]{$\Phi((b_m,\ldots,b_{m\pp+1}, b_{m\pp},\ldots, b_{m'}),PT_2\ppp)$};

	\node at (\z,-\x+\w) {$e_1\pp$};
	\node at (\z,-2*\x+\w) {$e_2\pp$};
	\node at (\z,-3*\x+\w) {$\vup$};
	\node at (\z,-4*\x+\w) {$e_{q\pp}\pp$};
	\draw[rounded corners] (-.5+\z, +\w) rectangle (.5+\z, -4*\x-.5+\w);
	\node[fill=white] at (\z,0+\w) {$PT_2\pp$};
	
	\node at (\zz,-\x+\ww) {$e_1$};
	\node at (\zz,-2*\x+\ww) {$e_2$};
	\node at (\zz,-3*\x+\ww) {$\vup$};
	\node at (\zz,-4*\x+\ww) {$e_{q}$};
	\draw[rounded corners] (-.5+\zz, +\ww) rectangle (.5+\zz, -4*\x-.5+\ww);
	\node[fill=white] at (\zz,0+\ww) {$PT_2$};
	
\end{tikzpicture}
\caption{Definition of $PT_2\ppp, PT_2\pp$, and $PT_2$ when $b_{m\pp+1}$ is in \csii{}}\label{fig:bmc2}
\end{figure}

Also, we set $r,s\geq 1$ and $0\leq m'-1=u(r)<\cdots<u(r+s)={m\pp+1}$ to be such that $e_{r+j}\ppp<b_{u(r+j-1)+1}<\dots< b_{u(r+j)}$ for $j \in [1,s]$ and
$$\cL_{be}\colonequals\{e_{r+1}\ppp, \ldots, e_{r+s}\ppp, b_{m'},b_{m'+1}, \ldots, b_{m\pp+1}\}$$
is a ladder in $\cP$. Note that we always have $e_i=e_i\pp=e_i\ppp$ for $i \in [1, r]$ by Lemma \ref{lem:smallelt} because $e_r\ppp\pl b_{j}$ for $j \in [m', m]$ by assumption. Also one may easily check that $e_i\pp=e_i\ppp$ for $i \geq r+s+1$. Here we list all the possibilities. (Also see Figure \ref{fig:bmcase}.) 

\begin{figure}[!htbp]
\captionsetup[subfigure]{justification=centering}
\subcaptionbox*{$u(r+s-1)=m\pp$\\$b_{m\pp}$ was in \csii[(a)]}{
\begin{tikzpicture}[baseline=(current bounding box.center)]
	\tikzmath{\x=0.7;\y=1.2;}
	\node (1) at (\y, 0) {$\vup$};
	\node (2) at (\y, -2*\x) {$e_{r+s}\ppp$};
	\node (3) at (\y, -4*\x) {$\vup$};

	\node (5) at (-.5*\y, -1*\x) {$b_{m\pp}$};
	\node (6) at (-.5*\y, -3*\x) {$b_{m\pp+1}$};

	\draw [->] (2) to (1);
	\draw [->] (3) to (2);
	\draw [->] (6) to (5);

	\draw [->,dashed] (2) to (5);
	\draw [->,dashed] (6) to (2);

	\draw[rounded corners] (\y-0.6, 0.5) rectangle (\y+.6, -3.7);
	\node[fill=white] at (\y,0.5) {$PT_2\ppp$};
\end{tikzpicture} 
}\qquad\quad\subcaptionbox*{$u(r+s-1)<m\pp$\\$b_{m\pp}$ was in \csii[(b)]}{
\begin{tikzpicture}[baseline=(current bounding box.center)]
	\tikzmath{\x=0.7;\y=1.2;}
	\node (1) at (\y, 0) {$\vup$};
	\node (2) at (\y, -2*\x) {$e_{r+s}\ppp$};
	\node (3) at (\y, -4*\x) {$\vup$};

	\node (5) at (-.5*\y, -3*\x) {$b_{m\pp}$};
	\node (6) at (-.5*\y, -5*\x) {$b_{m\pp+1}$};

	\draw [->] (2) to (1);
	\node[label={[rotate=30]$>$}] at (.25*\y,-2.9*\x)  {};
	\draw [->] (3) to (2);
	\draw [->] (6) to (2);
	\draw [->,dashed] (6) to (5);

	\draw[rounded corners] (\y-0.6, 0.5) rectangle (\y+.6, -3.7);
	\node[fill=white] at (\y,0.5) {$PT_2\ppp$};
\end{tikzpicture} 
}
\caption{Two possible cases of $b_{m\pp}$ on the calculation of $\Phi((b_{m\pp},\ldots, b_{m'}),PT_2\ppp)$} \label{fig:bmcase}
\end{figure}

\begin{enumerate}[label=$\circ$]
\item \csii[(a)], $u(r+s-1)={m\pp}$. (See Figure \ref{fig:ptsa=}.) This case includes $s=1$ and $u(r)+1=u(r+1)={m\pp+1}$. This means that $b_{m'}, \ldots, b_{m\pp}$ were originally in \csii{(a)} as well. We have $r+s+1\leq q\pp \leq q$, $e_i\pp=e_i\ppp$ for all $i$ and also $e_i=e_i\pp$ for $i \in [r+1, r+s-1]$ by Lemma \ref{lem:smallelt} because $e_{r+s-1}\pp\pl b_{j}$ for $j \in [m\pp+1, m]$.

\begin{figure}[!htbp]
\begin{tikzpicture}[baseline=(current bounding box.center)]
	\tikzmath{\x=0.5;\y=2.8;\z=\y-1.8;\yy=1.6;}

	\node (1) at (\z, 0) {$\vup$};
	\node (2) at (\z, -\x) {$r$};
	\node (3) at (\z, -2*\x) {$r+1$};
	\node (4) at (\z, -3*\x) {$\vup$};
	\node (-) at (\z, -4*\x) {$r+s-1$};
	\node (5) at (\z, -5*\x) {$r+s$};
	\node (6) at (\z, -6*\x) {$r+s+1$};
	\node (7) at (\z, -7*\x) {$\vup$};
	\node (8) at (\z, -8*\x) {$q\pp$};
	\node (9) at (\z, -9*\x) {$\vup$};
	\node (0) at (\z, -10*\x) {$q$};

	\node (11) at (\y, 0) {$\vup$};
	\node (12) at (\y, -\x) {$e_r$};
	\node (13) at (\y, -2*\x) {$e_{r+1}$};
	\node (14) at (\y, -3*\x) {$\vup$};
	\node (1-) at (\y, -4*\x) {$e_{r+s-1}$};
	\node (15) at (\y, -5*\x) {$e_{r+s}\pp$};
	\node (16) at (\y, -6*\x) {$e_{r+s+1}\pp$};
	\node (17) at (\y, -7*\x) {$\vup$};
	\node (18) at (\y, -8*\x) {$e_{q\pp}\pp$};

	\node (21) at (2*\y, 0) {$\vup$};
	\node (22) at (2*\y, -\x) {$e_r$};
	\node (23) at (2*\y, -2*\x) {$e_{r+1}$};
	\node (24) at (2*\y, -3*\x) {$\vup$};
	\node (2-) at (2*\y, -4*\x) {$e_{r+s-1}$};
	\node (25) at (2*\y, -5*\x) {$e_{r+s}\pp$};
	\node (26) at (2*\y, -6*\x) {$e_{r+s+1}\pp$};
	\node (27) at (2*\y, -7*\x) {$\vup$};
	\node (28) at (2*\y, -8*\x) {$e_{q\pp}\pp$};

	\node (31) at (3*\y, 0) {$\vup$};
	\node (32) at (3*\y, -\x) {$e_r$};
	\node (33) at (3*\y, -2*\x) {$e_{r+1}$};
	\node (34) at (3*\y, -3*\x) {$\vup$};
	\node (3-) at (3*\y, -4*\x) {$e_{r+s-1}$};
	\node (35) at (3*\y, -5*\x) {$e_{r+s}$};
	\node (36) at (3*\y, -6*\x) {$e_{r+s+1}$};
	\node (37) at (3*\y, -7*\x) {$\vup$};
	\node (38) at (3*\y, -8*\x) {$e_{q\pp}$};
	\node (39) at (3*\y, -9*\x) {$\vup$};
	\node (30) at (3*\y, -10*\x) {$e_{q}$};

	\node (41) at (4*\y, 0) {$\vup$};
	\node (42) at (4*\y, -\x) {$e_r$};
	\node (43) at (4*\y, -2*\x) {$e_{r+1}$};
	\node (44) at (4*\y, -3*\x) {$\vup$};
	\node (4-) at (4*\y, -4*\x) {$e_{r+s-1}$};
	\node (45) at (4*\y, -5*\x) {$e_{r+s}\pp$};
	\node (46) at (4*\y, -6*\x) {$e_{r+s+1}\pp$};
	\node (47) at (4*\y, -7*\x) {$\vup$};
	\node (48) at (4*\y, -8*\x) {$e_{q\pp}\pp$};

	\eee{11}{21}\eee{21}{31}\eee{31}{41}
	\eee{12}{22}\eee{22}{32}\eee{32}{42}
	\eee{13}{23}\eee{23}{33}\eee{33}{43}
	\eee{14}{24}\eee{24}{34}\eee{34}{44}
	\eee{1-}{2-}\eee{2-}{3-}\eee{3-}{4-}
	\eee{15}{25}
	\eee{16}{26}
	\eee{17}{27}
	\eee{18}{28}

	\draw[rounded corners] (\y-0.5*\yy, 0.5) rectangle (\y+0.5*\yy, -4.3);
	\node[fill=white] at (\y,0.5) {$PT_2\ppp$};
	\draw[rounded corners] (2*\y-0.5*\yy, 0.5) rectangle (2*\y+0.5*\yy, -4.3);
	\node[fill=white] at (2*\y,0.5) {$PT_2\pp$};
	\draw[rounded corners] (3*\y-0.5*\yy, 0.5) rectangle (3*\y+0.5*\yy, -5.3);
	\node[fill=white] at (3*\y,0.5) {$PT_2$};
	\draw[rounded corners] (4*\y-0.5*\yy, 0.5) rectangle (4*\y+0.5*\yy, -4.3);
	\node[fill=white] at (4*\y,0.5) {$PT_2\ppp$};
	
	\node[fill=white] at (\z,0.5) {index};
\end{tikzpicture}
\caption{$PT_2\ppp$, $PT_2\pp$, and $PT_2$ when $b_{m\pp+1}$ in \csii[(a)], $u(r+s-1)={m\pp}$}\label{fig:ptsa=}
\end{figure}

\item \csii[(a)], $u(r+s-1)< {m\pp}$. (See Figure \ref{fig:ptsa<}.) This means that $b_{m'}, \ldots, b_{m\pp}$ were originally in \csii{(b)}. We have $r+s+1\leq q\pp \leq q$, $e_{r+i}\pp=b_{u(r+i)}$ for $i \in [1, s-1]$, $e_{r+s}\pp=b_{m\pp}$,  and $e_i=e_i\ppp$ for $i \in [r+1, r+s]$ by Lemma \ref{lem:smallelt} because \csii{(a)} does not alter the given chain and $e_{r+s}\ppp\pl b_{j}$ for $j \in [m\pp+1, m]$ by assumption.

\begin{figure}[!htbp]
\begin{tikzpicture}[baseline=(current bounding box.center)]
	\tikzmath{\x=0.5;\y=2.8;\z=\y-1.8;\yy=1.6;}

	\node (1) at (\z, 0) {$\vup$};
	\node (2) at (\z, -\x) {$r$};
	\node (3) at (\z, -2*\x) {$r+1$};
	\node (4) at (\z, -3*\x) {$\vup$};
	\node (-) at (\z, -4*\x) {$r+s-1$};
	\node (5) at (\z, -5*\x) {$r+s$};
	\node (6) at (\z, -6*\x) {$r+s+1$};
	\node (7) at (\z, -7*\x) {$\vup$};
	\node (8) at (\z, -8*\x) {$q\pp$};
	\node (9) at (\z, -9*\x) {$\vup$};
	\node (0) at (\z, -10*\x) {$q$};

	\node (11) at (\y, 0) {$\vup$};
	\node (12) at (\y, -\x) {$e_r$};
	\node (13) at (\y, -2*\x) {$e_{r+1}$};
	\node (14) at (\y, -3*\x) {$\vup$};
	\node (1-) at (\y, -4*\x) {$e_{r+s-1}$};
	\node (15) at (\y, -5*\x) {$e_{r+s}$};
	\node (16) at (\y, -6*\x) {$e_{r+s+1}\pp$};
	\node (17) at (\y, -7*\x) {$\vup$};
	\node (18) at (\y, -8*\x) {$e_{q\pp}\pp$};

	\node (21) at (2*\y, 0) {$\vup$};
	\node (22) at (2*\y, -\x) {$e_r$};
	\node (23) at (2*\y, -2*\x) {$b_{u(r+1)}$};
	\node (24) at (2*\y, -3*\x) {$\vup$};
	\node (2-) at (2*\y, -4*\x) {$b_{u(r+s-1)}$};
	\node (25) at (2*\y, -5*\x) {$b_{m\pp}$};
	\node (26) at (2*\y, -6*\x) {$e_{r+s+1}\pp$};
	\node (27) at (2*\y, -7*\x) {$\vup$};
	\node (28) at (2*\y, -8*\x) {$e_{q\pp}\pp$};

	\node (31) at (3*\y, 0) {$\vup$};
	\node (32) at (3*\y, -\x) {$e_r$};
	\node (33) at (3*\y, -2*\x) {$e_{r+1}$};
	\node (34) at (3*\y, -3*\x) {$\vup$};
	\node (3-) at (3*\y, -4*\x) {$e_{r+s-1}$};
	\node (35) at (3*\y, -5*\x) {$e_{r+s}$};
	\node (36) at (3*\y, -6*\x) {$e_{r+s+1}$};
	\node (37) at (3*\y, -7*\x) {$\vup$};
	\node (38) at (3*\y, -8*\x) {$e_{q\pp}$};
	\node (39) at (3*\y, -9*\x) {$\vup$};
	\node (30) at (3*\y, -10*\x) {$e_{q}$};
	
	\node (41) at (4*\y, 0) {$\vup$};
	\node (42) at (4*\y, -\x) {$e_r$};
	\node (43) at (4*\y, -2*\x) {$e_{r+1}$};
	\node (44) at (4*\y, -3*\x) {$\vup$};
	\node (4-) at (4*\y, -4*\x) {$e_{r+s-1}$};
	\node (45) at (4*\y, -5*\x) {$e_{r+s}$};
	\node (46) at (4*\y, -6*\x) {$e_{r+s+1}\pp$};
	\node (47) at (4*\y, -7*\x) {$\vup$};
	\node (48) at (4*\y, -8*\x) {$e_{q\pp}\pp$};
	
	\eee{11}{21}\eee{21}{31}\eee{31}{41}
	\eee{12}{22}\eee{22}{32}\eee{32}{42}
	\eee{33}{43}
	\eee{34}{44}
	\eee{3-}{4-}
	\eee{35}{45}
	\eee{16}{26}
	\eee{17}{27}
	\eee{18}{28}

	\draw[rounded corners] (\y-0.5*\yy, 0.5) rectangle (\y+0.5*\yy, -4.3);
	\node[fill=white] at (\y,0.5) {$PT_2\ppp$};
	\draw[rounded corners] (2*\y-0.5*\yy, 0.5) rectangle (2*\y+0.5*\yy, -4.3);
	\node[fill=white] at (2*\y,0.5) {$PT_2\pp$};
	\draw[rounded corners] (3*\y-0.5*\yy, 0.5) rectangle (3*\y+0.5*\yy, -5.3);
	\node[fill=white] at (3*\y,0.5) {$PT_2$};
	\draw[rounded corners] (4*\y-0.5*\yy, 0.5) rectangle (4*\y+0.5*\yy, -4.3);
	\node[fill=white] at (4*\y,0.5) {$PT_2\ppp$};
	
	\node[fill=white] at (\z,0.5) {index};
\end{tikzpicture} 
\caption{$PT_2\ppp$, $PT_2\pp$, and $PT_2$ when $b_{m\pp+1}$ in \csii[(a)], $u(r+s-1)< {m\pp}$}\label{fig:ptsa<}
\end{figure}

\item \csii[(b)], $u(r+s-1)< {m\pp}$. (See Figure \ref{fig:ptsb<}.)  This means that $b_{m'}, \ldots, b_{m\pp}$ were originally in \csii{(b)} as well. We have $e_{i}=e_i\pp$ for $i \in [r+1,r+s-1]$ and $e_{r+s}\pp=b_{m\pp}$ by Lemma \ref{lem:smallelt} since $e_{r+s-1}\pp\pl b_{j}$ for $j \in [m\pp+1, m]$ by assumption.

\begin{figure}[!htbp]
\begin{tikzpicture}[baseline=(current bounding box.center)]
	\tikzmath{\x=0.5;\y=2.8;\z=\y-1.8;\yy=1.6;}

	\node (1) at (\z, 0) {$\vup$};
	\node (2) at (\z, -\x) {$r$};
	\node (3) at (\z, -2*\x) {$r+1$};
	\node (4) at (\z, -3*\x) {$\vup$};
	\node (-) at (\z, -4*\x) {$r+s-1$};
	\node (5) at (\z, -5*\x) {$r+s$};
	\node (6) at (\z, -6*\x) {$r+s+1$};
	\node (7) at (\z, -7*\x) {$\vup$};
	\node (8) at (\z, -8*\x) {$q\pp$};
	\node (9) at (\z, -9*\x) {$\vup$};
	\node (0) at (\z, -10*\x) {$q$};

	\node (11) at (\y, 0) {$\vup$};
	\node (12) at (\y, -\x) {$e_r$};
	\node (13) at (\y, -2*\x) {$e_{r+1}\ppp$};
	\node (14) at (\y, -3*\x) {$\vup$};
	\node (1-) at (\y, -4*\x) {$e_{r+s-1}\ppp$};
	\node (15) at (\y, -5*\x) {$e_{r+s}\ppp$};
	\node (16) at (\y, -6*\x) {$e_{r+s+1}\pp$};
	\node (17) at (\y, -7*\x) {$\vup$};
	\node (18) at (\y, -8*\x) {$e_{q\pp}\pp$};

	\node (21) at (2*\y, 0) {$\vup$};
	\node (22) at (2*\y, -\x) {$e_r$};
	\node (23) at (2*\y, -2*\x) {$e_{r+1}$};
	\node (24) at (2*\y, -3*\x) {$\vup$};
	\node (2-) at (2*\y, -4*\x) {$e_{r+s-1}$};
	\node (25) at (2*\y, -5*\x) {$b_{m\pp}$};
	\node (26) at (2*\y, -6*\x) {$e_{r+s+1}\pp$};
	\node (27) at (2*\y, -7*\x) {$\vup$};
	\node (28) at (2*\y, -8*\x) {$e_{q\pp}\pp$};

	\node (31) at (3*\y, 0) {$\vup$};
	\node (32) at (3*\y, -\x) {$e_r$};
	\node (33) at (3*\y, -2*\x) {$e_{r+1}$};
	\node (34) at (3*\y, -3*\x) {$\vup$};
	\node (3-) at (3*\y, -4*\x) {$e_{r+s-1}$};
	\node (35) at (3*\y, -5*\x) {$e_{r+s}$};
	\node (36) at (3*\y, -6*\x) {$e_{r+s+1}$};
	\node (37) at (3*\y, -7*\x) {$\vup$};
	\node (38) at (3*\y, -8*\x) {$e_{q\pp}$};
	\node (39) at (3*\y, -9*\x) {$\vup$};
	\node (30) at (3*\y, -10*\x) {$e_{q}$};
	
	\node (41) at (4*\y, 0) {$\vup$};
	\node (42) at (4*\y, -\x) {$e_r$};
	\node (43) at (4*\y, -2*\x) {$e_{r+1}\ppp$};
	\node (44) at (4*\y, -3*\x) {$\vup$};
	\node (4-) at (4*\y, -4*\x) {$e_{r+s-1}\ppp$};
	\node (45) at (4*\y, -5*\x) {$e_{r+s}\ppp$};
	\node (46) at (4*\y, -6*\x) {$e_{r+s+1}\pp$};
	\node (47) at (4*\y, -7*\x) {$\vup$};
	\node (48) at (4*\y, -8*\x) {$e_{q\pp}\pp$};
	
	\eee{11}{21}\eee{21}{31}\eee{31}{41}
	\eee{12}{22}\eee{22}{32}\eee{32}{42}
	\eee{23}{33}
	\eee{24}{34}
	\eee{2-}{3-}
	\eee{16}{26}
	\eee{17}{27}
	\eee{18}{28}

	\draw[rounded corners] (\y-0.5*\yy, 0.5) rectangle (\y+0.5*\yy, -4.3);
	\node[fill=white] at (\y,0.5) {$PT_2\ppp$};
	\draw[rounded corners] (2*\y-0.5*\yy, 0.5) rectangle (2*\y+0.5*\yy, -4.3);
	\node[fill=white] at (2*\y,0.5) {$PT_2\pp$};
	\draw[rounded corners] (3*\y-0.5*\yy, 0.5) rectangle (3*\y+0.5*\yy, -5.3);
	\node[fill=white] at (3*\y,0.5) {$PT_2$};
	\draw[rounded corners] (4*\y-0.5*\yy, 0.5) rectangle (4*\y+0.5*\yy, -4.3);
	\node[fill=white] at (4*\y,0.5) {$PT_2\ppp$};
	
	\node[fill=white] at (\z,0.5) {index};
\end{tikzpicture} 
\caption{$PT_2\ppp$, $PT_2\pp$, and $PT_2$ when $b_{m\pp+1}$ in \csii[(b)], $u(r+s-1)< {m\pp}$}\label{fig:ptsb<}
\end{figure}

\item \csii[(b)], $u(r+s-1) = {m\pp}$. (See Figure \ref{fig:ptsb=}.) This case includes $s=1$ and $u(r)+1=u(r+1)={m\pp+1}$. This means that $b_{m'}, \ldots, b_{m\pp}$ were originally in \csii{(a)}. We have $e_i\pp=e_i\ppp$ for any $i$ and $e_{r+i}=b_{u(r+i)}$ for $i \in [1, s-1]$ by Lemma \ref{lem:smallelt} because $b_{u(r+s-1)}=b_{m\pp}\pl b_{j}$ for $j \in [m\pp+1, m]$. Also, it is easy to observe that $e_{r+i}\pp <e_{r+i}$ for $i \in [1,s-1]$.

\begin{figure}[!htbp]
\begin{tikzpicture}[baseline=(current bounding box.center)]
	\tikzmath{\x=0.5;\y=2.8;\z=\y-1.8;\yy=1.6;}

	\node (1) at (\z, 0) {$\vup$};
	\node (2) at (\z, -\x) {$r$};
	\node (3) at (\z, -2*\x) {$r+1$};
	\node (4) at (\z, -3*\x) {$\vup$};
	\node (-) at (\z, -4*\x) {$r+s-1$};
	\node (5) at (\z, -5*\x) {$r+s$};
	\node (6) at (\z, -6*\x) {$r+s+1$};
	\node (7) at (\z, -7*\x) {$\vup$};
	\node (8) at (\z, -8*\x) {$q\pp$};
	\node (9) at (\z, -9*\x) {$\vup$};
	\node (0) at (\z, -10*\x) {$q$};

	\node (11) at (\y, 0) {$\vup$};
	\node (12) at (\y, -\x) {$e_r$};
	\node (13) at (\y, -2*\x) {$e_{r+1}\pp$};
	\node (14) at (\y, -3*\x) {$\vup$};
	\node (1-) at (\y, -4*\x) {$e_{r+s-1}\pp$};
	\node (15) at (\y, -5*\x) {$e_{r+s}\pp$};
	\node (16) at (\y, -6*\x) {$e_{r+s+1}\pp$};
	\node (17) at (\y, -7*\x) {$\vup$};
	\node (18) at (\y, -8*\x) {$e_{q\pp}\pp$};

	\node (21) at (2*\y, 0) {$\vup$};
	\node (22) at (2*\y, -\x) {$e_r$};
	\node (23) at (2*\y, -2*\x) {$e_{r+1}\pp$};
	\node (24) at (2*\y, -3*\x) {$\vup$};
	\node (2-) at (2*\y, -4*\x) {$e_{r+s-1}\pp$};
	\node (25) at (2*\y, -5*\x) {$e_{r+s}\pp$};
	\node (26) at (2*\y, -6*\x) {$e_{r+s+1}\pp$};
	\node (27) at (2*\y, -7*\x) {$\vup$};
	\node (28) at (2*\y, -8*\x) {$e_{q\pp}\pp$};

	\node (31) at (3*\y, 0) {$\vup$};
	\node (32) at (3*\y, -\x) {$e_r$};
	\node (33) at (3*\y, -2*\x) {$b_{u(r+1)}$};
	\node (34) at (3*\y, -3*\x) {$\vup$};
	\node (3-) at (3*\y, -4*\x) {$b_{u(r+s-1)}$};
	\node (35) at (3*\y, -5*\x) {$e_{r+s}$};
	\node (36) at (3*\y, -6*\x) {$e_{r+s+1}$};
	\node (37) at (3*\y, -7*\x) {$\vup$};
	\node (38) at (3*\y, -8*\x) {$e_{q\pp}$};
	\node (39) at (3*\y, -9*\x) {$\vup$};
	\node (30) at (3*\y, -10*\x) {$e_{q}$};
	
	\node (41) at (4*\y, 0) {$\vup$};
	\node (42) at (4*\y, -\x) {$e_r$};
	\node (43) at (4*\y, -2*\x) {$e_{r+1}\pp$};
	\node (44) at (4*\y, -3*\x) {$\vup$};
	\node (4-) at (4*\y, -4*\x) {$e_{r+s-1}\pp$};
	\node (45) at (4*\y, -5*\x) {$e_{r+s}\pp$};
	\node (46) at (4*\y, -6*\x) {$e_{r+s+1}\pp$};
	\node (47) at (4*\y, -7*\x) {$\vup$};
	\node (48) at (4*\y, -8*\x) {$e_{q\pp}\pp$};
	
	\eee{11}{21}\eee{21}{31}\eee{31}{41}
	\eee{12}{22}\eee{22}{32}\eee{32}{42}
	\eee{13}{23}
	\eee{14}{24}
	\eee{1-}{2-}
	\eee{15}{25}
	\eee{16}{26}
	\eee{17}{27}
	\eee{18}{28}

	\draw[rounded corners] (\y-0.5*\yy, 0.5) rectangle (\y+0.5*\yy, -4.3);
	\node[fill=white] at (\y,0.5) {$PT_2\ppp$};
	\draw[rounded corners] (2*\y-0.5*\yy, 0.5) rectangle (2*\y+0.5*\yy, -4.3);
	\node[fill=white] at (2*\y,0.5) {$PT_2\pp$};
	\draw[rounded corners] (3*\y-0.5*\yy, 0.5) rectangle (3*\y+0.5*\yy, -5.3);
	\node[fill=white] at (3*\y,0.5) {$PT_2$};
	\draw[rounded corners] (4*\y-0.5*\yy, 0.5) rectangle (4*\y+0.5*\yy, -4.3);
	\node[fill=white] at (4*\y,0.5) {$PT_2\ppp$};
	
	\node[fill=white] at (\z,0.5) {index};
\end{tikzpicture} 
\caption{$PT_2\ppp$, $PT_2\pp$, and $PT_2$ when $b_{m\pp+1}$ in \csii[(b)], $u(r+s-1) = {m\pp}$}\label{fig:ptsb=}
\end{figure}
\end{enumerate}

\end{enumerate}

From now on we verify the conditions $p\geq q$ and $d_i \not \pr e_i$ for $i\in [1,q]$ in each case. Note that $p=p\pp\geq q\pp$ and $d_{i}\pp \not\pr e_i\pp$ for any $i$ by induction assumption, and thus $p\geq q$ if $q=q\pp$ and $d_i \not \pr e_i$ if $d_i=d_i\pp$, $e_i=e_i\pp$.

\subsubsection*{\textup{\bfseries I.} \und{$a_m$ is in \textup{\csi[(b)]}}} We have $m\pp=m-1$. Here it suffices to check $d_k \not \pr e_k$ only for $k$ satisfying $k\leq q$ and either $k> q\pp$ or $e_k < e_k\pp$, which follows from \cond{}, induction assumption, and the fact that $d_k \leq d_k\pp$ for $k \in [1,p]$.

\noindent {\bfseries I.i.} $b_m$ is in \csi[(a)]: we have $q=q\pp+1$ and $PT_2=(b_m)+PT\pp_2$. First in this case $p \geq \varrho \geq q$ since $e\pp_{q-1}=e\pp_{q\pp}  \pl b_{m}=d\pp_{\varrho}$. Now it suffices to check $d_q \not \pr e_q$. Now if $\varrho>q$, then $e_{q}=d\pp_{\varrho} \not\pl d_{q}\pp =d_{q}$. If $\varrho=q$, then $e_{q}=d\pp_{\varrho} \not\pl a_m=d_{\varrho}$.

\noindent {\bfseries I.ii.}  $b_m$ is in \csi[(b)]: we have $q=q\pp$, and $e_i\neq e_i\pp$ only when $i=r$ in which case $e_r=b_{m}=d_\varrho\pp$. Thus it suffices to check $d_r \not \pr e_r$. Note that $r\leq \varrho$ since $d\pp_{\varrho} =b_m\pr e\pp_{r-1}$. If $r<{\varrho}$ then $d_{r}=d_{r}\pp\not\pr d\pp_{\varrho}=e_{r}$. If $r={\varrho}$ then $d_{\varrho}=a_m \not\pr d\pp_{\varrho}=e_{r}$ thus the condition still holds. 

\noindent {\bfseries I.iii.}  $b_m$ is in \csii[(a)], $u(r+s-1)=m-1$: we have $PT_2\pp=PT_2$, and thus there is nothing to check.

\noindent {\bfseries I.iv.}  $b_m$ is in \csii[(a)], $u(r+s-1)< m-1$: in this case we have $q=q\pp$,  $e_{r+s}\pp = b_{m-1}$, $e\pp_{r+j} = b_{u(r+j)}$ for $j \in [1,s-1]$, and $e\pp_{k}=e_k$ if $k\not\in[r+1, r+s]$. Thus it suffices to check $d_k \not \pr e_k$ for $k\in [r+1, r+s]$. Note that $r+s\leq \varrho$ since $e_{r+s}\pp=b_{m-1}<b_m=d_{\varrho}\pp $.
First suppose that $k\in [r+1,r+s]$, $k< \varrho$, and $d_k \pr e_k$. Then as $d_k =d_k\pp\pl d_{\varrho}\pp=b_m$ it implies that $d_k$ is climbing $\cL_{be}$ in $\cP$, which is a contradiction. 
Now suppose $k\in [r+1,r+s]$, $k\geq \varrho$ and $d_{k}\pr e_{k}$ (which forces that $k=\varrho=r+s$). Note that $\cL_{be}\cup \{e_{r+s+1}\}$ is a ladder in $\cP$ since $b_m$ is in \csii[(a)] and $\cL_{be}\cup \{e_{r+s+1}, d_{\varrho+1}\}$ is also a ladder as $d_{\varrho+1}=d_{\varrho+1} \pp \pr d_{\varrho}\pp=b_m$ and $d_{\varrho+1}=d_{\varrho+1}\pp\not \pr e_{r+s+1}\pp=e_{r+s+1}$. Since $d_\varrho \pl d_{\varrho+1}$, we see that $d_{\varrho}$ is climbing the ladder $\cL_{be}\cup \{e_{r+s+1}, d_{\varrho+1}\}$ which is a contradiction. (See Figure \ref{fig:Iiv}.)

\begin{figure}[!htbp]
\begin{tikzpicture}[baseline=(current bounding box.center)]
	\tikzmath{\x=0.9;\y=1.2;}

	\node (2) at (\y, -\x) {$e_{r+s}$};
	\node (8) at (\y, -6*\x) {$e_{r+s+1}$};

	\node (3) at (-\y, -1.5*\x) {$b_{u(r+s-1)+1}$};
	\node (4) at (0, -2.5*\x) {$\vup$};
	\node (5) at (-\y, -3.5*\x) {$\vup$};
	\node (6) at (0, -4.5*\x) {$b_{m-1}$};
	\node (7) at (-\y, -5.5*\x) {$b_m$};

	\node (1) at (-2*\y, -3*\x) {$d_{\varrho}$};
	\node (9) at (-2*\y, -6.5*\x) {$d_{\varrho+1}$};

	\draw [->] (8) to (2);
	\draw [->] (4) to (2);
	\draw [->] (5) to (3);
	\draw [->] (6) to (4);
	\draw [->] (7) to (5);
	\draw [->] (8) to (6);
	\draw [->] (9) to (7);
	\draw [->] (9) to (1);
	\draw [->,dashed] (3) to (2);
	\draw [->,dashed] (4) to (3);
	\draw [->,dashed] (5) to (4);
	\draw [->,dashed] (6) to (5);
	\draw [->,dashed] (7) to (6);
	\draw [->,dashed] (8) to (7);
	\draw [->,dashed] (9) to (8);
\end{tikzpicture} 
\caption{$a_m$ in \csi[(b)],  $b_m$ in \csii[(a)], $u(r+s-1)< m-1$, $\varrho=r+s$}\label{fig:Iiv}
\end{figure}
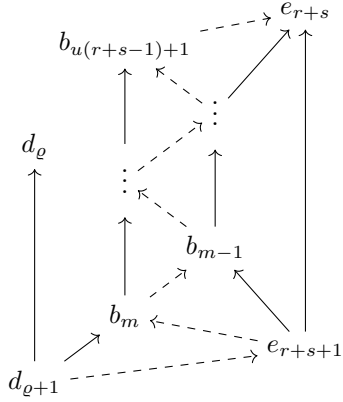

\noindent {\bfseries I.v.}  $b_m$ is in \csii[(b)], $u(r+s-1)< m-1$: we have  $q=q\pp$, $e_{r+s}=b_m>b_{m-1}=e_{r+s}\pp$ and $e_{k}=e_k\pp$ for $k \neq r+s$, and thus there is nothing to check.

\noindent {\bfseries I.vi.}  $b_m$ is in \csii[(b)], $u(r+s-1) = m-1$: we have $q=q\pp$, $e_{r+i}\pp<b_{u(r+i)}=e_{r+i}$ for $i \in [1,s]$ and $e_k\pp=e_k$ for $k \not\in [r+1,r+s]$, and thus there is nothing to check.

\subsubsection*{\textup{\bfseries II.} \und{$a_m$ is in \textup{\csii[(a)]}}}
Recall that $b_i=d_{i+\varrho-m}\pp$ for $i \in [m\pp+1, m]$. Since $d_i=d_i\pp$ for all $i$, similarly to above it suffices to check $d_k \not \pr e_k$ only for $k$ satisfying $k\leq q$ and either $k> q\pp$ or $e_k < e_k\pp$. 

\noindent {\bfseries II.i.} $b_{m\pp+1}$ is in \csi[(a)]: direct calculation shows that $q=q\pp+1$ and $PT_2=(a_m)+PT_2\pp$. First we have $p\geq \varrho+1\geq q\pp+1=q$ since $d_{\varrho+1}>a_{m\pp+1}=b_{m\pp+1} \pr e_{q\pp}$. It remains to check $d_q \not \pr e_q$. If $\varrho+1=q$, then $d_{\varrho+1}\pdr a_m=e_q$, and thus the result holds. Otherwise, i.e. if $\varrho+1>q$, then $d_{q} \leq d_{\varrho}<a_m=e_q$, and thus again the result holds.


\noindent {\bfseries II.ii.} $b_{m\pp+1}$ is in \csi[(b)]:  Note that $\varrho+1 \geq r$ since $d_{\varrho+1}\pp>a_{m\pp+1}=b_{m\pp+1} \pr e_{r-1}\pp$. First assume that $r=q$ or $e_{r+1}\pp\pr b_m= a_m$. (This includes the case $m=m\pp+1$.) Then we have $q=q\pp$, $e_r=a_m$, and $e_k =e_k\pp$ for $k \neq r$, and thus it suffices to show that $d_r \not \pr e_r = a_m$. However, if $d_r \pr a_m$ then $d_{\varrho+1} \geq d_r \pr a_m$ thus $d_{\varrho+1} \pr a_m$, which contradicts the assumption.

From now on we suppose that $e_{r+1}\pp\not\pr a_m$ (so that $m\geq m\pp+2$). In this case $a_m>e_r\pp$ by \cond{} since $e_{r+1}\pp\not\pr a_m$ and $e_{r+1}\pp \pr e_r\pp$. First we assume $m=m\pp+2$. Then we have $a_m\pdl e_{r+1}\pp$; otherwise we should have $a_m>e_{r+1}\pp$, but it is impossible by \cond{} since $e_{r+1}\pp \pr e_r\pp >a_{m-1}$ but $a_m \pdr a_{m-1}$. Now direct calculation shows that $q=q\pp$, $e_r=a_{m-1}$, and $e_k=e_k\pp$ for $k\neq r$, and thus it suffices to show that $d_r \not\pr e_r=a_{m-1}$. If $\varrho+1>r$, then we have $d_r \leq d_{\varrho}<a_{m-1}=e_r$, and thus we are done. Now we claim that $\varrho+1\not=r$; first note that $d_\varrho > e_{r-1}$ since $a_{m-1} \pr e_{r-1}\pp=e_{r-1}$ and $a_{m-1} \pdr d_\varrho$. Together with $d_\varrho =d_{\varrho}\pp \not\pr e_{\varrho}\pp=e_{r-1}$, this means that $d_\varrho \pdr e_{r-1}$. Also $e_{r+1} \pr a_{m-1}$ since $e_{r+1}=e_{r+1}\pp \pr e_r\pp >a_{m-1}$. But this means that $\{e_{r-1}, d_\varrho, a_{m-1}, a_m, e_{r+1}\}$ is a ladder in $\cP$ which $e_r\pp$ is climbing ($e_{r+1} \pr e_r\pp \pr e_{r-1})$, which is a contradiction. (See Figure \ref{fig:IIii}.)

\begin{figure}[!htbp]
\begin{tikzpicture}[baseline=(current bounding box.center)]
	\tikzmath{\x=0.9;\y=1.2;}

	\node (1) at (0,0) {$d_\varrho$};
	\node (2) at (\y,-.5*\x) {$a_{m-1}$};
	\node (3) at (\y,-2.5*\x) {$a_m$};
	\node (4) at (0,-3*\x) {$d_{\varrho+1}$};
	
	\node (5) at (2*\y,.5*\x) {$e_{r-1}$};
	\node (6) at (2*\y,-3*\x) {$e_{r+1}$};

	\node (7) at (2*\y,-1.53*\x) {$e_{r}\pp$};
	
	\draw [->] (3) to (1);
	\draw [->] (4) to (2);
	\draw [->] (2) to (5);
	\draw [->] (6) to (2);
	\draw [->] (6) to (7);
	\draw [->] (7) to (5);

	\draw [->, dashed] (1) to (5);
	\draw [->, dashed] (6) to (3);
	\draw [->,dashed] (2) to (1);	
	\draw [->,dashed] (3) to (2);
	\draw [->,dashed] (4) to (3);
\end{tikzpicture} 
\caption{$a_m$ in \csii[(a)],  $b_m$ in \csi[(b)], $m=m\pp+2$, $\varrho+1=r$}\label{fig:IIii}
\end{figure}

It remains to assume that $m\geq m\pp+3$. We claim that $e_r\pp \not\pr a_{m\pp+1}$. Otherwise, $a_{m-1} <e_r\pp$ by Lemma \ref{lem:ladend} which in turn implies that $a_{m-1} \pl e_{r+1}\pp$ and $a_m \pdl e_{r+1}\pp$ by assumption. However, it means that $\{a_{m\pp+1}, \ldots, a_m, e_{r+1}\pp\}$ is a ladder that $e_r\pp$ is climbing, which is a contradiction. Since $e_r \pp > a_{m\pp+1}$ by assumption, it follows that $e_r\pp \pdr a_{m\pp+1}$.  Now we apply Lemma \ref{lem:ladmid} to $\cL_{ad}$ (of length $m-m\pp+2\geq 5$) and $e_r\pp$ and conclude that $d_{\varrho} \pdl e_{r}\pp$ and $a_{m\pp+2} \pdr e_r\pp$.

Note that $e_{r+1}\pp \pr a_{m\pp+1}$ by \cond{} as $e_r\pp \pl e_{r+1}\pp$ and $e_{r}\pp>a_{m\pp+1}$. Thus by Lemma \ref{lem:ladend} we have $e_r\pp \pr a_{j}$ for $j\in[m\pp+1,m-2]$ and we may apply Lemma \ref{lem:ladtail4} to $\{a_{m-1}, a_m\}$ and $e_{r+1}\pp$. As $a_m \not\pl e_{r+1}\pp$ by assumption, case (\ref{lem:ladtail4}.4) is excluded. In other cases, direct calculations shows that $e_r=a_{m\pp+1}$, $e_k=e_k\pp$ if $k \not \in \{r, r+1\}$, and
\begin{enumerate}[label=$\circ$]
\item (\ref{lem:ladtail4}.1): either $e_{r+1}=e_{r+1}\pp$ or $e_{r+1}=a_m>e_{r+1}\pp$.
\item (\ref{lem:ladtail4}.2): $e_{r+1}=a_m$.
\item (\ref{lem:ladtail4}.3): $e_{r+1}=e_{r+1}\pp$.
\end{enumerate}
(See Figure \ref{fig:IIii2}.) In any case we have $q=q\pp$. Thus it suffices to check that $d_r \not \pr e_r$, and $d_{r+1} \not \pr e_{r+1}$ in case (\ref{lem:ladtail4}.2).
Note that here we have $\varrho+1>r$ since $d_{\varrho+1}\pr a_{m-1} \geq a_{m\pp+2}>e_r\pp$. If $d_r \pr e_r = a_{m\pp+1}$, then since $d_r \pl d_{\varrho+1}$ it means that $d_r$ is climbing $\cL_{ad}$, which is a contradiction. Also in case (\ref{lem:ladtail4}.2) we have $d_{r+1}\not\pr e_{r+1}=a_m$ by \cond{} since $d_{r+1}\leq d_{\varrho+1}$ and $d_{\varrho+1} \pdr a_m$.

\begin{figure}[!htbp]
\subcaptionbox*{(\ref{lem:ladtail4}.1)}{
\begin{tikzpicture}[baseline=(current bounding box.center)]
	\tikzmath{\x=0.9;\y=1.2;}

	\node (1) at (0,0) {$a_{m\pp+1}$};
	\node (2) at (\y,-1*\x) {$a_{m\pp+2}$};
	\node (3) at (0.5*\y,-2*\x) {$\vup$};
	\node (4) at (1.5*\y,-3*\x) {$a_{m-1}$};
	\node (5) at (1*\y,-4*\x) {$a_m$};
	\node (6) at (0,-4.5*\x) {$d_{\varrho+1}$};

	\node (10) at (2.*\y,0.5*\x) {$e_{r-1}\pp$};
	\node (11) at (2.25*\y,-0.5*\x) {$e_{r}\pp$};
	\node (12) at (2.75*\y,-3.5*\x) {$e_{r+1}\pp$};

	\draw [->] (3) to (1);
	\draw [->] (4) to (2);
	\draw [->] (5) to (3);
	\draw [->] (12) to (11);
	\draw [->] (11) to (10);
	\draw [->, out=120,in=0] (12) to (3);
	\draw [->] (1) to (10);
	\draw [->,out=90,in=180] (6) to (4);
%

	\draw [->,dashed] (2) to (1);	
	\draw [->,dashed] (3) to (2);
	\draw [->,dashed] (4) to (3);
	\draw [->,dashed] (5) to (4);
	\draw [->,dashed] (6) to (5);
	\draw [->,dashed] (11) to (1);	
	\draw [->,dashed] (2) to (11);	
	\draw [->,dashed] (12) to (4);
	\draw [->,dashed] (5) to (12);	
\end{tikzpicture} 
}\quad\subcaptionbox*{(\ref{lem:ladtail4}.2)}{
\begin{tikzpicture}[baseline=(current bounding box.center)]
	\tikzmath{\x=0.9;\y=1.2;}

	\node (1) at (0,0) {$a_{m\pp+1}$};
	\node (2) at (\y,-1*\x) {$a_{m\pp+2}$};
	\node (3) at (0.5*\y,-2*\x) {$\vup$};
	\node (4) at (1.5*\y,-3*\x) {$a_{m-1}$};
	\node (5) at (1*\y,-4*\x) {$a_m$};
	\node (6) at (0,-4.5*\x) {$d_{\varrho+1}$};

	\node (10) at (2.*\y,0.5*\x) {$e_{r-1}\pp$};
	\node (11) at (2.25*\y,-0.5*\x) {$e_{r}\pp$};
	\node (12) at (2.75*\y,-4.5*\x) {$e_{r+1}\pp$};

	\draw [->] (3) to (1);
	\draw [->] (4) to (2);
	\draw [->] (5) to (3);
	\draw [->] (12) to (11);
	\draw [->] (11) to (10);
	\draw [->, out=120,in=0] (12) to (3);
	\draw [->] (1) to (10);
	\draw [->,out=90,in=180] (6) to (4);
%

	\draw [->,dashed] (2) to (1);	
	\draw [->,dashed] (3) to (2);
	\draw [->,dashed] (4) to (3);
	\draw [->,dashed] (5) to (4);
	\draw [->,dashed] (6) to (5);
	\draw [->,dashed] (11) to (1);	
	\draw [->,dashed] (2) to (11);	
	\draw [->,dashed] (12) to (4);
	\draw [->,dashed] (12) to (5);	
\end{tikzpicture} 
}\quad\subcaptionbox*{(\ref{lem:ladtail4}.3)}{
\begin{tikzpicture}[baseline=(current bounding box.center)]
	\tikzmath{\x=0.9;\y=1.2;}

	\node (1) at (0,0) {$a_{m\pp+1}$};
	\node (2) at (\y,-1*\x) {$a_{m\pp+2}$};
	\node (3) at (0.5*\y,-2*\x) {$\vup$};
	\node (4) at (1.5*\y,-3*\x) {$a_{m-1}$};
	\node (5) at (1*\y,-4*\x) {$a_m$};
	\node (6) at (0,-4.5*\x) {$d_{\varrho+1}$};

	\node (10) at (2.*\y,0.5*\x) {$e_{r-1}\pp$};
	\node (11) at (2.25*\y,-0.5*\x) {$e_{r}\pp$};
	\node (12) at (2.75*\y,-4.5*\x) {$e_{r+1}\pp$};

	\draw [->] (3) to (1);
	\draw [->] (4) to (2);
	\draw [->] (5) to (3);
	\draw [->] (12) to (11);
	\draw [->] (11) to (10);
	\draw [->, out=120,in=0] (12) to (3);
	\draw [->] (1) to (10);
	\draw [->,out=90,in=180] (6) to (4);
%

	\draw [->,dashed] (2) to (1);	
	\draw [->,dashed] (3) to (2);
	\draw [->,dashed] (4) to (3);
	\draw [->,dashed] (5) to (4);
	\draw [->,dashed] (6) to (5);
	\draw [->,dashed] (11) to (1);	
	\draw [->,dashed] (2) to (11);	
	\draw [->] (12) to (4);
	\draw [->,dashed] (12) to (5);
\end{tikzpicture} 
}
\caption{$a_m$ in \csii[(a)],  $b_m$ in \csi[(b)], $m\geq m\pp+3$}\label{fig:IIii2}
\end{figure}

\noindent {\bfseries II.iii.} $b_{m\pp+1}$ is in \csii[(a)], $u(r+s-1)=m\pp$: when $m=m\pp+1$, then $PT_2=PT_2\pp$, and thus there is nothing to check. Thus suppose that $m\geq m\pp+2$. Note that $e_{r+s}\pp \pdl a_{m\pp+1}=b_{m\pp+1}$ since $b_{m\pp+1}$ is in \csii[(a)] and $u(r+s-1)+1=m\pp$, and $\varrho\geq r+s$ since $d_{\varrho+1}\pp \pr a_{m\pp+1}>e_{r+s}\pp$. First assume that $e_{r+s}\pp \pl a_{m\pp+2}$. Then $\cL_{ad}\cup\{e_{r+s}\pp\}-\{d_\varrho\}$ is a ladder and $e_{r+s}\pp \pl e_{r+s+1}\pp$, and thus by Lemma \ref{lem:ladend} we have $e_{r+s}\pp \pr a_{j}$ for $j\in[m\pp+1,m-2]$ and we may apply Lemma \ref{lem:ladtail4} to $\{a_{m-1}, a_m\}$ and $e_{r+s+1}\pp$. (\ref{lem:ladtail4}.4) is impossible since $b_{m\pp+1}$ is in \csii[(a)], i.e. there exists $j \in [m\pp+1, m]$ such that $a_j=b_{j} \pdl e_{r+s+1}\pp$. In other cases, direct calculation shows that $e_k=e_k\pp$ if $k \neq r+s+1$ and
\begin{enumerate}[label=$\circ$]
\item (\ref{lem:ladtail4}.1): either $e_{r+s+1}=e_{r+s+1}\pp$ or $e_{r+s+1}=a_m>e_{r+s+1}\pp$.
\item (\ref{lem:ladtail4}.2): $e_{r+s+1}=a_m$.
\item (\ref{lem:ladtail4}.3): $e_{r+s+1}=e_{r+s+1}\pp$.
\end{enumerate}
(See Figure \ref{fig:IIiii}.) In any case we have $q=q\pp$. Thus we only need to check that $d_{r+s+1} \not\pr e_{r+s+1}=a_m$ in case (\ref{lem:ladtail4}.2). But this holds by \cond{} since $d_{r+s+1}\leq d_{\varrho+1}$ and $d_{\varrho+1} \pdr a_m=e_{r+s+1}$. 

\begin{figure}[!htbp]
\subcaptionbox*{(\ref{lem:ladtail4}.1)}{
\begin{tikzpicture}[baseline=(current bounding box.center)]
	\tikzmath{\x=0.9;\y=1.2;}

	\node (1) at (0,0) {$a_{m\pp+1}$};
	\node (2) at (\y,-1*\x) {$a_{m\pp+2}$};
	\node (3) at (0.5*\y,-2*\x) {$\vup$};
	\node (4) at (1.5*\y,-3*\x) {$a_{m-1}$};
	\node (5) at (1*\y,-4*\x) {$a_m$};
	\node (6) at (0,-4.5*\x) {$d_{\varrho+1}$};

	\node (10) at (2.*\y,0.5*\x) {$e_{r+s}\pp$};
	\node (12) at (2.75*\y,-3.5*\x) {$e_{r+s+1}\pp$};

	\draw [->] (3) to (1);
	\draw [->] (2) to (10);
	\draw [->] (4) to (2);
	\draw [->] (5) to (3);
	\draw [->] (12) to (10);
	\draw [->, out=120,in=0] (12) to (3);
	\draw [->,dashed] (1) to (10);
	\draw [->,out=90,in=180] (6) to (4);

	\draw [->,dashed] (2) to (1);	
	\draw [->,dashed] (3) to (2);
	\draw [->,dashed] (4) to (3);
	\draw [->,dashed] (5) to (4);
	\draw [->,dashed] (6) to (5);
	\draw [->,dashed] (12) to (4);
	\draw [->,dashed] (5) to (12);	
\end{tikzpicture} 
}\quad\subcaptionbox*{(\ref{lem:ladtail4}.2)}{
\begin{tikzpicture}[baseline=(current bounding box.center)]
	\tikzmath{\x=0.9;\y=1.2;}

	\node (1) at (0,0) {$a_{m\pp+1}$};
	\node (2) at (\y,-1*\x) {$a_{m\pp+2}$};
	\node (3) at (0.5*\y,-2*\x) {$\vup$};
	\node (4) at (1.5*\y,-3*\x) {$a_{m-1}$};
	\node (5) at (1*\y,-4*\x) {$a_m$};
	\node (6) at (0,-4.5*\x) {$d_{\varrho+1}$};

	\node (10) at (2.*\y,0.5*\x) {$e_{r+s}\pp$};
	\node (12) at (2.75*\y,-4.5*\x) {$e_{r+s+1}\pp$};

	\draw [->] (3) to (1);
	\draw [->] (2) to (10);
	\draw [->] (4) to (2);
	\draw [->] (5) to (3);
	\draw [->] (12) to (10);
	\draw [->, out=120,in=0] (12) to (3);
	\draw [->,dashed] (1) to (10);
	\draw [->,out=90,in=180] (6) to (4);
%

	\draw [->,dashed] (2) to (1);	
	\draw [->,dashed] (3) to (2);
	\draw [->,dashed] (4) to (3);
	\draw [->,dashed] (5) to (4);
	\draw [->,dashed] (6) to (5);
	\draw [->,dashed] (12) to (4);
	\draw [->,dashed] (12) to (5);	
\end{tikzpicture} 
}\quad\subcaptionbox*{(\ref{lem:ladtail4}.3)}{
\begin{tikzpicture}[baseline=(current bounding box.center)]
	\tikzmath{\x=0.9;\y=1.2;}

	\node (1) at (0,0) {$a_{m\pp+1}$};
	\node (2) at (\y,-1*\x) {$a_{m\pp+2}$};
	\node (3) at (0.5*\y,-2*\x) {$\vup$};
	\node (4) at (1.5*\y,-3*\x) {$a_{m-1}$};
	\node (5) at (1*\y,-4*\x) {$a_m$};
	\node (6) at (0,-4.5*\x) {$d_{\varrho+1}$};

	\node (10) at (2.*\y,0.5*\x) {$e_{r+s}\pp$};

	\node (12) at (2.75*\y,-4.5*\x) {$e_{r+s+1}\pp$};

	\draw [->] (3) to (1);
	\draw [->] (2) to (10);
	\draw [->] (4) to (2);
	\draw [->] (5) to (3);
	\draw [->] (12) to (10);
	\draw [->, out=120,in=0] (12) to (3);
	\draw [->,dashed] (1) to (10);
	\draw [->,out=90,in=180] (6) to (4);
%

	\draw [->,dashed] (2) to (1);	
	\draw [->,dashed] (3) to (2);
	\draw [->,dashed] (4) to (3);
	\draw [->,dashed] (5) to (4);
	\draw [->,dashed] (6) to (5);
	\draw [->] (12) to (4);
	\draw [->,dashed] (12) to (5);
\end{tikzpicture} 
}
\captionsetup{justification=centering}
\caption{ $a_m$ in \csii[(a)],  $b_m$ in \csii[(a)], $u(r+s-1)=m\pp$,\\ $m\geq m\pp+2$, $e_{r+s}\pp \pl a_{m\pp+2}$}\label{fig:IIiii}
\end{figure}

It remains to assume that $e_{r+s}\pp \pdl a_{m\pp+2}$. Since $e_{r+s}\pp \pl e_{r+s+1}\pp$, this means that $a_{m\pp+2}<e_{r+s+1}\pp$ by \cond{}. 
On the other hand,  in this case $b_{m\pp+1}$ is in \csii[(a)] only if $a_{m\pp+1}\pdl e_{r+s+1}\pp$. Thus we may apply Lemma \ref{lem:ladmid} to $\cL_{ad}$ and $e_{r+s+1}\pp$. Here (\ref{lem:ladmid}.3) is the only one satisfying $a_{m\pp+2}<e_{r+s+1}\pp$, in which case $m=m\pp+2$ and $PT_2=PT_2\pp$. (See Figure \ref{fig:IIiii2}.) Thus there is nothing to prove.

\begin{figure}[!htbp]
\begin{tikzpicture}[baseline=(current bounding box.center)]
	\tikzmath{\x=0.9;\y=1.2;}

	\node (1) at (0,0) {$d_\varrho$};
	\node (2) at (\y,-.5*\x) {$a_{m-1}$};
	\node (3) at (\y,-2*\x) {$a_m$};
	\node (4) at (0,-3*\x) {$d_{\varrho+1}$};

	\node (5) at (2.5*\y,0*\x) {$e_{r+s}\pp$};

	\node (6) at (2.5*\y,-2.5*\x) {$e_{r+s+1}\pp$};

	\draw [->] (3) to (1);
	\draw [->] (4) to (2);
	\draw [->] (6) to (5);

	\draw [->,dashed] (2) to (5);
	\draw [->,dashed] (3) to (5);
	\draw [->,dashed] (2) to (1);	
	\draw [->,dashed] (3) to (2);
	\draw [->,dashed] (4) to (3);
	\draw [->,dashed] (6) to (2);
	\draw [->,dashed] (6) to (3);
	\draw [->,dashed] (4) to (6);
\end{tikzpicture} 
\captionsetup{justification=centering}
\caption{$a_m$ in \csii[(a)],  $b_m$ in \csii[(a)], \\$u(r+s-1)=m\pp$, $m\geq m\pp+2$, $e_{r+s}\pp \pdl a_{m\pp+2}$}\label{fig:IIiii2}
\end{figure}

\noindent {\bfseries II.iv.} $b_{m\pp+1}$ is in \csii[(a)], $u(r+s-1) < m\pp$:  first assume that $m=m\pp+1$. Then $q=q\pp$, $e_{r+s}\pp = b_{m-1}$, $e\pp_{r+j} = b_{u(r+j)}$ for $j \in [1,s-1]$, and $e_j=e_j\pp$ if $j \not\in[r+1, r+s]$. Thus it suffices to check $d_k \not \pr e_k$ for $k\in [r+1, r+s]$. Since $d_{\varrho+1}\pp=d_{\varrho+1}> a_m=b_m>b_{m-1}=e_{r+s}\pp$, we have $r+s\leq \varrho+1$. Also since $d_{\varrho+1} \pdr a_m=b_m$, by Lemma \ref{lem:extlad} either $\cL_{be}\cup\{d_{\varrho+1}\}$ or $\cL_{be}\cup\{d_{\varrho+1}\}-\{b_m\}$ is a ladder. Now if $d_k \pr e_k$ for some $k \in [r+1,r+s]$, $k<\varrho+1$ then as $d_k \pl d_{\varrho+1}$ it implies that $d_k$ is climbing a ladder, which is impossible. It remains to check the case when $k=r+s=\varrho+1$. We claim that this is impossible. First, $\varrho+2 = r+s+1 \geq q \geq p$, and thus $d_{\varrho+2}$ is well-defined. Furthermore, $d_{\varrho+2} = d_{\varrho+2}\pp \not \pr e_{r+s+1}\pp=e_{r+s+1}$, and thus $d_{\varrho+1}<e_{r+s+1}$ by \cond{} since $d_{\varrho+2}\pr d_{\varrho+1}$. On the other hand, if $d_{\varrho+2}<e_{r+s+1}$ then $a_m <d_{\varrho+1}\pl d_{\varrho+2}<e_{r+s+1}$ which implies $b_m = a_m \pl e_{r+s+1}$, but it contradicts the assumption that $b_{m}$ is in \csii[(a)]. Therefore, we have $d_{\varrho+1}<e_{r+s+1}<d_{\varrho+2}$ and also $d_{\varrho+2} \pdr e_{r+s+1}$. Now one may check that $\cP$ restricted to $\{d_\varrho, a_m, d_{\varrho+1},e_{r+s+1}, d_{\varrho+2}\}$ is isomorphic to $\cP_{(3,1,1),5}$, which is a contradiction. (See Figure \ref{fig:IIiv}.)

\begin{figure}[!htbp]
\begin{tikzpicture}[baseline=(current bounding box.center)]
	\tikzmath{\x=0.9;\y=1.2;}

	\node (1) at (0,0) {$d_\varrho$};
	\node (2) at (0,-2*\x) {$d_{\varrho+1}$};
	\node (3) at (0,-4*\x) {$d_{\varrho+2}$};

	\node (4) at (2*\y,-1*\x) {$a_m$};
	\node (5) at (2.5*\y,-3*\x) {$e_{r+s+1}$};
	
	\draw [->] (2) to (1);
	\draw [->] (3) to (2);
	\draw [->] (5) to (1);
	\draw [->] (3) to (4);

	\draw [->,dashed] (4) to (1);	
	\draw [->,dashed] (2) to (4);
	\draw [->,dashed] (5) to (2);
	\draw [->,dashed] (5) to (4);
	\draw [->,dashed] (3) to (5);
\end{tikzpicture} 
\caption{$a_m$ in \csii[(a)],  $b_m$ in \csii[(a)], $m=m\pp+1$, $\varrho+1=r+s$}\label{fig:IIiv}
\end{figure}

Now assume that $m\geq m\pp+2$. First $d_{\varrho+1}\pp=d_{\varrho+1}\pr a_{m\pp+1}>b_{m\pp}=e_{r+s}\pp$, and thus $\varrho \geq r+s$. As $\cL_{be}$ is a ladder of size $\geq 3$ by assumption, $\cL_{be} \cup(\cL_{ad}-\{d_\varrho\})$ is also a ladder by Lemma \ref{lem:ladjoin}. Now by Lemma \ref{lem:ladmid} we have $e_{r+s+1}\pp\pr x$ for any $x\in \cL_{be} \cup(\cL_{ad}-\{d_\varrho\})-\{a_{m-1}, a_m, d_{\varrho+1}\}$, and we may apply \ref{lem:ladtail4} to $\{a_{m-1}, a_m\}$ and $e_{r+s+1}\pp$. Keep in mind that $b_{m\pp+1}$ is in \csii[(a)] which implies that (\ref{lem:ladtail4}.4) is impossible. Direct calculation shows that $e\pp_{r+j} = b_{u(r+j)}$ for $j \in [1,s-1]$, $e\pp_{r+s} = b_{m\pp}$,  $e_k=e_k\pp$ if $k\not\in[r+1, r+s+1\}$ and
\begin{enumerate}[label=$\circ$]
\item (\ref{lem:ladtail4}.1): either $e_{r+s+1}=e_{r+s+1}\pp$ or $e_{r+s+1}=a_m>e_{r+s+1}\pp$.
\item (\ref{lem:ladtail4}.2): $e_{r+s+1}=a_m$.
\item (\ref{lem:ladtail4}.3): $e_{r+s+1}=e_{r+s+1}\pp$.
\end{enumerate}
(See Figure \ref{fig:IIiv2}.) In any case we have $q=q\pp$. 
Thus it suffices to show that $d_k \not \pr e_k$ if $k \in [r+1,r+s]$ and $d_{r+s+1} \not \pr e_{r+s+1}$ in case (\ref{lem:ladtail4}.2). The former is clear. Indeed, if $k \in [r+1,r+s]$ then $k<\varrho+1$, and thus $d_k \pl d_{\varrho+1}$. Therefore, if $d_k \pr e_k$ then $d_k$ is climbing $\cL_{be} \cup (\cL_{ad}-\{d_{\varrho}\})$ which is a contradiction. Now if we are in case (\ref{lem:ladtail4}.2) then $d_{\varrho+1}\geq d_{r+s+1}$ and $d_{\varrho+1} \pdr a_m = e_{r+s+1}$, and thus $d_{r+s+1} \not \pr e_{r+s+1}$.

\begin{figure}[!htbp]
\subcaptionbox*{(\ref{lem:ladtail4}.1)}{
\begin{tikzpicture}[baseline=(current bounding box.center)]
	\tikzmath{\x=0.9;\y=1.1;\z=0.15;}
	\node (1) at (0,0) {$b_{u(r+s-1)+1}$};
	\node (2) at (\y-1*\z,-1*\x) {$b_{u(r+s-1)+2}$};
	\node (3) at (0-2*\z,-2*\x) {$\vup$};
	\node (4) at (\y-3*\z,-3*\x) {$b_{m\pp}$};
	\node (5) at (0-4*\z,-4*\x) {$a_{m\pp+1}$};
	\node (6) at (\y-5*\z,-5*\x) {$a_{m\pp+2}$};
	\node (7) at (0-6*\z,-6*\x) {$\vup$};
	\node (8) at (\y-7*\z,-7*\x) {$a_{m-1}$};
	\node (9) at (0-8*\z,-8*\x) {$a_m$};

	\node (10) at (2*\y,1*\x) {$e_{r+s}$};
	\node (11) at (2*\y,-7.5*\x) {$e_{r+s+1}\pp$};

	\draw [->] (3) to (1);
	\draw [->] (4) to (2);
	\draw [->] (5) to (3);
	\draw [->] (6) to (4);
	\draw [->] (7) to (5);
	\draw [->] (8) to (6);
	\draw [->] (9) to (7);

	\draw [->] (2) to (10);
	\draw [->] (11) to (10);

	\draw [->, out=120,in=0] (11) to (7);

	\draw [->,dashed] (2) to (1);	
	\draw [->,dashed] (3) to (2);
	\draw [->,dashed] (4) to (3);
	\draw [->,dashed] (5) to (4);
	\draw [->,dashed] (6) to (5);	
	\draw [->,dashed] (7) to (6);
	\draw [->,dashed] (8) to (7);
	\draw [->,dashed] (9) to (8);

	\draw [->,dashed] (1) to (10);
	\draw [->,dashed] (11) to (8);
	\draw [->,dashed] (9) to (11);
\end{tikzpicture} 
}\subcaptionbox*{(\ref{lem:ladtail4}.2)}{
\begin{tikzpicture}[baseline=(current bounding box.center)]
	\tikzmath{\x=0.9;\y=1.1;\z=0.15;}
	\node (1) at (0,0) {$b_{u(r+s-1)+1}$};
	\node (2) at (\y-1*\z,-1*\x) {$b_{u(r+s-1)+2}$};
	\node (3) at (0-2*\z,-2*\x) {$\vup$};
	\node (4) at (\y-3*\z,-3*\x) {$b_{m\pp}$};
	\node (5) at (0-4*\z,-4*\x) {$a_{m\pp+1}$};
	\node (6) at (\y-5*\z,-5*\x) {$a_{m\pp+2}$};
	\node (7) at (0-6*\z,-6*\x) {$\vup$};
	\node (8) at (\y-7*\z,-7*\x) {$a_{m-1}$};
	\node (9) at (0-8*\z,-8*\x) {$a_m$};

	\node (10) at (2.*\y,1*\x) {$e_{r+s}$};
	\node (11) at (2*\y,-8.5*\x) {$e_{r+s+1}\pp$};

	\draw [->] (3) to (1);
	\draw [->] (4) to (2);
	\draw [->] (5) to (3);
	\draw [->] (6) to (4);
	\draw [->] (7) to (5);
	\draw [->] (8) to (6);
	\draw [->] (9) to (7);

	\draw [->] (2) to (10);
	\draw [->] (11) to (10);
	\draw [->, out=120,in=0] (11) to (7);


	\draw [->,dashed] (2) to (1);	
	\draw [->,dashed] (3) to (2);
	\draw [->,dashed] (4) to (3);
	\draw [->,dashed] (5) to (4);
	\draw [->,dashed] (6) to (5);	
	\draw [->,dashed] (7) to (6);
	\draw [->,dashed] (8) to (7);
	\draw [->,dashed] (9) to (8);

	\draw [->,dashed] (1) to (10);
	\draw [->,dashed] (11) to (8);
	\draw [->,dashed] (11) to (9);
\end{tikzpicture} 
}\subcaptionbox*{(\ref{lem:ladtail4}.3)}{
\begin{tikzpicture}[baseline=(current bounding box.center)]
	\tikzmath{\x=0.9;\y=1.1;\z=0.15;}
	\node (1) at (0,0) {$b_{u(r+s-1)+1}$};
	\node (2) at (\y-1*\z,-1*\x) {$b_{u(r+s-1)+2}$};
	\node (3) at (0-2*\z,-2*\x) {$\vup$};
	\node (4) at (\y-3*\z,-3*\x) {$b_{m\pp}$};
	\node (5) at (0-4*\z,-4*\x) {$a_{m\pp+1}$};
	\node (6) at (\y-5*\z,-5*\x) {$a_{m\pp+2}$};
	\node (7) at (0-6*\z,-6*\x) {$\vup$};
	\node (8) at (\y-7*\z,-7*\x) {$a_{m-1}$};
	\node (9) at (0-8*\z,-8*\x) {$a_m$};

	\node (10) at (2.*\y,1*\x) {$e_{r+s}$};
	\node (11) at (2*\y,-8.5*\x) {$e_{r+s+1}\pp$};

	\draw [->] (3) to (1);
	\draw [->] (4) to (2);
	\draw [->] (5) to (3);
	\draw [->] (6) to (4);
	\draw [->] (7) to (5);
	\draw [->] (8) to (6);
	\draw [->] (9) to (7);

	\draw [->] (2) to (10);
	\draw [->] (11) to (10);
	\draw [->, out=120,in=0] (11) to (7);	


	\draw [->,dashed] (2) to (1);	
	\draw [->,dashed] (3) to (2);
	\draw [->,dashed] (4) to (3);
	\draw [->,dashed] (5) to (4);
	\draw [->,dashed] (6) to (5);	
	\draw [->,dashed] (7) to (6);
	\draw [->,dashed] (8) to (7);
	\draw [->,dashed] (9) to (8);

	\draw [->,dashed] (1) to (10);
	\draw [->] (11) to (8);
	\draw [->,dashed] (11) to (9);
\end{tikzpicture} 
}
\caption{$a_m$ in \csii[(a)],  $b_m$ in \csii[(a)], $u(r+s-1) < m\pp$, $m\geq m\pp+2$}\label{fig:IIiv2}
\end{figure}

\noindent {\bfseries II.v.} $b_{m\pp+1}$ is in \csii[(b)], $u(r+s-1)< m\pp$: note that $\cL_{be} \owns a_{m\pp+1}, b_{m\pp}, e_{r+s}\ppp$, and thus $\cL_{be}$ is of length $\geq 3$. If $m>m\pp+1$ then $\cL_{be} \cup (\cL_{ad}-\{d_{\varrho}\})$ is a ladder by Lemma \ref{lem:ladjoin}, and thus so is $\cL_{be} \cup (\cL_{ad}-\{d_{\varrho},d_{\varrho+1}\})$. If $m=m\pp+1$ then $\cL_{be} \cup (\cL_{ad}-\{d_{\varrho},d_{\varrho+1}\})=\cL_{be}$ is clearly a ladder. Then direct calculation shows that $q=q\pp$, $e_{r+s}=b_m>b_{m-1}=e_{r+s}\pp$, and $e_i = e_i\pp$ otherwise, and thus there is nothing to check here.

\noindent {\bfseries II.vi.} $b_{m\pp+1}$ is in \csii[(b)], $u(r+s-1)= m\pp$: first assume that $\cL_{be}\cup(\cL_{ad}-\{d_\varrho, d_{\varrho+1}\}) =\cL_{be}\sqcup \{a_{m\pp+2}, \ldots, a_{m}\}$ is a ladder. Then direct calculation shows that $e_{r+i}=b_{u(r+i)}>e_{r+i}\pp$ for $i \in [1,s-1]$, $e_{r+s} = b_m > e_{r+s}\pp$, and $e_i=e_i\pp$ if $i \not\in[r+1, r+s]$, and thus there is nothing to check. Thus suppose otherwise, i.e. $\cL_{be} \cup \{a_{m\pp+2}, \ldots, a_{m}\}$ is not a ladder and in particular $m\geq m\pp+2$. Since $\{a_{m\pp+1}, \ldots, a_m, d_{\varrho+1}\}$ is a ladder of length $\geq 3$, by Lemma \ref{lem:ladjoin} the length of $\cL_{be}$ should be at most 2, i.e. $\cL_{be} = \{e_{r+1}\pp, a_{m\pp+1}\}$ and also $e_{r+1}\pp \pdl a_{m\pp+2}$. Also we have either $q\pp=r+1$ or $e_{r+2}\pp \pr a_{m\pp+1}=b_{m\pp+1}$ since $b_{m\pp+1}$ is in \csii[(b)]. If $q\pp>r+1$ then we may apply Lemma \ref{lem:ladtail4} to $\{a_{m-1}, a_m\}$ and $e_{r+2}\pp$. Again since $b_{m\pp+1}$ is in \csii[(b)], the only possible cases is (\ref{lem:ladtail4}.4), i.e. $a_m \pl e_{r+2}\pp$. Now direct calculation shows that $q=q\pp$, $e_{r+1}=a_m>e_{r+1}\pp$ and $e_i=e_i\pp$ otherwise. Thus there is nothing to check. (See Figure \ref{fig:IIvi}.)

\begin{figure}[!htbp]
\begin{tikzpicture}[baseline=(current bounding box.center)]
	\tikzmath{\x=0.9;\y=1.2;}

	\node (1) at (0,0) {$e_{r+1}\pp$};
	\node (2) at (-2*\y,-0.5*\x) {$a_{m\pp+1}$};
	\node (3) at (-1*\y,-1.5*\x) {$a_{m\pp+2}$};
	\node (4) at (-2*\y,-2.5*\x) {$\vup$};
	\node (5) at (-1*\y,-3.5*\x) {$a_m$};
	\node (6) at (0,-4*\x) {$e_{r+2}\pp$};

	\draw [->] (4) to (2);
	\draw [->] (5) to (3);
	\draw [->] (6) to (5);
	\draw [->] (6) to (1);

	\draw [->, dashed] (3) to (1);
	\draw [->,dashed] (2) to (1);	
	\draw [->,dashed] (3) to (2);
	\draw [->,dashed] (4) to (3);
	\draw [->,dashed] (5) to (4);

\end{tikzpicture} 
\captionsetup{justification=centering}
\caption{$a_m$ in \csii[(a)],  $b_m$ in \csii[(b)], $u(r+s-1)= m\pp$, \\$\cL_{be}\sqcup \{a_{m\pp+2}, \ldots, a_{m}\}$ is not a ladder}\label{fig:IIvi}
\end{figure}

\subsubsection*{\textup{\bfseries III.} \und{$a_m$ is in \textup{\csii[(b)]}}}
Recall that $b_i=a_i$ for $i \in [m\pp+1, m]$. Here, we have $d_i=d_i\pp$ if $i\in[1,\sigma]\cup[\varrho+1,p]$. Thus here we only need to check that $p\geq q$ when $q\neq q\pp$ and $d_k \not\pr e_k$ if $k\leq q$ and either $k> q\pp$, $e_k < e_k\pp$, or $k \in [\sigma+1, \varrho]$.

\noindent {\bfseries III.i.} $b_{m\pp+1}$ is in \csi[(a)]: direct calculation shows that $q=q\pp+m-m\pp$ and $PT_2=(d_{\varrho}\pp, \ldots, d_{\sigma+1}\pp)+PT_2\pp$. Note that $\sigma\geq q\pp$ since $d_{\sigma+1}\pp=b_{m\pp+1} \pr e_{q\pp}$, and thus in particular $q\leq \sigma+m-m\pp=\varrho\leq p$. It remains to check $d_k \not\pr e_k$ for $k\in [q\pp+1,q]$. However, since $e_{q\pp+i}=b_{m\pp+i}=d\pp_{\sigma+i}\pdl a_{m\pp+i}= d_{\sigma+i}$ and $d_{\sigma+i}\geq d_{q\pp+i}$ for $i \in [1, q-q\pp]$, we have $d_{q\pp+i} \not \pr e_{q\pp+i}$ by \cond{} as desired.

\noindent {\bfseries III.ii.} $b_{m\pp+1}$ is in \csi[(b)]: we claim that
\begin{enumerate}[label=--]
\item $q=\max\{q\pp, r+m-m\pp-1\}$,
\item $e_{r+i}=b_{m\pp+i+1}=d\pp_{\sigma+i+1}$ for $i \in [0, m-m\pp-1]$ and $e_{k}=e_k\pp$ otherwise, and
\item $b_{m\pp+i+1}$ are in \csi (either (a) or (b)) if $i<m-m\pp-1$.
\end{enumerate}
(See Figure \ref{fig:IIIii}.) Note that the first part follows from the other parts. To this end, we use induction on $i$. If $i=0$ then it is obvious by assumption. Now assume that the result is true up to $i-1$. First suppose that $i<m-m\pp-1$. If $b_{m\pp+i+1} >e_{r+i}\pp$, then we have $b_m \pr b_{m\pp+i+1} >e_{r+i}\pp\pr e_{r+i-1}=b_{m\pp+i}$ which means that $e_{r+i}\pp$ is climbing $\cL_{ad}$, a contradiction. Thus $b_{m\pp+i+1} <e_{r+i}\pp$, and direct calculation shows that $b_{m\pp+i+1}$ is in \csi{} and $e_{r+i}=b_{m\pp+i+1}$.

Finally, we assume $i=m-m\pp-1>0$. If $q\pp<r+m-m\pp-1$ or $b_m<e_{r+m-m\pp-1}\pp$ then $b_m$ is in \csi{} and the result is obvious, and thus suppose otherwise. If $e_{r+m-m\pp-1}\pp \pl b_m$, then $b_m \pr e_{r+m-m\pp-1}\pp\pr b_{m-1}$, $a_{m-1} \pdl b_m$, and $a_{m-1} \pdr b_{m-1}$, which contradicts Lemma \ref{lem:3122}(1). Thus we have $e_{r+m-m\pp-1}\pp \pdl b_m$. If $q\pp<r+m-m\pp$ or $b_m \pl e_{r+m-m\pp}\pp$, then $b_m$ is in \csii[(b)] and $e_{r+m-m\pp-1}=b_m$, and thus we are done. It remains to consider the case when $q\pp \geq r+m-m\pp$ and $e_{r+m-m\pp}\pp \pdr b_m$, so that $b_m$ is in \csii[(a)]. But then $a_{m-1} \pl e_{r+m-m\pp}\pp$ by Lemma \ref{lem:3122}(1) applied to $b_{m-1}, e_{r+m-m\pp-1}\pp, e_{r+m-m\pp}\pp$, and $a_{m-1}$. This means that $\{b_{m-1}, a_{m-1}, b_m, e_{r+m-m\pp}\pp\}$ is a ladder in $\cP$ which $e_{r+m-m\pp-1}\pp$ is climbing, which is a contradiction. (See Figure \ref{fig:IIIii2}.)

\begin{figure}[!htbp]
\begin{tikzpicture}[baseline=(current bounding box.center)]
	\tikzmath{\x=1.2;\y=1;}
	\node at (\y, 0) {$\vup$};
	\node (1) at (\y, -.5*\x) {$e_{r}$};
	\node (2) at (\y, -1.5*\x) {$e_{r+1}\pp$};
	\node (3) at (\y, -4.5*\x) {$e_{r+2}\pp$};
	\node at (\y, -5*\x) {$\vup$};

	\node (6) at (-.5*\y, -1*\x) {$b_{m\pp+1}$};
	\node (7) at (-.5*\y, -2*\x) {$b_{m\pp+2}$};
	\node (8) at (-.5*\y, -3*\x) {$\vup$};
	\node (9) at (-.5*\y, -4*\x) {$b_{m-1}$};
	\node (10) at (-.5*\y, -5*\x) {$b_m$};

	\draw [->] (2) to (1);
	\draw [->] (6) to (1);

	\draw [->] (7) to (6);
	\draw [->] (8) to (7);
	\draw [->] (9) to (8);
	\draw [->] (10) to (9);

	\draw [->] (3) to (2);

	\draw[rounded corners] (\y-0.5, 0.5) rectangle (\y+.5, -6.5);
	\node[fill=white] at (\y,0.5) {$PT_2\pp$};
\end{tikzpicture} 
$\textnormal{\Large $\Rightarrow$}$
\begin{tikzpicture}[baseline=(current bounding box.center)]
	\tikzmath{\x=1.2;\y=1;}
	\node at (\y, 0) {$\vup$};
	\node (1) at (\y, -.5*\x) {$e_{r}$};
	\node (2) at (\y, -1.5*\x) {$b_{m\pp+1}$};
	\node (3) at (\y, -4.5*\x) {$e_{r+2}\pp$};
	\node at (\y, -5*\x) {$\vup$};

	\node (7) at (-.5*\y, -2*\x) {$b_{m\pp+2}$};
	\node (8) at (-.5*\y, -3*\x) {$\vup$};
	\node (9) at (-.5*\y, -4*\x) {$b_{m-1}$};
	\node (10) at (-.5*\y, -5*\x) {$b_m$};

	\draw [->] (2) to (1);
	
	\draw [->] (7) to (2);
	\draw [->] (8) to (7);
	\draw [->] (9) to (8);
	\draw [->] (10) to (9);

	\draw [->] (3) to (2);

	\draw[rounded corners] (\y-0.5, 0.5) rectangle (\y+.5, -6.5);
\end{tikzpicture} 
$\textnormal{\Large $\Rightarrow\cdots\Rightarrow$}$
\begin{tikzpicture}[baseline=(current bounding box.center)]
	\tikzmath{\x=1.2;\y=1;}
	\node at (\y, 0) {$\vup$};
	\node (1) at (\y, -.5*\x) {$e_{r}$};
	\node (2) at (\y, -1.5*\x) {$b_{m\pp+1}$};
	\node (3) at (\y, -2.5*\x) {$\vup$};
	\node (4) at (\y, -3.5*\x) {$b_{m-1}$};
	\node (5) at (\y, -4.5*\x) {$e_{r+m-m\pp-1}\pp$};
	\node at (\y, -5*\x) {$\vup$};

	\node (10) at (-.5*\y, -4.5*\x) {$b_m$};

	\draw [->] (2) to (1);
	\draw [->] (3) to (2);
	\draw [->] (4) to (3);
	\draw [->] (5) to (4);
	
	\draw [->] (10) to (4);

	\draw [->] (3) to (2);

	\draw[rounded corners] (\y-1, 0.5) rectangle (\y+1, -6.5);
\end{tikzpicture} 
$\textnormal{\Large $\Rightarrow$}$
\begin{tikzpicture}[baseline=(current bounding box.center)]
	\tikzmath{\x=1.2;\y=1;}
	\node at (\y, 0) {$\vup$};
	\node (1) at (\y, -.5*\x) {$e_{r}$};
	\node (2) at (\y, -1.5*\x) {$b_{m\pp+1}$};
	\node (3) at (\y, -2.5*\x) {$\vup$};
	\node (4) at (\y, -3.5*\x) {$b_{m-1}$};
	\node (5) at (\y, -4.5*\x) {$b_m$};
	\node at (\y, -5*\x) {$\vup$};


	\draw [->] (2) to (1);
	\draw [->] (3) to (2);
	\draw [->] (4) to (3);
	\draw [->] (5) to (4);
	

	\draw [->] (3) to (2);

	\draw[rounded corners] (\y-0.5, 0.5) rectangle (\y+.5, -6.5);
	\node[fill=white] at (\y,0.5) {$PT_2$};
\end{tikzpicture} 
\caption{$a_m$ in \csii[(b)], $b_{m\pp+1}$ in \csi[(b)]}\label{fig:IIIii}
\end{figure}

\begin{figure}[!htbp]
\begin{tikzpicture}[baseline=(current bounding box.center)]
	\tikzmath{\x=0.9;\y=1.4;}

	\node (1) at (0,0) {$b_{m-1}$};
	\node (2) at (0,-2.5*\x) {$b_m$};
	
	\node (3) at (\y,-1.5*\x) {$e_{r+m-m\pp-1}\pp$};
	\node (4) at (\y,-3.5*\x) {$e_{r+m-m\pp}\pp$};
	
	\node (5) at (-1*\y,-1*\x) {$a_{m-1}$};

	\draw [->] (2) to (1);	
	\draw [->] (3) to (1);
	\draw [->] (4) to (3);
	\draw [->,out=180, in=270] (4) to (5);

	\draw [->,dashed] (2) to (3);
	\draw [->,dashed] (4) to (2);
	\draw [->,dashed] (5) to (1);	
	\draw [->,dashed] (2) to (5);	

\end{tikzpicture} 
\captionsetup{justification=centering}
\caption{$a_m$ in \csii[(b)], $b_{m\pp+1}$ in \csi[(b)], \\$e_{r+m-m\pp} \pdr b_m$, $b_m \pdr e_{r+m-m\pp-1}$}\label{fig:IIIii2}
\end{figure}

Note that $\sigma+1 \geq r$ as $d_{\sigma+1}\pp=b_{m\pp+1}\pr e_{r-1}\pp$. Thus $r+m-m\pp-1\leq \sigma+m-m\pp=\varrho\leq p$, which means that $q\leq p$. It remains to check that $d_k \not \pr e_k$ for $k \in [\sigma+1, \varrho] \cup [r, r+m-m\pp-1]$, $k\leq q$. If $k \in [\sigma+1, \varrho]$ then we have $d_{k}=a_{k+m-\varrho}\pdr b_{k+m-\varrho}=e_{r+k+m-\varrho-m\pp-1}=e_{r+k-\sigma-1}$ and $e_{r+k-\sigma-1}\leq e_k$, which means that $d_k \not \pr e_k$ by \cond{}. Similarly, if $k \in [r, r+m-m\pp-1]$ then $d_{k+\sigma+1-r}\pdr e_{k}$ and $d_{k+\sigma+1-r} \geq d_k$, and thus $d_k \not \pr e_k$ by \cond{}.

\noindent {\bfseries III.iii.} $b_{m\pp+1}$ is in \csii[(a)], $u(r+s-1)=m\pp$: note that $\sigma+1 \geq r+s$ since $d_{\sigma+1}\pp=b_{m\pp+1} >e_{r+s}\pp$. First we assume that $m=m\pp+1$, in which case $PT_2=PT_2\pp$. Here it suffices to check that $d_k \not \pr e_k$ for $k =\sigma+1=\varrho$, $k\leq q$.
It is obvious when $\sigma+1=\varrho>r+s$ by \cond{} since $d_{\varrho}=a_m\pdr b_m$ and $b_m<e_{r+s+1}\leq e_\varrho$. Now we assume that $k=\varrho=r+s$. Then $\varrho+1\leq p$ since $\varrho+1=r+s+1\leq q \leq p$, and $d_{\varrho+1} > e_{r+s+ 1}$ by \cond{} since $d_{\varrho+1}=d_{\varrho+1}\pp \pr d_{\varrho}\pp=b_m$ and $b_m\pdl e_{r+s+1}$. As $d_{\varrho}=d_{\varrho+1}\pp  \not \pr e_{r+s+1}\pp=e_{r+s+1}$ by assumption, it means that $d_{\varrho+1}   \pdr e_{r+s+1}$, i.e. $\{e_{r+s}, b_m, e_{r+s+1}, d_{\varrho+1}\}$ is a ladder. Now if $d_\varrho \pr e_{r+s}$ then $d_\varrho$ is climbing $\{e_{r+s}, b_m, e_{r+s+1}, d_{\varrho+1}\}$, which is a contradiction. (See Figure \ref{fig:IIIiii}.)

\begin{figure}[!htbp]
\begin{tikzpicture}[baseline=(current bounding box.center)]
	\tikzmath{\x=1.2;\y=1.4;}

	\node (1) at (0,0) {$b_m$};
	
	\node (2) at (\y,0.5*\x) {$e_{r+s}$};
	\node (3) at (\y,-1*\x) {$e_{r+s+1}$};
	
	\node (4) at (-1*\y,-.5*\x) {$d_{\varrho}$};
	\node (5) at (-1*\y,-1.5*\x) {$d_{\varrho+1}$};

	\draw [->] (5) to (4);	
	\draw [->] (3) to (2);
	\draw [->] (5) to (1);


	\draw [->,dashed] (1) to (2);
	\draw [->,dashed] (3) to (1);
	\draw [->,dashed] (5) to (3);	
	

\end{tikzpicture} 
\captionsetup{justification=centering}
\caption{$a_m$ in \csii[(b)], $b_{m\pp+1}$ in \csii[(a)], \\$u(r+s-1)=m\pp$, $m=m\pp+1$, $\varrho=r+s$ }\label{fig:IIIiii}
\end{figure}

Now we show that the case $m>m\pp+1$ is impossible. For the sake of contradiction we assume this condition.
Here $b_{m\pp+1}$ is in \csii[(a)] only if $e_{r+s+1}\pp \pdr b_{m\pp+1}$, which in turn implies that $e_{r+s+1}\pp <b_{m\pp+2}$ by \cond{}. Since $b_{m\pp+1} \pdr e_{r+s}\pp$, either $\cL_{ad}\cup \{e_{r+s}\pp\}$ or $\cL_{ad}\cup \{e_{r+s}\pp\}- \{b_{m\pp+1}\}$ is a ladder by Lemma \ref{lem:extlad}. If $m>m\pp+2$ or $\cL_{ad} \cup \{e_{r+s}\pp\}$ is a ladder then $b_{m\pp+1}\pl e_{r+s+1}\pp$ by Lemma \ref{lem:ladend} since $e_{r+s}\pp \pl e_{r+s+1}\pp$, but this is a contradiction. Thus $m=m\pp+2$ and $\cL_{ad}\cup \{e_{r+s}\pp\}- \{b_{m-1}\}$ is a ladder, i.e. $e_{r+s}\pp \pdl a_{m-1}$. Now we apply Lemma \ref{lem:ladtail4} to $\{a_{m-1}, b_m\}$ and $e_{r+s+1}\pp$; the only possible case is (\ref{lem:ladtail4}.1) since $e_{r+s+1}\pp <b_{m}$. However, in this case $b_{m-1}$ is in \csii[(b)] (together with $b_m$), which is a contradiction. (See Figure \ref{fig:IIIiii2}.)

\begin{figure}[!htbp]
\begin{tikzpicture}[baseline=(current bounding box.center)]
	\tikzmath{\x=1.2;\y=1.6;}

	\node (1) at (0,0) {$b_{m-1}$};
	\node (2) at (0,-1.5*\x) {$b_{m}$};
	
	\node (3) at (\y,0.5*\x) {$e_{r+s}\pp$};
	\node (4) at (\y,-1*\x) {$e_{r+s+1}\pp$};
	
	\node (5) at (-1*\y,-.5*\x) {$a_{m-1}$};
	\node (6) at (-1*\y,-2*\x) {$a_m$};

	\draw [->] (2) to (1);	
	\draw [->] (4) to (3);
	\draw [->] (6) to (5);


	\draw [->,dashed] (1) to (3);
	\draw [->,dashed] (4) to (1);
	\draw [->,dashed] (4) to (5);
	\draw [->,dashed] (2) to (4);
	\draw [->,dashed] (6) to (2);	
	\draw [->,dashed] (2) to (5);	
	\draw [->,dashed] (5) to (1);	
	\draw [->,dashed,out=60, in=180] (5) to (3);	
	

\end{tikzpicture} 
\captionsetup{justification=centering}
\caption{$a_m$ in \csii[(b)], $b_{m\pp+1}$ in \csii[(a)], \\$u(r+s-1)=m\pp$, $m>m\pp+1$}\label{fig:IIIiii2}
\end{figure}

\noindent {\bfseries III.iv.} $b_{m\pp+1}$ is in \csii[(a)], $u(r+s-1)< m\pp$:  first we assume that $m=m\pp+1$, in which case $q=q\pp$, $e_{r+i}\pp=b_{u(r+i)}$ for $i \in [1,s-1]$, $e_{r+s}\pp=b_{m-1}$, and $e_i=e_i\pp$ otherwise. Thus here it suffices to check that $d_k \not \pr e_k$ for $k\in \{\varrho\}\cup [r+1, r+s]$, $k \leq q$. Note that $\varrho \leq  r+s$ since $d_{\varrho}\pp=b_{m} > b_{m-1}=e_{r+s}\pp$. For $k <\varrho$,  first note that either $\cL_{be}\cup\{d_\varrho\}$ or $\cL_{be}\cup\{d_\varrho\}-\{b_m\}$ is a ladder by Lemma \ref{lem:extlad} since $b_m \pdl a_m=d_\varrho$. Thus if $d_k \pr e_k$ then $d_k$ is climbing a ladder as $d_k \pl d_{\varrho}$, which is a contradiction. It remains to consider the case $k=\varrho$. If $\varrho<r+s$, then it follows from \cond{} since $d_\varrho =a_m \pdr b_m$ and $b_m <e_{r+s+1}\pp=e_{r+s+1}\leq e_\varrho$. If $k=\varrho=r+s$, then first $\varrho+1=r+s+1 \leq q\leq p$, and $d_{\varrho+1} =d_{\varrho+1}\pp \pr d_{\varrho}\pp=b_m$. Since $d_{\varrho+1} \not \pr e_{r+s+1}$ and $e_{r+s+1}\pdr b_m$, by \cond{} it follows that $d_{\varrho+1} \pdr e_{r+s+1}$, i.e. $\cL_{be}\cup \{e_{r+s+1}, d_{\varrho+1}\}$ is a ladder in $\cP$. Now if $d_\varrho  \pr e_{r+s}$ then as $d_\varrho \pl d_{\varrho+1}$ it implies that $d_\varrho$ is climbing a ladder, which is contradiction. (See Figure \ref{fig:IIIiv}.)

\begin{figure}[!htbp]
\begin{tikzpicture}[baseline=(current bounding box.center)]
	\tikzmath{\x=1.2;\y=1.2;}

	\node (2) at (\y, -\x) {$e_{r+s}$};
	\node (8) at (\y, -4.5*\x) {$e_{r+s+1}$};

	\node (3) at (-\y, -1.5*\x) {$b_{u(r+s-1)+1}$};
	\node (4) at (0, -2*\x) {$\vup$};
	\node (5) at (-\y, -2.5*\x) {$\vup$};
	\node (6) at (0, -3*\x) {$b_{m-1}$};
	\node (7) at (-\y, -3.5*\x) {$b_m$};

	\node (1) at (-2*\y, -4*\x) {$d_{\varrho}$};
	\node (9) at (-2*\y, -5*\x) {$d_{\varrho+1}$};

	\draw [->] (8) to (2);
	\draw [->] (4) to (2);
	\draw [->] (5) to (3);
	\draw [->] (6) to (4);
	\draw [->] (7) to (5);
	\draw [->] (8) to (6);
	\draw [->] (9) to (7);
	\draw [->] (9) to (1);
	\draw [->,dashed] (3) to (2);
	\draw [->,dashed] (4) to (3);
	\draw [->,dashed] (5) to (4);
	\draw [->,dashed] (6) to (5);
	\draw [->,dashed] (7) to (6);
	\draw [->,dashed] (8) to (7);
	\draw [->,dashed] (9) to (8);
\end{tikzpicture} 
\captionsetup{justification=centering}
\caption{$a_m$ in \csii[(b)],  $b_m$ in \csii[(a)], $u(r+s-1)< m\pp$, 
\\$m=m\pp+1$, $\varrho=r+s$}\label{fig:IIIiv}
\end{figure}

Now assume that $m>m\pp+1$, and we prove that this is impossible. Indeed,  $b_{m\pp+1}\pdl e_{r+s+1}\ppp =e_{r+s+1}\pp$ since $b_{m\pp+1} \pl b_{m\pp+2}$ and $b_{m\pp+1}$ is in \csii[(a)]. However, since $\cL_{be}\cup\cL_{ad}$ is a ladder by Lemma \ref{lem:ladjoin} and $e_{r+s}=e_{r+s}\ppp \pl e_{r+s+1}\ppp=e_{r+s+1}\pp$, we should have $b_{m\pp+1} \pl e_{r+s+1}\pp$ by Lemma \ref{lem:ladend}. This is a contradiction. (See Figure \ref{fig:IIIiv2}.)

\begin{figure}[!htbp]
\begin{tikzpicture}[baseline=(current bounding box.center)]
	\tikzmath{\x=1.2;\y=1.6;}

	\node (2) at (0,0) {$b_{m\pp}$};
	\node (3) at (0,-1*\x) {$b_{m\pp+1}$};
	\node (4) at (0,-2*\x) {$\vup$};
	\node (5) at (0,-3*\x) {$b_m$};
	
	\node (6) at (\y,.5*\x) {$e_{r+s}$};	
	\node (7) at (\y,-1.5*\x) {$e_{r+s+1}\pp$};
	
	\node (9) at (-\y,-1.5*\x) {$a_{m\pp+1}$};
	\node (10) at (-\y,-2.5*\x) {$\vup$};
	\node (11) at (-\y,-3.5*\x) {$a_m$};

	\draw [->] (3) to (6);
	\draw [->] (9) to (2);

	\draw [->] (7) to (6);
	\draw [->] (4) to (3);
	\draw [->] (5) to (4);
	\draw [->] (10) to (9);
	\draw [->] (11) to (10);


	\draw [->,dashed] (3) to (2);
	\draw [->,dashed] (9) to (3);
	\draw [->,dashed] (4) to (9);
	\draw [->,dashed] (10) to (4);
	\draw [->,dashed] (5) to (10);
	\draw [->,dashed] (11) to (5);


\end{tikzpicture} 
\captionsetup{justification=centering}
\caption{$a_m$ in \csii[(b)], $b_{m\pp+1}$ in \csii[(a)], \\$u(r+s-1)<m\pp$, $m>m\pp+1$}\label{fig:IIIiv2}
\end{figure}

\noindent {\bfseries III.v.} $b_{m\pp+1}$ is in \csii[(b)], $u(r+s-1)< m\pp$: first we assume that $m=m\pp+1$, in which case $q=q\pp$, $e_{r+s}=b_m>b_{m-1}=e_{r+s}\pp$, and $e_i=e_i\pp$ otherwise. Thus here it suffices to check that $d_k \not \pr e_k$ for $k=\varrho$, $k \leq q$. Note that $\varrho \geq r+s$ since $d_{\varrho}\pp=b_{m} > b_{m-1}=e_{r+s}\pp$, which implies $e_{\varrho}\geq e_{r+s}$. As $d_\varrho=a_m \pdr b_m=e_{r+s}$, the result follow from \cond{}.


Now we assume that $m> m\pp+1$. Then $\cL_{be}\cup\cL_{ad}$ is a ladder by Lemma \ref{lem:ladjoin}, and thus if $q\leq r+s+1$ then $b_{m\pp+1}\pl e_{r+s+1}\ppp=e_{r+s+1}\pp$ by Lemma \ref{lem:ladend} as $e_{r+s+1}\pp=e_{r+s+1}\ppp \pr e_{r+s}\ppp$. Now almost the same argument as in \mbox{III.ii.} applies here and one can show that $q=\max\{q\pp, r+s+m-m\pp-1\}$, $e_{r+s+i}=b_{m\pp+i+1}=d_{\sigma+i+1}\pp$ for $i \in [0,m-m\pp-1]$, and $e_i = e_i\pp$ otherwise. First, we have $\sigma+1\geq r+s$ since $d_{\sigma+1}\pp=b_{m\pp+1}>b_{m\pp}=e_{r+s}\pp$. Thus $r+s+m-m\pp-1\leq \sigma+m-m\pp=\varrho\leq p$, which means that $q\leq p$. It remains to verify that $d_k \not \pr e_k$ for $k \in [\sigma+1, \varrho] \cup [r+s, r+s+m-m\pp-1]$, $k \leq q$. If $k \in [\sigma+1, \varrho]$ then we have $d_{k}=a_{k+m-\varrho}\pdr b_{k+m-\varrho}=e_{r+s+k+m-\varrho-m\pp-1}=e_{r+s+k-\sigma-1}$ and $e_{r+s+k-\sigma-1}\leq e_k$, which means that $d_k \not \pr e_k$ by \cond{}. Similarly, if $k \in [r+s, r+s+m-m\pp-1]$ then $d_{k+\sigma+1-r-s}\pdr e_{k}$ and $d_{k+\sigma+1-r-s} \geq d_k$, and thus $d_k \not \pr e_k$ by \cond{}.

%

\noindent {\bfseries III.vi.} $b_{m\pp+1}$ is in \csii[(b)], $u(r+s-1)= m\pp$: note that $\sigma+1 \geq r+s$ as $d_{\sigma+1}\pp=b_{m\pp+1}>e_{r+s}\pp$. First suppose that $r+s+1\leq q\pp$ and $e_{r+s+1}\pp\pdr b_{m\pp+1}$. Since $b_{m\pp+1}$ is in \csii[(b)], this forces that $m >m\pp+1$ and $e_{r+s+1}\pp\pdl b_{m\pp+2}$. On the other hand, by Lemma \ref{lem:extlad} either $\cL_{ad}\cup \{e_{r+s}\pp\}$ or $\cL_{ad}\cup \{e_{r+s}\pp\}-\{b_{m\pp+1}\}$ is a ladder, 
thus we may apply Lemma \ref{lem:ladtail4} to $\{a_{m-1}, b_m\}$ and $e_{r+s+1}\pp$. The only possible case is when $m=m\pp+2$ and we are in (\ref{lem:ladtail4}.1), i.e. $a_{m-1} \pdl e_{r+s+1}\pp$ and $e_{r+s+1}\pp \pdl b_m$.

Now direct calculation shows that $q=q\pp$, $e_{r+i}=b_{u(r+i)}>e_{r+i}\pp$ for $i \in [1,s-1]$, $e_{r+s}=b_{m-1}>e_{r+s}\pp$, $e_{r+s+1}=b_m >e_{r+s+1}\pp$, and $e_i=e_i\pp$ otherwise, and thus it suffices to check $d_k \not \pr e_k$ for $k=\{\sigma+1=\varrho-1,\varrho\}$, $k \leq q$. We have $d_{\varrho-1} \not \pr e_{\varrho-1}$ by \cond{} since $d_{\varrho-1}=a_{m-1} \pdr b_{m-1}=e_{r+s}$ and $e_{\varrho-1}\geq e_{r+s}$ and similarly $d_{\varrho} \not \pr e_{\varrho}$  since $d_{\varrho}=a_{m} \pdr b_{m}=e_{r+s+1}$ and $e_{\varrho}\geq e_{r+s+1}$. (See Figure \ref{fig:IIIvi}.)

\begin{figure}[!htbp]
\begin{tikzpicture}[baseline=(current bounding box.center)]
	\tikzmath{\x=1.2;\y=1.6;}

	\node (1) at (0,0) {$b_{m-1}$};
	\node (2) at (0,-1.5*\x) {$b_{m}$};
	
	\node (3) at (\y,0.5*\x) {$e_{r+s}\pp$};
	\node (4) at (\y,-1*\x) {$e_{r+s+1}\pp$};
	
	\node (5) at (-1*\y,-.5*\x) {$a_{m-1}$};
	\node (6) at (-1*\y,-2*\x) {$a_m$};

	\draw [->] (2) to (1);	
	\draw [->] (4) to (3);
	\draw [->] (6) to (5);


	\draw [->,dashed] (1) to (3);
	\draw [->,dashed] (4) to (1);
	\draw [->,dashed] (4) to (5);
	\draw [->,dashed] (2) to (4);
	\draw [->,dashed] (6) to (2);	
	\draw [->,dashed] (2) to (5);	
	\draw [->,dashed] (5) to (1);	
	

\end{tikzpicture} 
\captionsetup{justification=centering}
\caption{$a_m$ in \csii[(b)], $b_{m\pp+1}$ in \csii[(b)], \\$u(r+s-1)=m\pp$, $e_{r+s+1}\pp\pdr b_{m\pp+1}$}\label{fig:IIIvi}
\end{figure}

This time suppose that either $r+s+1>q\pp$ or $e_{r+s+1}\pp\pr b_{m\pp+1}$. Then almost the same argument as in \mbox{III.ii.} applies here and one can show that $q=\max\{q\pp, r+s+m-m\pp-1\}$, $e_{r+i}=b_{u(r+i)}>e_{r+i}\pp$ for $i \in [1,s-1]$, $e_{r+s}=b_{m\pp+1}>e_{r+s}\pp$, $e_{r+s+i}=b_{m\pp+i+1}=d_{\sigma+i+1}\pp$ for $i \in [1,m-m\pp-1]$, and $e_i = e_i\pp$ otherwise. First, we have $\sigma+1\geq r+s$ since $d_{\sigma+1}\pp=b_{m\pp+1}>b_{m\pp}=e_{r+s}\pp$. Thus $r+s+m-m\pp-1\leq \sigma+m-m\pp=\varrho\leq p$, which means that $q\leq p$.


 It remains to verify that $d_k \not \pr e_k$ for $k \in [\sigma+1, \varrho] \cup [r+s+1, r+s+m-m\pp-1]$, $k \leq q$. If $k \in [\sigma+1, \varrho]$ then we have $d_{k}=a_{k+m-\varrho}\pdr b_{k+m-\varrho}=e_{r+s+k+m-\varrho-m\pp-1}=e_{r+s+k-\sigma-1}$ and $e_{r+s+k-\sigma-1}\leq e_k$, which means that $d_k \not \pr e_k$ by \cond{}. Similarly, if $k \in [r+s+1, r+s+m-m\pp-1]$ then $d_{k+\sigma+1-r-s}\pdr e_{k}$ and $d_{k+\sigma+1-r-s} \geq d_k$, and thus $d_k \not \pr e_k$ by \cond{}.


We exhausted all the possibilities and thus proved Theorem \ref{thm:main}\ref{mainthm1}.

\subsection{Proof of Theorem \ref{thm:main}\ref{mainthm2}}
Let us write $w=(a_n, \ldots, a_1)$. Then $n-i \in \des_\cP(w)$ if and only if $a_i \pl a_{i+1}$. We first show that $a_i \pl a_{i+1}$ if and only if $i \in \des(QT)$. The result is trivial if $i$ and $i+1$ are in the same column of $QT$, and thus suppose otherwise. Then by Proposition \ref{prop:column}, $i+1$ is in the former column than that of $i$. Since $QT$ is a standard Young tableau by part \ref{mainthm1}, it follows that $i$ is in the upper row than $i+1$, and thus $i\in \des(QT)$. Now suppose that $a_i \not\pl a_{i+1}$. By Proposition \ref{prop:column}, $i$ should be in the former column than that of $i+1$. In this case it is easy to see that $i+1$ cannot be in the lower row than that of $i$ because of the standard Young tableau condition, and thus $i\not\in \des(QT)$.

\subsection{Proof of Theorem \ref{thm:main}\ref{mainthm3}}
For $w\in \sym_n$, let $\prs(w)=(PT=(PT_1, \ldots, PT_p), QT)$ and choose $w=w_0, w_1, \ldots, w_p =(\vn, \ldots, \vn) \in \fA$ such that $\Phi(w_i, \emptyset) = (PT_{i+1}, w_{i+1})$ for $i \in [0,p-1]$. Then by Proposition \ref{prop:column}\ref{mainprop1},  we have (recall that $\alpha^f$ is the word obtained from $\alpha$ by removing $\vn$)
$$w\psim PT_1+w_1^f\psim PT_1+PT_2+w_2^f\psim\cdots \psim PT_1+PT_2+\cdots PT_p=\rw(PT)$$
as desired.

\subsection{Proof of Theorem \ref{thm:main}\ref{mainthm3.5}}
By Theorem \ref{thm:main}\ref{mainthm3} if $\prs(w) = (PT, QT)$ then $w \sim_\cP \rw(PT)$. Thus by Proposition \ref{prop:ght} we have $\ght_\cP(w) = \ght_\cP(\rw(PT))$. Furthermore, if $w \sim_\cP w'$ and $\prs(w') = (PT', QT')$ then by the same reason $\ght_\cP(\rw(PT)) \sim_\cP \ght_\cP(w) \sim_\cP \ght_\cP(w') \sim_\cP\ght_\cP(\rw(PT'))$. 

Thus it suffices to show that $\ght_\cP(\rw(T))$ equals the length of the first column of $T$ for any $T \in \ptab$. Let us denote by $l$ the length of the first column of $T$. Since the first column of $T$ is a subword of genuine $\cP$-inversions in $\rw(T)$, it follows that $\ght_\cP(\rw(T))\geq l$. Now for the sake of contradiction suppose $\ght_\cP(\rw(T))> l$. Then by pigeonhole principle there exists $a, b \in [1,n]$ such that $(a,b) \in \ght_\cP(\rw(T))$, the column containing $a$ is on the left of that of $b$, and $b$ is not in the upper row than $a$.

Let $c$ be the element located in the intersection of the row of $a$ and the column of $b$. We claim that $(a,c) \in \ght_\cP(\rw(T))$. Indeed, if $b=c$ then we are done, and thus suppose otherwise. As $b$ and $c$ are in the same column we have $b\pr c$. Since $a\pr b$ by assumption, if $(a,c)\not\in\ght_\cP(\rw(T))$ then by Lemma \ref{lem:ght} there exists a subword $ae_1\cdots e_k c$ in $\rw(T)$ such that $a \pdr e_1 \pdr \cdots \pdr e_k \pdr c$ and $\{a, e_1, \ldots, e_k, c\}$ is a ladder in $\rw(T)$. Note that $b$ cannot be any of $e_i$ since any element between $b$ and $c$ in $\rw(T)$ is bigger than $c$ with respect to $\cP$. However, this means that $b$ is climbing the ladder $\{a, e_1, \ldots, e_k, c\}$ which contradicts the assumption.

Now let $d_1, d_2, \ldots, d_s$ be the elements between $a$ and $c$ in the row of $T$ containing $a$ and $c$. We also set $d_0=a$ and $d_{s+1}=c$ for simplicity. (Note that $d_0d_1\cdots d_s d_{s+1}$ is a subword of $\rw(T)$.) Then by the condition of $\cP$-tableaux, we have  $d_i \not\pr d_{i+1}$ for $i \in [0,s]$. (In particular, we have $s\geq 1$.) We claim that there exists $e_1\in \{d_1, \ldots, d_{s}\}$ such that $d_0 \dash_\cP e_1$. Suppose otherwise, then as $d_0 \pr d_{s+1}$ and $d_0 \not\pr d_1$, there exists $j \in [1,k]$ such that $d_0 \pl d_j$ and $d_0 \pr d_{j+1}$. However, this is impossible since  $d_j \not\pr d_{j+1}$.

Since $d_0 \dash_\cP e_1$ and $d_0 \pr d_{s+1}$, we have $e_1 >d_{s+1}$ by \cond{}. If $e_1 \pdr d_{s+1}$ then it contradicts that $(a,c) = (d_0,d_{s+1}) \in \ght_\cP(\rw(T))$, and thus we should have $e_1 \pr d_{s+1}$. Now we can iterate this process forever and find $e_1, e_2, \ldots  \in \{d_1, \ldots, d_s\}$ such that $e_1 \dash_\cP e_2 \dash_\cP \cdots$. This is clearly a contradiction, and thus we conclude that $\ght_\cP(\rw(T))= l$ which is what we want to prove.

\subsection{Proof of Theorem \ref{thm:main}\ref{mainthm4}}
Let $\prs(w)=(PT', QT')$. We first show that $PT'=PT$. Let $PT=(PT_1, \ldots, PT_k, PT_{k+1}, \ldots, PT_p)$ where the length of $PT_i$ equals 1 if and only if $i>k$. Then direct calculation shows that $\Phi(PT_k+PT_{k+1}+\cdots+PT_p, \emptyset) = ((PT_k), (\vn,\ldots,\vn)+PT_{k+1}+\cdots+PT_p+(\vn))$. Now if we assume $PT_{i}=(\ldots, a)$ and $PT_{i+1} = (b, \ldots, c)$ for $i \in [1,k-1]$, then we have $a<b$ since otherwise $a\pr c$ by \cond{} which contradicts the assumption that $PT$ is a $\cP$-tableau. (Here it is crucially used that the length of $PT_{i+1}$ is $\geq 2$.) In other words, if $x$ and $y$ are processed in the same step then $x, y$ should be in the same column of $PT$. Thus by Proposition \ref{prop:column}\ref{mainprop2}, we have $\Phi(w, \emptyset) = (PT_1, (\vn,\cdots,\vn)+PT_{2}+(\vn, \cdots, \vn)+PT_{3}+\cdots+(\vn, \cdots, \vn)+PT_{k+1}+\cdots+PT_p+(\vn))$, i.e. the first column of $PT'$ is equal to that of $PT$. Now we iterate this argument to conclude that $PT=PT'$ as desired.

It remains to show that $QT'=\omega(T_\lambda)$. By part \ref{mainthm2} of the theorem, if we set $l_1, \ldots, l_p$ to be the column lengths of $\lambda$ then we have $\des(QT') = \{n-x \mid x\in [1,n-1], x\neq \sum_{i=1}^k l_i \textup{ for some } k \in [1, p]\}$. By the property of evacuation, it follows that $\des(\omega(QT'))= [1,n-1]-\{\sum_{i=1}^k l_i \mid k \in [1, p]\}$. Now the result follows from the fact that $T_\lambda$ is the only standard Young tableau of shape $\lambda$ that satisfies this property.

\subsection{Proof of Theorem \ref{thm:main}\ref{mainthm5}} Suppose that $\prs(\alpha) = \prs(\alpha')=(PT, QT)$ where $PT=(PT_1, \ldots, PT_p) \in \ptab_{\lambda}$ and $QT=(QT_1, \ldots, QT_p) \in \SYT_\lambda$.  Also, let $\alpha=\alpha_0, \alpha_1, \ldots, \alpha_p=(\vn, \ldots, \vn)$ and $\alpha'=\alpha_0', \alpha_1', \ldots, \alpha_p'=(\vn, \ldots, \vn)$ be the elements in $\fA$ such that $\Phi(\alpha_i, \emptyset) = (PT_{i+1}, \alpha_{i+1})$ and $\Phi(\alpha_i', \emptyset) = (PT_{i+1}, \alpha_{i+1}')$ for $i \in [0,p-1]$. Clearly $\alpha_p=\alpha'_p$, and thus it suffices to show that $\alpha_{i+1}=\alpha'_{i+1} \Rightarrow \alpha_i=\alpha'_i$ for $i \in [0,p-1]$. However it follows from Lemma \ref{lem:invrel} since $(\emptyset, \hh{\alpha_{i}})= \Psi^\ccP_{X_{i+1}}(\hh{\alpha_{i+1}}, PT_{i+1})$ and $(\emptyset, \hh{\alpha_{i}'})= \Psi^\ccP_{X_{i+1}}(\hh{\alpha_{i+1}'}, PT_{i+1})$ where $X_{i+1} = \{|\alpha|+1-j\mid j \in QT_{i+1}\}$.

\subsection{Proof of Theorem \ref{thm:main}\ref{mainthm6}} By part \ref{mainthm5} of the theorem, it suffices to show that $\prs: \sym_n \rightarrow \bigsqcup_{\lambda \vdash n} \ptab_\lambda \times \SYT_\lambda$ is surjective. 
We have (here $\omega$ is Sch\"utzenberger's evacuation)
\begin{align*}
&\sum_{\lambda \vdash n}\sum_{PT\in \ptab_\lambda}\sum_{QT\in \SYT_\lambda} F_{\des(\omega(QT))}
\\&=\sum_{\lambda \vdash n}\sum_{PT\in \ptab_\lambda}\sum_{QT\in \SYT_\lambda} F_{\des(QT)} &(\because \omega \textup{ is an involution on } \SYT_\lambda)
\\&=\sum_{\lambda \vdash n} |\ptab_\lambda| \cdot s_\lambda&(\because \textnormal{\cite{ges84}})
\\&=\sum_{w \in \sym_n} F_{\des_\cP(w)} &(\because \textnormal{\cite[Theorem 4]{gas96}})
\\&=\sum_{(PT, QT) \in \im \prs} F_{\des(\omega(QT))} &(\because \textnormal{part \ref{mainthm2} and injectivity of $\prs$})
\\&=\sum_{\lambda \vdash n}\sum_{PT \in \ptab_\lambda} \sum_{(PT, QT) \in \im \prs} F_{\des(\omega(QT))}.
\end{align*}
This equality holds only when $\prs$ is surjective, from which the result follows.

\begin{rmk} It is possible to prove surjectivity of $\prs$ without relying on Gasharov's result. Indeed, instead one may argue similarly to the proof of Proposition \ref{prop:phiinj} and use properties of $\cP$-tableaux for $\cP$ which avoids $\cP_{(3,1,1),5}$ and $\cP_{(4,2,1,1),6}$. Or conversely, from the theorem above we obtain the following bi-product.
\end{rmk}
\begin{thm} Suppose that $\cP$ avoids $\cP_{(3,1,1),5}$ and $\cP_{(4,2,1,1),6}$. Let $PT=(PT_1, \ldots,PT_l) \in \ptab_\lambda$ and $QT=(QT_1, \ldots, QT_l)\in \SYT_\lambda$ for some $\lambda\vdash n$. Define $\alpha_0, \alpha_1, \ldots, \alpha_l\in \fA$, $c_0,c_1, \ldots, c_l\in \fC$ successively so that $\alpha_l=(\vn,\ldots,\vn)$, $|\alpha_l|=n$, and $\Psi_{X_i}^{\hh{\cP}}(\hh{\alpha_{i}}, \hh{PT_i}) = (\hh{c_{i-1}}, \hh{\alpha_{i-1}})$ for $i \in [1,l]$ where $X_i=\{n+1-x\mid x\in \und{QT_i}\}$. Then we have $c_0=c_1=\cdots=c_l=\emptyset$, $\prs(\alpha_0) = (PT, QT)$, and $\alpha_0\in \sym_n$.
\end{thm}
\begin{proof} By Theorem \ref{thm:main}\ref{mainthm6}, there exists $w\in \sym_n$ such that $\prs(w) = (PT, QT)$. It means that there exists $w=w_0, w_1, \ldots, w_l=(\vn,\ldots, \vn)$ such that $\Phi(w_{i-1}, \emptyset)=(PT_i, w_i)$ and $\und{QT_i}=\{j \in [1,n] \mid (w_{i-1})_j\neq\vn, (w_i)_j=\vn\}$ for $i \in [1,l]$ where $(w_{i-1})_j$ and $(w_i)_j$ are the $j$-th coordinates of $w_{i-1}$ and $w_i$, respectively. Now the result follows from successively applying Lemma \ref{lem:invrel}.
\end{proof}

\subsection{Proof of Theorem \ref{thm:mainstrong}}
Suppose that $\Gamma=(V, \des_\cP, \{E_i\})$ is a connected $\cP$-Knuth equivalence graph. We claim that $PT\in \{PT_1, \ldots, PT_k\}$ if $\prs(w)=(PT, QT)$ for some $w\in V$. Indeed, we have $\rw(PT) \psim w\psim w_i \psim \rw(PT_i)$ for any $i \in [1,k]$ by Theorem \ref{thm:main}\ref{mainthm3}. By assumption, it means that $\rw(PT) \in \{\rw(PT_1), \ldots, \rw(PT_k)\}$. Since a $\cP$-tableau is completely determined by its reading word, it means that $PT \in \{PT_1, \ldots, PT_k\}$ as desired.

It is clear that $\sym_n$ is a disjoint union of connected $\cP$-Knuth equivalence graphs. Since $\prs$ is a bijection between $\sym_n$ and $\bigsqcup_{\lambda \vdash n} \ptab_\lambda \times \SYT_\lambda$, it follows that $\prs(V) = \bigsqcup_{i=1}^k\{PT_i\}\times \SYT_{\lambda_i}$. Thus we have (here $|\finv_\cP(V)|$ is a fake $\cP$-inversion number of any $w\in V$ and $\omega$ is Sch\"utzenberger's evacuation)
\begin{align*}
\gamma_V&=t^{|\finv_\cP(V)|}\sum_{w \in V} F_{\des_\cP(w)}
\\&=t^{|\finv_\cP(V)|}\sum_{(PT, QT) \in \prs(V)} F_{\des(\omega(QT))} & (\because \textnormal{Theorem \ref{thm:main}\ref{mainthm2} and injectivity of $\prs$})
\\&=t^{|\finv_\cP(V)|}\sum_{i=1}^k\sum_{T \in \SYT_{\lambda_i}} F_{\des(\omega(T))} & (\because \prs(V) = \bigsqcup_{i=1}^k\{PT_i\}\times \SYT_{\lambda_i})
\\&=t^{|\finv_\cP(V)|}\sum_{i=1}^k\sum_{T \in \SYT_{\lambda_i}} F_{\des(T)} &(\because \omega \textup{ is an involution on each } \SYT_{\lambda_i})
\\&=t^{|\finv_\cP(V)|}\sum_{i=1}^ks_{\lambda_i} &(\because \textnormal{\cite{ges84}})
\end{align*}
as desired.

Finally, the last sentense of the theorem follows directly from Theorem \ref{thm:main}\ref{mainthm3.5}.

\bibliographystyle{amsalpha}
\bibliography{PRS}

\end{document}